\documentclass{book}%
\usepackage{amsfonts}
\usepackage{amsmath}
\usepackage{amssymb}
\usepackage{graphicx}%
\newtheorem{theorem}{Theorem}

\newtheorem{claim}[theorem]{Claim}

\newtheorem{conjecture}[theorem]{Conjecture}
\newtheorem{corollary}[theorem]{Corollary}

\newtheorem{definition}[theorem]{Definition}

\newtheorem{lemma}[theorem]{Lemma}

\newtheorem{problem}[theorem]{Problem}
\newtheorem{proposition}[theorem]{Proposition}

\newenvironment{proof}[1][Proof]{\noindent\textbf{#1.} }{\ \rule{0.5em}{0.5em}}
\begin{document}

\frontmatter
\title{P-points, MAD families and cardinal invariants }
\author{Osvaldo Guzm\'{a}n Gonz\'{a}lez}
\maketitle
\tableofcontents

\chapter{Introduction}

\qquad\ \ \ 

\textit{\textquotedblleft Set theory was born on that December 1873 day when
Cantor established that the real numbers are uncountable.\textquotedblright}

\hspace{0in}\hfill\textit{Akihiro Kanamori}

\qquad\ \ \ \ \ \ \ \ \ \ \ \ \ \ \ \ \ \ \ \qquad\ \ \ \ \ \ \qquad
\ \ \ \ \ \ \qquad\ \ \ \ 

Georg Cantor is considered to be the father of set theory, he dared to do
something that seemed impossible: comparing the size of two different
infinities. Currently, we may use the infinite cardinal numbers $\{\aleph
_{\alpha}\mid\alpha\in$ \textsf{OR}$\}$ to compare the size of infinite sets
(where \textsf{OR }denotes the class of ordinals). For every infinite $X$
there is an ordinal number $\alpha$ such that $X$ and $\aleph_{\alpha}$ are
equipotent. Moreover, if $\alpha<\beta$ then the size of $\aleph_{\alpha}$ is
strictly smaller than the one of $\aleph_{\beta}.$ In this way, $\aleph_{0}$
is the size of the smallest infinite set, this is the cardinality of the
natural numbers, $\omega$. Cantor showed that the set $\mathbb{R}$ of real
numbers is uncountable, so there is $\alpha>0$ such that $\aleph_{\alpha}$ and
$\mathbb{R}$ are equipotent. The assertion that $\mathbb{R}$ has size
$\aleph_{1}$ is known as the Continuum Hypothesis ($\mathsf{CH}$). Using the
constructible universe, G\"{o}del showed that the Continuum Hypothesis is
consistent with the axioms of $\mathsf{ZFC.}$ Several years later, Paul Cohen
developed the technique of forcing and showed that the negation of the
Continuum Hypothesis is also consistent. In fact, the size of the real numbers
(denoted by $\mathfrak{c}$) can be as big as we want it to be.

\qquad\ \ \ \ 

The main topics of this thesis are cardinal invariants, $P$-points and
\textsf{MAD} families. \emph{Cardinal invariants of the continuum }are
cardinal numbers that are bigger than $\aleph_{0}$ and smaller or equal than
$\mathfrak{c}.$ Of course, they are only interesting when they have some
combinatorial or topological definition. Currently, there is a long list of
cardinal invariants that are being studied and compared by set theorists. An
\emph{almost disjoint family }is a family of infinite subsets of
$\omega\emph{\ }$such that the intersection of any two of its elements is
finite. A \textsf{MAD} family is a maximal almost disjoint family. The study
of these families has become very important in set theory and topology. It is
easy to construct \textsf{MAD} families; however, it is very hard to construct
\textsf{MAD} families with interesting combinatorial or topological
properties. Perhaps paradoxically, it is also very hard to construct models of
$\mathsf{ZFC}$ where certain types of \textsf{MAD} families do not exist. An
ultrafilter $\mathcal{U}$ on $\omega$ is called a \emph{P-point }if every
countable $\mathcal{B\subseteq U}$ there is $X\in$ $\mathcal{U}$ such that
\thinspace$X\setminus B$ is finite for every $B\in\mathcal{B}.$ This kind of
ultrafilters has been extensively studied, however there is still a large
number of open questions about them.

\qquad\ \ \qquad\ \ 

In the preliminaries we recall the principal properties of filters,
ultrafilters, ideals, \textsf{MAD} families and cardinal invariants of the
continuum. We present the construction of Shelah of a completely separable
\textsf{MAD} family under $\mathfrak{s\leq a}.$ None of the results in this
chapter are due to the author.\qquad\ \ \ 

\qquad\ \ 

The second chapter is dedicated to a principle of Sierpi\'{n}ski. The
\emph{principle }$\left(  \ast\right)  $\emph{ of Sierpi\'{n}ski }is the
following statement: There is a family of functions $\left\{  \varphi
_{n}:\omega_{1}\longrightarrow\omega_{1}\mid n\in\omega\right\}  $ such that
for every $I\in\left[  \omega_{1}\right]  ^{\omega_{1}}$ there is $n\in\omega$
for which $\varphi_{n}\left[  I\right]  =\omega_{1}.$ This principle was
recently studied by Arnie Miller. He showed that this principle is equivalent
to the following statement: There is a set $X=\left\{  f_{\alpha}\mid
\alpha<\omega_{1}\right\}  \subseteq\omega^{\omega}$ such that for every
$g:\omega\longrightarrow\omega$ there is $\alpha$ such that if $\beta>\alpha$
then $f_{\beta}\cap g$ is infinite (sets with that property are referred to
as$\ \mathcal{IE}$\emph{-Luzin sets}). Miller showed that the principle of
Sierpi\'{n}ski implies that \textsf{non}$\left(  \mathcal{M}\right)
=\omega_{1}.$ He asked if the converse was true, i.e. does \textsf{non}%
$\left(  \mathcal{M}\right)  =\omega_{1}$ imply the principle\emph{ }$\left(
\ast\right)  $\emph{ }of Sierpi\'{n}ski? We answer his question affirmatively.
In other words, we show that \textsf{non}$\left(  \mathcal{M}\right)
=\omega_{1}$ is enough to construct an $\mathcal{IE}$-Luzin set. It is not
hard to see that the $\mathcal{IE}$-Luzin set we constructed is meager. This
is no coincidence, because with the aid of an inaccessible cardinal, we
construct a model where \textsf{non}$\left(  \mathcal{M}\right)  =\omega_{1}$
and every $\mathcal{IE}$-Luzin set is meager.

\qquad\ \ \ 

The third chapter is dedicated to a conjecture of Hru\v{s}\'{a}k. In
\cite{KatetovOrderonBorelIdeals} Michael Hru\v{s}\'{a}k conjectured the
following: Every Borel cardinal invariant is either at most \textsf{non}%
$\left(  \mathcal{M}\right)  $ or at least \textsf{cov}$\left(  \mathcal{M}%
\right)  $ (it is known that the definability is an important requirement,
otherwise $\mathfrak{a}$ would be a counterexample). Although the veracity of
this conjecture is still an open problem, we were able to obtain some partial
results: The conjecture is false for \textquotedblleft Borel invariants of
$\omega_{1}^{\omega}$\textquotedblright\ nevertheless, it is true for a large
class of definable invariants. This is part of a joint work with Michael
Hru\v{s}\'{a}k and Jind\v{r}ich Zapletal.

\qquad\ \ \qquad\ \ 

In the fourth chapter we present a survey on destructibility of ideals and
\textsf{MAD} families. If $\mathbb{P}$ is a forcing notion and $\mathcal{A}$
is a \textsf{MAD }family, we say that $\mathbb{P}$ \emph{destroys
}$\mathcal{A}$ if $\mathcal{A}$ is not maximal after forcing with $\mathbb{P}%
$. It is well known that there is a \textsf{MAD }family that is destroyed by
every forcing adding a new real, but construting indestructible \textsf{MAD
}families is much more difficult and there are still many fundamental open
questions in this topic. We prove several classic theorems, but we also prove
some new results. For example, we show that every almost disjoint family of
size less than $\mathfrak{c}$ can be extended to a Cohen indestructible
\textsf{MAD} family is equivalent to $\mathfrak{b=c}$ (this is part of a joint
work with Michael Hru\v{s}\'{a}k, Ariet Ramos and Carlos Mart\'{\i}nez). A
\textsf{MAD} family $\mathcal{A}$ is \emph{Shelah-Stepr\={a}ns} if for every
$X\subseteq\left[  \omega\right]  ^{<\omega}\setminus\left\{  \emptyset
\right\}  $ either there is $A\in\mathcal{I}\left(  \mathcal{A}\right)  $ such
that $s\cap A\neq\emptyset$ for every $s\in X$ or there is $B\in
\mathcal{I}\left(  \mathcal{A}\right)  $ that contains infinitely many
elements of $X$ (where $\mathcal{I}\left(  \mathcal{A}\right)  $ denotes the
ideal generated by $\mathcal{A}$). This concept was introduced by Raghavan in
\cite{AModelwithnoStronglySeparableAlmostDisjointFamilies}, which is connected
to the notion of \textquotedblleft strongly separable\textquotedblright%
\ introduced by Shelah and Stepr\={a}ns in \cite{MasasintheCalkinAlgebra}. We
prove that Shelah-Stepr\={a}ns \textsf{MAD} families have very strong
indestructibility properties: Shelah-Stepr\={a}ns \textsf{MAD} families are
indestructible for \textquotedblleft many\textquotedblright\ definable
forcings that does not add dominating reals (this statement will be formalized
in the fourth chapter). According to the author's best knowledge, this is the
strongest notion (in terms of indestructibility) that has been considered in
the literature so far. In spite of their strong indestructibility,
Shelah-Stepr\={a}ns \textsf{MAD} families can be destroyed by a ccc forcing
that does not add unsplit or dominating reals. We also consider some strong
combinatorial properties of \textsf{MAD} families and show the relationships
between them (This is part of a joint work with Michael Hru\v{s}\'{a}k, Dilip
Raghavan and Joerg Brendle).

\qquad\ \ \ 

The fifth chapter is one of the most important chapters in the thesis. A
\textsf{MAD} family $\mathcal{A}$ is called $+$\emph{-Ramsey }if every tree
that branches into $\mathcal{I}\left(  \mathcal{A}\right)  $-positive sets has
an $\mathcal{I}\left(  \mathcal{A}\right)  $-positive branch. Michael
Hru\v{s}\'{a}k's first published question is the following: Is there a
$+$-Ramsey \textsf{MAD} family? It was previously known that such families can
consistently exist. However, there was no construction of such families using
only the axioms of \textsf{ZFC. }We solve this problem by constructing such a
family without any extra assumptions. Our proof is divided by cases: in case
$\mathfrak{a<s}$ we show that there is a Miller-indestructible \textsf{MAD}
family and that every Miller indestructible \textsf{MAD} family is $+$-Ramsey.
In case $\mathfrak{s\leq a}$ we construct a $+$-Ramsey \textsf{MAD} family
using Shelah's technique for constructing a completely separable \textsf{MAD}
family. The existence of $+$-Ramsey \textsf{MAD} families has interesting
applications in topological games on Fr\'{e}chet spaces, the reader may
consult \cite{SelectivityofAlmostDisjointFamilies} for more details.

\qquad\ \ \qquad\ \ \ 

In the fourth and fifth chapters, we introduce several notions of \textsf{MAD}
families, in the sixth chapter we prove several implications and non
implications between them. We construct (under $\mathsf{CH}$) several
\textsf{MAD} families with different properties.

\qquad\ \ \ 

In the seventh chapter we build models without $P$-points. We show that there
are no $P$-points after adding Silver reals either iteratively or by the side
by side product. These results have some important consequences: The first one
is that is its possible to get rid of $P$-points using only definable
forcings. This answers a question of Michael Hru\v{s}\'{a}k. We can also use
our results to build models with no $P$-points and with arbitrarily large
continuum, which was also an open question. These results were obtained with
David Chodounsk\'{y}.

\qquad\ \ 

In the last chapter we collect some important open problems concerning
\textsf{MAD} families.

\qquad\ \ \ \ 

The main contributions of this thesis are the following:

\begin{enumerate}
\item There is a $+$-Ramsey \textsf{MAD} family. This answers an old question
of Michael Hru\v{s}\'{a}k (see \cite{ThereisaRamseyMADFamily}).

\item There are no $P$-points in the Silver model, answering a question of
Michael Hru\v{s}\'{a}k (this is joint work with David Chodounsk\'{y}
\cite{NoPpointsintheSilverModel}).

\item The statement \textquotedblleft There are no $P$%
-points\textquotedblright\ is consistent with the continuum being arbitrarily
large, this answers an open question regarding $P$-points (see
\cite{WohofskyThesis}, this is joint work with David Chodounsk\'{y}
\cite{NoPpointsintheSilverModel}).

\item Every Miller indestructible \textsf{MAD} family is $+$-Ramsey. This
improves a result of Hru\v{s}\'{a}k and Garc\'{\i}a Ferreira (see
\cite{ThereisaRamseyMADFamily}).

\item A Borel ideal is Shelah-Stepr\={a}ns if and only if it is Kat\v{e}tov
above \textsf{FIN}$\times$\textsf{FIN}$.$ This entails that
Shelah-Stepr\={a}ns \textsf{MAD} families have very strong indestructibility
properties (This is part of a joint work with Michael Hru\v{s}\'{a}k, Dilip
Raghavan and Joerg Brendle \cite{CombinatoricsofMADFamilies}).

\item Shelah-Stepr\={a}ns \textsf{MAD} families exist under $\Diamond\left(
\mathfrak{b}\right)  .$ In particular, $\Diamond\left(  \mathfrak{b}\right)  $
is strong enough to produce Cohen or random indestructible \textsf{MAD}
families (This answers a question of Hru\v{s}\'{a}k and Garc\'{\i}a Ferreira,
see \cite{CombinatoricsofMADFamilies}).

\item Cohen indestructible \textsf{MAD} families exist generically if and only
if $\mathfrak{b=c}$ (This is part of a joint work with Michael Hru\v{s}\'{a}k,
Ariet Ramos and Carlos Mart\'{\i}nez \cite{GenericExistenceofMADfamilies}).

\item \textsf{non}$\left(  \mathcal{M}\right)  =\omega_{1}$ implies the
$\left(  \ast\right)  $ principle of Sierpi\'{n}ski. This answers a question
of Arnie Miller (see \cite{SierpinskiOsvaldo}).

\item The statement \textquotedblleft\textsf{non}$\left(  \mathcal{M}\right)
=\omega_{1}$\textquotedblright\ and every $\mathcal{IE}$-Luzin set is meager
is consistent (see \cite{SierpinskiOsvaldo}).

\item Partial solutions to the conjecture of Hru\v{s}\'{a}k; mainly there is a
\textquotedblleft Borel cardinal invariant of $\omega_{1}^{\omega}%
$\textquotedblright\ that is not below \textsf{non}$\left(  \mathcal{M}%
\right)  $ nor is above \textsf{cov}$\left(  \mathcal{M}\right)  .$ However,
the conjecture is true for a large class of Borel cardinal invariants (This is
part of a joint work with Michael Hru\v{s}\'{a}k and Jind\v{r}ich Zapletal).
\end{enumerate}

\mainmatter

\chapter{Preliminaries}

\section{Notation}

Our notation is mostly standard. We say that $T\subseteq\kappa^{<\kappa}$ is a
\emph{tree }if it is closed under taking initial segments. If $s\in T$ we
define $suc_{T}\left(  s\right)  =\left\{  \alpha\mid s^{\frown}\alpha\in
T\right\}  $ (where $s^{\frown}\alpha$ is the sequence that has $s$ as an
initial segment and $\alpha$ in the last entry). If $T\subseteq\omega
^{<\omega}$ we say that $f\in\omega^{\omega}$ is a \emph{branch of }$T$ if
$f\upharpoonright n\in T$ for every $n\in\omega.$ The set of all branches of
$T$ is denoted by $\left[  T\right]  $. For every $n\in\omega$ we define
$T_{n}=\left\{  s\in T\mid\left\vert s\right\vert =n\right\}  .$ If
$s\in\omega^{<\omega}$ then the \emph{cone of }$s$ is defined as $\left\langle
s\right\rangle =\left\{  f\in\omega^{\omega}\mid s\subseteq f\right\}  .$ If
$A\subseteq\omega$ we define $A^{0}=\omega\setminus A$ and $A^{1}=A.$ In this
thesis, the expression \textquotedblleft for almost all\textquotedblright%
\ means \textquotedblleft for all but finitely many\textquotedblright. We say
$A\subseteq^{\ast}B$ ($A$ is an almost subset of $B$) if $A\setminus B$ is finite.

\qquad\ \ \ \ 

With respect to forcing, in this thesis, if $p\leq q$ then $p$ is
\textquotedblleft stronger\textquotedblright\ than $q,$ or $p$ carries more
information than $q.$ We denote by $V$ the collection of all sets.

\bigskip

\newpage

\section{Filters and ideals}

\qquad\ Filters and ideals play a major role in set theory. Let $X$ be a non
empty set. Informally, we can think of filters on $X$ as being a collection of
\textquotedblleft big\textquotedblright\ subsets of $X$ while ideals are
collections of \textquotedblleft small\textquotedblright\ subsets of $X.$ The
formal definitions are the following: (for us, all ideals contain all finite sets).

\begin{definition}
Let $X$ be a set.

\begin{enumerate}
\item We say that $\mathcal{F\subseteq}$ $\wp\left(  X\right)  $ is a
\emph{filter on }$X$ if the following conditions hold:

\begin{enumerate}
\item $X\in\mathcal{F}$ and $\emptyset\notin\mathcal{F}.$

\item If $A\in\mathcal{F}$ and $A\subseteq^{\ast}B$ then $B\in\mathcal{F}.$

\item If $A,B\in\mathcal{F}$ then $A\cap B\in\mathcal{F}.$
\end{enumerate}

\item We say that $\mathcal{I\subseteq}$ $\wp\left(  X\right)  $ is an
\emph{ideal on }$X$ if the following conditions hold:

\begin{enumerate}
\item $X\notin\mathcal{I}$ and $\emptyset\in\mathcal{I}.$

\item If $A\in\mathcal{I}$ and $B\subseteq^{\ast}A$ then $B\in\mathcal{I}.$

\item If $A,B\in\mathcal{I}$ then $A\cup B\in\mathcal{I}.$
\end{enumerate}

\item An ideal $\mathcal{I}$ is a $\sigma$-ideal if it is closed under
countable unions.
\end{enumerate}
\end{definition}

\qquad\ \ \ \ \ \ 

In this thesis we will be mainly interested in the cases when $X$ is a
countable set or a Polish space. Given a family $\mathcal{B}$ of subsets of
$X,$ we define $\mathcal{B}^{\ast}=\left\{  X\smallsetminus B\mid
B\in\mathcal{B}\right\}  .$ It is easy to see that if $\mathcal{F}$ is a
filter then $\mathcal{F}^{\ast}$ is an ideal (called the \emph{dual ideal of
}$\mathcal{F}$) and if $\mathcal{I}$ an ideal then $\mathcal{I}^{\ast}$ is a
filter (called the dual filter of $\mathcal{I}$). If $\mathcal{I}$ is an ideal
on $X$, we let $\mathcal{I}^{+}=\wp\left(  X\right)  \smallsetminus
\mathcal{I}$ be the family of $\mathcal{I}$\emph{-positive sets. }If
$\mathcal{F}$ is a filter, we define $\mathcal{F}^{+}$ $\mathcal{=}$ $\left(
\mathcal{F}^{\ast}\right)  ^{+};$ it is easy to see that $\mathcal{F}^{+}$ is
the family of all sets that have non-empty (infinite) intersection with every
element of $\mathcal{F}.$ If $A\in\mathcal{I}^{+}$ then \emph{the restriction
of }$\mathcal{I}$ to $A$, defined as $\mathcal{I\upharpoonright}A=\wp\left(
A\right)  \cap\mathcal{I}$, is an ideal on $A.$ The \emph{ortoghonal of
}$\mathcal{I}$ (denoted by\emph{\ }$\mathcal{I}^{\perp}$) is the set of all
$X\subseteq\omega$ such that $X\cap A$ is finite for every $A\in\mathcal{I}$
(this definition also applies for arbitrary subfamilies of $\wp\left(
\omega\right)  ,$ not just ideals).

\begin{definition}
Let $\mathcal{I}$ be an ideal on $\omega$ (or any countable set).

\begin{enumerate}
\item $\mathcal{I}$ is \emph{tall }if for every $X\in\left[  \omega\right]
^{\omega}$ there is $Y\in\mathcal{I}$ such that $Y\cap X$ is infinite (this
definition also applies to arbitrary subfamilies of $\wp\left(  \omega\right)
,$ not just ideals).

\item $\mathcal{I}$ is $\omega$\emph{-hitting }if for every $\left\{
X_{n}\mid n\in\omega\right\}  \subseteq\left[  \omega\right]  ^{\omega}$ there
is $Y\in\mathcal{I}$ such that $Y\cap X_{n}$ is infinite for every $n\in
\omega.$

\item $\mathcal{I}$ is a \emph{P}$^{+}$\emph{-ideal }if for every $\subseteq
$-decreasing family $\left\{  X_{n}\mid n\in\omega\right\}  \subseteq
\mathcal{I}^{+}$ there is $Y\in\mathcal{I}^{+}$ such that $Y\subseteq^{\ast
}X_{n}$ for every $n\in\omega.$

\item $\mathcal{I}$ is a \emph{P-ideal }if for every family $\left\{
X_{n}\mid n\in\omega\right\}  \subseteq\mathcal{I}$ there is $Y\in\mathcal{I}$
such that $X_{n}\subseteq^{\ast}Y$ for every $n\in\omega.$

\item $\mathcal{I}$ is a \emph{Q}$^{+}$\emph{-ideal }if for every
$X\in\mathcal{I}^{+}$ and every partition $\left\{  P_{n}\mid n\in
\omega\right\}  $ of $X$ into finite sets, there is $A\in\mathcal{I}^{+}%
\cap\wp\left(  X\right)  $ such that $\left\vert A\cap P_{n}\right\vert \leq1$
for every $n\in\omega.$

\item $\mathcal{I}$ is \emph{selective }if for every $\subseteq$-decreasing
family $\left\{  Y_{n}\mid n\in\omega\right\}  \subseteq\mathcal{I}^{+}$ there
is $X=\left\{  x_{n}\mid n\in\omega\right\}  $ such that the following holds:

\begin{enumerate}
\item $X\in\mathcal{I}^{+}.$

\item $X\subseteq Y_{0}.$

\item $X\setminus\left(  x_{n}+1\right)  \subseteq Y_{x_{n}}.$
\end{enumerate}
\end{enumerate}
\end{definition}

\qquad\ \ \ 

The previous definitions extend to filters as well. We will say that a filter
$\mathcal{F}$ is $P^{+}$ if $\mathcal{F}^{\ast}$ is $P^{+}$ and similarly for
the other definitions.\qquad\qquad\ \ \ \ 

\begin{lemma}
If $\mathcal{I}$ is selective then $\mathcal{I}$ is both $P^{+}$ and $Q^{+}.$
\end{lemma}

\begin{proof}
If $\mathcal{I}$ is selective it is clearly $P^{+}$ so we only need to prove
that it is also $Q^{+}.$ Let $X\in\mathcal{I}^{+}$ and $\left\{  P_{n}\mid
n\in\omega\right\}  $ a partition of $X$ into finite sets. For every
$n\in\omega$ define $Y_{n}=X\setminus max(%
{\textstyle\bigcup\limits_{i\leq n}}
P_{i})$. Then $\left\{  Y_{n}\mid n\in\omega\right\}  \subseteq\mathcal{I}%
^{+}$ is a decreasing sequence and if $Z$ witnesses the selectiveness of
$\mathcal{I}$ then $Z$ is the set we were looking for.
\end{proof}

\qquad\qquad\qquad\ \qquad\ \ 

A filter $\mathcal{U}$ is an \emph{ultrafilter }if it is a maximal filter.
Ultrafilters are of fundamental importance in practically every branch of set
theory. It is easy to construct an ultrafilter using the Axiom of Choice.

\begin{definition}
Let $\mathcal{U}$ be an ultrafilter in $\omega.$

\begin{enumerate}
\item $\mathcal{U}$ is a $P$\emph{-point }if for every decreasing $\left\{
X_{n}\mid n\in\omega\right\}  \subseteq\mathcal{U}$ there is $X\in\mathcal{U}$
such that $X\subseteq^{\ast}X_{n}$ for every $n\in\omega.$

\item $\mathcal{U}$ is a $Q$\emph{-point }if for every partition $\left\{
P_{n}\mid n\in\omega\right\}  $ of $\omega$ into finite sets, there is
$X\in\mathcal{U}$ such that $\left\vert X\cap P_{n}\right\vert \leq1$ for
every $n\in\omega.$

\item $\mathcal{U}$ is a \emph{Ramsey ultrafilter }if for every partition
$\left\{  P_{n}\mid n\in\omega\right\}  $ of $\omega,$ either there is
$n\in\omega$ such that $P_{n}\in\mathcal{U}$ or there is $X\in\mathcal{U}$
such that $\left\vert X\cap P_{n}\right\vert \leq1$ for every $n\in\omega.$
\end{enumerate}
\end{definition}

\qquad\qquad\ \ 

We now have the following:

\begin{proposition}
[Mathias \cite{HappyFamilies}]Let $\mathcal{U}$ be an ultrafilter. The
following are equivalent:

\begin{enumerate}
\item $\mathcal{U}$ is Ramsey.

\item $\mathcal{U}$ is a $P$-point and a $Q$-point.

\item $\mathcal{U}$ is selective.

\item For every coloring $c:\left[  \omega\right]  ^{2}\longrightarrow2$ there
is $X\in\mathcal{U}$ that is $c$-monochromatic.
\end{enumerate}
\end{proposition}

\begin{proof}
Fix $\mathcal{U}$ an ultrafilter. Clearly any Ramsey ultrafilter is both a
$P$-point and a $Q$-point. We will first prove that 2 implies 3. Let $\left\{
Y_{n}\mid n\in\omega\right\}  \subseteq\mathcal{U}$ be a decreasing sequence.
We may assume $n\notin Y_{n}$ so $%
{\textstyle\bigcap}
Y_{n}=\emptyset.$ We now define $P_{0}=\omega\setminus Y_{0}$ and
$P_{n+1}=Y_{n}\setminus Y_{n+1}.$ Clearly $\mathcal{P=}\left\{  P_{n}\mid
n\in\omega\right\}  $ is a partition of $\omega$ and $\mathcal{P}%
\cap\mathcal{U}=\emptyset.$ Since $\mathcal{U}$ is a $P$-point, there is
$X\in\mathcal{U}$ such that $X\cap P_{n}$ is finite for every $n\in\omega.$ We
can then find an increasing function $g:\omega\longrightarrow\omega$ such that
$X\setminus g\left(  n\right)  \subseteq Y_{n}$ for every $n\in\omega.$ Now,
we define an interval partition $\mathcal{R}=\left\{  R_{n}\mid n\in
\omega\right\}  $ such that if $i\in R_{n}$ then $g\left(  i\right)
<max\left(  R_{n+1}\right)  .$ Since $\mathcal{U}$ is a $Q$-point, there is
$Z\in\mathcal{U}$ such that $\left\vert Z\cap R_{n}\right\vert \leq1$ for each
$n\in\omega,$ we may even assume $Z\subseteq X\cap Y_{0}.$ Let $Z_{0}=%
{\textstyle\bigcup\limits_{n\in\omega}}
\left(  Z\cap R_{2n}\right)  $ and $Z_{1}=%
{\textstyle\bigcup\limits_{n\in\omega}}
\left(  Z\cap R_{2n+1}\right)  $, since $\mathcal{U}$ is an ultrafilter, then
there is $i<2$ such that $Z_{i}\in\mathcal{U},$ this is the set we were
looking for.

\qquad\ \ 

We will now show that 3 implies 4. Let $\mathcal{U}$ be a selective
ultrafilter and $c:\left[  \omega\right]  ^{2}\longrightarrow2.$ For every
$n\in\omega$ and $i<2$ let $H_{i}\left(  n\right)  =\left\{  m>n\mid c\left(
\left\{  n,m\right\}  \right)  =i\right\}  .$ Since $\mathcal{U}$ is an
ultrafilter, for every $n\in\omega$ there is $i_{n}$ such that $H_{i_{n}%
}\left(  n\right)  \in\mathcal{U}.$ Let $X_{n}=%
{\textstyle\bigcap\limits_{m\leq n}}
H_{i_{m}}\left(  m\right)  .$ Since $\mathcal{U}$ is selective, there is
$Y=\left\{  y_{n}\mid n\in\omega\right\}  \in\mathcal{U}$ such that
$Y\subseteq X_{0}$ and $Y\setminus\left(  y_{n}+1\right)  \subseteq X_{y_{n}}$
for every $n\in\omega.$ Note that if $n<m$ then $c\left(  y_{n},y_{m}\right)
=i_{y_{n}}.$ Since $\mathcal{U}$ is an ultrafilter, we can find $Y_{1}%
\subseteq Y$ such that $Y_{1}\in\mathcal{U}$ and $i_{n}=i_{m}$ for every
$n,m\in Y_{1}.$ Clearly $Y_{1}$ is monochromatic.

\qquad\ \ \ \ \ 

We will now prove that 4 implies 1. Let $\mathcal{P=}\left\{  P_{n}\mid
n\in\omega\right\}  $ be a partition of $\omega.$ We now define the coloring
$c:\left[  \omega\right]  ^{2}\longrightarrow2$ where $c\left(  \left\{
n,m\right\}  \right)  =1$ if and only if $n$ and $m$ belong to the same
element of the partition. Clearly any $0$-monochromatic set is a partial
selector and any $1$-monochromatic set is contained in a single element of
$\mathcal{P}.$
\end{proof}

\qquad\ \ \ \ 

If $\mathcal{I}$ is $\sigma$-ideal on a Polish space, we denote by
$\mathbb{P}_{\mathcal{I}}$ the forcing of Borel sets modulo $\mathcal{I}$
(i.e. if $B$ and $C$ are Borel sets then $B\leq C$ if and only if $B\setminus
C\in\mathcal{I}$). The book \cite{ForcingIdealized} is a very interesting
reference about this type of forcings. By \textsf{ctble }we denote the
$\sigma$-ideal of all countable subsets of $2^{\omega}$, by $\mathcal{M}$ we
denote the $\sigma$-ideal of meager sets, $\mathcal{N}$ denotes the $\sigma
$-ideal of Lebesgue measure zero sets and $\mathcal{K}_{\sigma}$ denotes the
$\sigma$-ideal on $\omega^{\omega}$ generated by compact sets. Given
$F:\omega^{<\omega}\longrightarrow\omega$ we define $C_{\exists}\left(
F\right)  =\left\{  g\in\omega^{\omega}\mid\exists^{\infty}n\left(  g\left(
n\right)  \leq F\left(  g\upharpoonright n\right)  \right)  \right\}  .$ The
Laver ideal $\mathcal{L}$ is the $\sigma$-ideal generated by $\left\{
C_{\exists}\left(  F\right)  \mid F:\omega^{<\omega}\longrightarrow
\omega\right\}  .$ It is well known (see \cite{ForcingIdealized}) that
$\mathbb{P}_{\text{\textsf{ctble}}}$ is equivalent to the Sacks forcing,
$\mathbb{P}_{\mathcal{M}}$ is equivalent to the Cohen forcing, $\mathbb{P}%
_{\mathcal{N}}$ is equivalent to random forcing, $\mathbb{P}_{\mathcal{K}%
_{\sigma}}$ is the Miller forcing and $\mathbb{P}_{\mathcal{L}}$ is equivalent
to the Laver forcing.

\qquad\ \ \ 

Let $\mathcal{I}$ be an ideal on a Polish space $X$ and assume $W$ is a model
of $\mathsf{ZFC}$ extending the $\mathsf{ZFC}$ model $V.$ Given $r$ such that
$W\models r\in X$. We say that $r$ is $\mathcal{I}$\emph{-quasigeneric over
}$V$\emph{ }if $W\models r\notin B^{W}$ for every Borel set $B\in
V\cap\mathcal{I}$ (by $B^{W}$ we denote the interpretation of $B$ in $W$). In
this way, the \textsf{ctble}-quasigeneric reals are the new reals, the
$\mathcal{M}$-quasigeneric reals are the Cohen reals, the $\mathcal{N}%
$-quasigeneric reals are the random reals and the $\mathcal{K}_{\sigma}%
$-quasigeneric reals are the unbounded reals.

\qquad\ \ \ 

If $\mathcal{I}$ is an ideal on $\omega$ (or on any countable set) we define
the \emph{Mathias forcing }$\mathbb{M}\left(  \mathcal{I}\right)  $ \emph{with
respect to }$\mathcal{I}$ as the set of all pairs $\left(  s,A\right)  $ where
$s\in\left[  \omega\right]  ^{<\omega}$ and $A\in\mathcal{I}$. If $\left(
s,A\right)  ,\left(  t,B\right)  \in\mathbb{M}\left(  \mathcal{I}\right)  $
then $\left(  s,A\right)  \leq\left(  t,B\right)  $ if the following
conditions hold:

\begin{enumerate}
\item $t$ is an initial segment of $s.$

\item $B\subseteq A.$

\item $\left(  s\setminus t\right)  \cap B=\emptyset.$
\end{enumerate}

\qquad\ \ \ \ 

If $\mathcal{F}$ is a filter, then by $\mathbb{M}\left(  \mathcal{F}\right)  $
we denote $\mathbb{M}\left(  \mathcal{F}^{\ast}\right)  .$

\newpage

\section{Borel filters and ideals}

Given a filter (ideal) on $\omega,$ we may view it as a subspace of the Cantor
space and then study its topological properties. In this way, we say a filter
(ideal) is Borel ($F_{\sigma},G_{\sigma}...$) if it is is Borel ($F_{\sigma
},G_{\sigma}...$) as a subspace of $\wp\left(  \omega\right)  .$ Note that
$\mathcal{F}$ and $\mathcal{F}^{\ast}$ are homeomorphic since taking
complement is a homeomorphism of the Cantor space.\qquad\ \ \ \ \ \ \ 

\begin{lemma}
\qquad\ \ \ \ \qquad\ 

\begin{enumerate}
\item There are no closed ideals.

\item There are no $G_{\delta}$ ideals.

\item If an ideal has the Baire property then it is meager.

\item There are no meager ultrafilters (i.e. no ultrafilter has the Baire property).
\end{enumerate}
\end{lemma}

\begin{proof}
The first point follows since $\left[  \omega\right]  ^{<\omega}$ is a dense
subset of $\wp\left(  \omega\right)  .$ If $\mathcal{I}$ was a $G_{\delta}$
set, then by the Baire category theorem, it would be comeager, but then
$\mathcal{I}^{\ast}$ would be a comeager set disjoint with $\mathcal{I},$
which is a contradiction. If there was a non-meager ideal with the Baire
property then it would be comeager in an open set, and we would get a
contradiction as before.
\end{proof}

\ \ \ \ \ \ \ \ \ \ \ \ \ \ \ \ \ \ \ \ \ \ \ \ \ \ \ \ \ \ \ \ \ \ \ \ \ \ \ \ \ \ \ \ \qquad
\ \qquad\ \ 

In contrast to 1 and 2 of the previous lemma, there are $F_{\sigma}$ ideals
($\left[  \omega\right]  ^{<\omega}$ being the easiest example). This kind of
ideals has many interesting combinatorial and forcing properties.

\qquad\ \ \ \ \ \ 

If $X\in\left[  \omega\right]  ^{\omega}$ we define $e_{X}:\omega
\longrightarrow\omega$ as the unique increasing function whose image is $X.$
If $\mathcal{F}$ is a filter we denote by $\widetilde{\mathcal{F}}=\left\{
e_{X}\mid X\in\mathcal{F}\right\}  .$ We now define two orders in
$\omega^{\omega}.$ Let $f,g\in\omega^{\omega}$ then $f\leq g$ if and only if
$f\left(  n\right)  \leq g\left(  n\right)  $ for every $n\in\omega$ and
$f\leq^{\ast}g$ if and only if $f\left(  n\right)  \leq g\left(  n\right)  $
for almost all $n\in\omega.$ Recall that in this thesis, \textquotedblleft for
almost all\textquotedblright\ means for all except finitely many. In the same
way, we say that $f=^{\ast}g$ if $f\left(  n\right)  =g\left(  n\right)  $
holds for almost all $n\in\omega.$ We say a family $\mathcal{B}\subseteq
\omega^{\omega}$ is \emph{unbounded }if $\mathcal{B}$ is unbounded with
respect to $\leq^{\ast}$.We say that $\mathcal{P}=\left\{  P_{n}\mid
n\in\omega\right\}  $ is an \emph{interval partition }if it is a partition of
$\omega$ into consecutive intervals, by $PART$ we denote the set of all
interval partitions. The following notions are very useful for studying meager
sets in $2^{\omega}$ (or in $\wp\left(  \omega\right)  $).

\begin{definition}
\qquad\ \ \ 

\begin{enumerate}
\item A \emph{chopped real }is a pair $\left(  x,\mathcal{P}\right)  $ where
$x\in2^{\omega}$ and $\mathcal{P}$ is an interval partition. We denote by
$\mathbb{CR}$ the set of all chopped reals.

\item If $\left(  x,\mathcal{P}\right)  $ is a chopped real and $y\in
2^{\omega}$ then we say that $y$ \emph{matches }$\left(  x,\mathcal{P}\right)
$ if there are infinitely many $P\in\mathcal{P}$ such that $y\upharpoonright
P=x\upharpoonright P.$

\item We define $Match\left(  x,\mathcal{P}\right)  $ as the set of all $y$
that matches $\left(  x,\mathcal{P}\right)  $ and the set $\lnot Match\left(
x,\mathcal{P}\right)  $ is defined as the collection of all $y$ that does not
match $\left(  x,\mathcal{P}\right)  $.
\end{enumerate}
\end{definition}

\qquad\ \ \ \ \qquad\ \ 

We now have the following:

\begin{proposition}
[\cite{HandbookBlass}]The set $\left\{  \lnot Match\left(  x,\mathcal{P}%
\right)  \mid\left(  x,\mathcal{P}\right)  \in\mathbb{CR}\right\}  $ is a
cofinal set for the meager sets in $2^{\omega}.$
\end{proposition}

\begin{proof}
Let $\left\langle T_{n}\mid n\in\omega\right\rangle $ be an increasing
sequence of well pruned subtrees of $2^{<\omega}$ such that each $\left[
T_{n}\right]  $ is a nowhere dense set. We recursively define $\mathcal{P}%
=\left\{  P_{n}\mid n\in\omega\right\}  $ and $\left\{  s_{n}\mid n\in
\omega\right\}  $ such that for all $n\in\omega$ the following holds:

\begin{enumerate}
\item $\mathcal{P}$ is an interval partition.

\item $s_{n}:P_{n}\longrightarrow2.$

\item $s_{0}\notin T_{0}.$

\item If $t\in2^{max\left(  P_{n}\right)  }$ then $t\cup s_{n+1}\notin
T_{n+1}.$
\end{enumerate}

\qquad\ \ 

This is easy to do since each $\left[  T_{n}\right]  $ is nowhere dense and
$2^{k}$ is finite for every $k\in\omega.$ Let $x=%
{\textstyle\bigcup}
s_{n}.$ It is easy to see that $%
{\textstyle\bigcup\limits_{n\in\omega}}
\left[  T_{n}\right]  \subseteq\lnot Match\left(  x,\mathcal{P}\right)  $
(since the sequence $\left\langle T_{n}\mid n\in\omega\right\rangle $ is
increasing, any $y\in2^{\omega}$ matching $\left(  x,\mathcal{P}\right)  $
will not be in $%
{\textstyle\bigcup\limits_{n\in\omega}}
\left[  T_{n}\right]  $).
\end{proof}

\qquad\ \ 

We can now prove the following important result:

\begin{proposition}
[Talagrand, Jalali-Naini see \cite{Barty}]Let $\mathcal{F}$ be a filter on
$\omega.$ The following are equivalent:

\begin{enumerate}
\item $\mathcal{F}$ is a non-meager filter.

\item $\widetilde{\mathcal{F}}$ is an unbounded family.

\item For every increasing function $f:\omega\longrightarrow\omega$ there is
$X\in\mathcal{F}$ such that $X\cap\left[  n,f\left(  n\right)  \right]
=\emptyset$ for infinitely many $n\in\omega.$

\item If $\mathcal{P}=\left\{  P_{n}\mid n\in\omega\right\}  $ is an interval
partition then there is $X\in\mathcal{F}$ such that $X\cap P_{n}=\emptyset$
for infinitely many $n\in\omega.$
\end{enumerate}
\end{proposition}

\begin{proof}
We first prove that 1 implies 2 by contrapositive. Assume that there is
$g:\omega\longrightarrow\omega$ that is an\ upper bound for $\widetilde
{\mathcal{F}}.$ For every $n\in\omega$ define $A_{n}=\left\{  X\mid e_{X}%
<_{n}g\right\}  ,$ which is a closed nowhere dense set. Then $\mathcal{F}%
\subseteq%
{\textstyle\bigcup\limits_{n\in\omega}}
A_{n}$ so $\mathcal{F}$ is a meager set.

\qquad\ \ 

We will now show that 2 implies 3. Let $f:\omega\longrightarrow\omega$ be an
increasing function, we may assume $f\left(  n\right)  >n$ for every
$n\in\omega.$ We now define the function $h:\omega\longrightarrow\omega$ given
by $h\left(  n\right)  =f^{n}\left(  n\right)  $ (where $f^{n}$ is the
$n$-iteration of $f$). Note that $n<f\left(  n\right)  <ff\left(  n\right)
<...<f^{n-1}\left(  n\right)  <f^{n}\left(  n\right)  =h\left(  n\right)  .$
Since $\widetilde{\mathcal{F}}$ is unbounded, there is $X\in\mathcal{F}$ such
that $e_{X}$ is not dominated by $h.$ Let $n$ such that $e_{X}\left(
n\right)  >h\left(  n\right)  .$ This means that the $n$-th element of $X$ is
bigger than $h\left(  n\right)  ,$ so $X$ must have empty intersection with
one of the following intervals: $\left[  n,f\left(  n\right)  \right]
,\left[  f\left(  n\right)  ,ff\left(  n\right)  \right]  ,...,\left[
f^{n-1}\left(  n\right)  ,f^{n}\left(  n\right)  \right]  .$

\qquad\ \ \ \ 

We will prove that 3 implies 4. Let $\mathcal{P}=\left\{  P_{n}\mid n\in
\omega\right\}  $ be an interval partition. We now define the function
$f:\omega\longrightarrow\omega$ given by $f\left(  n\right)  $ as the smallest
$m$ such that there is $k$ such that $P_{k}\subseteq\left[  n,m\right]  .$ By
3 there is $X\in\mathcal{F}$ such that $X\cap\left[  n,f\left(  n\right)
\right]  =\emptyset$ for infinitely many $n\in\omega.$ It is then easy to see
that $X$ has empty intersection with infinitely many of the intervals.

\qquad\ \ 

Finally, we will prove that 4 implies 1. Let $\left(  y,\mathcal{P}\right)  $
be a chopped real, we will show that $\mathcal{F}$ is not contained in $\lnot
Match\left(  y,\mathcal{P}\right)  .$ Using 4, we find $X\in\mathcal{F}$ such
that $X$ has empty intersection with infinitely many intervals. We now define
$A$ as follows: If $X\cap P_{n}\neq\emptyset$ then $A\cap P_{n}=X\cap P_{n}$
and if $X\cap P_{n}=\emptyset$ then $A\cap P_{n}=y\cap P_{n}.$ Since
$X\subseteq A$, it follows that $A\in\mathcal{F}$ and clearly $A\in
Match\left(  y,\mathcal{P}\right)  .$
\end{proof}

\newpage

\qquad\ \ \ 

For the convenience of the reader, we include the 4 most useful versions of
Talagrand's theorem.

\begin{proposition}
[Talagrand, Jalali-Naini theorem for filters]Let $\mathcal{F}$ be a filter on
$\omega.$

\begin{enumerate}
\item The following are equivalent:

\begin{enumerate}
\item $\mathcal{F}$ is a non-meager filter.

\item $\widetilde{\mathcal{F}}$ is an unbounded family.

\item If $\mathcal{P}=\left\{  P_{n}\mid n\in\omega\right\}  $ is an interval
partition then there is $X\in\mathcal{F}$ such that $X\cap P_{n}=\emptyset$
for infinitely many $n\in\omega.$
\end{enumerate}

\item The following are equivalent:

\begin{enumerate}
\item $\mathcal{F}$ is a meager filter.

\item $\widetilde{\mathcal{F}}$ is a bounded family.

\item There is an interval partition $\mathcal{P}=\left\{  P_{n}\mid
n\in\omega\right\}  $ such that if $X\in\mathcal{F}$ then $X\cap P_{n}%
\neq\emptyset$ for for almost all $n\in\omega.$
\end{enumerate}
\end{enumerate}
\end{proposition}

\qquad\ \ \ \ 

\bigskip

\begin{proposition}
[Talagrand, Jalali-Naini theorem for ideals]Let $\mathcal{I}$ be an ideal on
$\omega.$

\begin{enumerate}
\item The following are equivalent:

\begin{enumerate}
\item $\mathcal{I}$ is a non-meager ideal.

\item $\widetilde{\mathcal{I}^{\ast}}$ is an unbounded family.

\item If $\mathcal{P}=\left\{  P_{n}\mid n\in\omega\right\}  $ is an interval
partition then there is $X\in\mathcal{I}$ such that $P_{n}$ $\subseteq X$ for
infinitely many $n\in\omega.$
\end{enumerate}

\item The following are equivalent:

\begin{enumerate}
\item $\mathcal{I}$ is a meager ideal.

\item $\widetilde{\mathcal{I}^{\ast}}$ is a bounded family.

\item There is an interval partition $\mathcal{P}=\left\{  P_{n}\mid
n\in\omega\right\}  $ such that if $X\in\mathcal{I}$ then $X$ contains only
finitely many intervals of $\mathcal{P}.$
\end{enumerate}
\end{enumerate}
\end{proposition}

\qquad\qquad\ \ \qquad\ \ \ 

The $F_{\sigma}$ ideals have many interesting combinatorial properties. A very
useful way to construct such ideals is with the aid of lower semicontinuous submeasures:

\begin{definition}
We say $\varphi:\wp\left(  \omega\right)  \longrightarrow\omega\cup\left\{
\omega\right\}  $ is a \emph{lower semicontinuous submeasure} if the following hold:

\begin{enumerate}
\item $\varphi\left(  \omega\right)  =\omega.$

\item $\varphi\left(  A\right)  =0$ if and only if $A=\emptyset.$

\item $\varphi\left(  A\right)  \leq\varphi\left(  B\right)  $ whenever
$A\subseteq B.$

\item $\varphi\left(  A\cup B\right)  \leq\varphi\left(  A\right)
+\varphi\left(  B\right)  $ for every $A,B\subseteq X.$

\item (lower semicontinuity) if $A\subseteq\omega$ then $\varphi\left(
A\right)  =sup\left\{  \varphi\left(  A\cap n\right)  \mid n\in\omega\right\}
.$
\end{enumerate}
\end{definition}

\qquad\ \ \qquad\ \ 

Given a lower semicontinuous submeasure $\varphi$ we define \textsf{Fin}%
$\left(  \varphi\right)  $ as the family of those subsets of $\omega$ with
finite submeasure. We then have the following:

\begin{lemma}
If $\varphi:\wp\left(  \omega\right)  \longrightarrow\omega\cup\left\{
\omega\right\}  $ is a lower semicontinuous submeasure then \textsf{Fin}%
$\left(  \varphi\right)  $ is an $F_{\sigma}$-ideal.
\end{lemma}

\begin{proof}
Given $n\in\omega$ define $\mathcal{C}_{n}=\left\{  A\mid\varphi\left(
A\right)  \leq n\right\}  $ which is a closed set by lower semicontinuity.
Clearly \textsf{Fin}$\left(  \varphi\right)  =%
{\textstyle\bigcup\limits_{n\in\omega}}
\mathcal{C}_{n}.$
\end{proof}

\qquad\ \ 

It is a very interesting result of Mazur that the converse of the previous
lemma is also true: if $\mathcal{I}$ is an $F_{\sigma}$-ideal then there is a
lower semicontinuous submeasure $\varphi$ such that $\mathcal{I}=$
\textsf{Fin}$\left(  \varphi\right)  .$ Such $\varphi$ is closely related with
the representation of $\mathcal{I}$ as an increasing union of compact sets.

\begin{proposition}
[Mazur \cite{Mazur}]$\mathcal{I}$ is an $F_{\sigma}$-ideal if and only if
there is a lower semicontinuous submeasure such that $\mathcal{I}=$
\textsf{Fin}$\left(  \varphi\right)  .$
\end{proposition}

\begin{proof}
Let $\mathcal{I}=%
{\textstyle\bigcup\limits_{n\in\omega}}
\mathcal{C}_{n}$ where $\left\langle \mathcal{C}_{n}\right\rangle _{n\in
\omega}$ is an increasing union of compact sets such that each $\mathcal{C}%
_{n}$ is closed under subsets and if $X_{1},X_{2}\in\mathcal{C}_{n}$ then
$X_{1}\cup X_{2}\in\mathcal{C}_{n+1}$ (note every $F_{\sigma}$ ideal can be
represented in this way). \ Define $\varphi:\wp\left(  \omega\right)
\longrightarrow\omega\cup\left\{  \infty\right\}  $ as $\varphi\left(
s\right)  =min\left\{  n+1\mid s\in\mathcal{C}_{n}\right\}  $ if $s$ is a
finite set and $\varphi\left(  A\right)  =sup\left\{  \varphi\left(  A\cap
n\right)  \mid n\in\omega\right\}  $ in case $A$ is an infinite set. It is
easy to see that $\varphi$ is a lower semicontinuous submeasure and
$\mathcal{I}=$ \textsf{Fin}$\left(  \varphi\right)  .$
\end{proof}

\qquad\ \ \ 

We will now define some Borel ideals that will be used in this thesis. For
every $n\in\omega$ we define $C_{n}=\left\{  \left(  n,m\right)  \mid
m\in\omega\right\}  $ and if $f:\omega\longrightarrow\omega$ let $D\left(
f\right)  =\left\{  \left(  n,m\right)  \mid m\leq n\right\}  .$

\begin{definition}
We define the following ideals:

\begin{enumerate}
\item \textsf{FIN }is the ideal of all finite subsets of $\omega.$

\item $\mathcal{ED}$ is the ideal on $\omega\times\omega$ generated by
$\left\{  C_{n}\mid n\in\omega\right\}  $ and (the graphs of) functions from
$\omega$ to $\omega.$

\item \textsf{FIN}$\times$\textsf{FIN} is the ideal on $\omega\times\omega$
generated by $\left\{  C_{n}\mid n\in\omega\right\}  \cup\left\{  D\left(
f\right)  \mid f\in\omega^{\omega}\right\}  .$

\item $\emptyset\times$\textsf{FIN} is the ideal on $\omega\times\omega$
generated by $\left\{  D\left(  f\right)  \mid f\in\omega^{\omega}\right\}  .$

\item \textsf{conv} is the ideal on $\left[  0,1\right]  \cap\mathbb{Q}$
generated by all sequences converging to a real number.

\item \textsf{nwd} is the ideal on $\mathbb{Q}$ generated by all nowhere dense sets.

\item The \emph{summable ideal} is defined as $\mathcal{J}_{1/n}%
=\{A\subseteq\omega\mid%
{\textstyle\sum\limits_{n\in A}}
\frac{1}{n+1}<\omega\}.$
\end{enumerate}
\end{definition}

\qquad\ \ 

With the exception of \textsf{FIN }and $\emptyset\times$\textsf{FIN}$,$ all of
the previous ideals are tall. The ideals $\mathcal{ED}$ and $\mathcal{J}%
_{1/n}$ are $F_{\sigma}$ while the others are not.

\begin{definition}
If $a\subseteq\omega^{<\omega}$ we define $\pi\left(  a\right)  =\left\{
f\in\omega^{\omega}\mid\exists^{\infty}n\left(  f\upharpoonright n\in
a\right)  \right\}  .$ Let $\mathcal{I}$ be a $\sigma$-ideal on $\omega
^{\omega}$ (or $2^{\omega}$). We define $tr\left(  \mathcal{I}\right)  $
\emph{the trace ideal of }$\mathcal{I}$ (which will be an ideal on
$\omega^{<\omega}$ or $2^{<\omega}$) where $a\in tr\left(  \mathcal{I}\right)
$ if and only if $\pi\left(  a\right)  \in\mathcal{I}.$
\end{definition}

\qquad\ \ \ 

Note that if $a\subseteq\omega^{<\omega}$ then $\pi\left(  a\right)  $ is a
$G_{\delta}$ set (furthermore, every $G_{\delta}$ set is of this form). While
both $tr\left(  \mathcal{M}\right)  $ and $tr\left(  \mathcal{N}\right)  $ are
Borel, in general, the trace ideals are not Borel (see
\cite{ForcingwithQuotients} for more information).

\qquad\ \ 

\newpage

\section{\textsf{MAD} Families}

A family $\mathcal{A}\subseteq\left[  \omega\right]  ^{\omega}$ is
\emph{almost disjoint (AD) }if the intersection of any two different elements
of $\mathcal{A}$ is finite, a \textsf{MAD }\emph{family}\textsf{\emph{ }}is
a\textsf{ }maximal almost disjoint family. Almost disjoint families and
\textsf{MAD} families have become very important in set theory, topology and
functional analysis (see \cite{AlmostDisjointFamiliesandTopology}). It is very
easy to prove that the Axiom of Choice implies the existence of \textsf{MAD}
families. However, constructing \textsf{MAD} families with special
combinatorial or topological properties is a very difficult task without the
an additional hypothesis beyond \textsf{ZFC}. Constructing models of set
theory where there are no certain kinds of \textsf{MAD} families is also very
difficult. We would like to mention some important examples regarding the
existence or non-existence of special \textsf{MAD} families:

\begin{enumerate}
\item (Simon \cite{ACompactFrechetSpacewhoseSquareisnotFrechet}) There is a
\textsf{MAD} family which can be partitioned into two nowhere \textsf{MAD}
families. \footnote{$\mathcal{A}$ is \emph{nowhere \textsf{MAD} }if for every
$X\in\mathcal{I}\left(  \mathcal{A}\right)  ^{+}$ there is $Y\in\left[
X\right]  ^{\omega}$ such that $Y$ is almost disjoint with every element of
$\mathcal{A}.$}

\item (Mr\'{o}wka \cite{Mrowka}) There is a \textsf{MAD} family for which its
$\Psi$-space has a unique compactification.

\item (Raghavan \cite{ThereisavanDouwenMADfamily}) There is a van Douwen
\textsf{MAD} family.\footnote{A family of functions $\mathcal{A}%
\subseteq\omega^{\omega}$ is a \emph{Van Douwen }\textsf{MAD }family if for
every infinite partial function $f$ from $\omega$ to $\omega$ there is
$h\in\mathcal{A}$ such that $\left\vert f\cap h\right\vert =\omega.$}

\item (Raghavan \cite{AModelwithnoStronglySeparableAlmostDisjointFamilies})
There is a model with no Shelah-Stepr\={a}ns \textsf{MAD} families (this
notion will be defined in the fourth chapter).
\end{enumerate}

\qquad\ \ 

In this thesis, we will add another result to the list, we will show that
there is a $+$-Ramsey \textsf{MAD} family. In this chapter we will recall the
basic properties of \textsf{AD} families. Note that an \textsf{AD} family
$\mathcal{A}$ is \textsf{MAD} if and only if for every $X\in\left[
\omega\right]  ^{\omega}$ there is $A\in\mathcal{A}$ such that $A\cap X$ is
infinite.\qquad\ \ 

\begin{definition}
If $\mathcal{A}$ is an \textsf{AD} family we define:

\begin{enumerate}
\item $\mathcal{I}\left(  \mathcal{A}\right)  $ is the ideal generated by
$\mathcal{A}.$ In other words, $X\in\mathcal{I}\left(  \mathcal{A}\right)  $
if and only if there are $A_{0},...,A_{n}\in\mathcal{A}$ such that
$X\subseteq^{\ast}A_{0}\cup...\cup A_{n}.$

\item $\mathcal{I}\left(  \mathcal{A}\right)  ^{++}$ is the set of all
$X\subseteq\omega$ for which there is $\mathcal{B\in}$ $\left[  \mathcal{A}%
\right]  ^{\omega}$ such that if $A\in\mathcal{B}$ then $X\cap A$ is infinite.

\item $\mathcal{A}^{\perp}$ is the set of all $X\subseteq\omega$ such that
$\left\vert X\cap A\right\vert <\omega$ for every $A\in\mathcal{A}.$
\end{enumerate}
\end{definition}

\qquad\qquad\ \ \ 

Recall that $\mathcal{I}\left(  \mathcal{A}\right)  ^{+}$ is the collection of
all subsets of $\omega$ that are not in $\mathcal{I}\left(  \mathcal{A}%
\right)  .$ Then we have the following:

\begin{lemma}
If $\mathcal{A}$ is an AD family then the following holds:

\begin{enumerate}
\item $\mathcal{A}$ is a \textsf{MAD} family if and only if $\mathcal{I}%
\left(  \mathcal{A}\right)  $ is a tall ideal.

\item $\mathcal{I}\left(  \mathcal{A}\right)  ^{++}\subseteq\mathcal{I}\left(
\mathcal{A}\right)  ^{+}.$

\item $\mathcal{A}$ is a \textsf{MAD} family if and only if $\mathcal{I}%
\left(  \mathcal{A}\right)  ^{++}=\mathcal{I}\left(  \mathcal{A}\right)
^{+}.$
\end{enumerate}
\end{lemma}

\begin{proof}
The first part follows directly by the definitions. Let $X\in\mathcal{I}%
\left(  \mathcal{A}\right)  ^{++}$ and $A_{0},...,A_{n}\in\mathcal{A}$. We
need to see that $X$ is not almost contained in $A_{0}\cup...\cup A_{n}.$
Since $X\in\mathcal{I}\left(  \mathcal{A}\right)  ^{++}$ then there is
$B\in\mathcal{A}\setminus\left\{  A_{0},...,A_{n}\right\}  $ such that $X\cap
B$ is infinite, so $X$ can not be almost contained in $A_{0}\cup...\cup
A_{n}.$

\qquad\ \ 

Now assume $\mathcal{A}$ is a \textsf{MAD} family and let $X\notin
\mathcal{I}\left(  \mathcal{A}\right)  ^{++}$ we must show $X\notin
\mathcal{I}\left(  \mathcal{A}\right)  ^{+}.$ Let $\mathcal{B\subseteq A}$ be
the collection of all elements of $\mathcal{A}$ that $X$ intersects
infinitely. We then know $\mathcal{B}$ is finite, lets say $\mathcal{B}%
=\left\{  A_{i}\mid i<n\right\}  $ and define $Y=X\setminus%
{\textstyle\bigcup\limits_{i<n}}
A_{i}.$ Note that no element of $\mathcal{A}$ intersects $Y$ infinitely and
since $\mathcal{A}$ is \textsf{MAD}, $Y$ must be finite so \thinspace
$X\subseteq^{\ast}%
{\textstyle\bigcup\limits_{i<n}}
A_{i}$ and then $X\notin\mathcal{I}\left(  \mathcal{A}\right)  ^{+}.$ For the
other implication, assume $\mathcal{I}\left(  \mathcal{A}\right)
^{++}=\mathcal{I}\left(  \mathcal{A}\right)  ^{+}$ we want to prove
$\mathcal{I}\left(  \mathcal{A}\right)  $ is tall but this immediate if
$\mathcal{I}\left(  \mathcal{A}\right)  ^{++}=\mathcal{I}\left(
\mathcal{A}\right)  ^{+}.$
\end{proof}

\qquad\ \ 

The following lemma establishes the basic properties of the ideals generated
by AD families.

\begin{lemma}
[Mathias \cite{HappyFamilies}]Let $\mathcal{A}$ be an AD family then:

\begin{enumerate}
\item $\mathcal{I}\left(  \mathcal{A}\right)  $ is meager.

\item $\mathcal{I}\left(  \mathcal{A}\right)  $ is selective (hence $P^{+}$
and $Q^{+}$).

\item $\mathcal{I}\left(  \mathcal{A}\right)  $ is not a $P$-ideal.

\item $\mathcal{I}\left(  \mathcal{A}\right)  $ is not $\omega$-hitting.
\end{enumerate}
\end{lemma}

\begin{proof}
Let $\left\{  A_{n}\mid n\in\omega\right\}  \subseteq\mathcal{A},$ note that
no element of $\mathcal{I}\left(  \mathcal{A}\right)  $ has infinite
intersection with every $A_{n}$ so $\mathcal{I}\left(  \mathcal{A}\right)  $
is not a $P$-ideal nor $\omega$-hitting. Define an interval partition
$\mathcal{P}=\left\{  P_{n}\mid n\in\omega\right\}  $ such that $P_{n}\cap
A_{i}\neq\emptyset$ for every $i\leq n,$ then $\mathcal{I}\left(
\mathcal{A}\right)  $ is meager by Talagrand's theorem.

\qquad\ \ 

We will now show that $\mathcal{I}\left(  \mathcal{A}\right)  $ is selective.
Let $\mathcal{Y}=\left\{  Y_{n}\mid n\in\omega\right\}  \subseteq
\mathcal{I}\left(  \mathcal{A}\right)  ^{+}$ be a decreasing family. First
assume there is $Z\in\mathcal{A}^{\perp}$ such that $Z$ is a
pseudointersection of $\mathcal{Y}.$ Then we recursively construct $X=\left\{
x_{n}\mid n\in\omega\right\}  $ such that $x_{0}\in Z\cap Y_{0}$ and
$x_{n+1}\in Z\cap Y_{x_{n}}$ (with $x_{n}<x_{n+1}$). Then $X$ is the set we
were looking for. Now assume $\mathcal{Y}$ does not have a pseudointersection
in $\mathcal{A}^{\perp}.$ Recursively we can find a family $\mathcal{B}%
=\left\{  B_{n}\mid n\in\omega\right\}  \subseteq\mathcal{I}\left(
\mathcal{A}\right)  $ such that the following holds:

\begin{enumerate}
\item Each $B_{n}\subseteq Y_{0}$ is a pseudointersection of $\mathcal{Y}.$

\item There is $A_{n}\in\mathcal{A}$ such that $B_{n}\subseteq A_{n}.$

\item If $n\neq m$ then $A_{n}\neq A_{m}.$
\end{enumerate}

\qquad\ \ \ 

Let $f:\omega\longrightarrow\mathcal{B}$ such that if $B\in\mathcal{B}$ then
$f^{-1}\left(  B\right)  $ is infinite. Then we recursively construct
$X=\left\{  x_{n}\mid n\in\omega\right\}  $ such that $x_{0}\in Y_{0}$ and
$x_{n+1}\in f\left(  n\right)  \cap Y_{x_{n}}$ (with $x_{n}<x_{n+1}$). Then
$X$ is the set we were looking for.
\end{proof}

\qquad\ \ \ 

The following is a useful lemma that will be used implicitly in several occasions:

\begin{proposition}
Let $\mathcal{A}$ be a \textsf{MAD} family and let $X\in\mathcal{I}\left(
\mathcal{A}\right)  ^{+}.$ Then there is an almost disjoint family
$\mathcal{C\subseteq I}\left(  \mathcal{A}\right)  ^{+}$ of subsets of $X$ of
size $\mathfrak{c}.$\qquad\ \ \ 
\end{proposition}

\begin{proof}
Since $X\in\mathcal{I}\left(  \mathcal{A}\right)  ^{+}$, there is a countable
family $\mathcal{B=}\left\{  B_{n}\mid n\in\omega\right\}  \subseteq
\mathcal{A}$ such that $X$ has infinite intersection with every element of
$\mathcal{B}.$ Let $\mathcal{P}=\left\{  P_{n}\mid n\in\omega\right\}  $ be a
partition of $X$ into finite sets such that $P_{n}\cap B_{i}\neq\emptyset$ for
every $i\leq n.$ Let $\mathcal{D}$ be an almost disjoint family of size
$\mathfrak{c,}$ for every $D\in\mathcal{D}$ we define $Y_{D}=%
{\textstyle\bigcup\limits_{n\in D}}
P_{n}.$ The family $\left\{  Y_{D}\mid D\in\mathcal{D}\right\}  $ has the
desired properties.
\end{proof}

\qquad\ \ \ \qquad\ \ 

The following types of \textsf{MAD} families will play a very important role
in this thesis:

\begin{definition}
Let $\mathcal{A}$ be an \textsf{AD} family.\ \ \ \qquad\ \ 

\begin{enumerate}
\item $\mathcal{A}$ is \emph{weakly tight }if for every $\left\{  X_{n}\mid
n\in\omega\right\}  \subseteq\mathcal{I}\left(  \mathcal{A}\right)  ^{+}$
there is $B\in\mathcal{I}\left(  \mathcal{A}\right)  $ such that $\left\vert
B\cap X_{n}\right\vert =\omega$ for infinitely many $n\in\omega.$\qquad

\item $\mathcal{A}$ is\ \emph{tight }if for every $\left\{  X_{n}\mid
n\in\omega\right\}  \subseteq\mathcal{I}\left(  \mathcal{A}\right)  ^{+}$
there is $B\in\mathcal{I}\left(  \mathcal{A}\right)  $ such that $B\cap X_{n}$
is infinite for every $n\in\omega.$
\end{enumerate}
\end{definition}

\qquad\ \ \ \ \qquad\ \ \ \qquad\ 

Clearly every weakly tight AD family is \textsf{MAD} and tightness imply weak
tightness. By the previous result, $\mathcal{A}$ is tight if and only if for
every $\left\{  X_{n}\mid n\in\omega\right\}  \subseteq\mathcal{I}\left(
\mathcal{A}\right)  ^{+}$ there is $B\in\mathcal{I}\left(  \mathcal{A}\right)
$ such that $B\cap X_{n}\neq\emptyset$ for every $n\in\omega.$ The following
is a simple equivalence of weak tightness:

\begin{lemma}
Let $\mathcal{A}$ be a \textsf{MAD} family. The following are equivalent:

\begin{enumerate}
\item $\mathcal{A}$ is weakly tight.

\item If $X=\left\{  X_{n}\mid n\in\omega\right\}  \subseteq\mathcal{I}\left(
\mathcal{A}\right)  ^{+}$ is a partition, then there is $A\in\mathcal{A}$ such
that $A\cap X_{n}$ is infinite for infinitely many $X_{n}.$

\item If $X=\left\{  X_{n}\mid n\in\omega\right\}  \subseteq\mathcal{I}\left(
\mathcal{A}\right)  ^{+}$ are pairwise disjoint, then there is $A\in
\mathcal{A}$ such that $A\cap X_{n}$ is infinite for infinitely many $X_{n}.$
\end{enumerate}
\end{lemma}

\begin{proof}
Obviously 1 implies 2. Moreover, it is easy to see that 2 and 3 are equivalent
since from an infinite family of pairwise disjoint sets we can get a partition
by performing only finite changes. We will now prove the that 3 implies 1. Let
$\mathcal{A}$ be a \textsf{MAD} family that satisfies 3, we will show that
$\mathcal{A}$ is weakly tight. Let $X=\left\{  X_{n}\mid n\in\omega\right\}
\subseteq\mathcal{I}\left(  \mathcal{A}\right)  ^{+},$ we now define the
forcing $\mathbb{P}$ whose elements are functions $p$ with the following properties:

\begin{enumerate}
\item $p:n_{p}\times m_{p}\longrightarrow2.$

\item If $\left(  i,j\right)  \in dom\left(  p\right)  $ and $j\notin X_{i}$
then $p\left(  i,j\right)  =0.$
\end{enumerate}

\qquad\ \ 

If $p,q\in\mathbb{P}$ then $p\leq q$ if $q\subseteq p$ and the following
holds: if $m_{q}\leq j<m_{p}$ and $i_{1},i_{2}$ are two distinct elements of
$n_{q}$ either $p\left(  i_{1},j\right)  =0$ or $p\left(  i_{2},j\right)  =0.$
It is easy to see that $\mathbb{P}$ adds an almost disjoint family $\left\{
Y_{n}\mid n\in\omega\right\}  $ such that $Y_{n}\subseteq X_{n}.$ Moreover,
each $Y_{n}$ is forced to be in $\mathcal{I}\left(  \mathcal{A}\right)  ^{+}$
and we only need to meet countably many dense set to achieve this. The result
clearly follows.
\end{proof}

\qquad\ \ \ \qquad\ \ \ 

If $\mathcal{P}$ is a property of almost disjoint families, we will say that
\emph{\textsf{MAD} families with property }$\mathcal{P}$ \emph{exist
generically }if every AD family of size less than $\mathfrak{c}$ can be
extended to a \textsf{MAD} family with property $\mathcal{P}.$

\newpage

\section{The Kat\v{e}tov order}

In \cite{ProductsofFilters} Kat\v{e}tov introduced a partial order on ideals.
The Kat\v{e}tov order is a very powerful tool for studying ideals over
countable sets. It plays a very important role in understanding
destructibility of ideals. Another important feature of the Kat\v{e}tov order
is its usefulness for classifying non-definable objects like ultrafilters. It
can be proved that an ultrafilter $\mathcal{U}$ is a Ramsey ultrafilter if and
only if its dual $\mathcal{U}^{\ast}$ is not Kat\v{e}tov above the ideal
$\mathcal{ED},$ $\mathcal{U}$ is a $P$-point if and only if $\mathcal{U}%
^{\ast}$ is not Kat\v{e}tov above \textsf{FIN}$\times$\textsf{FIN }and
$\mathcal{U}$ is a nowhere dense ultrafilter if and only if $\mathcal{U}%
^{\ast}$ is not Kat\v{e}tov above the ideal \textsf{nwd} (see
\cite{KatetovOrderonBorelIdeals}).

\begin{definition}
Let $A\ $and $B$ be two countable sets,$\ \mathcal{I},\mathcal{J}$ ideals on
$X$ and $Y$ respectively and $f:Y\longrightarrow X$.

\begin{enumerate}
\item We say $f$ is a \emph{Kat\v{e}tov morphism from }$\left(  Y,\mathcal{J}%
\right)  $ \emph{to} $\left(  X,\mathcal{I}\right)  $ if $f^{-1}\left(
A\right)  \in\mathcal{J}$ for every $A\in\mathcal{I}.$

\item We define $\mathcal{I}\leq_{\mathsf{K}}$ $\mathcal{J}$ ($\mathcal{I}$
i\emph{s Kat\v{e}tov smaller than} $\mathcal{J}$ or $\mathcal{J}$ \emph{is
Kat\v{e}tov above} $\mathcal{I}$) if there is a Kat\v{e}tov morphism from
$\left(  Y,\mathcal{J}\right)  $ to $\left(  X,\mathcal{I}\right)  .$

\item We define $\mathcal{I\simeq}_{K}$ $\mathcal{J}$ ($\mathcal{I}$ \emph{is
Kat\v{e}tov equivalent to} $\mathcal{J}$) if $\mathcal{I}\leq_{\mathsf{K}}$
$\mathcal{J}$ and $\mathcal{J}\leq_{\mathsf{K}}$ $\mathcal{I}.$

\item We say $f$ is a \emph{Kat\v{e}tov-Blass morphism from }$\left(
Y,\mathcal{J}\right)  $ \emph{to} $\left(  X,\mathcal{I}\right)  $ if $f$ is a
finite to one Kat\v{e}tov morphism from\emph{ }$\left(  Y,\mathcal{J}\right)
$ to $\left(  X,\mathcal{I}\right)  .$

\item We define $\mathcal{I}\leq_{\mathsf{KB}}$ $\mathcal{J}$ \ if there is a
Kat\v{e}tov-Blass morphism from $\left(  Y,\mathcal{J}\right)  $ to $\left(
X,\mathcal{I}\right)  .$

\item $\mathcal{I}$ \emph{is Kat\v{e}tov-Blass equivalent to} $\mathcal{J}$ if
$\mathcal{I}\leq_{\mathsf{KB}}$ $\mathcal{J}$ and $\mathcal{J}\leq
_{\mathsf{KB}}$ $\mathcal{I}.$
\end{enumerate}
\end{definition}

\qquad\ \ \ \qquad\qquad

The following are some simple observations regarding the Kat\v{e}tov order:

\begin{lemma}
Let $\mathcal{I},\mathcal{J},\mathcal{L}$ be ideals.

\begin{enumerate}
\item $\mathcal{I\simeq}_{K}$ $\mathcal{I}.$

\item If $\mathcal{I}\leq_{\mathsf{K}}\mathcal{J}$ and $\mathcal{J}%
\leq_{\mathsf{K}}$ $\mathcal{L}$ then $\mathcal{I}\leq_{\mathsf{K}}%
\mathcal{L}.$

\item \textsf{FIN }is the smallest element in the Kat\v{e}tov order.

\item $\mathcal{I}$ is Kat\v{e}tov equivalent to \textsf{FIN} if and only if
$\mathcal{I}$ is not tall..

\item If $X\in\mathcal{I}^{+}$ then $\mathcal{I}\leq_{\mathsf{K}}$
$\mathcal{I}\upharpoonright X.$
\end{enumerate}
\end{lemma}

\qquad\ \ \ 

An ideal $\mathcal{I}$ is \emph{Kat\v{e}tov uniform }if $\mathcal{I}$ is
Kat\v{e}tov equivalent to all its restrictions (equivalently, if
$X\in\mathcal{I}^{+}$ then $\mathcal{I}\upharpoonright X\leq_{\mathsf{K}%
}\mathcal{I}$). Since every tall ideal contains a \textsf{MAD} family then the
ideals generated by \textsf{MAD} families are coinitial in the Kat\v{e}tov
order. On the other hand, the dual ideals of ultrafilters form a cofinal
family. If $\mathcal{A}$ and $\mathcal{B}$ are \textsf{AD} families then we
define $\mathcal{A}\leq_{\mathsf{K}}$ $\mathcal{B}$ if $\mathcal{I}\left(
\mathcal{A}\right)  $ $\leq_{\mathsf{K}}$ $\mathcal{I}\left(  \mathcal{B}%
\right)  .$ We then have the following:

\begin{lemma}
Let $\mathcal{A},\mathcal{B}$ be \textsf{AD} families.

\begin{enumerate}
\item $\mathcal{A}$ is \textsf{MAD} if and only if $\mathcal{A\nleq}_{K}$
\textsf{FIN}.

\item If $X\in\mathcal{I}\left(  \mathcal{A}\right)  ^{+}$ then $\mathcal{A}%
\leq_{\mathsf{K}}$ $\mathcal{A}\upharpoonright X.$

\item If $\mathcal{A}\leq_{\mathsf{K}}$ $\mathcal{B}$ then $\left\vert
\mathcal{B}\right\vert \leq\left\vert \mathcal{A}\right\vert .$\qquad\ \ 
\end{enumerate}
\end{lemma}

\qquad\ \ \ \qquad\ \ 

Every AD family is Kat\v{e}tov below \textsf{FIN}$\times$\textsf{FIN} as we
will prove now.

\begin{proposition}
If $\mathcal{A}$ is an \textsf{AD} family then $\mathcal{I}\left(
\mathcal{A}\right)  \leq_{K}$\textsf{FIN}$\times$\textsf{FIN}$.$
\end{proposition}

\begin{proof}
Let $\left\{  A_{n}\mid n\in\omega\right\}  \subseteq\mathcal{I}\left(
\mathcal{A}\right)  $ be a partition of $\omega$ into infinite sets. We then
find a bijection $f:\omega\times\omega\longrightarrow\omega$ such that
$f\left[  C_{n}\right]  =A_{n}$ where $C_{n}$ is the set $\left\{  \left(
n,m\right)  \mid m\in\omega\right\}  .$ It is easy to see that $f$ is a
Kat\v{e}tov morphism from $(\omega\times\omega,$\textsf{FIN}$\times
$\textsf{FIN}$)$ to $\left(  \omega,\mathcal{I}\left(  \mathcal{A}\right)
\right)  .$
\end{proof}

\newpage

\section{Cardinal invariants of the continuum}

\qquad\ \ \ Let $\left(  \mathbb{P},\leq\right)  $ be a partial order, we say
$D\subseteq\mathbb{P}$ is $\leq$\emph{-dominating }(or just dominating if
$\leq$ is clear from the context) if for every $p\in\mathbb{P}$ there is $q\in
D$ such that $p\leq q.$ Meanwhile, a set $B\subseteq\mathbb{P}$ is called
$\leq$\emph{-unbounded }(or just unbounded) if there is no $p\in\mathbb{P}$
such that $q\leq p$ for every $q\in B$. If $\mathbb{P}$ does not have a
maximum we can then define the following invariants:

\qquad\ \ 

\begin{enumerate}
\item $\mathfrak{d}\left(  \mathbb{P}\right)  $ is the smallest size of a
dominating family of $\mathbb{P}.$

\item $\mathfrak{b}\left(  \mathbb{P}\right)  $ is the smallest size of an
unbounded family of $\mathbb{P}.$
\end{enumerate}

\qquad\ \ \ \qquad\ \ 

Since every dominating family is unbounded (in case there is no maximum) then
$\mathfrak{b}\left(  \mathbb{P}\right)  \leq\mathfrak{d}\left(  \mathbb{P}%
\right)  .$ We can then define two of the most important cardinal invariants:

\begin{definition}
\qquad\ \ \qquad\ 

\begin{enumerate}
\item The \emph{unboundedness number }$\mathfrak{b}$ is $\mathfrak{b}\left(
\omega^{\omega},\leq^{\ast}\right)  $ i.e. the smallest size of an $\leq
^{\ast}$-unbounded family of functions.

\item The \emph{dominating number }$\mathfrak{d}$ is $\mathfrak{d}\left(
\omega^{\omega},\leq^{\ast}\right)  $ i.e. the smallest size of a $\leq^{\ast
}$-dominating family of functions.
\end{enumerate}
\end{definition}

\qquad\ \ \ 

Note that if $f$ is a function then $f\leq f+1$ so $\left(  \omega^{\omega
},\leq^{\ast}\right)  $ has no maximum and then $\mathfrak{b\leq d.}$ From now
on, when talking about elements of $\omega^{\omega},$ unbounded will mean
$\leq^{\ast}$-unbounded and dominating will mean $\leq^{\ast}$-dominating. The
basic properties of $\mathfrak{b}$ and $\mathfrak{d}$ are the following:

\begin{lemma}
\qquad\ \ 

\begin{enumerate}
\item $\omega<\mathfrak{b\leq}$ \textsf{cof}$\left(  \mathfrak{d}\right)
\leq\mathfrak{d\leq c}.$

\item $\mathfrak{b}\left(  \omega^{\omega},\leq\right)  =\omega$ and
$\mathfrak{d}\left(  \omega^{\omega},\leq\right)  =\mathfrak{d}.$

\item There is an unbounded family $B=\left\{  f_{\alpha}\mid\alpha
<\mathfrak{b}\right\}  \subseteq\omega^{\omega}$ such that if $\alpha<\beta$
then $f_{\alpha}<^{\ast}f_{\beta}.$

\item $\mathfrak{b}$ is a regular cardinal.
\end{enumerate}
\end{lemma}

\begin{proof}
We first prove that $\omega<\mathfrak{b}$ or in other words, that every
countable family of $\omega^{\omega}$ is $\leq^{\ast}$-bounded. Given
$B=\left\{  f_{n}\mid n\in\omega\right\}  \subseteq\omega^{\omega}$ define $g$
$\in\omega^{\omega}$ such that $g\left(  n\right)  =f_{0}\left(  n\right)
+...+f_{n}\left(  n\right)  .$ It is easy to see that $g$ is an upper bound
for $B.$

\qquad\ \ 

We will now prove that $\mathfrak{b\leq}$ \textsf{cof}$\left(  \mathfrak{d}%
\right)  .$ We will proceed by contradiction, so assume that \textsf{cof}%
$\left(  \mathfrak{d}\right)  <\mathfrak{b}.$ Let $D\subseteq\omega^{\omega}$
be a dominating family of size $\mathfrak{d}.$ Let $D=\bigcup\left\{
D_{\alpha}\mid\alpha\in\mathsf{cof}\left(  \mathfrak{d}\right)  \right\}  $
where each $D_{\alpha}$ has size less than $\mathfrak{d}.$ Since each
$D_{\alpha}$ is not dominating, there is $g_{\alpha}\in\omega^{\omega}$ that
is not dominated by any element of $D_{\alpha}$ i.e. $g_{\alpha}\nleq^{\ast}f$
for every $f\in D_{\alpha}.$ Since \textsf{cof}$\left(  \mathfrak{d}\right)
<\mathfrak{b}$ we can then find $h\in\omega^{\omega}$ such that $g_{\alpha
}\leq^{\ast}h$ for every $\alpha\in$ \textsf{cof}$\left(  \mathfrak{d}\right)
.$ But then $h$ is not dominated by any element of $D,$ which is a contradiction.

\qquad\ \ \ 

If $c_{n}\in\omega^{\omega}$ is the function with constant value $n$, then
$\left\{  c_{n}\mid n\in\omega\right\}  $ is $\leq$-unbounded, so
$\mathfrak{b}\left(  \omega^{\omega},\leq\right)  =\omega.$ Obviously
$\mathfrak{d\leq d}\left(  \omega^{\omega},\leq\right)  $ so we only need to
prove the other inequality. Given $g:\omega\longrightarrow\omega$ and
$n\in\omega$ we define the function $g^{n}:\omega\longrightarrow\omega$ where
$g^{n}\left(  m\right)  =g\left(  m\right)  +n.$ If $D=\left\{  g_{\alpha}%
\mid\alpha\in\mathfrak{d}\right\}  $ is a dominating family, then
$D_{1}=\left\{  g_{\alpha}^{n}\mid n\in\omega,\alpha\in\mathfrak{d}\right\}  $
is $\leq$-dominating, so $\mathfrak{d}\left(  \omega^{\omega},\leq\right)  $
is at most $\mathfrak{d.}$

\qquad\ \ \ 

We now prove 3. Let $A=\left\{  g_{\alpha}\mid\alpha\in\mathfrak{b}\right\}  $
be an unbounded family. We then recursively construct $B=\left\{  f_{\alpha
}\mid\alpha<\mathfrak{b}\right\}  $ such that $g_{\alpha}\leq^{\ast}f_{\alpha
}$ and if $\alpha<\beta$ then $f_{\alpha}<^{\ast}f_{\beta}.$ This can be done
since at each step we have less than $\mathfrak{b}$ functions. It is clear
that $B$ has the desired properties.

\qquad

We will now prove that $\mathfrak{b}$ is a regular cardinal. Let $B=\left\{
f_{\alpha}\mid\alpha<\mathfrak{b}\right\}  $ be as above and let
$S\subseteq\mathfrak{b}$ be a cofinal set of size \textsf{cof}$\left(
\mathfrak{b}\right)  $. Since $S$ is cofinal in $\mathfrak{b}$, then
$B^{\prime}=\left\{  f_{\alpha}\mid\alpha\in S\right\}  $ is unbounded so
$\left\vert S\right\vert =\mathfrak{b}.$
\end{proof}

\qquad\ \ \ 

In this way, changing $\leq^{\ast}$ to $\leq$ makes a difference for
$\mathfrak{b}$ but not for $\mathfrak{d}.$ It is important to remark that
while $\mathfrak{b}$ is regular, $\mathfrak{d}$ can be singular. Given
$f,g\in\omega^{\omega}$ and $n\in\omega.$ we define $f\leq_{n}g$ if $f\left(
m\right)  \leq g\left(  m\right)  $ for every $m\geq n.$ In this way,
$f\leq^{\ast}g$ if and only if there is $n\in\omega$ such that $f\leq_{n}g.$

\begin{definition}
We say $S=\left\{  f_{\alpha}\mid\alpha\in\kappa\right\}  \subseteq
\omega^{\omega}$ is a \emph{scale }if $\kappa$ is regular, $S$ is dominating
and $f_{\alpha}\leq^{\ast}f_{\beta}$ whenever $\alpha<\beta.$
\end{definition}

\qquad\ \ \ 

Note that the requirement of the regularity is harmless, if there was a scale
of singular size, then there would be a scale of regular size.\qquad\ \ 

\begin{lemma}
There is a scale if and only if $\mathfrak{b=d.}$ Moreover, the size of any
scale is $\mathfrak{b}.$
\end{lemma}

\begin{proof}
First assume $\mathfrak{b}=\mathfrak{d}$ and let $D=\left\{  f_{\alpha}%
\mid\alpha\in\mathfrak{d}\right\}  $ be a dominating family. If we do the
construction used in 3 of the previous lemma we get a scale. Now assume that
$S=\left\{  f_{\alpha}\mid\alpha\in\kappa\right\}  $ is a scale and note that
$\mathfrak{d}\leq\kappa.$ Now assume that $\mathfrak{b}<\kappa$ and let
$B=\left\{  g_{\beta}\mid\beta\in\mathfrak{b}\right\}  $ be an unbounded
family. For every $\beta\in\mathfrak{b}$ find $\alpha_{\beta}\in\kappa$ such
that $g_{\beta}\leq^{\ast}f_{\alpha_{\beta}}.$ Since $\mathfrak{b}<\kappa$ and
$\kappa$ is regular, there is $\gamma$ such that $\alpha_{\beta}<\gamma$ for
every $\beta\in\mathfrak{b}$ and then $f_{\gamma}$ will bound $B,$ which is a contradiction.
\end{proof}

\qquad\ \ \ \ \ \ \ \ \ \ \ \ \ \ \qquad\ \ \ \ \ \ \ \ \ \ 

Recall that $\mathcal{P}=\left\{  P_{n}\mid n\in\omega\right\}  $ is an
\emph{interval partition }if it is a partition of $\omega$ into consecutive
intervals and by $PART$ we denoted the set of all interval partitions. Given
interval partitions $\mathcal{P}$ and $\mathcal{Q}$ define $\mathcal{Q\leq P}$
if for every $P_{n}\in\mathcal{P}$ there is $Q_{m}\in\mathcal{Q}$ such that
$Q_{m}\subseteq P_{n}$ (in other words, every interval in $\mathcal{P}$
contains at least one interval of $\mathcal{Q}$) and $\mathcal{Q\leq}^{\ast
}\mathcal{P}$ if for almost all $P_{n}\in\mathcal{P}$ there is $Q_{m}%
\in\mathcal{Q}$ such that $Q_{m}\subseteq P_{n}$ (i.e. almost every interval
in $\mathcal{P}$ contains at least one interval of $\mathcal{Q}$). The proof
of the following useful lemma can be found in \cite{HandbookBlass}:

\begin{lemma}
[\cite{HandbookBlass}]\qquad\ \ \ \qquad\ \ \ \ \ \ \ \ \ \ \ \ \qquad\ \ \ 

\begin{enumerate}
\item $\mathfrak{d}=\mathfrak{d}\left(  PART,\leq^{\ast}\right)  .$

\item $\mathfrak{b}=\mathfrak{b}\left(  PART,\leq^{\ast}\right)  .$
\end{enumerate}
\end{lemma}

The \emph{almost disjointness number }$\mathfrak{a}$ is the smallest size of a
\textsf{MAD} family. Since every \textsf{MAD} family is Kat\v{e}tov below
\textsf{FIN}$\times$\textsf{FIN} we conclude that $\mathfrak{b\leq a}.$
\qquad\qquad\ \ \ 

\begin{definition}
\qquad\ \ \qquad\ \ 

\begin{enumerate}
\item We say that $S$ \emph{splits }$X$ if $S\cap X$ and $X\setminus S$ are
both infinite.

\item $\mathcal{S\subseteq}$ $\left[  \omega\right]  ^{\omega}$ is a
\emph{splitting family }if for every $X\in\left[  \omega\right]  ^{\omega}$
there is $S\in\mathcal{S}$ such that $S$ splits $X.$

\item The \emph{splitting number }$\mathfrak{s}$ is the smallest size of a
splitting family.

\item $\mathcal{R\subseteq}$ $\left[  \omega\right]  ^{\omega}$ is a
\emph{reaping family }if for every $X\in\left[  \omega\right]  ^{\omega}$
there is $A\in\mathcal{R}$ such that either $A\subseteq^{\ast}X$ or
$A\subseteq^{\ast}\omega\setminus X.$\qquad\ \ \ \ \ 

\item The \emph{reaping number }$\mathfrak{r}$ is the smallest size of a
reaping family.
\end{enumerate}
\end{definition}

\qquad\ \ \ \ 

It is easy to see that $\left[  \omega\right]  ^{\omega}$ is a splitting
family, so the invariant $\mathfrak{s}$ is well defined. Note that a filter
$\mathcal{F}$ is an ultrafilter if and only if $\mathcal{F}$ is a reaping
family, so the invariant $\mathfrak{r}$ is also well defined.

\begin{lemma}
$\mathfrak{s\leq d}$ and $\mathfrak{b\leq r}.$
\end{lemma}

\begin{proof}
We will first prove the inequality $\mathfrak{s\leq d.}$ Let $\mathcal{D=}$
$\left\{  \mathcal{P}_{\alpha}\mid\alpha<\mathfrak{d}\right\}  $ be a
dominating family of interval partitions where $\mathcal{P}_{\alpha}=\left\{
P_{\alpha}\left(  n\right)  \mid n\in\omega\right\}  .$ For every
$\alpha<\mathfrak{d}$ we define $S_{\alpha}=%
{\textstyle\bigcup\limits_{n\in\omega}}
P_{\alpha}\left(  2n\right)  $ and we will show that $\mathcal{S=}$ $\left\{
S_{\alpha}\mid\alpha<\mathfrak{d}\right\}  $ is a splitting family. Let
$X\in\left[  \omega\right]  ^{\omega}.$ We now find an interval partition
$\mathcal{R}$ such that every interval of $\mathcal{R}$ contains at least one
point of $X.$ Since $\mathcal{D}$ is a dominating family of interval
partitions, there is $\alpha<\mathfrak{d}$ such that $\mathcal{P}_{\alpha}$
dominates $\mathcal{R}.$ It is easy to see that $S_{\alpha}$ splits $X.$

\qquad\ \ \ \ 

The inequality $\mathfrak{b\leq r}$ is similar: Let $\kappa<\mathfrak{b}$ and
$\mathcal{R=}$ $\left\{  A_{\alpha}\mid\alpha<\kappa\right\}  \subseteq\left[
\omega\right]  ^{\omega}.$ We will show that $\mathcal{R}$ is not a reaping
family. For every $\alpha<\kappa$ we find an interval partition $\mathcal{P}%
_{\alpha}$ such that every interval of $\mathcal{P}_{\alpha}$ contains at
least one point of $A_{\alpha}.$ Since $\kappa<\mathfrak{b}$ then there is an
interval partition $\mathcal{R}=\left\{  R\left(  n\right)  \mid n\in
\omega\right\}  $ dominating every $\mathcal{P}_{\alpha}.$ It is easy to see
that $X=%
{\textstyle\bigcup\limits_{n\in\omega}}
R\left(  2n\right)  $ witness that $\mathcal{R}$ is not a reaping family.
\end{proof}

\qquad\ \ \ \ \qquad\ \ \ \ \qquad\ \ 

The following is a stronger notion than of a splitting family:

\begin{definition}
\qquad\ \ \ \qquad\ \ \ \ \ \ \qquad\ \ \ \ 

\begin{enumerate}
\item Let $S\in\left[  \omega\right]  ^{\omega}$ and $\mathcal{P}=\left\{
P_{n}\mid n\in\omega\right\}  $ be an interval partition. We say $S$
\emph{block-splits }$\mathcal{P}$ if both of the sets $\left\{  n\mid
P_{n}\subseteq S\right\}  $ and $\left\{  n\mid P_{n}\cap S=\emptyset\right\}
$ are infinite.

\item A family $\mathcal{S}\subseteq\left[  \omega\right]  ^{\omega}$ is
called a \emph{block-splitting family }if every interval partition is
block-split by some element of $\mathcal{S}.$
\end{enumerate}
\end{definition}

\qquad\ \ \ 

It is easy to see that every block-splitting family is splitting. The
following is a result of Kamburelis and Weglorz:

\begin{proposition}
[\cite{Splittings}]The smallest size of a block-splitting family is
$max\left\{  \mathfrak{b},\mathfrak{s}\right\}  .$
\end{proposition}

\begin{proof}
Let $\kappa$ be the smallest size of a block-splitting family. Obviously
$\mathfrak{s}\leq\kappa$ and now we will prove that $\mathfrak{b}\leq\kappa.$
To prove this it is enough to show that no family of size less than
$\mathfrak{b}$ is a block-splitting family. Let $\mu<\mathfrak{b}$ and
$\mathcal{S}=\left\{  S_{\alpha}\mid\alpha<\mu\right\}  $ be a family of
infinite subsets of $\omega.$ For every $\alpha<\mu$ define an interval
partition $P_{\alpha}=\left\{  P_{n}\left(  \alpha\right)  \mid n\in
\omega\right\}  $ such that each $P_{n}\left(  \alpha\right)  $ has non empty
intersection with both $S_{\alpha}$ and $\omega\setminus S_{\alpha}.$ Since
$\mu<\mathfrak{b}$ then there is an interval partition $\mathcal{R}=\left\{
R_{n}\mid n\in\omega\right\}  $ dominating each $P_{\alpha}$ i.e. almost all
intervals of $\mathcal{R}$ contains one of $P_{\alpha}.$ It is easy to see
that no element of $\mathcal{S}$ can block-split $\mathcal{R}.$

\qquad\ \ \ 

Now we will construct a block-splitting family of size $max\left\{
\mathfrak{b},\mathfrak{s}\right\}  .$ First find an unbounded family of
interval partitions $\mathcal{B}=\left\{  P_{\alpha}\mid\alpha<\mathfrak{b}%
\right\}  $ (where $P_{\alpha}=\left\{  P_{\alpha}\left(  n\right)  \mid
n\in\omega\right\}  $) and a splitting family $\mathcal{S}=\left\{  S_{\beta
}\mid\beta<\mathfrak{s}\right\}  $. Given $\alpha<\mathfrak{b}$ and
$\beta<\mathfrak{s}$ define $D_{\alpha,\beta}=\bigcup\limits_{n\in S_{\beta}%
}P_{\alpha}\left(  n\right)  $ we will prove that $\left\{  D_{\alpha,\beta
}\mid\alpha<\mathfrak{b},\beta<\mathfrak{s}\right\}  $ is a block-splitting
family. Let $\mathcal{R}=\left\{  R_{n}\mid n\in\omega\right\}  $ be an
interval partition. Since $\mathcal{B}$ is unbounded, there is $\alpha
<\mathfrak{b}$ such that $P_{\alpha}$ is not dominated by $\mathcal{R}%
\mathbf{.}$ We can then find an infinite set $W=\left\{  w_{n}\mid n\in
\omega\right\}  $ such that for every $n<\omega$ there is $k<\omega$ for which
$R_{k}\subseteq P_{\alpha}\left(  w_{n}\right)  $ (this is possible since
$P_{\alpha}$ is not dominated by $\mathcal{R}$). Since $\mathcal{S}$ is a
splitting family, there is $\beta<\mathfrak{s}$ such that\ both $S_{\beta}\cap
W$ and $\left(  \omega\setminus S_{\beta}\right)  \cap W$ are infinite. It is
easy to see that $D_{\alpha,\beta}$ block-splits $\mathcal{R}.$
\end{proof}

\qquad\ \ \qquad\ \ \ 

We will need the following notions:

\begin{definition}
Let $S\in\left[  \omega\right]  ^{\omega}$ and $\overline{X}=\left\{
X_{n}\mid n\in\omega\right\}  \subseteq\left[  \omega\right]  ^{\omega}$.

\begin{enumerate}
\item We say that $S$ $\omega$\emph{-splits} $\overline{X}$ if $S$ splits
every $X_{n}.$

\item We say that $S$ $\left(  \omega,\omega\right)  $\emph{-splits}
$\overline{X}$ if both the sets $\left\{  n\mid\left\vert X_{n}\cap
S\right\vert =\omega\right\}  $ and $\left\{  n\mid\left\vert X_{n}\cap\left(
\omega\backslash S\right)  \right\vert =\omega\right\}  $ are infinite.

\item We say that $\mathcal{S}\subseteq\left[  \omega\right]  ^{\omega}$ is an
$\omega$\emph{-splitting family }if every countable collection of infinite
subsets of $\omega$ is $\omega$-split by some element of $S.$

\item We say that $\mathcal{S}\subseteq\left[  \omega\right]  ^{\omega}$ is an
$\left(  \omega,\omega\right)  $\emph{-splitting family }if every countable
collection of infinite subsets of $\omega$ is $\left(  \omega,\omega\right)
$-split by some element of $S.$
\end{enumerate}
\end{definition}

\qquad\qquad\ \ \ \ \ \qquad\ \ 

We now have the following:

\begin{lemma}
Every block-splitting family is an $\omega$-splitting family.
\end{lemma}

\begin{proof}
Let $\mathcal{S}$ be a block-splitting family and $\overline{X}=\left\{
X_{n}\mid n\in\omega\right\}  \subseteq\left[  \omega\right]  ^{\omega}.$
Define an interval partition $\mathcal{P}=\left\{  P_{n}\mid n\in
\omega\right\}  $ such that if $i\leq n$ then $P_{n}\cap X_{i}\neq\emptyset.$
Since $\mathcal{S}$ is a block-splitting family, there is $S\in\mathcal{S}$
that block-splits $\mathcal{P}.$ It is then easy to see that $S$ $\omega
$-splits $\overline{X}.$
\end{proof}

\qquad\ \ \ \qquad\ \ \ 

We will need the following lemma.

\begin{lemma}
[\cite{SplittingFamiliesandCompleteSeparability}]Every splitting family of
size less than $\mathfrak{b}$ is an $\left(  \omega,\omega\right)  $-splitting family.
\end{lemma}

\begin{proof}
Let $\mathcal{S}=\left\{  S_{\alpha}\mid\alpha<\kappa\right\}  $ be a
splitting family of size less than $\mathfrak{b}.$ Assume $\mathcal{S}$ is not
an $\left(  \omega,\omega\right)  $-splitting family so there is $\overline
{X}=\left\{  X_{n}\mid n\in\omega\right\}  \subseteq\left[  \omega\right]
^{\omega}$ that is not $\left(  \omega,\omega\right)  $-split by any element
of $\mathcal{S}$. This means that for every $\alpha<\kappa$ there $i_{\alpha
}<2$ such that $X_{n}\subseteq^{\ast}S_{\alpha}^{i_{\alpha}}$ for almost all
$n\in\omega.$ We can then define a function $f_{\alpha}:\omega\longrightarrow
\omega$ such that if $n<\omega$ then $X_{n}\setminus f_{\alpha}\left(
n\right)  $ $\subseteq S_{\alpha}^{i_{\alpha}}$ if $X_{n}\subseteq^{\ast
}S_{\alpha}^{i_{\alpha}}$ and $f_{\alpha}\left(  n\right)  =0$ in the other
case. Since $\kappa<\mathfrak{b}$ there is $g:\omega\longrightarrow\omega$
dominating each $f_{\alpha}.$ Recursively define $A=\left\{  a_{n}\mid
n\in\omega\right\}  $ such that $a_{n}\in X_{n}\setminus g\left(  n\right)  $
and $a_{n}\neq a_{m}$ whenever $n\neq m.$ Since $\mathcal{S}$ is a splitting
family, there is $\alpha<\kappa$ such that $A\cap S_{\alpha}$ and
$A\cap\left(  \omega\setminus S_{\alpha}\right)  $ are infinite. However,
since $g$ dominates $f_{\alpha}$ we conclude that $A\subseteq^{\ast}S_{\alpha
}^{i_{\alpha}}$ which is a contradiction.
\end{proof}

\qquad\ \ \ \qquad\ \ 

We can then conclude the following important result of Mildenberger, Raghavan
and Stepr\={a}ns:

\begin{corollary}
[\cite{SplittingFamiliesandCompleteSeparability}]There is an $\left(
\omega,\omega\right)  $-splitting family of size $\mathfrak{s}.$
\end{corollary}

\begin{proof}
If $\mathfrak{b\leq s}$ then there is a block-splitting family of size
$\mathfrak{s}$ and if $\mathfrak{s<b}$ then every splitting family of minimal
size is $\left(  \omega,\omega\right)  $-splitting.
\end{proof}

\qquad\ \ \ \qquad\ \ \ 

Many important cardinal invariants come from ideals as we will now see.

\begin{definition}
Let $\mathcal{I\subseteq}\wp\left(  X\right)  $ be an ideal in $X.$ Then we
define the following invariants:

\begin{enumerate}
\item \textsf{add}$\left(  \mathcal{I}\right)  =min\left\{  \left\vert
\mathcal{A}\right\vert \mid\mathcal{A}\subseteq\mathcal{I\wedge}%
\bigcup\mathcal{A\notin I}\right\}  .$

\item \textsf{cov}$\left(  \mathcal{I}\right)  =min\left\{  \left\vert
\mathcal{A}\right\vert \mid\mathcal{A}\subseteq\mathcal{I\wedge}%
\bigcup\mathcal{A}=X\right\}  .$

\item \textsf{non}$\left(  \mathcal{I}\right)  =min\left\{  \left\vert
B\right\vert \mid B\subseteq X\wedge B\notin\mathcal{I}\right\}  .$

\item \textsf{cof}$\left(  \mathcal{I}\right)  =min\left\{  \left\vert
\mathcal{A}\right\vert \mid\mathcal{A}\subseteq\mathcal{I\wedge}\left(
\forall B\in\mathcal{I}\left(  \exists A\in\mathcal{A}\left(  B\subseteq
A\right)  \right)  \right)  \right\}  .$
\end{enumerate}
\end{definition}

\qquad\ \ 

Note that $\mathcal{I}$ is a $\sigma$-ideal if and only if $\omega<$
\textsf{add}$\left(  \mathcal{I}\right)  .n$For ideals on countable sets, we
have the following definitions:

\begin{definition}
Let $\mathcal{I}$ be a tall ideal in $\omega$ (or any countable set). Then we
define the following invariants:

\begin{enumerate}
\item \textsf{add*}$\left(  \mathcal{I}\right)  $ is the smallest size of a
family $\mathcal{A}\subseteq\mathcal{I}$ such that $\mathcal{A}$ does not have
a pseudounion in $\mathcal{A}.$

\item \textsf{cov*}$\left(  \mathcal{I}\right)  $ is the smallest size of a
family $\mathcal{A}\subseteq\mathcal{I}$ such that $\mathcal{A}$ is tall.

\item \textsf{non*}$\left(  \mathcal{I}\right)  $ is the smallest size of a
family $\mathcal{B}\subseteq\left[  \omega\right]  ^{\omega}$ such that there
is no $A\in\mathcal{I}$ that has infinite intersection with every element of
$\mathcal{B}.$
\end{enumerate}
\end{definition}

\qquad\ \ \ \qquad\ \ 

It is easy to see that \textsf{add*}$\left(  \mathcal{I}\right)  \leq$
\textsf{cov*}$\left(  \mathcal{I}\right)  ,$ \textsf{non*}$\left(
\mathcal{I}\right)  \leq$\textsf{cof}$\left(  \mathcal{I}\right)  .$ Note that
if $\mathcal{A}$ is a \textsf{MAD} family then \textsf{cov*}$\left(
\mathcal{I}\left(  \mathcal{A}\right)  \right)  =\left\vert \mathcal{A}%
\right\vert .$

\begin{lemma}
If $\mathcal{I}\leq_{\mathsf{K}}$ $\mathcal{J}$ then \textsf{cov}$^{\ast
}\left(  \mathcal{J}\right)  \leq$\textsf{cov}$^{\ast}\left(  \mathcal{I}%
\right)  .$
\end{lemma}

\begin{proof}
Let $f:\left(  \omega,\mathcal{J}\right)  \longrightarrow\left(
\omega,\mathcal{I}\right)  $ be a Kat\v{e}tov morphism. It is easy to see that
if $\left\{  B_{\alpha}\mid\alpha\in\kappa\right\}  \subseteq\mathcal{I}$ is a
tall family such that $%
{\textstyle\bigcup}
B_{\alpha}=\omega$ then $\left\{  f^{-1}\left(  B_{\alpha}\right)  \mid
\alpha\in\kappa\right\}  $ is also tall.
\end{proof}

\qquad\ \ \qquad\ \ \ \ \qquad\ \ 

We will need the following definition:

\begin{definition}
We say $\left(  A,B,\longrightarrow\right)  $ is an \emph{invariant }if,

\begin{enumerate}
\item $\longrightarrow$ $\subseteq A\times B.$

\item For every $a\in A$ there is a $b\in B$ such that $a\longrightarrow b$
(which means $\left(  a,b\right)  \in$ $\longrightarrow$).

\item There is no $b\in B$ such that $a\longrightarrow b$ for all $a\in A.$
\end{enumerate}
\end{definition}

\qquad\ \ \ \ \ \qquad\ \ \ \ 

The \emph{evaluation of }$\left(  A,B,\longrightarrow\right)  $ (denoted by
$\left\Vert A,B,\longrightarrow\right\Vert $)\emph{ }is defined as the minimum
size a family $D\subseteq B$ such that for every $a\in A$ there is a $d\in D$
such that $a\longrightarrow d$. The invariant $\left(  A,B,\longrightarrow
\right)  $ is called a \emph{Borel invariant }if $A,B$ and $\longrightarrow$
are Borel subsets of some Polish space. Most (but not all) of the usual
invariants are in fact Borel invariants.

\qquad\ \ 

\begin{center}
The statement $\Diamond\left(  A,B,\longrightarrow\right)  $ means the
following: For every \textbf{Borel}$\ C:2^{<\omega_{1}}\longrightarrow A$
there is a $g:\omega_{1}\longrightarrow B$ such that for every $R\in$
$^{\omega_{1}}2$ the set $\left\{  \alpha\mid C\left(  R\upharpoonright
\alpha\right)  \longrightarrow g\left(  \alpha\right)  \right\}  $ is stationary.

\qquad\ \ \ \ 
\end{center}

Here a function $C:2^{<\omega_{1}}\longrightarrow A$ is Borel if
$C\upharpoonright\alpha$ is Borel for every $\alpha<\omega_{1}.$ We will write
$\Diamond\left(  \mathfrak{d}\right)  $ instead of $\Diamond\left(
\omega^{\omega},\omega^{\omega},\leq^{\ast}\right)  $ and $\Diamond\left(
\mathfrak{b}\right)  $ instead of $\Diamond\left(  \omega^{\omega}%
,\omega^{\omega},^{\ast}\ngeq\right)  .$

\newpage

\section{Completely separable \textsf{MAD} families}

A \textsf{MAD} family $\mathcal{A}$ is \emph{completely separable }if for
every $X\in\mathcal{I}\left(  \mathcal{A}\right)  ^{+}$ there is
$A\in\mathcal{A}$ such that $A\subseteq X.$ This type of \textsf{MAD} families
was introduced by Hechler in \cite{Hechler}. A year later, Shelah and
Erd\"{o}s asked the following question:

\begin{problem}
[Erd\"{o}s-Shelah]Is there a completely separable \textsf{MAD} family?
\end{problem}

\qquad\ \ \ \qquad\ \ \ 

It is easy to construct models where the previous question has a positive
answer. It was shown by Balcar and Simon (see \cite{DisjointRefinement}) that
such families exist assuming one of the following axioms: $\mathfrak{a=c},$
$\mathfrak{b=d},$ $\mathfrak{d\leq a}$ and $\mathfrak{s}=\omega_{1}.$ In
\cite{SANEplayer} (see also \cite{AlmostDisjointFamiliesandTopology} and
\cite{SplittingFamiliesandCompleteSeparability}) Shelah developed a novel and
powerful method to construct completely separable \textsf{MAD} families. He
used it to prove that there are such families if either $\mathfrak{s\leq a}$
or $\mathfrak{a<s}$ and a certain (so called) \textsf{PCF }hypothesis holds
(which holds for example, if the continuum is less than $\aleph_{\omega}$).
Since Shelah's construction of a completely separable \textsf{MAD} family
under $\mathfrak{s\leq a}$ is the key for our construction of a $+$-Ramsey
\textsf{MAD} family, we will present his construction in the
following\ section. It is worth mentioning that the method of Shelah has been
further developed in \cite{OnWeaklytightFamilies} and
\cite{SplittingFamiliesandCompleteSeparability} where it is proved that weakly
tight \textsf{MAD} families exist under $\mathfrak{s\leq b.}$ Dilip Raghavan
has recently found even more applications of this method, unfortunately, his
results are still unpublished.

\qquad\ \ \ \qquad\ 

In this section, we expose the construction of Shelah of a completely
separable \textsf{MAD} family under $\mathfrak{s\leq a}$. This exposition is
based on \cite{SplittingFamiliesandCompleteSeparability} and
\cite{AlmostDisjointFamiliesandTopology}, none of the results in this section
are due to the author.

\begin{lemma}
$\mathfrak{s}$ has uncountable cofinality.
\end{lemma}

\begin{proof}
We argue by contradiction. Let $\mathcal{S}$ be a splitting family of size
$\mathfrak{s}.$ We can then find $\left\{  \mathcal{S}_{n}\mid n\in
\omega\right\}  $ such that $\mathcal{S=}%
{\textstyle\bigcup}
\mathcal{S}_{n}$ and each $\mathcal{S}_{n}$ has size less than $\mathfrak{s}$
(so they are \textquotedblleft nowhere splitting\textquotedblright). We can
then recursively find a decreasing sequence $\mathcal{P}=\left\{  A_{n}\mid
n\in\omega\right\}  $ such that no element of $\mathcal{S}_{n}$ splits
$A_{n}.$ Let $B$ be a pseudointersection of $\mathcal{P}.$ It is easy to see
than no element of $\mathcal{S}$ splits $B,$ which is a contradiction.
\end{proof}

\qquad\ 

We will need the following proposition:

\begin{proposition}
[\cite{OnWeaklytightFamilies}]If $\mathcal{S}$ is an $\left(  \omega
,\omega\right)  $-splitting family, $\mathcal{A}$ an \textsf{AD} family and
$X\in\mathcal{I}\left(  \mathcal{A}\right)  ^{+}$ then there is $S\in
\mathcal{S}$ such that $X\cap S,$ $X\cap\left(  \omega\setminus S\right)
\in\mathcal{I}\left(  \mathcal{A}\right)  ^{+}.$
\end{proposition}

\begin{proof}
We may assume $\mathcal{A}$ is a \textsf{MAD} family (in other case we extend
it to a \textsf{MAD} family keeping $X$ as a positive set). Since
$X\in\mathcal{I}\left(  \mathcal{A}\right)  ^{+}$, there is $\left\{
A_{n}\mid n\in\omega\right\}  \subseteq\mathcal{A}$ such that $X\cap A_{n}$ is
infinite for every $n\in\omega.$ Since $\mathcal{S}$ is an $\left(
\omega,\omega\right)  $-splitting family there is $S\in\mathcal{S}$ that
$\left(  \omega,\omega\right)  $-splits $\left\{  X\cap A_{n}\mid n\in
\omega\right\}  $ and then $X\cap S,$ $X\cap\left(  \omega\setminus S\right)
\in\mathcal{I}\left(  \mathcal{A}\right)  ^{+}.$
\end{proof}

\qquad\ \ \ \ \ 

By the previous result, if $\mathcal{A}$ is an \textsf{AD} family,
$X\in\mathcal{I}\left(  \mathcal{A}\right)  ^{+}$ and $\mathcal{S}=\left\{
S_{\alpha}\mid\alpha<\mathfrak{s}\right\}  $ is an $\left(  \omega
,\omega\right)  $-splitting family then there are $\alpha<\mathfrak{s}$ and
$\tau_{X}^{\mathcal{A}}\in2^{\alpha}$ such that:

\begin{enumerate}
\item If $\beta<\alpha$ then $X\cap S_{\beta}^{1-\tau_{X}^{\mathcal{A}}\left(
\beta\right)  }\in\mathcal{I}\left(  \mathcal{A}\right)  .$

\item $X\cap S_{\alpha},$ $X\setminus S_{\alpha}\in\mathcal{I}\left(
\mathcal{A}\right)  ^{+}.$
\end{enumerate}

\qquad\ \ \ \ 

Clearly $\tau_{X}^{\mathcal{A}}$ $\in2^{<\mathfrak{s}}$ is unique and if
$Y\in\left[  X\right]  ^{\omega}\cap\mathcal{I}\left(  \mathcal{A}\right)
^{+}$ then $\tau_{Y}^{\mathcal{A}}$ extends $\tau_{X}^{\mathcal{A}}.$ We can
now prove the main result of this section:\qquad\ \ \ 

\begin{theorem}
[Shelah \cite{SANEplayer}]If $\mathfrak{s\leq a}$ then there is a completely
separable \textsf{MAD} family.
\end{theorem}

\begin{proof}
Let $\left[  \omega\right]  ^{\omega}=\left\{  X_{\alpha}\mid\alpha
<\mathfrak{c}\right\}  .$ We will recursively construct $\mathcal{A}=\left\{
A_{\alpha}\mid\alpha<\mathfrak{c}\right\}  $ and $\left\{  \sigma_{\alpha}%
\mid\alpha<\mathfrak{c}\right\}  \subseteq2^{<\mathfrak{s}}$ such that for
every $\alpha<\mathfrak{c}$ the following holds: (where $\mathcal{A}_{\alpha
}=\left\{  A_{\xi}\mid\xi<\alpha\right\}  $)

\begin{enumerate}
\item $\mathcal{A}_{\alpha}$ is an AD family.

\item If $X_{\alpha}\in\mathcal{I}\left(  \mathcal{A}_{\alpha}\right)  ^{+}$
then $A_{\alpha}\subseteq X_{\alpha}.$

\item If $\alpha\neq\beta$ then $\sigma_{\alpha}\neq\sigma_{\beta}.$

\item If $\xi<dom\left(  \sigma_{\alpha}\right)  $ then $A_{\alpha}%
\subseteq^{\ast}S_{\xi}^{\sigma_{\alpha}\left(  \xi\right)  }.$
\end{enumerate}

\qquad\ \ 

It is clear that if we manage to do this then we will have achieved to
construct a completely separable \textsf{MAD} family. Assume $\mathcal{A}%
_{\delta}=\left\{  A_{\xi}\mid\xi<\delta\right\}  $ has already been
constructed. Let $X=$ $X_{\delta}$ if $X_{\delta}\in\mathcal{I}\left(
\mathcal{A}_{\delta}\right)  ^{+}$ and if $X_{\delta}\in\mathcal{I}\left(
\mathcal{A}_{\delta}\right)  $ let $X$ be any other element of $\mathcal{I}%
\left(  \mathcal{A}_{\delta}\right)  ^{+}.$ We recursively find $\left\{
X_{s}\mid s\in2^{<\omega}\right\}  \subseteq\mathcal{I}\left(  \mathcal{A}%
_{\delta}\right)  ^{+}$ ,$\left\{  \eta_{s}\mid s\in2^{<\omega}\right\}
\subseteq2^{<\mathfrak{s}}$ and $\left\{  \alpha_{s}\mid s\in2^{<\omega
}\right\}  $ as follows:

\begin{enumerate}
\item $X_{\emptyset}=X.$

\item $\eta_{s}=\tau_{X_{s}}^{\mathcal{A}_{\delta}}$ and $\alpha
_{s}=dom\left(  \eta_{s}\right)  .$

\item $X_{s^{\frown}0}=X_{s}\cap S_{\alpha_{s}}$ and $X_{s^{\frown}1}%
=X_{s}\cap\left(  \omega\setminus S_{\alpha_{s}}\right)  .$
\end{enumerate}

\qquad\ \ 

Note that if $t\subseteq s$ then $X_{s}\subseteq X_{t}$ and $\eta_{t}%
\subseteq\eta_{s}.$ On the other hand, if $s$ is incompatible with $t$ then
$\eta_{s}$ and $\eta_{t}$ are incompatible. For every $f\in2^{\omega}$ let
$\eta_{f}=\bigcup\limits_{n\in\omega}\eta_{f\upharpoonright n}$. Since
$\mathfrak{s}$ has uncountable cofinality, each $\eta_{f}$ is an element of
$2^{<\mathfrak{s}}$ and if $f\neq g$ then $\eta_{f}$ and $\eta_{g}$ are
incompatible nodes of $2^{<\mathfrak{s}}.$ Since $\delta$ is smaller than
$\mathfrak{c}$ there is $f\in2^{\omega}$ such that there is no $\alpha<\delta$
such that $\sigma_{\alpha}$ extends $\eta_{f}.$ Since $\left\{
X_{f\upharpoonright n}\mid n\in\omega\right\}  $ is a decreasing sequence of
elements in $\mathcal{I}\left(  \mathcal{A}_{\delta}\right)  ^{+}$ so there is
$Y\in\mathcal{I}\left(  \mathcal{A}_{\delta}\right)  ^{+}$ such that
$Y\subseteq^{\ast}X_{f\upharpoonright n}$ for every $n\in\omega.$

\qquad\ \ \ \qquad\ \ \ \ \ 

Letting $\beta=dom\left(  \eta_{f}\right)  ,$ we claim that if $\xi<\beta$
then $Y\cap S_{\xi}^{1-\eta_{f}\left(  \xi\right)  }\in\mathcal{I}\left(
\mathcal{A}\right)  .$ To prove this, let $n$ be the first natural number such
that $\xi<dom\left(  \eta_{f\upharpoonright n}\right)  .$ By our construction,
we know that $X_{f\upharpoonright n}\cap S_{\xi}^{1-\eta_{f}\left(
\xi\right)  }\in\mathcal{I}\left(  \mathcal{A}\right)  $ and since
$Y\subseteq^{\ast}X_{f\upharpoonright n}$ the result follows.

\qquad\ \ \ 

For every $\xi<\beta$ let $F_{\xi}\in\left[  \mathcal{A}\right]  ^{<\omega}$
such that $Y\cap S_{\xi}^{1-\eta_{f}\left(  \xi\right)  }\subseteq^{\ast
}\bigcup F_{\xi}$ and let $W=\left\{  A_{\alpha}\mid\sigma_{\alpha}%
\subseteq\eta_{f}\right\}  .$ Let $\mathcal{D}=W\cup\bigcup\limits_{\xi<\beta
}F_{\xi}$ and note that $\mathcal{D}$ has size less than $\mathfrak{s},$ hence
it has size less than $\mathfrak{a}.$ In this way we conclude that
$Y\upharpoonright\mathcal{D}$ is not a \textsf{MAD} family, so there is
$A_{\delta}\in\left[  Y\right]  ^{\omega}$ that is almost disjoint with every
element of $\mathcal{D}$ and define $\sigma_{\delta}=\eta_{f}.$ We claim that
$A_{\delta}$ is almost disjoint with $\mathcal{A}_{\delta}.$ To prove this,
let $\alpha<\delta,$ in case $A_{\alpha}\in W$ we already know $A_{\alpha}\cap
A_{\delta}$ is finite so assume $A_{\alpha}\notin W.$ Letting $\xi
=\Delta\left(  \sigma_{\delta},\sigma_{\alpha}\right)  $ we know that
$A_{\alpha}\subseteq^{\ast}S_{\xi}^{1-\sigma_{\delta}\left(  \xi\right)  }$ so
$A_{\alpha}\cap A_{\delta}\subseteq^{\ast}\bigcup F_{\xi}$ but since $F_{\xi
}\subseteq\mathcal{D}$ we conclude that $A_{\delta}$ is almost disjoint with
$\bigcup F_{\xi}$ and then $A_{\alpha}\cap A_{\delta}$ must be finite.
\end{proof}

\qquad\ \ 

A key feature in the previous proof is that each $\mathcal{A}_{\delta
}=\left\{  A_{\xi}\mid\xi<\delta\right\}  $ is nowhere \textsf{MAD}.

\qquad\ \ \ 

\chapter{The principle $\left(  \ast\right)  $ of Sierpi\'{n}ski}

The \emph{principle }$\left(  \ast\right)  $\emph{ of Sierpi\'{n}ski }is the
following statement:

\qquad

\begin{center}
\textit{There is a family of functions }$\left\{  \varphi_{n}:\omega
_{1}\longrightarrow\omega_{1}\mid n\in\omega\right\}  $\textit{ such that for
every }$I\in\left[  \omega_{1}\right]  ^{\omega_{1}}$\textit{ there is }%
$n\in\omega$\textit{ for which }$\varphi_{n}\left[  I\right]  =\omega_{1}.$
\end{center}

\qquad\ \ \ 

It was introduced by Sierpi\'{n}ski and he proved that it is a consequence of
the Continuum Hypothesis. It was recently studied by Arnold W. Miller in
\cite{TheOntoMappingofSierpinski} and this was the motivation for this work.
This principle is related to the following type of sets:\qquad\ \ 

\begin{definition}
Let $\mathcal{I}$ be a $\sigma$-ideal on $\omega^{\omega}.$ We say $X=\left\{
f_{\alpha}\mid\alpha<\omega_{1}\right\}  \subseteq\omega^{\omega}$ is an
$\mathcal{I}$\emph{-Luzin set }if $X\cap A$ is at most countable for every
$A\in\mathcal{I}.$
\end{definition}

\qquad\ \ 

In this terminology, the Luzin sets are the $\mathcal{M}$-Luzin sets and the
Sierpi\'{n}ski sets are the $\mathcal{N}$-Luzin sets. Clearly the existence of
an $\mathcal{I}$-Luzin set implies \textsf{non}$\left(  \mathcal{I}\right)
=\omega_{1},$ but the converse is usually not true. For example, it was shown
by Shelah and Judah in \cite{KillingLuzinandSierpinskiSets} that there are no
Luzin or Sierpi\'{n}ski sets in the Miller model while \textsf{non}$\left(
\mathcal{M}\right)  =$ \textsf{non}$\left(  \mathcal{N}\right)  =\omega_{1}$
holds. \qquad\ \ \ \ \ \ \qquad\ \ \ 

\begin{definition}
\qquad\ \ \ \ 

\begin{enumerate}
\item Given $f\in\omega^{\omega}$ we define $ED\left(  f\right)  =\left\{
g\in\omega^{\omega}\mid\left\vert f\cap g\right\vert <\omega\right\}  .$

\item $\mathcal{IE}$ is the $\sigma$-ideal generated by $\left\{  ED\left(
f\right)  \mid f\in\omega^{\omega}\right\}  .$
\end{enumerate}
\end{definition}

\qquad\ \ \ 

It is easy to see that each $ED\left(  f\right)  $ is a meager set so
$\mathcal{IE}\subseteq\mathcal{M}.$ It is well known that \textsf{non}$\left(
\mathcal{IE}\right)  =$ \textsf{non}$\left(  \mathcal{M}\right)  $ (see
\cite{HandbookBlass}). The following result was proved by Miller, although the
implication from 3 to 1 is not explicit in \cite{TheOntoMappingofSierpinski}
(it is implicitly proved in lemma 6). The referee of \cite{SierpinskiOsvaldo}
found a very elegant and short proof of this result, which we reproduce here.

\begin{proposition}
[Miller \cite{TheOntoMappingofSierpinski}]The following are equivalent:

\qquad\ \ \ 

\begin{enumerate}
\item The principle $\left(  \ast\right)  $ of Sierpi\'{n}ski.

\item There is a family $\left\{  g_{\alpha}:\omega\longrightarrow\omega
_{1}\mid\alpha<\omega_{1}\right\}  $ with the property that for every
$g:\omega\longrightarrow\omega_{1}$ there is $\alpha<\omega_{1}$ such that if
$\beta>\alpha$ then $g_{\beta}\cap g$ is infinite.

\item There is an $\mathcal{IE}$-Luzin set.
\end{enumerate}
\end{proposition}

\begin{proof}
We will first show that 2 implies 1. Let $\left\{  g_{\alpha}:\omega
\longrightarrow\omega_{1}\mid\alpha<\omega_{1}\right\}  $ with the property
that for every $g:\omega\longrightarrow\omega_{1}$ there is $\alpha<\omega
_{1}$ such that if $\beta>\alpha$ then $g_{\beta}\cap g$ is infinite. For each
$n\in\omega,$ we define $\varphi_{n}:\omega_{1}\longrightarrow\omega_{1}$
where $\varphi_{n}\left(  \alpha\right)  =g_{\alpha}\left(  n\right)  .$ We
will show that this works. We argue by contradiction, assume there is
$I\in\left[  \omega_{1}\right]  ^{\omega_{1}}$ such that no $\varphi_{n}$ maps
$I$ onto $\omega_{1}.$ Define $g:\omega\longrightarrow\omega_{1}$ such that
$g\left(  n\right)  \notin\varphi_{n}\left[  I\right]  $ for every $n\in
\omega.$ Let $\alpha<\omega_{1}$ such that if $\alpha<\beta<\omega_{1}$ then
$g_{\beta}\cap g$ is infinite. Since $I$ is uncountable, we can find $\beta\in
I\setminus\alpha.$ Let $n\in\omega$ such that $g\left(  n\right)  =g_{\beta
}\left(  n\right)  =\varphi_{n}\left(  \beta\right)  .$ SInce $\beta\in I$
then $g\left(  n\right)  \in\varphi_{n}\left[  I\right]  $ which is a contradiction.

\qquad\ \ 

We will now prove that i implies 2. Let $\left\{  \varphi_{n}:\omega
_{1}\longrightarrow\omega_{1}\mid n\in\omega\right\}  $ such that for every
$I\in\left[  \omega_{1}\right]  ^{\omega_{1}}$ there is $n\in\omega$ for which
$\varphi_{n}\left[  I\right]  =\omega_{1}.$ We first claim that in fact, there
are infinitely many $n\in\omega$ for which $\varphi_{n}\left[  I\right]
=\omega_{1}.$ Assume this is not the case, so there is $I\in\left[  \omega
_{1}\right]  ^{\omega_{1}}$ for which the set $A=\left\{  n\mid\varphi
_{n}\left[  I\right]  =\omega_{1}\right\}  $ is finite. For each $n\in A,$ we
can find $\alpha_{n}<\omega_{1}$ such that $J=I\setminus%
{\textstyle\bigcup\limits_{n\in\omega}}
\varphi_{n}^{-1}\left(  \left\{  \alpha_{n}\right\}  \right)  $ is
uncountable. Clearly, $J$ can not be mapped onto $\omega_{1}$ by any
$\varphi_{n},$ which is a contradiction. We now define $g_{\alpha}%
:\omega\longrightarrow\omega_{1}$ where $g_{\alpha}\left(  n\right)
=\varphi_{n}\left(  \alpha\right)  ,$ we will prove that the family $\left\{
g_{\alpha}\mid\alpha<\omega_{1}\right\}  $ has the desired properties. Assume
this is not the case, so there is $g:\omega\longrightarrow\omega_{1}$ and
$I\in\left[  \omega_{1}\right]  ^{\omega_{1}}$ such that $g\cap g_{\alpha}$ is
finite for every $\alpha\in I.$ We can even assume there is $m\in\omega$ such
that if $n>m$ and $\alpha\in I$ then $g_{\alpha}\left(  n\right)  \neq
g\left(  n\right)  .$ Therefore, if $n>m$ then $g\left(  n\right)
\notin\varphi_{n}\left[  I\right]  ,$ which is a contradiction.

We will now prove that 3 implies 2. Let $\mathcal{A}=\left\{  A_{\alpha}%
\mid\omega\leq\alpha<\omega_{1}\right\}  $ be an almost disjoint family. Since
there is an $\mathcal{IE}$-Luzin set, for each $\alpha$ we can find a family
$\mathcal{F}_{\alpha}=\left\{  f_{\alpha\beta}:A_{\alpha}\longrightarrow
\alpha\mid\beta<\omega_{1}\right\}  $ such that for every $g:A_{\alpha
}\longrightarrow\alpha$ there is $\delta$ such that if $\beta>\delta$ then
$f_{\alpha\beta}\cap g$ is infinite. Since $\mathcal{A}$ is an almost disjoint
family, we can then construct a family $\mathcal{G}=\left\{  g_{\beta}%
:\omega\longrightarrow\omega_{1}\mid\omega\leq\beta<\omega_{1}\right\}  $ such
that $f_{\alpha\beta}=^{\ast}g_{\beta}\upharpoonright A_{\alpha}$ for every
$\alpha<\beta<\omega_{1}.$ We need to prove that for every $g:\omega
\longrightarrow\omega_{1}$ there is $\alpha<\omega_{1}$ such that if
$\beta>\alpha$ then $g_{\beta}\cap g$ is infinite. First we find $\delta$ such
that $g:\omega\longrightarrow\delta$ and then we know there is $\gamma$ such
that if $\beta>\gamma$ then $f_{\delta\beta}\cap\left(  g\upharpoonright
A_{\delta}\right)  $ is infinite. It then follows that if $\beta>max\left\{
\delta,\gamma\right\}  $ then $g_{\beta}\upharpoonright A_{\delta}=^{\ast
}f_{\delta\beta},$ so $\left\vert g_{\beta}\cap g\right\vert =\omega.$
\end{proof}

In this way, the existence of a Luzin set implies the principle $\left(
\ast\right)  $ of Sierpi\'{n}ski while it implies \textsf{non}$\left(
\mathcal{M}\right)  =\omega_{1}.$ Miller then asked in
\cite{TheOntoMappingofSierpinski} if the principle $\left(  \ast\right)  $ of
Sierpi\'{n}ski is a consequence of \textsf{non}$\left(  \mathcal{M}\right)
=\omega_{1}$ and we will show that this is indeed the case. We will then prove
(with the aid of an inaccessible cardinal) that while \textsf{non}$\left(
\mathcal{M}\right)  =\omega_{1}$ implies the existence of a $\mathcal{IE}%
$-Luzin set, it does not imply the existence of a non-meager $\mathcal{IE}%
$-Luzin set.

\qquad\ \ \qquad\ \ 

We will now show that the principle $\left(  \ast\right)  $ of Sierpi\'{n}ski
follows by \textsf{non}$\left(  \mathcal{M}\right)  =\omega_{1},$ answering
the question of Miller. By $Partial\left(  \omega^{\omega}\right)  $ we shall
denote the set of all infinite partial functions from $\omega$ to $\omega.$ We
start with the \ following lemma:

\begin{lemma}
If \textsf{non}$\left(  \mathcal{M}\right)  =\omega_{1}$ then there is a
family $X=\left\{  f_{\alpha}\mid\alpha<\omega_{1}\right\}  $ with the
following properties:

\qquad\ \ 

\begin{enumerate}
\item Each $f_{\alpha}$ is an infinite partial function from $\omega$ to
$\omega.$

\item The set $\left\{  dom\left(  f_{\alpha}\right)  \mid\alpha<\omega
_{1}\right\}  $ is an almost disjoint family.

\item For every $g:\omega\longrightarrow\omega$ there is $\alpha<\omega_{1}$
such that $f_{\alpha}\cap g$ is infinite.
\end{enumerate}
\end{lemma}

\begin{proof}
Let $\omega^{<\omega}=\left\{  s_{n}\mid n\in\omega\right\}  $ and we define
$H:\omega^{\omega}\longrightarrow Partial\left(  \omega^{\omega}\right)  $
where the domain of $H\left(  f\right)  $ is $\left\{  n\mid s_{n}\sqsubseteq
f\right\}  $ and if $n\in dom\left(  H\left(  f\right)  \right)  $ then
$H\left(  f\right)  \left(  n\right)  =f\left(  \left\vert s_{n}\right\vert
\right)  .$ It is easy to see that if $f\neq g$ then $dom\left(  H\left(
f\right)  \right)  $ and $dom\left(  H\left(  g\right)  \right)  $ are almost disjoint.

\qquad\ \ 

Given $g:\omega\longrightarrow\omega$ we define $N\left(  g\right)  =\left\{
f\in\omega^{\omega}\mid\left\vert H\left(  f\right)  \cap g\right\vert
<\omega\right\}  .$ It then follows that $N\left(  g\right)  $ is a meager set
since $N\left(  g\right)  =\bigcup\limits_{k\in\omega}N_{k}\left(  g\right)  $
where $N_{k}\left(  g\right)  =\left\{  f\in\omega^{\omega}\mid\left\vert
H\left(  f\right)  \cap g\right\vert <k\right\}  $ and it is easy to see that
each $N_{k}\left(  g\right)  $ is a nowhere dense set. Finally, if $X=\left\{
h_{\alpha}\mid\alpha<\omega_{1}\right\}  $ is a non-meager set then $H\left[
X\right]  $ is the family we were looking for.
\end{proof}

\qquad\ \ \ 

\qquad\ \ \ With the previous lemma we can prove the following:

\begin{proposition}
If \textsf{non}$\left(  \mathcal{M}\right)  =\omega_{1}$ then the principle
$\left(  \ast\right)  $ of Sierpi\'{n}ski is true.
\end{proposition}

\begin{proof}
Let $X=\left\{  f_{\alpha}\mid\alpha<\omega_{1}\right\}  $ be a family as in
the previous lemma. We will build a $\mathcal{IE}$-Luzin set $Y=\left\{
h_{\alpha}\mid\alpha<\omega_{1}\right\}  .$ For simplicity, we may assume
$\left\{  dom\left(  f_{n}\right)  \mid n\in\omega\right\}  $ is a partition
of $\omega.$

\qquad\ \ \ 

For each $n\in\omega,$ let $h_{n}$ be any constant function. Given $\alpha
\geq\omega,$ enumerate it as $\alpha=\left\{  \alpha_{n}\mid n\in
\omega\right\}  $ and then we recursively define $B_{0}=dom\left(
f_{\alpha_{0}}\right)  $ and $B_{n+1}=dom\left(  f_{\alpha_{n}}\right)
\,\backslash\left(  B_{0}\cup...\cup B_{n}\right)  .$ Clearly $\left\{
B_{n}\mid n\in\omega\right\}  $ is a partition of $\omega.$ Let $h_{\alpha
}=\bigcup\limits_{n\in\omega}f_{\alpha_{n}}\upharpoonright B_{n},$ it then
follows that $Y=\left\{  h_{\alpha}\mid\alpha<\omega_{1}\right\}  $ is an
$\mathcal{IE}$-Luzin set.
\end{proof}

\qquad\ \ \ \ \ 

It is not hard to see that the $\mathcal{IE}$-Luzin set constructed in the
previous proof is meager. One may then wonder if it is possible to construct a
non-meager $\mathcal{IE}$-Luzin set from \textsf{non}$\left(  \mathcal{M}%
\right)  =\omega_{1}.$ We will prove that this is not the case. This will be
achieved by using Todorcevic's method of forcing with models as side
conditions (see \cite{PartitionProblems} for more on this very useful
technique). It is currently unknown if there is a non-meager $\mathcal{IE}%
$-Luzin set in the Miller model.

\begin{definition}
We define the forcing $\mathbb{P}_{cat}$ as the set of all $p=\left(
s_{p},\overline{M}_{p},F_{p}\right)  \ $with the following properties:

\qquad\ \ \ 

\begin{enumerate}
\item $s_{p}\in\omega^{<\omega}$ (this is usually referred as \emph{the stem
}of $p$).

\item $\overline{M}_{p}=\left\{  M_{0},...,M_{n}\right\}  $ is an $\in$-chain
of countable elementary submodels of $\mathsf{H(}\left(  2^{\mathfrak{c}%
}\right)  ^{++}).$

\item $F_{p}:\overline{M}_{p}\longrightarrow\omega^{\omega}.$

\item $s_{p}\cap F_{p}\left(  M_{i}\right)  =\emptyset$ for every $i\leq n.$

\item $F_{p}\left(  M_{i}\right)  \notin M_{i}$ and if $i<n$ then
$F_{p}\left(  M_{i}\right)  \in M_{i+1}.$

\item $F_{p}\left(  M_{i}\right)  $ is a Cohen real over $M_{i}$ (i.e. if
$Y\in M_{i}$ is a meager set then $F_{p}\left(  M_{i}\right)  \notin Y$).
\end{enumerate}

\qquad\ \ 

Finally, if $p,q\in\mathbb{P}_{cat}$ then $p\leq q$ if $s_{q}\subseteq s_{p},$
$\overline{M}_{q}\subseteq\overline{M}_{p}$ and $F_{q}\subseteq F_{p}.$
\end{definition}

\qquad\ \ \qquad\ \ \qquad\ 

The following lemma is easy and it is left to the reader:

\begin{lemma}
\qquad\ \ \ 

\begin{enumerate}
\item If $M\preceq\mathsf{H(}\left(  2^{\mathfrak{c}}\right)  ^{+++})$ is
countable and $p\in M\cap\mathbb{P}_{cat}$ then there is $f\in\omega^{\omega}$
such that if $N=M\cap\mathsf{H(}\left(  2^{\mathfrak{c}}\right)  ^{++})$ then
$\overline{p}=\left(  s_{p},\overline{M}_{p}\cup\left\{  N\right\}  ,F_{p}%
\cup\left\{  \left(  N,f\right)  \right\}  \right)  $ is a condition of
$\mathbb{P}_{cat}$ and it extends $p.$

\item If $n\in\omega$ then $D_{n}=\left\{  p\in\mathbb{P}_{cat}\mid n\subseteq
dom\left(  s_{p}\right)  \right\}  $ is an open dense subset of $\mathbb{P}%
_{cat}.$
\end{enumerate}
\end{lemma}

\qquad\ \ \ \qquad\ \ \ 

We will now prove that $\mathbb{P}_{cat}$ is a proper forcing by applying the
usual \textquotedblleft side conditions trick\textquotedblright.

\begin{lemma}
$\mathbb{P}_{cat}$ is a proper forcing.
\end{lemma}

\begin{proof}
Let $p\in\mathbb{P}_{cat}$ and $M$ a countable elementary submodel of
$\mathsf{H(}\left(  2^{\mathfrak{c}}\right)  ^{+++})$ such that $p\in M$. By
the previous lemma, we know there is $f\in\omega^{\omega}$ such that
$\overline{p}=\left(  s_{p},\overline{M}_{p}\cup\left\{  N\right\}  ,F_{p}%
\cup\left\{  \left(  N,f\right)  \right\}  \right)  \in\mathbb{P}_{cat}$
(where $N=M\cap\mathsf{H(}\left(  2^{\mathfrak{c}}\right)  ^{++})$). We will
now prove $\overline{p}$ is an $\left(  M,\mathbb{P}_{cat}\right)  $-generic condition.

\qquad\ \ 

Let $D\in M$ be an open dense subset of $\mathbb{P}_{cat}$ and $q\leq
\overline{p}$ (we may even assume $q\in D$). We must prove that $q$ is
compatible with an element of $M\cap D.$ In order to achieve this, let
$q_{M}=\left(  s_{q},\overline{M}_{q}\cap M,F_{q}\cap M\right)  .$ It is easy
to see $q_{M}$ is a condition as well as an element of $M$. By elementarity,
we can find $r\in M\cap D$ such that $r\leq q_{M}$ and $s_{r}=s_{q}.$ It is
then easy to see that $r$ and $q$ are compatible (this is easy since $r$ and
$q$ share the same stem).
\end{proof}

\qquad\ 

The next lemma shows that $\mathbb{P}_{cat}$ destroys all the ground model
non-meager $\mathcal{IE}$-Luzin families.

\begin{lemma}
If $X=\left\{  f_{\alpha}\mid\alpha<\omega_{1}\right\}  \subseteq
\omega^{\omega}$ is a non-meager set then $\mathbb{P}_{cat}$ adds a function
that is almost disjoint with uncountably many elements of $X.$
\end{lemma}

\begin{proof}
Given a generic filter $G\subseteq\mathbb{P}_{cat}$, we denote the
\emph{generic real }by $f_{gen}$ i.e. $f_{gen}$ is the union of all the stems
of the elements in $G.$ We will show $f_{gen}$ is forced to be almost disjoint
with uncountably many elements of $X.$ Let $p\in\mathbb{P}_{cat}$ with stem
$s_{p}$ and $\alpha<\omega_{1}.$ Choose $t\in\omega^{<\omega}$ with the same
length as $s_{p}$ but disjoint with it. Let $Y=\left\{  g_{\beta}\mid
\alpha<\beta<\omega_{1}\right\}  $ where $g_{\beta}=t\cup\left(  f_{\beta
}\upharpoonright\lbrack\left\vert t\right\vert ,\omega)\right)  .$ It is easy
to see that $Y$ is a non-meager set and then we can find $\beta>\alpha$ and
$q\leq p$ such that $g_{\beta}$ is in the image of $F_{q}.$ In this way,
$f_{gen}$ is forced by $q$ to be disjoint from $g_{\beta},$ so it will be
almost disjoint with $f_{\beta}.$
\end{proof}

\qquad\ \ 

We say a forcing notion $\mathbb{P}$ \emph{destroys category }if there is
$p\in\mathbb{P}$ such that $p\Vdash``\omega^{\omega}\cap V\in\mathcal{M}%
\textquotedblright.$ The following is a well known result, but I was unable to
find a reference for it:

\begin{proposition}
Let $\mathbb{P}$ be a partial order. Then $\mathbb{P}$ destroys category if
and only if $\mathbb{P}$ adds an eventually different real.
\end{proposition}

\begin{proof}
If $\mathbb{P}$ adds an eventually different real then clearly $\mathbb{P}$
destroys category, so we only need to prove the other implication. Let
$\mathbb{P}$ be a partial order that destroys category. If $\mathbb{P}$ adds a
dominating real the result is obvious, so let us assume $\mathbb{P}$ does not
add dominating reals.

\qquad\ \ 

Let $G\subseteq\mathbb{P}$ be a generic filter. Then there is a chopped real
$\left(  x,\mathcal{P}\right)  \in V\left[  G\right]  $ such that $2^{\omega
}\cap V\subseteq\lnot Match\left(  x,\mathcal{P}\right)  .$ Since $\mathbb{P}$
does not add dominating reals, then there is a ground model interval partition
$\mathcal{R}$ such that there are infinitely many intervals of $\mathcal{R}$
that contain an interval of $\mathcal{P}.$ Let $W=2^{<\omega}\times\left[
\mathcal{R}\right]  ^{<\omega}$ which clearly is a ground model countable set.
We work in $V\left[  G\right]  ,$ let $Z=\left\{  R_{i}\mid i\in
\omega\right\}  \subseteq\mathcal{R}$ be such that every $R_{i}$ contains an
interval from $\mathcal{P}.$ We now define the function $f:\omega
\longrightarrow W$ where $f\left(  n\right)  =\left(  x\upharpoonright\left(
max\left(  R_{2^{n}}\right)  +1\right)  ,\left\{  R_{0},...,R_{2^{n}}\right\}
\right)  .$ We claim that $f$ is an eventually different real. Assume this is
not the case, so there is $g\in V$ such that $g\cap f$ is infinite. We may
assume each $g\left(  n\right)  $ is of the form $\left(  s_{n},F_{n}\right)
$ where $F_{n}\in\left[  \mathcal{R}\right]  ^{2^{n}+1}$ and $dom\left(
s_{n}\right)  $ is a superset of all intervals of $F_{n}.$ We can then
recursively define $y\in2^{\omega}\cap V$ such that if $g\left(  n\right)
=\left(  s_{n},F_{n}\right)  $ then there is $R\in F_{n}$ for which
$y\upharpoonright F_{n}=s_{n}\upharpoonright F_{n}.$ It is then easy to see
that $y$ matches $\left(  x,\mathcal{P}\right)  ,$ which is a contradiction.
\end{proof}

\qquad\ \ \ \qquad\ \ \ 

Given a Polish space $X,$ we denote by \textsf{nwd}$\left(  X\right)  $ the
ideal of all nowhere dense subsets of $X.$ We will need the following result
of Kuratowski and Ulam (see \cite{Kechris}):

\begin{proposition}
[Kuratowski-Ulam]Let $X$ and $Y$ two Polish spaces. If $N\subseteq X\times Y$
is a nowhere dense set, then $\{x\in X\mid N_{x}\in$ \textsf{nwd}$\left(
Y\right)  \}$ is comeager (where $N_{x}=\left\{  y\mid\left(  x,y\right)  \in
N\right\}  $).
\end{proposition}

\qquad\ \ 

As a consequence of the Kuratowski-Ulam theorem we get the following result:

\begin{lemma}
Let $p\in\mathbb{P}_{cat},$ $\overline{M}_{p}=\left\{  M_{0},...,M_{n}%
\right\}  $ and $i\leq n.$ Let $g_{j}=F_{p}\left(  M_{i+j}\right)  $ and
$m=n-i.$ If $D\in M_{i}$ and $D\subseteq\left(  \omega^{\omega}\right)
^{m+1}$ is a nowhere dense set, then $\left(  g_{0},...,g_{m}\right)  \notin
D.$
\end{lemma}

\begin{proof}
We prove it by induction on $m.$ If $m=0$ this is true just by the definition
of $\mathbb{P}_{cat}$. Assume this is true for $m$ and we will show it is also
true for $m+1.$ Since $D\subseteq\left(  \omega^{\omega}\right)  ^{m+2}$ is a
nowhere dense set, then by the Kuratowski-Ulam theorem we conclude that
$A=\{h\in\omega^{\omega}\mid D_{h}\in$ \textsf{nwd}$(\left(  \omega^{\omega
}\right)  ^{m+1})\}$ is comeager and note it is an element of $M_{i}.$ In this
way, $g_{0}\in A$ so $D_{g_{0}}\in$ \textsf{nwd}$(\left(  \omega^{\omega
}\right)  ^{m+1})$ and it is an element of $M_{i+1}.$ By the inductive
hypothesis we know $\left(  g_{1},...,g_{m+1}\right)  \notin D_{g_{0}}$ which
implies $\left(  g_{0},...,g_{m+1}\right)  \notin D.$
\end{proof}

\qquad\ \ \ 

We will prove that $\mathbb{P}_{cat}$ does not destroy category and this is a
consequence of the following result:

\begin{lemma}
Let $p\in\mathbb{P}_{cat}$ and $\dot{g}$ a $\mathbb{P}_{cat}$-name for an
element of $\omega^{\omega}.$ Letting $\overline{M}=\left\langle M_{n}\mid
n\in\omega\right\rangle $ be an $\in$-chain of elementary submodels of
$\mathsf{H(}\left(  2^{\mathfrak{c}}\right)  ^{+++}),$ $h:\omega
\longrightarrow\omega$ and $\left\{  A_{n}\mid n\in\omega\right\}
\subseteq\left[  \omega\right]  ^{\omega}$ a family of infinite pairwise
disjoint sets with the following properties:

\qquad\ \ 

\begin{enumerate}
\item $p,\dot{g}\in M_{0}.$

\item $h\upharpoonright A_{n}\in M_{n+1}.$

\item If $f\in M_{n}\cap\omega^{\omega}$ then $f\cap\left(  h\upharpoonright
A_{n}\right)  $ is infinite.
\end{enumerate}

\qquad\ \ 

Then there is a condition $q\leq p$ such that $q\Vdash``\left\vert h\cap
\dot{g}\right\vert =\omega\textquotedblright.$
\end{lemma}

\begin{proof}
Let $M=\bigcup\limits_{n\in\omega}M_{n}$ and define $h_{n}=h\upharpoonright
A_{n}\in M_{n+1}.$ We know there is some $f\in\omega^{\omega}$ such that
$\overline{p}=\left(  s_{p},\overline{M}_{p}\cup\left\{  N\right\}  ,F_{p}%
\cup\left\{  \left(  N,f\right)  \right\}  \right)  \in\mathbb{P}_{cat}$
(where $N=M\cap\mathsf{H(}\left(  2^{\mathfrak{c}}\right)  ^{++})$). We will
now prove that $\overline{p}$ forces that $\dot{g}$ and $h$ will have infinite
intersection. We may assume $A_{n}\cap n=\emptyset$ for every $n\in\omega.$

\qquad\ \ 

Pick any $q\leq\overline{p}$ and $k_{0}\in\omega.$ We must find an extension
of $q$ that forces that $\dot{g}$ and $h$ share a common value bigger than
$k_{0}.$ We first find $n>k_{0}$ such that $q^{\prime}=\left(  s_{q}%
,\overline{M}_{q}\cap M,F_{q}\cap M\right)  \in M_{n}.$ Let $m=\left\vert
\overline{M}_{q}\backslash\overline{M}_{q^{\prime}}\right\vert $ and now we
define $D$ as the set of all $t\in\omega^{<\omega}$ such that there are $l\in
A_{n}$ and $r\in\mathbb{P}_{cat}$ with the following properties:\qquad\ \ 

\begin{enumerate}
\item $r\leq q^{\prime}$.

\item $r\in M_{n}.$

\item $s_{q}\subseteq t$ and the stem of $r$ is $t.$

\item $r\Vdash``\dot{g}\left(  l\right)  =h_{n}\left(  l\right)
\textquotedblright.$
\end{enumerate}

\qquad\ \ 

It is easy to see that $D$ is an element of $M_{n+1}.$ We now define $N\left(
D\right)  \subseteq\left(  \omega^{\omega}\right)  ^{m}$ as the set of all
$\left(  f_{1},...,f_{m}\right)  \in\left(  \omega^{\omega}\right)  ^{m}$ such
that $\left(  f_{1}\cup...\cup f_{m}\right)  \cap t\nsubseteq s_{q}$ for every
$t\in D.$ We claim that $N\left(  D\right)  $ is a nowhere dense set.

\qquad\ \ 

Let $z_{1},...,z_{m}\in\omega^{<\omega}$ and we may assume all of them have
the same length and it is bigger than the length of $s_{q}.$ We know
$q^{\prime}=\left(  s_{q},\overline{M}_{q^{\prime}},F_{q^{\prime}}\right)  $
and let $im\left(  F_{q^{\prime}}\right)  =\left\{  f_{_{1}},...,f_{_{k}%
}\right\}  $ (where $im$ denotes the image of the function). Let $t_{0}$ be
any extension of $s_{q}$ such that $t_{0}\cap\left(  f_{_{1}}\cup...\cup
f_{_{k}}\cup z_{1}\cup...z_{m}\right)  \subseteq s_{q}$ and $\left\vert
t_{0}\right\vert =\left\vert z_{1}\right\vert .$ In this way, $q_{0}=\left(
t_{0},\overline{M}_{q^{\prime}},F_{q^{\prime}}\right)  $ is a condition and is
an element of $M_{n}.$ Inside $M_{n},$ we build a decreasing sequence
$\left\langle q_{i}\right\rangle _{i\in\omega}$ (starting from the $q_{0}$ we
just constructed) in such a way that $q_{i}$ determines $\dot{g}%
\upharpoonright i$ andits stem is $t_{i}.$ In this way, there is a function
$u:\omega\longrightarrow\omega\in M_{n}$ such that $q_{i}\Vdash``\dot
{g}\upharpoonright i=u\upharpoonright i\textquotedblright.$ Since $u\in M_{n}$
we may then find $l\in A_{n}$ such that $u\left(  l\right)  =h_{n}\left(
l\right)  .$ Let $t=t_{l+1}$ and $r=q_{l+1},$ we may then find $z_{i}^{\prime
}\supseteq z_{i}$ such that $t\cap\left(  z_{1}^{\prime}\cup...\cup
z_{m}^{\prime}\right)  \subseteq s_{q}$ and $\left\vert z_{i}^{\prime
}\right\vert =\left\vert t\right\vert .$ In this way, we conclude that
$\left\langle z_{1}^{\prime},...,z_{m}^{\prime}\right\rangle \cap N\left(
D\right)  =\emptyset$ (where $\left\langle z_{1}^{\prime},...,z_{m}^{\prime
}\right\rangle =\left\{  \left(  g_{1},...,g_{m}\right)  \mid\forall i\leq
m\left(  z_{i}^{\prime}\subseteq g_{i}\right)  \right\}  $) so we conclude
$N\left(  D\right)  $ is a nowhere dense set.

\qquad\ \ 

Let $g_{1},...,g_{m}$ be the elements of $im\left(  F_{q}\right)  $ that are
not in $M.$ Since $D\in N$ then by the previous lemma, we know that $\left(
g_{1},...,g_{m}\right)  \notin N\left(  D\right)  .$ This means there are
$l\in A_{n},$ $t\in\omega^{<\omega}$ and $r\in M_{n}$ such that $r\leq
q^{\prime},$ whose stem is $t$ and $r\Vdash``\dot{g}\left(  l\right)
=h_{n}\left(  l\right)  \textquotedblright$ with the property that
$t\cap\left(  g_{1}\cup...\cup g_{m}\right)  \subseteq s_{q},$ but since $q$
is a condition, it follows that $t\cap\left(  g_{1}\cup...\cup g_{m}\right)
=\emptyset.$ In this way, $r$ and $q$ are compatible, which finishes the proof.
\end{proof}

\qquad\ \ \ \ \ \ \ \ \ \ \ \ \ \qquad\qquad\qquad\ \ \ \ \ \ \ \ \ \ 

As a corollary we get the following:

\begin{corollary}
$\mathbb{P}_{cat}$ does not destroy category.
\end{corollary}

\qquad\ \ 

Unfortunately, the iteration of forcings that do not destroy category may
destroy category (this may even occur at a two step iteration, see
\cite{Barty}). Luckily for us, the iteration of the $\mathbb{P}_{cat}$ forcing
does not destroy category as we will prove soon. First we need a couple of lemmas.

\qquad\ \ \ \qquad\ \ \ \ \ 

\begin{lemma}
Let $\mathbb{P}$ be a proper forcing that does not destroy category and
$p\in\mathbb{P}.$ If $\dot{S}$ is a $\mathbb{P}$-name for a countable set of
reals, then there is $q\leq p$ and $h:\omega\longrightarrow\omega$ such that
$q\Vdash``\forall f\in\dot{S}\left(  \left\vert f\cap h\right\vert
=\omega\right)  \textquotedblright.$
\end{lemma}

\begin{proof}
First note that if $\dot{f}_{0},...,\dot{f}_{n}$ are $\mathbb{P}$-names for
reals, then there is $q\leq p$ and $h:\omega\longrightarrow\omega$ such that
$q$ forces $\dot{f}_{i}$ and $h$ have infinite intersection for every $i\leq
n.$ To prove this, we choose a partition $\left\{  A_{0,},...,A_{n}\right\}  $
of $\omega$ into infinite sets and let $\dot{g}_{i}$ be the $\mathbb{P}$-name
of $\dot{f}_{i}\upharpoonright A_{i}.$ Since $\mathbb{P}$ does not destroy
category, there are $q\leq p$ and $h_{i}:A_{i}\longrightarrow\omega$ such that
$q$ forces that $h_{i}$ and $\dot{f}_{i}$ have infinite intersection. Clearly
$q$ and $h=\bigcup h_{i}$ have the desired properties.

\qquad\ \ 

To prove the lemma, let $\dot{S}=\left\{  \dot{g}_{n}\mid n\in\omega\right\}
$ and fix a partition $\left\{  A_{n}\mid n\in\omega\right\}  $ of $\omega$ in
infinite sets. By the previous remark, we know there is a $\mathbb{P}$-name
$\dot{F}$ such that $p\Vdash``\dot{F}:\omega\longrightarrow Partial\left(
\omega^{\omega}\right)  \cap V\textquotedblright$ such that every $\dot
{F}\left(  n\right)  $ is forced to be a function with domain $A_{n}$ and
intersects infinitely $\dot{g}_{0}\upharpoonright A_{n},...,\dot{g}%
_{n}\upharpoonright A_{n}.$ Since $\mathbb{P}$ is a proper forcing, we can
find $q\leq p$ and $M\in V$ a countable subset of $Partial\left(
\omega^{\omega}\right)  $ such that $q\Vdash``\dot{F}:\omega\longrightarrow
M\textquotedblright.$ We know $\mathbb{P}$ does not destroy category and $M$
is countable, so there must be $r\leq q$ and $H:\omega\longrightarrow M$ such
that $r\Vdash``\exists^{\infty}n(\dot{F}\left(  n\right)  =H\left(  n\right)
)\textquotedblright.$ We may assume that the domain of $H\left(  n\right)  $
is $A_{n}$ for every $n\in\omega.$ Finally, we define $h=\bigcup
\limits_{n\in\omega}H\left(  n\right)  $ and it is easy to see that $r$ forces
that $h$ has infinite intersection with every element of $\dot{S}.$
\end{proof}

\qquad\ \ 

We will also need the following lemma.

\begin{lemma}
Let $\mathbb{P}$ be a proper forcing that does not destroy category,
$G\subseteq\mathbb{P}$ a generic filter and $X$ any set. Then there are
$\overline{M}=\left\{  M_{n}\mid n\in\omega\right\}  \subseteq V,$ $P=\left\{
A_{n}\mid n\in\omega\right\}  \subseteq V$ and $h:\omega\longrightarrow\omega$
in $V$ with the following properties:

\qquad\ \ 

\begin{enumerate}
\item Each $M_{n}$ is a countable elementary submodel of $H\left(
\kappa\right)  $ for some big enough $\kappa$ (in $V$).

\item $X\in M_{0}$ and $M_{n}\in M_{n+1}$ for every $n\in\omega.$

\item $P$ is a family of infinite pairwise disjoint subsets of $\omega$.

\item $P,\overline{M}\in V\left[  G\right]  $ (while $\overline{M}$ is a
subset of $V,$ in general it will not be a ground model set, the same is true
for $P$).

\item $G\cap M_{n}$ is a $\left(  M_{n},\mathbb{P}\right)  $-generic filter
for every $n\in\omega.$

\item $h\upharpoonright A_{n}\in M_{n+1}$ and if $f\in M_{n}\left[  G\right]
$ then $h\upharpoonright A_{n}\cap f$ is infinite.
\end{enumerate}
\end{lemma}

\begin{proof}
Let $r$ be any condition of $\mathbb{P}.$ We will prove that there is an
extension of $r$ that forces the existence of the desired objects. Let
$\left\{  B_{n}\mid n\in\omega\right\}  $ be any definable partition of
$\omega$ into infinite sets.

\qquad\ \ 

\begin{claim}
If $G\subseteq\mathbb{P}$ is a generic filter with $r\in G$ then (in $V\left[
G\right]  $) there is a sequence $\left\langle \left(  N_{i},p_{i}%
,h_{i}\right)  \mid i\in\omega\right\rangle $ such that for every $i\in\omega$
the following holds:

\qquad\ 

\begin{enumerate}
\item $N_{i}\in V$ is a countable elementary submodel of $H\left(
\kappa\right)  $ (the $H\left(  \kappa\right)  $ of the ground model).

\item $r,X\in N_{0}$ and $N_{i}\in N_{i+1}$.

\item $p_{0}\leq r$ and \thinspace$\left\langle p_{k}\right\rangle
_{k\in\omega}$ is a decreasing sequence contained in $G.$

\item $p_{i}$ is $\left(  N_{i},\mathbb{P}\right)  $-generic.

\item $h_{i}:B_{i}\longrightarrow\omega$ is in $N_{i+1}$.

\item $p_{i}\Vdash``\forall f\in N_{i}[\dot{G}]\cap\omega^{\omega}\left(
\left\vert f\cap h_{i}\right\vert =\omega\right)  \textquotedblright.$
\end{enumerate}
\end{claim}

\qquad\ \ 

Assume the claim is false, so we can find $n\in\omega$ and a sequence
$R=\left\langle \left(  N_{i},p_{i},h_{i}\right)  \mid i\leq n\right\rangle $
that is maximal with the previous properties (the point $5$ is only demanded
for $i<n$). Let $p\in G$ be a condition forcing $R$ has all these features
(including the maximality). Back in $V,$ let $M$ be a countable elementary
submodel such that $\mathbb{P},p,R\in M.$ By the previous lemma, there is an
$\left(  M,\mathbb{P}\right)  $-generic condition $q\leq p$ and $g:B_{n+1}%
\longrightarrow\omega$ such that $g$ is forced by $q$ to intersect infinitely
every real of $M\left[  G\right]  .$ In this way, $q$ forces that $R$ could be
extended by adding $\left(  M,q,g\right)  $ but this is a contradiction since
$q\leq p$ so it forces $R$ was maximal. This finishes the proof of the claim.

\qquad\ \ 

Let $\langle(\dot{N}_{i},\dot{p}_{i},\dot{h}_{i})\mid i\in\omega\rangle$ be
the name of a sequence as in the claim. We can now define a name for a
function $\dot{F}$ from $\omega$ to $Partial\left(  \omega^{\omega}\right)
\cap V$ such that $r\Vdash``\forall n(\dot{F}\left(  n\right)  =\dot{h}%
_{n})\textquotedblright.$ As in the previous lemma, we can find a condition
$p\leq r$ and $H:\omega\longrightarrow Partial\left(  \omega^{\omega}\right)
$ such that $p\Vdash``\exists^{\infty}n(\dot{F}\left(  n\right)  =H\left(
n\right)  )\textquotedblright.$ We may assume the domain of $H\left(
n\right)  $ is $B_{n}$ and let $h=\bigcup\limits_{n\in\omega}H\left(
n\right)  .$ Let $\dot{Z}=\left\{  \dot{z}_{n}\mid n\in\omega\right\}  $ be a
name for a subset of $\omega$ such that $p\Vdash``\forall n\left(  F\left(
\dot{z}_{n}\right)  =H\left(  \dot{z}_{n}\right)  \right)  \textquotedblright%
.$ If $G\subseteq\mathbb{P}$ is a generic filter such that $p\in G$ then we
define $M_{n}=N_{\dot{z}_{n}\left[  G\right]  }$ and $A_{n}=B_{\dot{z}%
_{n}\left[  G\right]  },$ it is clear that these sets have the desired properties.
\end{proof}

\qquad\ \ 

From this we can conclude the following,

\qquad\ \ 

\begin{corollary}
If $\mathbb{P}$ is a proper forcing that does not destroy category then
$\mathbb{P\ast P}_{cat}$ does not destroy category.
\end{corollary}

\begin{proof}
Let $\dot{p}$ be a $\mathbb{P}$-name for a condition of $\mathbb{P}_{cat}$ and
$\dot{f}$ a $\mathbb{P}$-name for a $\mathbb{P}_{cat}$-name for a real. Let
$G\subseteq\mathbb{P}$ be a generic filter. By the previous lemma, there are
$h:\omega\longrightarrow\omega$ in $V,$ an $\in$-chain of elementary submodels
$\left\{  M_{n}\left[  G\right]  \mid n\in\omega\right\}  $ and a pairwise
disjoint family $\left\{  A_{n}\mid n\in\omega\right\}  $ of infinite subsets
of $\omega$ such that $\dot{p}\left[  G\right]  ,\dot{f}\left[  G\right]  \in
M_{0}\left[  G\right]  $ and $h\upharpoonright A_{n}\in M_{n+1}\left[
G\right]  $ has infinite intersection with every real in $M_{n}\left[
G\right]  .$ Finally, we can extend $\dot{p}\left[  G\right]  $ to a condition
forcing that $\dot{f}\left[  G\right]  $ and $h$ will have infinite intersection.
\end{proof}

\qquad\ \ \ 

As commented before, the iteration of forcings that does not destroy category
may destroy category, but the following preservation result of Dilip Raghavan
shows this can only happen at the successor steps of the iteration:

\qquad\ \ \qquad\ \ 

\begin{proposition}
[Raghavan \cite{MADnessinSetTheory}]Let $\delta$ be a limit ordinal and
$\langle\mathbb{P}_{\alpha},\mathbb{\dot{Q}}_{\alpha}\mid\alpha<\delta\rangle$
a countable support iteration of proper forcings. If $\mathbb{P}_{\alpha}$
does not destroy category for every $\alpha<\delta$ then $\mathbb{P}_{\delta}$
does not destroy category.
\end{proposition}

\qquad\ \ \ 

With the aid of the previous preservation theorem we conclude the following:

\qquad\ \ \ 

\begin{corollary}
The countable support iteration of $\mathbb{P}_{cat}$ does not destroy category.
\end{corollary}

\qquad\ \ 

Putting all the pieces together, we can finally prove our theorem:

\qquad\ \ \qquad\ \ 

\begin{proposition}
If the existence of an inaccessible cardinal is consistent, then so is the
following statement: \textsf{non}$\left(  \mathcal{M}\right)  =\omega_{1}$ and
every $\mathcal{IE}$-Luzin set is meager.
\end{proposition}

\begin{proof}
Let $\mu$ be an inaccessible cardinal. We perform a countable support
iteration $\{\mathbb{P}_{\alpha},\mathbb{\dot{Q}}_{\alpha}\mid\alpha<\mu\}$ in
which $\mathbb{\dot{Q}}_{\alpha}$ is forced by $\mathbb{P}_{\alpha}$ to be the
$\mathbb{P}_{cat}$ forcing. It is easy to see that if $\alpha<\mu$ then
$\mathbb{P}_{\alpha}$ has size less than $\mu$ so it has the $\mu$-chain
condition and then $\mathbb{P}_{\mu}$ has the $\mu$-chain condition (see
\cite{IteratedForcing}). The result then follows by the previous results.
\end{proof}

\chapter{Remarks on a conjecture of Hru\v{s}\'{a}k}

Based on his Category Dichotomy for Borel ideals
(\cite{CombinatoricsofFiltersandIdeals}), Hru\v{s}\'{a}k conjectured the following:

\begin{conjecture}
[Hru\v{s}\'{a}k]If $\left(  A,B,\longrightarrow\right)  $ is a Borel invariant
then either

$\left\Vert A,B,\longrightarrow\right\Vert \leq$\ \textsf{non}$\left(
\mathcal{M}\right)  $ or \textsf{cov}$\left(  \mathcal{M}\right)  \leq$
$\left\Vert A,B,\longrightarrow\right\Vert .$\ \qquad
\end{conjecture}

\qquad\ \ \ \ \ 

We will provide both a partial negative and a partial affirmative answer to
this conjecture: We show that the conjecture of Hru\v{s}\'{a}k is false if we
allowed $A$ and $B$ to be Borel subsets of $\omega_{1}^{\omega}.$
Nevertheless, we show that the conjecture is true for a large class of Borel
invariants. Note that the definability of $\left(  A,B,\longrightarrow\right)
$ is important: otherwise the almost disjointness number $\mathfrak{a}$ will
be a counterexample.

\qquad\ \ 

For every function $F:\omega_{1}^{<\omega}\longrightarrow\omega_{1}$ we define
the set $C\left(  F\right)  $ as the set of all $f\in\omega_{1}^{\omega}$ such
that there are infinitely many $n\in\omega$ for which $f\left(  n\right)  \in
F\left(  f\upharpoonright n\right)  $. The $\omega_{1}$\emph{-Namba ideal}
$\mathcal{L}_{\omega_{1}}$ is the ideal on $\omega_{1}^{\omega}$ generated by
$\left\{  C\left(  F\right)  \mid F:\omega_{1}^{<\omega}\longrightarrow
\omega_{1}\right\}  .$ We will be interested in the invariant \textsf{non}%
$\left(  \mathcal{L}_{\omega_{1}}\right)  .$ It is easy to see that this
invariant is uncountable. By $E_{\omega}^{\omega_{1}}$ we denote the set of
all ordinals smaller than $\omega_{1}$ with cofinality $\omega.$%
\ $\mathsf{CG}_{\omega}\left(  \omega_{1}\right)  $ is the statement that
there is a sequence $\overline{C}=\left\langle C_{\alpha}\mid\alpha\in
E_{\omega}^{\omega_{1}}\right\rangle $ where $C_{\alpha}\subseteq\alpha$ is a
cofinal set of order type $\omega$ such that for every club $D\subseteq
\omega_{1}$ there is $\alpha$ for which $C_{\alpha}\subseteq D.$ We call such
$\overline{C}$ a \emph{club guessing sequence}.

\qquad\ \ \ \ \ \ 

We will show that the existence of a Club Guessing sequence implies that the
uniformity of $\mathcal{L}_{\omega_{1}}$ is precisely $\omega_{1}.$

\begin{proposition}
The principle $\mathsf{CG}_{\omega}\left(  \omega_{1}\right)  $ implies
\textsf{non}$\left(  \mathcal{L}_{\omega_{1}}\right)  =\omega_{1}.$
\end{proposition}

\begin{proof}
Let $\overline{C}=\left\{  C_{\alpha}\mid\alpha\in E_{\omega}^{\omega_{1}%
}\right\}  $ be a club guessing sequence. Enumerate each $C_{\alpha}=\left\{
\alpha_{n}\mid n\in\omega\right\}  $ in an increasing way, we may further
assume $0\notin C_{\alpha}$ for every $\alpha\in LIM\left(  \omega_{1}\right)
.$ We now define $f_{\alpha}:\omega\longrightarrow\omega_{1}$ where
$f_{\alpha}\left(  n\right)  =\alpha_{n}.$ We will show that $X=\left\{
f_{\alpha}\mid\alpha\in E_{\omega}^{\omega_{1}}\right\}  \notin\mathcal{L}%
_{\omega_{1}}.$

\qquad\ \ \ 

Let $F:\omega_{1}^{<\omega}\longrightarrow\omega_{1},$ we must show that $X$
is not contained in $C\left(  F\right)  .$ Let $D\subseteq\omega_{1}$ be a
club such that if $\alpha\in D$ and $s\in\alpha^{<\omega}$ then $F\left(
s\right)  <\alpha.$ Since $\overline{C}$ is a club guessing sequence, there is
$\alpha\in D$ such that $C_{\alpha}\subseteq D.$ It is then easy to see that
$f_{\alpha}\notin C\left(  F\right)  .$
\end{proof}

\qquad\ \ \ 

We will now prove that the inequality \textsf{non}$\left(  \mathcal{L}%
_{\omega_{1}}\right)  >\omega_{1}$ is consistent and we will use Baumgartner's
forcing for adding a club with finite conditions. Let $\mathbb{BA}$ be the set
of all finite functions $p\subseteq\omega_{1}\times\omega_{1}$ with the
property that there is a function\ enumerating a club $g:\omega_{1}%
\longrightarrow\omega_{1}$ such that $p\subseteq g$ and $im\left(  g\right)  $
consists only of indecomposable ordinals. We order $\mathbb{BA}$ by inclusion.
It is well known that $\mathbb{BA}$ is a proper forcing and adds a club, whose
name we will denote by $\dot{D}_{gen}.$ Given a club $D\subseteq\omega_{1}$,
define a function $F_{D}:\omega_{1}^{<\omega}\longrightarrow\omega_{1}$ given
by $F_{D}\left(  s\right)  =min\left\{  \gamma\in D\mid im\left(  s\right)
\subseteq\gamma\right\}  .$ Recall that if $F:\omega_{1}^{<\omega
}\longrightarrow\omega_{1}$ we defined $C\left(  F\right)  =\left\{
f\in\omega_{1}^{\omega}\mid\exists^{\infty}n\left(  f\left(  n\right)  \in
F\left(  f\upharpoonright n\right)  \right)  \right\}  .$ Note that if
$f\in\omega_{1}^{\omega}$ then the following holds:

\begin{enumerate}
\item If $f\left[  \omega\right]  $ has a maximum then $f\in C\left(
F_{D}\right)  .$

\item If $\bigcup f\left[  \omega\right]  $ is not a limit point of $D$ then
$f\in C\left(  F_{D}\right)  .$
\end{enumerate}

\begin{lemma}
If $f\in\omega_{1}^{\omega}$ then $E_{f}=\{p\in\mathbb{BA}\mid p\Vdash``f\in
C(F_{\dot{D}_{gen}})\textquotedblright\}$ is a dense set.
\end{lemma}

\begin{proof}
Let $p\in\mathbb{BA}.$ We may assume $f\left[  \omega\right]  $ has no maximum
and $p$ forces that $\gamma=\bigcup f\left[  \omega\right]  $ is a limit point
of $\dot{D}_{gen}$ (in particular $\gamma$ must be an indecomposable ordinal)
so there must be a limit ordinal $\beta<\omega_{1}$ such that $p\left(
\beta\right)  =\gamma.$ Let $q\leq p$ and $n\in\omega.$ We must prove that
there is $q_{1}\leq q$ and $m>n$ such that $q_{1}\Vdash``f\left(  m\right)
<F_{\dot{D}_{gen}}\left(  f\upharpoonright m\right)  \textquotedblright.$ Let
$g:\omega_{1}\longrightarrow\omega_{1}$ be a function enumerating a club such
that $q\subseteq g$ and $im\left(  g\right)  $ consists only of indecomposable
ordinals. Let $\beta_{0}=max\left(  \beta\cap dom\left(  q\right)  \right)  $
and note we may assume that $f\left(  0\right)  ,...f\left(  n\right)
<q\left(  \beta_{0}\right)  $ (if this is not the case we just extend $q$ in
order to obtain this condition). Let $m$ be the smallest natural number for
which $q\left(  \beta_{0}\right)  <f\left(  m\right)  .$ Since $q$ forces that
$\gamma$ is a limit point of $\dot{D}_{gen},$ there must be $\beta_{0}%
<\beta_{1}<\beta$ such that $f\left(  m\right)  ,f\left(  m+1\right)
<g\left(  \beta_{1}\right)  .$ We then define $q_{1}$ and $g_{1}$ as follows:

\begin{enumerate}
\item $q_{1}=q\cup\left\{  \left(  \beta_{0}+1,g\left(  \beta_{1}\right)
\right)  \right\}  .$

\item $g\upharpoonright\left(  \beta_{0}+1\right)  ,$ $g\upharpoonright
\lbrack\beta,\omega_{1})\subseteq g_{1}.$

\item $g_{1}\left(  \beta_{0}+1\right)  =g\left(  \beta_{1}\right)  .$

\item If $\xi\in\left(  \beta_{0}+1,\beta\right)  $ then $g_{1}\left(
\xi\right)  =g\left(  \beta_{1}+\xi\right)  .$
\end{enumerate}

\qquad\ \ 

Note that $q_{1}$ is a condition of $\mathbb{BA}$ (as witnessed by $g_{1}$)
extending $q$ and $q_{1}\Vdash``f\left(  m+1\right)  <F_{\dot{D}_{gen}}\left(
f\upharpoonright m+1\right)  \textquotedblright.$
\end{proof}

\qquad\ \ \ \ \ \ \ 

Since $\mathbb{BA}$ is a proper forcing, we conclude the following:

\begin{proposition}
The Proper Forcing Axiom implies \textsf{non}$\left(  \mathcal{L}_{\omega_{1}%
}\right)  >\omega_{1}.$\qquad\ \ \ \ 
\end{proposition}

\qquad\ \ \ \ \ 

We will now show that the cardinal invariant \textsf{non}$\left(
\mathcal{L}_{\omega_{1}}\right)  $ does not satisfy the conjecture of
Hru\v{s}\'{a}k. Although \textsf{non}$\left(  \mathcal{L}_{\omega_{1}}\right)
$ is not a Borel invariant, it is still (in some sense) definable (it would be
a Borel invariant if we were allowed to use the space $\omega_{1}^{\omega}$
instead of a Polish space).

\begin{proposition}
Both the inequalities \textsf{non}$\left(  \mathcal{L}_{\omega_{1}}\right)  <$
\textsf{cov}$\left(  \mathcal{M}\right)  $\ and \textsf{non}$\left(
\mathcal{M}\right)  <$ \textsf{non}$\left(  \mathcal{L}_{\omega_{1}}\right)  $
are consistent with the axioms of $\mathsf{ZFC.}$
\end{proposition}

\begin{proof}
It is well known that Martin's Axiom is consistent with $\mathsf{CG}_{\omega
}\left(  \omega_{1}\right)  $ (see \cite{Kunen} chapter V.7.3) so the
inequality \textsf{non}$\left(  \mathcal{L}_{\omega_{1}}\right)  <$
\textsf{cov}$\left(  \mathcal{M}\right)  $ is consistent. In order to build a
model for the second inequality, we will perform a countable support iteration
$\langle\mathbb{P}_{\alpha},\mathbb{\dot{Q}}_{\alpha}\mid\alpha\leq\omega
_{2}\rangle$ where $\mathbb{P}_{\alpha}\Vdash``\mathbb{\dot{Q}}_{\alpha
}=\mathbb{BA}\textquotedblright$ for every $\alpha<\omega_{2}.$ It is enough
to show that \textsf{non}$\left(  \mathcal{M}\right)  =\omega_{1}$ holds after
forcing with $\mathbb{P}_{\omega_{2}}.$ To achieve this, it is enough to prove
that $\mathbb{BA}$ \emph{preserves Cohen reals: }(see \cite{Barty})

\qquad

\begin{itemize}
\item For any countable elementary submodel $M$ of $\mathsf{H}\left(
\theta\right)  $ (for some big enough $\theta$), $p\in M\cap\mathbb{BA}$ and
$c$ a Cohen real over $M,$ there is $q$ an $\left(  M,\mathbb{BA}\right)
$-generic condition extending $p$ such that $q$ forces that $c$ is a Cohen
real over $M[\dot{G}].$
\end{itemize}

\qquad\ \ 

Let $M,p$ and $c$ as above and let $\delta=M\cap\omega_{1}.$ Define
$q=p\cup\left\{  \left(  \delta,\delta\right)  \right\}  .$ It is easy to see
that $q\in\mathbb{BA}$ and $q$ is an $\left(  M,\mathbb{BA}\right)  $-generic
condition. We will now prove that it forces that $c$ remains a Cohen real over
$M[\dot{G}].$ Let $\dot{D}\in M$ be a name for an open dense set of
$\omega^{\omega}$ and $r\leq q.$ We must show we can find $r^{\prime}\leq r$
such that $r^{\prime}\Vdash``c\in\dot{D}\textquotedblright.$ Let
$r_{1}=r\upharpoonright\delta$ and note that $r_{1}\in M$ since $r\left(
\delta\right)  =\delta.$ Now we define $E=\bigcup\{\left\langle s\right\rangle
\mid\exists b\leq r_{1}(b\Vdash``\left\langle s\right\rangle \subseteq\dot
{D}\textquotedblright)\}$ which clearly is an open dense set and belonging to
$M.$ Since $c$ is a Cohen real over $M,$ we know that $c\in E$ so there is
$s\in\omega^{<\omega}$ and $b\leq r_{1}$ such that $b\Vdash``\left\langle
s\right\rangle \subseteq\dot{D}\textquotedblright$ and $c\in\left\langle
s\right\rangle .$ By elementarity, we may assume $b\in M.$ It is then easy to
see that $b$ and $r$ are compatible.
\end{proof}

\qquad\ \qquad\qquad\ \ \ \qquad\ \ 

Nevertheless, we will show that the conjecture of Hru\v{s}\'{a}k holds for a
large class of Borel cardinal invariants. A forcing notion $\mathbb{P}$
\emph{preserves category }if $\mathbb{P\Vdash}``A\notin\mathcal{M}%
\textquotedblright$ for every nonmeager set $A.$ Given a $\sigma$-ideal
$\mathcal{I}$ on a Polish space $X,$ we define $\mathbb{P}_{\mathcal{I}%
}=Borel\left(  X\right)  $ $/\mathcal{I}.$ We say $\mathcal{I}$ is a
universally Baire ideal if the set of all (codes for) analytic sets is
universally Baire.

\begin{definition}
Let $\left(  A,B,\longrightarrow\right)  $ be a Borel invariant, we say that
\textsf{non}$\left(  \mathcal{M}\right)  <\left\Vert A,B,\longrightarrow
\right\Vert $ is \emph{nicely consistent }if there is a universally Baire
$\sigma$-ideal $\mathcal{I}$ such that $\mathbb{P}_{\mathcal{I}}$ is proper,
preserves category and destroys $\left(  A,B,\longrightarrow\right)  .$
\end{definition}

\qquad\ \ 

By the methods of \cite{ForcingIdealized}, \textsf{non}$\left(  \mathcal{M}%
\right)  <\left\Vert A,B,\longrightarrow\right\Vert $ is nicely consistent if
there is an iterable $\sigma$-ideal $\mathcal{I}$ such that $\left(
\mathbb{P}_{\mathcal{I}}\right)  _{\omega_{2}}\Vdash``$\textsf{non}$\left(
\mathcal{M}\right)  <\left\Vert A,B,\longrightarrow\right\Vert
\textquotedblright$ where $\left(  \mathbb{P}_{\mathcal{I}}\right)
_{\omega_{2}}$ denotes the countable support iteration of the forcing
$\mathbb{P}_{\mathcal{I}}$ (see \cite{ForcingIdealized} for the definition of
an iterable ideal). Intuitively, \textsf{non}$\left(  \mathcal{M}\right)
<\left\Vert A,B,\longrightarrow\right\Vert $ is nicely consistent just means
that the inequality \textsf{non}$\left(  \mathcal{M}\right)  <\left\Vert
A,B,\longrightarrow\right\Vert $ can be forced with the a nice enough forcing.
Given a $\sigma$-ideal $\mathcal{I},$ define \textsf{cov}$^{-}\left(
\mathcal{I}\right)  $ as the smallest size of a family $\mathcal{A\subseteq
I}$ such that there is a Borel set $B\notin\mathcal{I}$ for which
$B\subseteq\bigcup\mathcal{A}$ (in \cite{ForcingIdealized} \textsf{cov}%
$^{-}\left(  \mathcal{I}\right)  $ is denoted as \textsf{cov}$^{\ast}\left(
\mathcal{I}\right)  $).

\qquad\ \ \ 

\begin{lemma}
Let $\left(  A,B,\longrightarrow\right)  $ be a Borel invariant and
$\mathcal{I}$ a $\sigma$-ideal such that $\mathbb{P}_{\mathcal{I}}$ is proper.
If $\mathbb{P}_{\mathcal{I}}$ destroys $\left(  A,B,\longrightarrow\right)  $
then \textsf{cov}$^{-}\left(  \mathcal{I}\right)  \leq\left\Vert
A,B,\longrightarrow\right\Vert .$
\end{lemma}

\begin{proof}
Let $\dot{r}$ be a $\mathbb{P}_{\mathcal{I}}$-name such that $\mathbb{P}%
_{\mathcal{I}}\Vdash``\dot{r}\in A\textquotedblright$ and if $b\in B$ then
$\mathbb{P}_{\mathcal{I}}\Vdash``\dot{r}\nrightarrow b\textquotedblright.$ By
the Borel reading of names (see \cite{ForcingIdealized}) there is a Borel set
$C\notin\mathcal{I}$ and a Borel function $F:C\longrightarrow A$ such that
$C\Vdash``F\left(  \dot{r}_{gen}\right)  =\dot{r}\textquotedblright.$ For
every $b\in B$ let $E_{b}=\left\{  x\in C\mid F\left(  x\right)
\longrightarrow b\right\}  .$ Note that each $E_{b}$ is a Borel set and
$E_{b}\in\mathcal{I}.$ Let $D\subseteq B$ be a dominating family for $\left(
A,B,\longrightarrow\right)  $ of minimum size and let $M$ be a countable
elementary submodel such that $\left(  A,B,\longrightarrow\right)  ,C,F\in M.$
Define $C_{1}$ as the set of all $M$-generic points of $C,$ since
$\mathbb{P}_{\mathcal{I}}$ is proper then $C_{1}$ is a Borel set extending
$C.$ It is easy to see that $C_{1}\subseteq\bigcup\limits_{b\in D}E_{b}$ and
then we conclude that \textsf{cov}$^{-}\left(  \mathcal{I}\right)
\leq\left\Vert A,B,\longrightarrow\right\Vert .$
\end{proof}

\qquad\ \ \ \ 

With this we can prove the following (where by \textsf{LC }we denote a large
cardinal hypothesis):

\begin{theorem}
[\emph{LC}]Let $\left(  A,B,\longrightarrow\right)  $ be a Borel invariant
such that \textsf{non}$\left(  \mathcal{M}\right)  <\left\Vert
A,B,\longrightarrow\right\Vert $ is nicely consistent, then \textsf{cov}%
$\left(  \mathcal{M}\right)  \leq\left\Vert A,B,\longrightarrow\right\Vert .$
\end{theorem}

\begin{proof}
Let $\mathcal{I}$ be a $\sigma$-ideal such that $\mathbb{P}_{\mathcal{I}}$ is
proper, preserves category and destroys $\left(  A,B,\longrightarrow\right)
.$ By the previous result, we know that \textsf{cov}$^{-}\left(
\mathcal{I}\right)  \leq\left\Vert A,B,\longrightarrow\right\Vert .$ By
\cite{ForcingIdealized} Corollary 3.5.4, Cohen forcing $\mathbb{C=P}%
_{\mathcal{M}}$ adds an $\mathcal{I}$-quasigeneric real, using the same method
as in the previous lemma, we can then conclude that \textsf{cov}$\left(
\mathcal{M}\right)  \leq$ \textsf{cov}$^{-}\left(  \mathcal{I}\right)  $ and
then \textsf{cov}$\left(  \mathcal{M}\right)  \leq\left\Vert
A,B,\longrightarrow\right\Vert .$
\end{proof}

\qquad\ \qquad\ \ 

We do not know if there is a Borel invariant $\left(  A,B,\longrightarrow
\right)  $ such that \textsf{non}$\left(  \mathcal{M}\right)  <\left\Vert
A,B,\longrightarrow\right\Vert $ is consistent but not nicely consistent.

\chapter{Indestructibility\qquad\ \ \ }

\section{Indestructibility of ideals}

If $\mathcal{I}$ is an ideal in $\omega$ and $\mathbb{P}$ is a partial order,
we say that $\mathbb{P}$ \emph{destroys }$\mathcal{I}$ if $\mathbb{P}$ forces
that $\mathcal{I}$ is no longer tall i.e. if $\mathbb{P}$ adds a new subset of
$\omega$ that is almost disjoint with every element of $\mathcal{I}.$ The
theory of destructibility of ideals is very important in forcing theory, since
many important forcing properties may be stated in these terms. The following
proposition is an example of this fact.

\begin{proposition}
Let $\mathbb{P}$ be a partial order.

\begin{enumerate}
\item $\mathbb{P}$ adds new reals if and only if $\mathbb{P}$ destroys
$tr\left(  \mathsf{ctble}\right)  .$

\item $\mathbb{P}$ adds unbounded reals if and only if $\mathbb{P}$ destroys
$tr\left(  \mathcal{K}_{\sigma}\right)  .$

\item $\mathbb{P}$ adds dominating reals if and only if $\mathbb{P}$ destroys
$tr\left(  \mathcal{L}\right)  $ if and only if $\mathbb{P}$ destroys
\textsf{FIN}$\times$\textsf{FIN}$.$

\item $\mathbb{P}$ adds eventually different reals if and only if $\mathbb{P}$
destroys $\mathcal{ED}.$
\end{enumerate}
\end{proposition}

\begin{proof}
The first and second points will be proved later in this chapter. We will now
prove that $\mathbb{P}$ adds dominating reals if and only if $\mathbb{P}$
destroys \textsf{FIN}$\times$\textsf{FIN}$.$ First assume $\mathbb{P}$
destroys \textsf{FIN}$\times$\textsf{FIN}$,$ so then $\mathbb{P}$ adds an
infinite partial function that is almost disjoint with $($\textsf{FIN}$\times
$\textsf{FIN}$)\cap V.$ We may assume $f$ is increasing. We now define
$g:\omega\longrightarrow\omega$ such that $g\left(  n\right)  =f\left(
m_{n}\right)  $ where $m_{n}=min\left\{  k\geq n\mid k\in dom\left(  f\right)
\right\}  .$ It is easy to see that $g$ is a dominating real. Clearly if
$\mathbb{P}$ adds a dominating real then $\mathbb{P}$ destroys \textsf{FIN}%
$\times$\textsf{FIN}$.$ By a theorem below, since Laver forcing adds
dominating reals then \textsf{FIN}$\times$\textsf{FIN}$\leq_{K}tr\left(
\mathcal{L}\right)  $ so by destroying $tr\left(  \mathcal{L}\right)  $ we
will add a dominating real. It is easy to see that adding a dominating real
will destroy $tr\left(  \mathcal{L}\right)  .$ Finally, if $\mathbb{P}$ adds
an eventually different real then $\mathbb{P}$ will destroy $\mathcal{ED}$.
Conversely, if $\mathbb{P}$ destroys $\mathcal{ED}$ then $\mathbb{P}$ will
make $V\cap\omega^{\omega}$ a meager set, hence it will add an eventually
different real.
\end{proof}

\qquad\ \ \ 

The previous proposition suggests the following conjecture: if $\mathcal{I}$
is a $\sigma$-ideal on $\omega^{\omega}$ and $\mathbb{P}$ is a partial order,
then $\mathbb{P}$ adds $\mathcal{I}$-quasigeneric reals if and only if
$\mathbb{P}$ destroys $tr\left(  \mathcal{I}\right)  .$ However, this
conjecture is false even for nice ideals and nice partial orders, as the
following example from \cite{ForcingwithQuotients} shows: the trace of the
null ideal can be destroyed by a $\sigma$-centered forcing (for example by
$\mathbb{M}\left(  tr\left(  \mathcal{N}\right)  \right)  $) however, it is
known that no $\sigma$-centered forcing can add random reals.

\qquad\ \ \ \qquad\ \ \ 

The Kat\v{e}tov order is a key tool for understanding the destructibility of ideals:

\begin{lemma}
Let $\mathcal{I},\mathcal{J}$ be two ideals such that $\mathcal{I\leq}_{K}$
$\mathcal{J}.$ If $\mathbb{P}$ destroys $\mathcal{J}$ then $\mathbb{P}$
destroys $\mathcal{I}.$
\end{lemma}

\begin{proof}
Let $f:\left(  \omega,\mathcal{J}\right)  \longrightarrow\left(
\omega,\mathcal{I}\right)  $ be a Kat\v{e}tov-morphism. Let $\dot{X}$ be a
$\mathbb{P}$-name for an infinite subset of $\omega$ that is forced to be
almost disjoint with every element of $\mathcal{J}.$ It is easy to see that
$f[\dot{X}]$ is forced to ba almost disjoint with $\mathcal{I}.$
\end{proof}

\qquad\ \ \ \ \ \ \qquad\ \ 

The following is an useful lemma:

\begin{lemma}
[\cite{ForcingIdealized}]Let $\mathcal{I}$ be $\sigma$-ideal in $\omega
^{\omega}$ such that $\mathbb{P}_{\mathcal{I}}$ is proper and has the
continuous reading of names. If $B\in\mathbb{P}_{\mathcal{I}}$ then there is
$D\leq B$ such that $D$ is $G_{\delta}.$
\end{lemma}

\begin{proof}
Let $B\in\mathbb{P}_{\mathcal{I}}$ and since $B$ is analytic, we can find $T$
a tree on $\omega\times\omega$ such that $B=\left\{  x\mid\exists y\left(
\left(  x,y\right)  \in\left[  T\right]  \right)  \right\}  $ (where $\left[
T\right]  =\left\{  \left(  x,y\right)  \mid\forall n\left(  \left(
x\upharpoonright n,y\upharpoonright n\right)  \in T\right)  \right\}  $). Let
$\dot{y}$ be a $\mathbb{P}_{\mathcal{I}}$-name such that $B\Vdash``\,(\dot
{r}_{gen},\dot{y})\in\left[  T\right]  \textquotedblright$ (where $\dot
{r}_{gen}$ is the name of the generic real). Since $\mathbb{P}_{\mathcal{I}}$
is proper and has the continuous reading of names we can find $C\leq B$ and a
continuous function $F:C\longrightarrow\omega^{\omega}$ such that
$C\Vdash``F\left(  \dot{r}_{gen}\right)  =\dot{y}\textquotedblright.$ Since
$C\Vdash``\left(  \dot{r}_{gen},F\left(  \dot{r}_{gen}\right)  \right)
\in\left[  T\right]  \textquotedblright$ we may assume (by possibly shrinking
the condition) that $\left(  x,F\left(  x\right)  \right)  \in\left[
T\right]  $ for every $x\in C.$

\qquad\ \ 

Let $L=\left\{  s\in\omega^{<\omega}\mid\left\langle s\right\rangle \cap
C\neq\emptyset\right\}  .$ We now define a function $\overline{F}%
:L\longrightarrow\omega^{<\omega}$ as follows: if $\overline{F}$ is not
constant on $\left\langle s\right\rangle \cap C$ we define $F\left(  s\right)
=t$ where $t$ is the largest such that $F\left[  \left\langle s\right\rangle
\cap C\right]  \subseteq\left\langle t\right\rangle .$ In case $F$ is constant
on $\left\langle s\right\rangle \cap C$ we define $\overline{F}\left(
s\right)  =r\upharpoonright\left\vert s\right\vert $ where $r$ is the constant
value of $F\left[  \left\langle s\right\rangle \cap C\right]  .$ We now define
$D$ as the set of all $r\in\omega^{\omega}$ such that for every $n\in\omega$
there is $s\subseteq r$ such that $\overline{F}\left(  s\right)  $ has length
at least $n.$ It is easy to see that $C\subseteq D\subseteq B$ so
$D\in\mathbb{P}_{\mathcal{I}}$ and it is a $G_{\delta}$ set.
\end{proof}

\qquad\ \ 

The relevance of the trace ideals in the study of destructibility is the
following important result of Hru\v{s}\'{a}k and Zapletal:

\begin{proposition}
[\cite{ForcingwithQuotients}]Let $\mathcal{I}$ be a $\sigma$-ideal in
$\omega^{\omega}$ such that $\mathbb{P}_{\mathcal{I}}$ is proper and has the
continuous reading of names. If $\mathcal{J}$ is an ideal on $\omega,$ then
the following are equivalent:\qquad\ \ \ \ \ 

\begin{enumerate}
\item There is a condition $B\in\mathbb{P}_{\mathcal{I}}$ such that $B$ forces
that $\mathcal{J}$ is not tall.

\item There is $a\in tr\left(  \mathcal{I}\right)  ^{+}$ such that
$\mathcal{J\leq}_{K}$ $tr\left(  \mathcal{I}\right)  \upharpoonright a.$
\end{enumerate}
\end{proposition}

\begin{proof}
We first show that 2 implies 1. Let $B=\pi\left(  a\right)  $ which is an
element of $\mathbb{P}_{\mathcal{I}}.$ It is enough to show that $B$ forces
that $tr\left(  \mathcal{I}\right)  \upharpoonright a$ is no longer tall. Let
$\dot{r}_{gen}$ be the name for the generic real. Note that since
$B\Vdash``\dot{r}_{gen}\in B\textquotedblright$, it follows that $B$ forces
that $\dot{x}=\left\{  \dot{r}_{gen}\upharpoonright n\mid\dot{r}%
_{gen}\upharpoonright n\in a\right\}  $ is infinite. We will prove that $B$
forces $\dot{x}$ to be AD with $tr\left(  \mathcal{I}\right)  \upharpoonright
a.$ Let $C\leq B$ and $d\in$ $tr\left(  \mathcal{I}\right)  \upharpoonright a$
then $C_{1}=C\setminus\pi\left(  d\right)  \in\mathbb{P}_{\mathcal{I}}$ and
$C_{1}$ forces $\dot{x}$ is almost disjoint with $d.$

\qquad\ \ \ 

Now assume that there is $B\in\mathbb{P}_{\mathcal{I}}$ and a $\mathbb{P}%
_{\mathcal{I}}$-name $\dot{X}$ such that $\dot{X}=\left\{  \dot{x}_{n}\mid
n\in\omega\right\}  $ is forced by $B$ to be almost disjoint with
$\mathcal{J}.$ We recursively define $\left\{  a_{n}\mid n\in\omega\right\}  $
and a function $g$ as follows:

\begin{enumerate}
\item Each $a_{n}\subseteq\omega^{<\omega}$ is an antichain.

\item If $s\in a_{n+1}$ then there is $t\in a_{n}$ such that $t\subseteq s.$

\item If $a=%
{\textstyle\bigcup}
a_{n}$ then $\pi\left(  a\right)  \subseteq B$ and $\pi\left(  a\right)
\in\mathbb{P}_{\mathcal{I}}.$

\item $g$ is a function from $a$ to $\omega.$

\item If $s\in a_{n}$ then $\left\langle s\right\rangle \cap\pi\left(
a\right)  \Vdash``\dot{x}_{n}=g\left(  s\right)  \textquotedblright.$
\end{enumerate}

\qquad\qquad\ 

This can be done since $\mathbb{P}_{\mathcal{I}}$ has the continuous reading
of names (and the previous lemma). It is then easy to see that $g:\left(
\omega^{<\omega},tr\left(  \mathcal{I}\right)  \upharpoonright a\right)
\longrightarrow\left(  \omega,\mathcal{J}\right)  $ is a Kat\v{e}tov-morphism.
\end{proof}

\qquad\ \ 

If $t\in2^{<\omega}$ we define $\left\langle t\right\rangle _{<\omega
}=\left\{  s\in2^{<\omega}\mid t\sqsubseteq s\right\}  .$ We now have the following:

\begin{lemma}
$tr\left(  \mathcal{M}\right)  $ is Kat\v{e}tov-Blass equivalent to
\textsf{nwd}$.$\qquad\ \ \ \ 
\end{lemma}

\begin{proof}
Let $\preceq$ be a well order for the rational numbers. We recursively define
$\left\{  U_{s}\mid s\in2^{<\omega}\right\}  $ and $\left\{  q_{s}\mid
s\in2^{<\omega}\right\}  \subseteq\mathbb{Q}$ as follows:

\begin{enumerate}
\item Each $U_{s}$ is a clopen set of the rational numbers and $q_{s}%
=min_{\preceq}\left(  U_{s}\right)  .$

\item $U_{\emptyset}=\mathbb{Q}.$

\item $\left\{  U_{s^{\frown}0},U_{s^{\frown}i}\right\}  $ is a partition
(into clopen sets) of $U_{s}-\left\{  q_{s}\right\}  .$

\item $\mathbb{Q}=\left\{  q_{s}\mid s\in2^{<\omega}\right\}  $ and $\left\{
U_{s}\mid s\in2^{<\omega}\right\}  $ is a $\pi$-base of open sets
\end{enumerate}

\qquad\ \ \ 

We then define $f:2^{<\omega}\longrightarrow\mathbb{Q}$ given by $f\left(
s\right)  =q_{s}.$ We will prove that $f$ is a Kat\v{e}tov-morphism from
$\left(  2^{<\omega},tr\left(  \mathcal{M}\right)  \right)  $ to
$(\mathbb{Q},$\textsf{nwd}$).$ Let $N\subseteq\mathbb{Q}$ be a nowhere dense
set, we will prove that $\pi\left(  f^{-1}\left(  N\right)  \right)  $ is a
nowhere dense set of $2^{\omega}.$ Let $s\in2^{<\omega}$ and since $N$ is
nowhere dense and $\left\{  U_{s}\mid s\in2^{<\omega}\right\}  $ is a $\pi
$-base of open sets we can find $t\in2^{<\omega}$ extending $s$ such that
$U_{t}\cap N=\emptyset.$ It then follows that $\left\langle t\right\rangle
\cap\pi\left(  f^{-1}\left(  N\right)  \right)  =\emptyset.$ Now we will prove
that $f^{-1}$ is a Kat\v{e}tov-morphism from $(\mathbb{Q},$\textsf{nwd}$)$ to
$\left(  2^{<\omega},tr\left(  \mathcal{M}\right)  \right)  .$ It is enough to
prove that if $a\in tr\left(  \mathcal{M}\right)  $ then $f\left[  a\right]
\in$\textsf{nwd}$.$ Let $s\in2^{<\omega}$ and since $\pi\left(  a\right)
\cap\left\langle s\right\rangle $ is a $G_{\delta}$ meager set in
$\left\langle s\right\rangle $ it can not be dense, so there is $t\in
2^{<\omega}$ extending $s$ such that $\left\langle t\right\rangle \cap
\pi\left(  a\right)  =\emptyset$ which implies that $a\cap\left\langle
t\right\rangle _{<\omega}$ is an off-branch set, so we can then find an
extension $r$ of $t$ such that $a\cap\left\langle r\right\rangle _{<\omega
}=\emptyset$ which then implies $f\left[  a\right]  \cap U_{r}=\emptyset.$
\end{proof}

\qquad\ \ \ \ \ \qquad\ \ 

The following result is useful for computing the covering numbers of the trace ideals:

\begin{proposition}
[\cite{ForcingwithQuotients}]Let $\mathcal{I}$ be a $\sigma$-ideal in
$\omega^{\omega}$ generated by analytic sets such that $\mathbb{P}%
_{\mathcal{I}}$ is proper and has the continuous reading of names. Then:

\qquad\ \ \ \ \ \qquad\ \ 

\textsf{cov}$\left(  \mathcal{I}\right)  \leq$ \textsf{cov}$^{\ast}\left(
tr\left(  \mathcal{I}\right)  \right)  \leq max\{$\textsf{cov}$\left(
\mathcal{I}\right)  ,\mathfrak{d}\}.$
\end{proposition}

\begin{proof}
Let $\kappa<$ \textsf{cov}$\left(  \mathcal{I}\right)  $ and $\left\{
a_{\alpha}\mid\alpha<\kappa\right\}  \subseteq tr\left(  \mathcal{I}\right)
.$ Since $\kappa<$\textsf{cov}$\left(  \mathcal{I}\right)  ,$ there is
$f\in\omega^{\omega}$ such that $f\notin\bigcup\limits_{\alpha<\kappa}%
\pi\left(  a_{\alpha}\right)  $ and then $\left\{  f\upharpoonright n\mid
n\in\omega\right\}  $ is almost disjoint with each $a_{\alpha}.$ Now we will
prove that \textsf{cov}$^{\ast}\left(  tr\left(  \mathcal{I}\right)  \right)
\leq max\{$\textsf{cov}$\left(  \mathcal{I}\right)  ,\mathfrak{d}\}.$ Let $S$
be the set of all $f:\omega^{<\omega}\longrightarrow\left[  \omega^{<\omega
}\right]  ^{<\omega}$ such that $s\in f\left(  s\right)  $ and $f\left(
s\right)  $ is a finite set of $\left\{  t\in\omega^{<\omega}\mid s\sqsubseteq
t\right\}  .$ Let $\mathcal{F}=\left\{  f_{\alpha}\mid\alpha<\mathfrak{d}%
\right\}  \subseteq S$ such that for every $f\in S$ there is $\alpha
<\mathfrak{d}$ such that $f\left(  s\right)  \subseteq f_{\alpha}\left(
s\right)  $ for every $s\in\omega^{<\omega}.$ Let $\kappa$ be the maximum of
\textsf{cov}$\left(  \mathcal{I}\right)  $ and $\mathfrak{d}$. Since every
analytic set is the union of at most $\mathfrak{d}$ compact sets, we can find
a family $\left\{  T_{\beta}\mid\beta<\kappa\right\}  $ of finitely branching
subtrees of $\omega^{<\omega}$ such that each $\left[  T_{\beta}\right]  $
belongs to $\mathcal{I}$ and $\omega^{\omega}=\bigcup\limits_{\beta<\kappa
}\left[  T_{\beta}\right]  .$ For every $\alpha<\mathfrak{d}$ and
$\beta<\kappa$ we define $\left\langle a_{\alpha,\beta}\left(  n\right)  \mid
n\in\omega\right\rangle $ and $\left\langle m_{n}\mid n\in\omega\right\rangle
$ with the following properties:

\begin{enumerate}
\item $m_{0}=\emptyset.$

\item If $l<k$ then $m_{l}<m_{k}.$

\item $a_{\alpha,\beta}\left(  n\right)  $ is a finite subset of
$\omega^{<\omega}.$

\item $a_{\alpha,\beta}\left(  0\right)  =f_{\alpha}\left(  \emptyset\right)
.$

\item $m_{n+1}$ is bigger than the length of all the elements of
$a_{\alpha,\beta}\left(  n\right)  .$

\item $a_{\alpha,\beta}\left(  n+1\right)  =%
{\textstyle\bigcup}
\left\{  f_{\alpha}\left(  t\right)  \mid t\in T_{\beta}\cap\omega^{m_{n+1}%
}\right\}  .$
\end{enumerate}

\qquad\ \ \ 

Let $a_{\alpha,\beta}=\bigcup a_{\alpha,\beta}\left(  n\right)  $ and note
that if $g\in\pi\left(  a_{\alpha,\beta}\right)  $ then there are infinitely
many $t\in T_{\beta}$ such that $t\sqsubseteq g$ so $g\in\left[  T_{\beta
}\right]  $ and then $a_{\alpha,\beta}\in tr\left(  \mathcal{I}\right)  .$ For
every $t\in\omega^{<\omega}$ and $\alpha<\mathfrak{d}$ define $b_{\alpha
,t}=\bigcup\left\{  f_{\alpha}\left(  t^{\frown}n\right)  \mid n\in
\omega\right\}  .$ Clearly $\pi\left(  b_{\alpha,t}\right)  =\emptyset.$ Hence
$W=\left\{  a_{\alpha,\beta}\mid\alpha<\mathfrak{d},\beta<\kappa\right\}
\cup\left\{  b_{\alpha,t}\mid\alpha<\mathfrak{d\wedge}t\in\omega^{<\omega
}\right\}  $ is a subset of $tr\left(  \mathcal{I}\right)  .$ We will now show
$W$ is a tall family. Let $X\subseteq\omega^{<\omega}$ be an infinite set.

\qquad\ \ 

Define $L$ as the tree of all $t\in\omega^{<\omega}$ such that $\left\{  s\mid
t\subseteq s\wedge s\in X\right\}  $ is infinite. We proceed by cases: first
assume there is $t\in L$ that is a maximal node. Since $t\in L$ and it is
maximal, it follows that $\left\{  n\mid\exists s\in X\left(  t^{\frown
}n\subseteq s\right)  \right\}  $ is infinite and then there is $\alpha
<\mathfrak{d}$ such that $b_{\alpha,t}\cap X$ is infinite. We now assume that
$L$ does not have a maximal node, so there is $r\in\left[  L\right]  .$ Since
$\omega^{\omega}=\bigcup\limits_{\beta<\kappa}\left[  T_{\beta}\right]  $
there is $\beta<\kappa$ such that $r\in\left[  T_{\beta}\right]  .$ We now
define a function $f:\omega^{<\omega}\longrightarrow\left[  \omega^{<\omega
}\right]  ^{<\omega}$ as follows: if $t\notin L$ then $f\left(  t\right)
=\left\{  t\right\}  $ and if $t\in L$ we choose $s\in X$ such that
$t\subseteq s$ and define $f\left(  t\right)  =\left\{  t,s\right\}  .$ Let
$\alpha<\mathfrak{d}$ such that $f_{\alpha}$ dominates $f,$ it is easy to see
that $a_{\alpha,\beta}\cap X$ is infinite.
\end{proof}

\qquad\ \ 

In particular, we may conclude that \textsf{cov}$\left(  \mathcal{M}\right)
\leq$ \textsf{cov}$^{\ast}($\textsf{nwd}$)\leq\mathfrak{d}.$ The following
result of Keremedis shows that \textsf{cov}$^{\ast}($\textsf{nwd}$)$ is
actually the covering number of the meager ideal:

\begin{proposition}
[Keremedis, see \cite{CombinatoricsofDenseSubsetsoftheRationals}]%
\textsf{cov}$\left(  \mathcal{M}\right)  =$ \textsf{cov}$^{\ast}\left(
tr\left(  \mathcal{M}\right)  \right)  .$
\end{proposition}

\begin{proof}
Let $\kappa<$ \textsf{cov}$^{\ast}\left(  tr\left(  \mathcal{M}\right)
\right)  $ and $\left\{  T_{\alpha}\mid\alpha<\kappa\right\}  $ be a family of
subtrees of $2^{<\omega}$ such that each $\left[  T_{\alpha}\right]  $ is
nowhere dense. We must prove that $2^{\omega}\neq%
{\textstyle\bigcup\limits_{\alpha<\kappa}}
\left[  T_{\alpha}\right]  .$ Since $\pi\left(  T_{\alpha}\right)  =\left[
T_{\alpha}\right]  ,$ therefore $T_{\alpha}\in tr\left(  \mathcal{M}\right)  $
for every $\alpha<\kappa.$ In this way, there is an infinite $Y\subseteq
2^{<\omega}$ that has finite intersection with each $T_{\alpha}.$ Furthermore,
we claim there is such $Y$ for which $\pi\left(  Y\right)  \neq\emptyset.$

\qquad\ \ \ 

Assume this is not the case. We then recursively build $\left\{  A_{n}\mid
n\in\omega\right\}  $ such that for every $n\in\omega$ the following holds:

\begin{enumerate}
\item $A_{n}\subseteq2^{<\omega}$ is an infinite antichain.

\item $A_{n}\cap T_{\alpha}$ is finite for every $\alpha<\kappa.$

\item Every element of $A_{n+1}$ extends an element of $A_{n};$ moreover,
every $t\in A_{n}$ has infinitely many extensions in $A_{n+1}.$
\end{enumerate}

\qquad\ \ 

Let $A_{0}$ be any infinite antichain almost disjoint with every $T_{\alpha}.$
Assuming we have constructed $A_{n}$ we will see how to construct $A_{n+1}.$
Given $s\in A_{n}$ let $B\left(  s\right)  =\left\{  t\in2^{<\omega}\mid
s\subseteq t\right\}  $ and since $\kappa<$ \textsf{cov}$^{\ast}\left(
tr\left(  \mathcal{M}\right)  \right)  $ we conclude that $\left\{  T_{\alpha
}\cap B\left(  s\right)  \mid\alpha<\kappa\right\}  $ is not tall. Let
$D_{s}\subseteq B\left(  s\right)  $ be an infinite antichain almost disjoint
with every $T_{\alpha}\cap B\left(  s\right)  .$ We now define $A_{n+1}=%
{\textstyle\bigcup\limits_{s\in A_{n}}}
D_{s}.$ Since each $T_{\alpha}$ is upward closed we know that $T_{\alpha}\cap
A_{n+1}$ is finite. Let $A_{n}=\left\{  s_{n}\left(  i\right)  \mid i\in
\omega\right\}  $ and for every $\alpha<\kappa$ we define $f_{\alpha}\in
\omega^{\omega}$ given by $f_{\alpha}\left(  n\right)  =min\left\{
m\mid\forall i\geq m\left(  s_{n}\left(  i\right)  \notin T_{\alpha}\right)
\right\}  .$ Since $\kappa<$\textsf{cov}$^{\ast}\left(  tr\left(
\mathcal{M}\right)  \right)  \leq\mathfrak{d}$ there is a function $g\in
\omega^{\omega}$ not dominated by any $f_{\alpha}.$ We then recursively build
$Y=\left\{  t_{n}\mid n\in\omega\right\}  $ such that for every $n\in\omega$
the following holds:

\begin{enumerate}
\item $t_{n}\in A_{n}.$

\item $t_{n}\subseteq t_{n+1}.$

\item There is $i\geq g\left(  n\right)  $ such that $t_{n}=s_{n}\left(
i\right)  .$
\end{enumerate}

\qquad\ \ 

It is easy to see that $Y\cap T_{\alpha}$ is finite for every $\alpha<\kappa,$
which is a contradiction.

\qquad\ \ \ 

Let $Y$ such that $\pi\left(  Y\right)  \neq\emptyset$ and $Y$ is almost
disjoint with each $T_{\alpha}.$ Clearly if $r\in\pi\left(  Y\right)  $ then
$r\notin%
{\textstyle\bigcup\limits_{\alpha<\kappa}}
\left[  T_{\alpha}\right]  .$
\end{proof}

\qquad\ \ \ \ \ 

As a consequence of the previous results we can conclude the following:\qquad\ \ \ \ \ \ \ \ \ \ \ 

\begin{proposition}
\qquad\ \ \ \qquad\ \ 

\begin{enumerate}
\item \textsf{cov}$^{\ast}\left(  tr\left(  \mathsf{ctble}\right)  \right)
=\mathfrak{c}.$

\item \textsf{cov}$^{\ast}\left(  tr\left(  K_{\sigma}\right)  \right)
=\mathfrak{d}.$

\item \textsf{cov}$^{\ast}\left(  tr\left(  \mathcal{M}\right)  \right)  =$
\textsf{cov}$\left(  \mathcal{M}\right)  .$

\item \textsf{cov}$^{\ast}\left(  tr\left(  \mathcal{L}\right)  \right)
=\mathfrak{b}.$
\end{enumerate}
\end{proposition}

\qquad\ \ \ \qquad\ \ \ 

The last equality follows since $tr\left(  \mathcal{L}\right)  $ is
Kat\v{e}tov above \textsf{FIN}$\times$\textsf{FIN}$.$

\begin{definition}
Let $\mathcal{I}$ and $\mathcal{J}$ be two $\sigma$-ideals in $\omega^{\omega
}.$ We say $\mathcal{I}$ is \emph{continuously Kat\v{e}tov smaller than
}$\mathcal{J}$ (denoted by $\mathcal{I\leq}_{CK}\mathcal{J}$ ) if there is a
continuous function $F:\omega^{\omega}\longrightarrow\omega^{\omega}$ such
that $F^{-1}\left(  A\right)  \in\mathcal{J}$ whenever $A\in\mathcal{I}.$
\end{definition}

\qquad\ \ \ \ \qquad\ \ \ \ \ \ 

Then we have the following result:

\begin{proposition}
[\cite{TesisDavid}]Let $\mathcal{I}$ and $\mathcal{J}$ be $\sigma$-ideals in
$\omega^{\omega}.$ If $\mathcal{I\leq}_{CK}\mathcal{J}$ \ then $tr\left(
\mathcal{I}\right)  \leq_{K}tr\left(  \mathcal{J}\right)  .$
\end{proposition}

\begin{proof}
Let $F:\omega^{\omega}\longrightarrow\omega^{\omega}$ be a continuous function
such that $F^{-1}\left(  A\right)  \in\mathcal{J}$ whenever $A\in\mathcal{I}.$
We now define $f:\omega^{<\omega}\longrightarrow\omega^{<\omega}$ as follows:
let $s\in\omega^{<\omega},$ if $F$ is not constant on $\left\langle
s\right\rangle $ we define $f\left(  s\right)  =max\left\{  t\mid F\left(
\left\langle s\right\rangle \right)  \subseteq\left\langle t\right\rangle
\right\}  $ and if $F$ is constant on $\left\langle s\right\rangle $ then
$f\left(  s\right)  =r\upharpoonright\left\vert s\right\vert $ where $r$ is
the constant value of $F\upharpoonright\left\langle s\right\rangle .$ We claim
that $f:\left(  \omega^{<\omega},tr\left(  \mathcal{J}\right)  \right)
\longrightarrow\left(  \omega^{<\omega},tr\left(  \mathcal{I}\right)  \right)
$ is a Kat\v{e}tov morphism. Let $a\in tr\left(  \mathcal{I}\right)  $ we will
show that $\pi\left(  f^{-1}\left(  a\right)  \right)  \in\mathcal{J}.$ Note
that if $x\in\pi\left(  f^{-1}\left(  a\right)  \right)  $ then $F\left(
x\right)  \in\pi\left(  a\right)  \in\mathcal{I}.$ However, $\mathbb{P}%
_{\mathcal{J}}$ forces that $F\left(  \dot{r}_{gen}\right)  $ is $\mathcal{I}%
$-quasigeneric (where $\dot{r}_{gen}$ denotes the name of the generic real)
so\ $\pi\left(  f^{-1}\left(  a\right)  \right)  $ can not be a condition of
$\mathbb{P}_{\mathcal{J}}.$
\end{proof}

\qquad\ \ \ \ 

Let $\mathcal{I}$ be a $\sigma$-ideal, we say that $\mathbb{P}_{\mathcal{I}}$
is \emph{continuously homogenous }if for every $B\in\mathbb{P}_{\mathcal{I}}$
it is the case that $\mathcal{I}\upharpoonright B$ $\mathcal{\leq}_{CK}$
$\mathcal{I}.$ Note that if $\mathbb{P}_{\mathcal{I}}$ is continuously
homogenous then $tr\left(  \mathcal{I}\right)  $ is Kat\v{e}tov uniform.

\begin{lemma}
The ideals $tr($\textsf{ctble}$),$ \textsf{nwd}$,$ $tr\left(  K_{\sigma
}\right)  ,$ $tr\left(  \mathcal{L}\right)  $ are all Kat\v{e}tov uniform.
\end{lemma}

\begin{proof}
It is well known that every uncountable Borel set of $2^{\omega}$ contains a
Cantor set, it then follows that \textsf{ctble} is continuously homogenous. A
similar argument works for Miller and Laver forcings. Finally, if $A\notin$
\textsf{nwd} then it contains a copy of the rational numbers. \qquad\ \ \ 
\end{proof}

\qquad\qquad\qquad\qquad\ 

We can then conclude the following:

\begin{proposition}
Let $\mathbb{P}$ be a partial order.

\begin{enumerate}
\item $\mathbb{P}$ destroys $tr($\textsf{ctble}$)$ if and only if $\mathbb{P}$
adds a new real.

\item $\mathbb{P}$ destroys $tr\left(  K_{\sigma}\right)  $ if and only if
$\mathbb{P}$ adds an unbounded real.
\end{enumerate}
\end{proposition}

\begin{proof}
Clearly if $\mathbb{P}$ destroys $tr($\textsf{ctble}$)$ then it must add a new
real. Conversely, if $r$ is a new real added by $\mathbb{P}$ then $\left\{
r\upharpoonright n\mid n\in\omega\right\}  $ destroys $tr($\textsf{ctble}$).$
The second point follows by the proof of \textsf{cov}$^{\ast}\left(  tr\left(
K_{\sigma}\right)  \right)  =\mathfrak{d}.$\qquad\ \ \ 
\end{proof}

\qquad\ \ \ 

We have the following characterizations:

\begin{proposition}
Let $\mathcal{J}$ be an ideal on $\omega.$

\begin{enumerate}
\item The following are equivalent:

\begin{enumerate}
\item $\mathcal{J}$ is destructible by Sacks forcing.

\item $\mathcal{J}$ is destructible by any forcing adding a new real.

\item $\mathcal{J}\leq_{K}tr($\textsf{ctble}$)$
\end{enumerate}

\item The following are equivalent:

\begin{enumerate}
\item $\mathcal{J}$ is destructible by Miller forcing.

\item $\mathcal{J}$ is destructible by any forcing adding an unbounded real.

\item $\mathcal{J}\leq_{K}tr\left(  K_{\sigma}\right)  .$
\end{enumerate}

\item The following are equivalent:

\begin{enumerate}
\item $\mathcal{J}$ is destructible by Cohen forcing.

\item $\mathcal{J}\leq_{K}$\textsf{nwd}$.$
\end{enumerate}
\end{enumerate}
\end{proposition}

\qquad\ \ \ \ \qquad\ \ \qquad\ \ \ 

The following definition is essentially the same as one considered by Brendle
and Yatabe in \cite{ForcingIndestructibilityofMADFamilies}:

\begin{definition}
Let $\mathcal{I}$ be a $\sigma$-ideal on $\omega^{\omega}$ (or $2^{\omega}$)
such that $\mathbb{P}_{\mathcal{I}}$ is proper and has the continuous reading
of names. We say $\mathcal{I}$ has \emph{very weak fusion }if for every ideal
$\mathcal{J}$ on $\omega,$ the following conditions are equivalent:

\begin{enumerate}
\item There is a condition $B\in\mathbb{P}_{\mathcal{I}}$ such that
$B\Vdash``\mathcal{J}$ is not tall$\textquotedblright.$

\item There is $a\in tr\left(  \mathcal{I}\right)  ^{+}$ such that
$\mathcal{J}\leq_{KB}tr\left(  \mathcal{I}\right)  \upharpoonright a.$
\end{enumerate}
\end{definition}

\qquad\ \ 

We then have the following:

\begin{proposition}
[\cite{ForcingIndestructibilityofMADFamilies}]\textsf{ctble}$,\mathcal{M}%
,\mathcal{N}$ and $\mathcal{K}_{\sigma}$ have very weak fusion.
\end{proposition}

\begin{proof}
Let $\mathcal{J}$ be an ideal on $\omega,$ we need to prove that if
$\mathbb{P}_{\mathcal{I}}$ destroys $\mathcal{J}$ (for $\mathcal{I}$ one of
the ideals mentioned in the proposition) then there is $a\in tr\left(
\mathcal{I}\right)  ^{+}$ such that $\mathcal{J}\leq_{KB}tr\left(
\mathcal{I}\right)  \upharpoonright a.$

\qquad\ \ 

We first prove it for Cohen forcing. Let $s\in\omega^{<\omega}$ and $\dot{X}$
be a name for an infinite set that is forced by $s$ to be almost disjoint with
$\mathcal{J}.$ Let $C=\left\{  t_{n}\mid n\in\omega\right\}  \subseteq
\omega^{<\omega}$ be the set of all extensions of $s.$ We now recursively find
$a=\{\overline{t}_{n}\mid n\in\omega\}\subseteq\omega^{<\omega}$ and $\left\{
m_{n}\mid n\in\omega\right\}  \subseteq\omega$ such that for every $n\in
\omega$ the following holds:

\begin{enumerate}
\item $\overline{t}_{n}$ extends $t_{n}.$

\item $m_{n}<m_{n+1}.$

\item $\overline{t}_{n}\Vdash``m_{n}\in\dot{X}\textquotedblright.$
\end{enumerate}

\qquad\ \ 

This is very easy to do and clearly $a\in tr\left(  \mathcal{M}\right)  ^{+}.$
We now define $f:a\longrightarrow\omega$ where $f\left(  \overline{t}%
_{n}\right)  =m_{n}.$ It is easy to see that $f$ is injective and is a
Kat\v{e}tov morphism from $\left(  a,tr\left(  \mathcal{M}\right)
\upharpoonright a\right)  $ to $\left(  \omega,\mathcal{J}\right)  .$

\qquad\ \ 

We now prove it for Sacks forcing. Let $p$ be a Sacks tree and $\dot{X}$ be a
name for an infinite set that is forced by $\mathbb{S}$ to be almost disjoint
with $\mathcal{J}.$ We recursively construct $\left\{  p_{s}\mid
s\in2^{<\omega}\right\}  \subseteq\mathbb{S},$ $a=\left\{  t_{s}\mid
s\in2^{<\omega}\right\}  \subseteq2^{<\omega}$ and $\left\{  m_{s}\mid
s\in2^{<\omega}\right\}  \subseteq\omega$ such that for every $s\in2^{<\omega
}$ the following holds:

\begin{enumerate}
\item $p_{\emptyset}\leq p.$

\item $p_{s^{\frown}0}$ and $p_{s^{\frown}1}$ are two incompatible extensions
of $p_{s}.$

\item $t_{s}$ is the stem of $p_{s}.$

\item $t_{s^{\frown}0}$ and $t_{s^{\frown}1}$ are incomparable nodes of
$2^{<\omega}$ and $t_{s^{\frown}0}\cap t_{s^{\frown}1}=t_{s}.$

\item $p_{s}\Vdash``m_{s}\in\dot{X}\textquotedblright.$

\item If $s\neq t$ then $m_{s}\neq m_{t}.$
\end{enumerate}

\qquad\ \ 

Once again, this is easy to do and $a\in tr($\textsf{ctble}$)^{+}.$ We now
define $f:a\longrightarrow\omega$ where $f\left(  t_{s}\right)  =m_{s}.$ It is
easy to see that $f$ is injective and is a Kat\v{e}tov morphism from
$(a,tr($\textsf{ctble}$)^{+}\upharpoonright a)$ to $\left(  \omega
,\mathcal{J}\right)  .$ A very similar proof works for Miller forcing.

\qquad\ \ 

Finally, we prove it for random forcing. Let $T\subseteq2^{<\omega}$ be a tree
such that $\left[  T\right]  $ has positive Lebesgue measure and $\dot
{X}=\{\dot{x}_{n}\mid n\in\omega\}$ be a name for an infinite set that is
forced by $\mathbb{B}$ to be almost disjoint with $\mathcal{J}.$ By the usual
proof that random forcing is $\omega^{\omega}$-bounding, we may assume there
are $\left\{  F_{n}\mid n\in\omega\right\}  $ and $\left\{  h_{n}\mid
n\in\omega\right\}  $ such that for every $n\in\omega$ the following holds:

\begin{enumerate}
\item $F_{n}$ is a finite maximal antichain of $T.$

\item $F_{n+1}$ refines $F_{n}.$

\item $h_{n}:F_{n}\longrightarrow\omega.$

\item If $s\in F_{n}$ then $T_{s}\Vdash``\dot{x}_{n}=h_{n}\left(  s\right)
\textquotedblright.$
\end{enumerate}

\qquad\ \ 

Let $W\in\left[  \omega\right]  ^{\omega}$ such that if $n,m\in W$ and $n<m$
then $h_{n}\left(  s\right)  <h_{m}\left(  t\right)  $ for every $s\in F_{n}$
and $t\in F_{m}.$ Let $a=%
{\textstyle\bigcup\limits_{n\in W}}
F_{n}$ then $a\in tr\left(  N\right)  ^{+}.$ We now define $f:a\longrightarrow
\omega$ where $f=%
{\textstyle\bigcup\limits_{n\in W}}
h_{n}.$ It is easy to see that $f$ is finite to one and is a Kat\v{e}tov
morphism from $\left(  a,tr\left(  \mathcal{N}\right)  \upharpoonright
a\right)  $ to $\left(  \omega,\mathcal{J}\right)  .$
\end{proof}

\qquad\ \ \ 

We will not need the following results, but we would like to mention them
since they are important in the theory of destructibility of ideals:

\begin{proposition}
\qquad\ \ 

\begin{enumerate}
\item (Laflamme \cite{ZappingSmallFilters}) Every $F_{\sigma}$-ideal can be
destroyed without adding unbounded reals.

\item (Zapletal \cite{ForcingIdealized}) Every $F_{\sigma}$-ideal can be
destroyed without adding unbounded reals or splitting reals.

\item (Raghavan, Shelah \cite{TwoinequalitiesbetweenCardinalInvariants}) The
density zero ideal can not be destroyed without adding unbounded reals.
\end{enumerate}
\end{proposition}

\section{Indestructibility of \textsf{MAD} families}

Let $\mathcal{A}$ be a \textsf{MAD} family and $\mathbb{P}$ a forcing notion.
We say that $\mathbb{P}$ \emph{destroys }$\mathcal{A}$ if $\mathcal{A}$ is not
longer maximal after forcing with $\mathbb{P}.$ Clearly, $\mathbb{P}$ destroys
$\mathcal{A}$ if and only if $\mathbb{P}$ destroys $\mathcal{I}\left(
\mathcal{A}\right)  .$ Recall that if $\mathcal{I}$, $\mathcal{J}$ are ideals
and $\mathcal{I\leq}_{K}$ $\mathcal{J}$ then \textsf{cov}$^{\ast}\left(
\mathcal{J}\right)  \leq$ \textsf{cov}$^{\ast}\left(  \mathcal{I}\right)  $
and that \textsf{cov}$^{\ast}\left(  \mathcal{I}\left(  \mathcal{A}\right)
\right)  =\left\vert \mathcal{A}\right\vert .$

\begin{corollary}
Let $\mathcal{A}$ be a \textsf{MAD} family.

\begin{enumerate}
\item If $\left\vert \mathcal{A}\right\vert <\mathfrak{c}$ then $\mathcal{A}$
is Sacks indestructible.

\item If $\left\vert \mathcal{A}\right\vert <\mathfrak{d}$ then $\mathcal{A}$
is Miller indestructible.

\item If $\left\vert \mathcal{A}\right\vert <$ \textsf{cov}$\left(
\mathcal{M}\right)  $ then $\mathcal{A}$ is Cohen indestructible.
\end{enumerate}
\end{corollary}

\qquad\ \ \ 

Since every tall ideal contains a \textsf{MAD} family, there are Sacks
destructible \textsf{MAD} families. However, the answers of the following
questions are unknown:

\begin{problem}
[Stepr\={a}ns]Is there a Cohen indestructible \textsf{MAD} family?
\end{problem}

\begin{problem}
[Hru\v{s}\'{a}k]\ \ \ Is there a Sacks indestructible \textsf{MAD} family?
\end{problem}

\qquad\ \ 

The answer is positive under many additional axioms, but it is currently
unknown if it is possible to build such families on the basis of \textsf{ZFC}
alone. Since every \textsf{AD} family is Kat\v{e}tov below \textsf{FIN}%
$\times$\textsf{FIN} we have the following:

\begin{proposition}
If $\mathbb{P}$ adds a dominating real then $\mathbb{P}$ destroys every ground
model \textsf{MAD} family.\qquad\ \ \ 
\end{proposition}

\qquad\ \ \ 

The following lemma shows that tight families are Cohen indestructible.
Moreover, these concepts are almost the same:

\begin{lemma}
[\cite{OrderingMADFamiliesalaKatetov}]\qquad\ \ \ \qquad\ \ \ \ 

\begin{enumerate}
\item If $\mathcal{A}$ is tight then $\mathcal{A}$ is Cohen-indestructible.

\item If $\mathcal{A}$ is Cohen indestructible then there is $X\in
\mathcal{I}\left(  \mathcal{A}\right)  ^{+}$ such that $\mathcal{A}%
\upharpoonright X$ is tight.
\end{enumerate}
\end{lemma}

\begin{proof}
Let $\mathcal{A}$ be a tight \textsf{MAD} family, we will show that
$\mathcal{I}\left(  \mathcal{A}\right)  \nleq_{K}$\textsf{nwd}$.$ Let
$f:\mathbb{Q}\longrightarrow\omega$ be a function, we will show it is not a
Kat\v{e}tov morphism. Let $\left\{  U_{n}\mid n\in\omega\right\}  $ be a base
of open sets for the rational numbers. If there is $n\in\omega$ such that
$f\left[  U_{n}\right]  \in\mathcal{I}\left(  \mathcal{A}\right)  $ then
clearly $f$ is not a Kat\v{e}tov morphism, so assume $\left\{  f\left[
U_{n}\right]  \mid n\in\omega\right\}  \subseteq\mathcal{I}\left(
\mathcal{A}\right)  ^{+}.$ Since $\mathcal{A}$ is tight, there is
$B\in\mathcal{I}\left(  \mathcal{A}\right)  $ such that $B\cap f\left[
U_{n}\right]  $ is infinite for each $n\in\omega.$ Then $f^{-1}\left(
B\right)  $ is dense so $f^{-1}\left(  B\right)  \notin$ \textsf{nwd}$.$

\qquad\ \ 

We prove 2 by the contrapositive. Assuming that $\mathcal{A}$ is a
\textsf{MAD} family without tight restrictions, we will show Cohen forcing
destroys $\mathcal{A}.$ We recursively build $\left\{  X_{s}\mid s\in
\omega^{<\omega}\right\}  $ such that for all $s\in\omega^{<\omega}$ the
following holds:

\begin{enumerate}
\item $X_{s}\in\mathcal{I}\left(  \mathcal{A}\right)  ^{+}.$

\item $\mathcal{A}\upharpoonright X_{s}$ is not tight.

\item $\left\{  X_{s^{\frown}n}\mid n\in\omega\right\}  $ witness the non
tightness of $\mathcal{A}\upharpoonright X_{s}$ (we may assume that for every
$A\in\mathcal{I}\left(  \mathcal{A}\upharpoonright X_{s}\right)  $ there is
$n\in\omega$ such that $A$ and $X_{s^{\frown}n}$ are disjoint).

\item $X_{\emptyset}=\omega.$
\end{enumerate}

\qquad\ \ 

This is easy to do since $\mathcal{A}$ does not have tight restrictions. Let
$c\in\omega^{\omega}$ be a Cohen real over $\omega^{<\omega}.$ Now, in
$V\left[  c\right]  $ we find a pseudointersection $X$ of $\left\{
X_{c\upharpoonright n}\mid n\in\omega\right\}  $ and we claim that $X$ is
forced to be almost disjoint with every element of $\mathcal{A}.$ Let
$s\in\omega^{<\omega}$ be a Cohen condition and $B\in\mathcal{I}\left(
\mathcal{A}\right)  .$ By our construction, there is $n\in\omega$ such that
$B\cap X_{s^{\frown}n}=\emptyset$ hence $s^{\frown}n$ forces that $\dot{X}$
and $B$ are almost disjoint.
\end{proof}

\qquad\ \ \ \ \qquad\ \ \ 

In this way, there are tight \textsf{MAD} families if and only if there are
Cohen indestructible \textsf{MAD} families. Nevertheless, Cohen
indestructibility (consistently) does not imply tightness, as we will prove
later. We will need the following lemma:

\begin{lemma}
Let $\mathcal{A}$ be an \textsf{AD} family of size less than $\mathfrak{b}.$
If $\left\{  X_{n}\mid n\in\omega\right\}  \subseteq\mathcal{I}\left(
\mathcal{A}\right)  ^{+}$ then there is $B\in\mathcal{A}^{\perp}$ such that
$B\cap X_{n}\neq\emptyset$ for every $n\in\omega.$
\end{lemma}

\begin{proof}
Since $\mathcal{A}$ is nowhere \textsf{MAD}, we may assume $X_{n}%
\in\mathcal{A}^{\perp}$ for every $n\in\omega$ and they are disjoint. For
every $A\in\mathcal{A}$ we define a function $f_{A}:\omega\longrightarrow
\omega$ where $A\cap X_{n}\subseteq f_{A}\left(  n\right)  $ for each
$n\in\omega.$ Since $\left\vert \mathcal{A}\right\vert <\mathfrak{b}$ there is
$g\in\omega^{\omega}$ dominating each $f_{A}.$ Choose any $B=\left\{
b_{n}\mid n\in\omega\right\}  $ such that $b_{n}\in X_{n}\setminus g\left(
n\right)  $ then $B$ has the desired properties. \ \ \ \ 
\end{proof}

\qquad\ \ \ 

With a usual bookkeeping argument we can then conclude the following:

\begin{proposition}
If $\mathfrak{b=c}$ then tight \textsf{MAD} families exist generically.
\end{proposition}

\qquad\qquad\ \ \ 

We will now prove the converse of the previous proposition.

\begin{proposition}
There is an \textsf{AD} family of size $\mathfrak{b}$ that can not be extended
to a tight \textsf{MAD} family.
\end{proposition}

\begin{proof}
Define $\pi:\omega\times\omega\longrightarrow\omega$ by $\pi\left(
n,m\right)  =n.$ If $A\subseteq\omega\times\omega$ and $n\in\omega$ we put
$\left(  A\right)  _{n}=\left\{  k\mid\left(  n,k\right)  \in A\right\}  $ and
let $\omega^{<\omega}=\left\{  s_{n}\mid n\in\omega\right\}  .$ We define
$H:\omega^{\omega}\longrightarrow\wp\left(  \omega\times\omega\right)  $ where
$\pi\left(  H\left(  f\right)  \right)  =\left\{  n\mid s_{n}\sqsubseteq
f\right\}  $ and if $n\in dom\left(  H\left(  f\right)  \right)  $ then
$\left(  H\left(  f\right)  \right)  _{n}=f\left(  \left\vert s_{n}\right\vert
\right)  .$ It is easy to see that if $f\neq g$ then $H\left(  f\right)  $ and
$H\left(  g\right)  $ are almost disjoint. Given $g:\omega\longrightarrow
\omega$ we define $N\left(  g\right)  =\left\{  f\in\omega^{\omega}%
\mid\left\vert H\left(  f\right)  \cap g\right\vert <\omega\right\}  .$ It
then follows that $N\left(  g\right)  $ is a bounded set since $N\left(
g\right)  =\bigcup\limits_{k\in\omega}N_{k}\left(  g\right)  $ where
$N_{k}\left(  g\right)  =\left\{  f\in\omega^{\omega}\mid\left\vert H\left(
f\right)  \cap g\right\vert <k\right\}  $ and it is easy to see that each
$N_{k}\left(  g\right)  $ is $\sigma$-compact.

\qquad\ \ \ \ 

Let $X\subseteq\omega^{\omega}$ of size $\mathfrak{b}$ that can not be covered
by $\sigma$-compact sets. Now, we may find $\mathcal{A}$ an AD family of size
$\mathfrak{b}$ such that $H\left[  X\right]  \subseteq\mathcal{A}$ and
$C_{n}=\left\{  \left(  n,m\right)  \mid m\in\omega\right\}  \in
\mathcal{I}\left(  \mathcal{A}\right)  ^{++}$ for every $n\in\omega.$ It is
easy to see that $\mathcal{A}$ can not be extended to a tight \textsf{MAD} family.
\end{proof}

\qquad\ \ \ \qquad\ \ 

We can then conclude the following result:

\begin{corollary}
The following statements are equivalent:

\begin{enumerate}
\item $\mathfrak{b=c}.$

\item Tight \textsf{MAD} families exist generically.
\end{enumerate}
\end{corollary}

\qquad\ \ \ 

We will mention some other results regarding the generic existence of
indestructible \textsf{MAD} families. In order to do so, we need the following definition.

\begin{definition}
Let $\mathcal{J}$ be a tall ideal on $\omega.$ We define $\mathfrak{a}\left(
\mathcal{J}\right)  $ as smallest size of an AD family $\mathcal{A}$ such that
$\mathcal{A}\cup\mathcal{A}^{\perp}\subseteq\mathcal{J}.$\qquad
\end{definition}

\qquad\ \ \ 

The following lemma follows by the definitions.

\begin{lemma}
Let $\mathcal{J}$ be a tall ideal on $\omega$ and $\mathcal{A}$ an infinite AD
family of size less than $\mathfrak{a}\left(  \mathcal{J}\right)  .$ If
$f:\omega\longrightarrow\omega$ is a finite to one function, then there is an
AD family $\mathcal{B}$ such that the following holds:

\begin{enumerate}
\item $\mathcal{A\subseteq B}.$

\item $\left\vert \mathcal{B}\right\vert =\left\vert \mathcal{A}\right\vert .$

\item There is $B\in\mathcal{I}\left(  \mathcal{B}\right)  $ such that
$f^{-1}\left(  B\right)  \in\mathcal{J}^{+}.$
\end{enumerate}
\end{lemma}

\qquad\ \ \ 

Then we have the following:

\begin{proposition}
Let $\mathcal{I}$ be a $\sigma$-ideal on $\omega^{\omega}$ that has very weak
fusion and $tr\left(  \mathcal{I}\right)  $ is Kat\v{e}tov uniform. The
following statements are equivalent:

\begin{enumerate}
\item $\mathbb{P}_{\mathcal{I}}$ indestructible \textsf{MAD} families exist generically.

\item $\mathfrak{a}\left(  tr\left(  \mathcal{I}\right)  \right)
=\mathfrak{c}.$\qquad\ \ \ \qquad\ \ 
\end{enumerate}
\end{proposition}

\begin{proof}
By a simple bookkeeping argument and the previous lemma we conclude that 2
implies 1. Clearly, every witness of $\mathfrak{a}\left(  tr\left(
\mathcal{I}\right)  \right)  $ can not be extended to a $\mathbb{P}%
_{\mathcal{I}}$ indestructible \textsf{MAD} family.
\end{proof}

\qquad\ \ \ 

The previous arguments show that $\mathfrak{a}\left(  tr\left(  \mathcal{M}%
\right)  \right)  =\mathfrak{b}.$ The following definition is useful for
studying this invariants:

\begin{definition}
Let $\mathcal{J}$ be a tall ideal. We define \textsf{cov}$^{+}\left(
\mathcal{J}\right)  $ as the smallest family $\mathcal{B}\subseteq\mathcal{J}$
such that for every $X\in\mathcal{J}^{+}$ there is $B\in\mathcal{B}$ such that
$B\cap X$ is infinite.
\end{definition}

\qquad\ \ \ 

Clearly \textsf{cov}$^{\ast}\left(  \mathcal{J}\right)  $ and $\mathfrak{a}%
\left(  \mathcal{J}\right)  $ are upperbounds for \textsf{cov}$^{+}\left(
\mathcal{J}\right)  .$ Given $s\in2^{<\omega}$ we define $\left\langle
s\right\rangle _{<\omega}=\left\{  t\in2^{<\omega}\mid s\sqsubseteq t\right\}
.$ It is clear that if $X\cap\left\langle s\right\rangle _{<\omega}%
\neq\emptyset$ for every $s\in2^{<\omega}$ then $X\notin tr($\textsf{ctble}%
$)\,.$ Let $\mathcal{BR}$ be the ideal of $2^{<\omega}$ generated by branches.
In this way $\mathcal{BR}^{\perp}$ is the ideal of all well founded subsets of
$2^{<\omega},$ its elements are called \emph{off-branch }and it is clear that
$\mathcal{BR}^{\perp}\subseteq$ $tr($\textsf{ctble}$)\,.$ We have the
following simpler characterization of \textsf{cov}$^{+}(tr($\textsf{ctble}%
$)\,).$

\begin{lemma}
\textsf{cov}$^{+}(tr($\textsf{ctble}$)\,)$ is the minimum size a family
$\mathcal{B\subseteq BR}^{\perp}$ such that for every $A\in tr($%
\textsf{ctble}$)^{+}$ there is $B\in\mathcal{B}$ such that $\left\vert A\cap
B\right\vert =\omega.$
\end{lemma}

\begin{proof}
Call $\mu$ the minimum size a family $\mathcal{B\subseteq BR}^{\perp}$ such
that for every $A\in tr($\textsf{ctble}$)^{+}$ there is $B\in\mathcal{B}$ such
that $\left\vert A\cap B\right\vert =\omega.$ It is clear that \textsf{cov}%
$^{+}(tr($\textsf{ctble}$)\,)\leq\mu$ and we shall now prove the other
inequality. In case \textsf{cov}$^{+}(tr($\textsf{ctble}$))=\mathfrak{c}$
there is nothing to prove, so assume \textsf{cov}$^{+}(tr($\textsf{ctble}$))$
is less than size of the continuum and let $\mathcal{B}\subseteq
tr($\textsf{ctble}$)$ witness this fact. Since $2^{\omega\times\omega}%
\cong2^{\omega}$ we may find a partition $\left\{  \left[  T_{\alpha}\right]
\mid\alpha<\mathfrak{c}\right\}  $ of $2^{\omega}$ where each $T_{\alpha}$ is
a Sacks tree. Since $\mathcal{B}\subseteq tr($\textsf{ctble}$)$ and has size
less than $\mathfrak{c},$ there is $T_{\alpha}$ such that $\pi\left(
B\right)  \cap\left[  T_{\alpha}\right]  =\emptyset$ for every $B\in
\mathcal{B}.$ The splitting nodes of $T_{\alpha}$ is isomorphic to
$2^{<\omega}$ and for every $B\in\mathcal{B}$ it is the case that $B\cap
T_{\alpha}$ is offbranch in $T_{\alpha}.$
\end{proof}

\qquad\ \ \ \ 

Now it is easy to prove the following:

\begin{proposition}
\textsf{cov}$\left(  \mathcal{M}\right)  \leq$ \textsf{cov}$^{+}%
(tr($\textsf{ctble}$)).$
\end{proposition}

\begin{proof}
Let $\kappa<$ \textsf{cov}$\left(  \mathcal{M}\right)  $ and $\mathcal{A=}%
\left\{  A_{\alpha}\mid\alpha\in\kappa\right\}  \subseteq\mathcal{BR}^{\perp}%
$. We ought to find $B\in tr($\textsf{ctble}$)^{+}$ that is AD with
$\mathcal{A}.$ Let $\mathbb{P}$ be the partial order of all finite trees
contained in $2^{<\omega}$ and we order it by end extension. Obviously,
$\mathbb{P}$ is isomorphic to Cohen forcing. Let $\dot{T}_{gen}$ be the name
for the generic tree, clearly $\dot{T}_{gen}$ is forced to be a Sacks tree.
For every $\alpha<\kappa\ $define the set $D_{\alpha}$ of all $T\in\mathbb{P}$
such that if $s\in T$ is a maximal node, then $\left\langle s\right\rangle
_{\leq\omega}\cap A_{\alpha}=\emptyset.$ It is straightforward to see that
$D_{\alpha}$ is dense. Since $\kappa<$ \textsf{cov}$\left(  \mathcal{M}%
\right)  $ then we can find in $V$ a filter that intersects every $D_{\alpha}$
and the result follows.
\end{proof}

\qquad\ \ 

We can then conclude the following:

\begin{corollary}
[\cite{ForcingIndestructibilityofMADFamilies}]If \textsf{cov}$\left(
\mathcal{M}\right)  =\mathfrak{c}$ then Sacks indestructible \textsf{MAD}
families exist generically.
\end{corollary}

\qquad\ \ 

For simplicity, we define $\mathfrak{a}_{Sacks}=\mathfrak{a}(tr($%
\textsf{ctble}$)).$ As before, $\mathfrak{a}_{Sacks}$ is the smallest size of
an almost disjoint family $\mathcal{A\subseteq BR}^{\perp}$ such that
$\mathcal{A\cup A}^{\perp}\subseteq$ $tr($\textsf{ctble}$).$ The reason we are
interested in this cardinal invariant is the following:

\begin{proposition}
If $\mathfrak{a}\leq\mathfrak{a}_{Sacks}$ then there is a Sacks indestructible
\textsf{MAD} family.
\end{proposition}

\begin{proof}
If $\mathfrak{a<c}$ then any \textsf{MAD} family of minimum size is Sacks
indestructible and if $\mathfrak{a=c}$ then $\mathfrak{a}_{Sacks}%
=\mathfrak{c}$ so Sacks indestructible \textsf{MAD} families exist generically.
\end{proof}

\qquad\ \ \ 

We do not know if the inequality $\mathfrak{a}_{Sacks}<\mathfrak{a}$ is consistent.

\qquad\ \ \ \ \qquad\ \ \ 

The following is a very important result of Shelah regarding the
destructibility of \textsf{MAD} families (see \cite{ProperandImproper}).

\begin{proposition}
Every \textsf{MAD} family can be destroyed with a proper forcing that does not
add dominating reals.\ \ \ 
\end{proposition}

\qquad\ \ \ \ \ \ 

Letting $\mathcal{I}$ be an ideal in $\omega,$ by $\left(  \mathcal{I}%
^{<\omega}\right)  ^{+}$ we denote the set of all $X\subseteq\left[
\omega\right]  ^{<\omega}\setminus\left\{  \emptyset\right\}  $ such that for
every $A\in\mathcal{I}$ there is $s\in X$ such that $s\cap A=\emptyset.$ If
$\mathcal{F}$ is a filter then we define $\left(  \mathcal{F}^{<\omega
}\right)  ^{+}$ as $\left(  \left(  \mathcal{F}^{\ast}\right)  ^{<\omega
}\right)  ^{+}.$ Note that if $X\subseteq\left[  \omega\right]  ^{<\omega
}\setminus\left\{  \emptyset\right\}  $ then $X\in\left(  \mathcal{F}%
^{<\omega}\right)  ^{+}$ if and only if for every $A\in\mathcal{F}$ there is
$s\in X$ such that $s\subseteq A.$ The following definition will be very
important in the rest of the chapter:

\begin{definition}
An ideal $\mathcal{I}$ is called \emph{Shelah-Stepr\={a}ns }if for every
$X\in\left(  \mathcal{I}^{<\omega}\right)  ^{+}$ there is $Y\in\left[
X\right]  ^{\omega}$ such that $\bigcup Y\in\mathcal{I}.$
\end{definition}

\qquad\ \ \ \qquad\ \ 

In other words, an ideal $\mathcal{I}$ is Shelah-Stepr\={a}ns if for every
$X\subseteq\left[  \omega\right]  ^{<\omega}\setminus\left\{  \emptyset
\right\}  $ either there is $A\in\mathcal{I}$ such that $s\cap A\neq\emptyset$
for every $s\in X$ or there is $B\in\mathcal{I}$ that contains infinitely many
elements of $X.$ The previous notion was introduced by Raghavan in
\cite{AModelwithnoStronglySeparableAlmostDisjointFamilies} for almost disjoint
families, which is connected to the notion of \textquotedblleft strongly
separable\textquotedblright\ introduced by Shelah and Steprans in
\cite{MasasintheCalkinAlgebra}.

\begin{lemma}
Every non-meager ideal is Shelah-Stepr\={a}ns.
\end{lemma}

\begin{proof}
Let $\mathcal{I}$ be a non-meager ideal and $X\in\left(  \mathcal{I}^{<\omega
}\right)  ^{+}.$ Note that since $X\in\left(  \mathcal{I}^{<\omega}\right)
^{+}$ (and $\mathcal{I}$ contains every finite set) for every $n\in\omega$
there is $s\in X$ such that $s\cap n=\emptyset.$ In this way we can find
$Z=\left\{  s_{n}\mid n\in\omega\right\}  \subseteq X$ such that if $n\neq m$
then $s_{n}\cap s_{m}=\emptyset.$ We then define $M=\left\{  A\subseteq
\omega\mid\forall^{\infty}n\left(  s_{n}\nsubseteq A\right)  \right\}  $ which
is clearly a meager set and then there must be $A\in\mathcal{I}$ such that
$A\notin M$ hence there is $Y\in\left[  X\right]  ^{\omega}$ such that
$\bigcup Y\subseteq A\in\mathcal{I}.$
\end{proof}

\qquad\ \ \ \qquad

Nevertheless, there are meager ideals that are also Shelah-Stepr\={a}ns as the
following result shows:

\begin{lemma}
\textsf{FIN}$\times$\textsf{FIN} is Shelah-Stepr\={a}ns.
\end{lemma}

\begin{proof}
It is easy to see that if $X\in($\textsf{FIN}$\times$\textsf{FIN)}$^{+}$ then
there must be infinitely many elements of $X$ that are below the graph of a
function, so there must be $Y\in\left[  X\right]  ^{\omega}$ such that
$\bigcup Y\in\mathcal{I}.$ \qquad\ \ \ 
\end{proof}

\qquad\ \ \ \ \ \ \ \ \ \ \ \ 

We will now show that the property of being Shelah-Stepr\={a}ns is upward
closed in the Kat\v{e}tov order:

\begin{lemma}
Let $\mathcal{I}$ and $\mathcal{J}$ be two ideals on $\omega.$ If the ideal
$\mathcal{I}$ is Shelah-Stepr\={a}ns and $\mathcal{I}\leq_{\mathsf{K}}$
$\mathcal{J}$ then $\mathcal{J}$ is also Shelah-Stepr\={a}ns.
\end{lemma}

\begin{proof}
Let $f:\omega\longrightarrow\omega$ be a Kat\u{e}tov-morphism from $\left(
\omega,\mathcal{J}\right)  $ to $\left(  \omega,\mathcal{I}\right)  .$ Letting
$X\in\left(  \mathcal{J}^{<\omega}\right)  ^{+}$ we must find $Y\in\left[
X\right]  ^{\omega}$ such that $\bigcup Y\in\mathcal{J}.$ Define
$X_{1}=\left\{  f\left[  s\right]  \mid s\in X\right\}  ,$ we will first argue
that \thinspace$X_{1}\in\left(  \mathcal{I}^{<\omega}\right)  ^{+}.$ To prove
this fact, let $A\in\mathcal{I}.$ Since $f$ is a Kat\u{e}tov-morphism,
$f^{-1}\left(  A\right)  \in\mathcal{J}$ so there is $s\in X$ for which $s\cap
f^{-1}\left(  A\right)  =\emptyset$ and then $f\left[  s\right]  \cap
A=\emptyset.$ Since $\mathcal{I}$ is Shelah-Stepr\={a}ns, there is $Y_{1}%
\in\left[  X_{1}\right]  ^{\omega}$ such that $\bigcup Y_{1}\in\mathcal{I}.$
Finally if $Y\in\left[  X\right]  ^{\omega}$ is such that $Y_{1}=\left\{
f\left[  s\right]  \mid s\in Y\right\}  $ then $\bigcup Y\in\mathcal{J}.$
\qquad\ \ 
\end{proof}

\qquad\ \ \ \qquad\ \ \ \ \ 

We will need the following game designed by Claude Laflamme: Let $\mathcal{I}$
be an ideal on $\omega,$ define the game $\mathcal{L}\left(  \mathcal{I}%
\right)  $ between players $\mathsf{I}$ and $\mathsf{II}$ as follows:

\qquad\ \ \ \ \ \ \ \ \ \qquad\ \ \ \ \qquad\ \ \ \qquad\ \ 

\begin{center}%
\begin{tabular}
[c]{|l|l|l|l|l|l|}\hline
$\mathsf{I}$ & $...$ & $A_{n}$ &  & $...$ & \\\hline
$\mathsf{II}$ & $...$ &  & $s_{n}$ & $...$ & $\bigcup s_{n}\in\mathcal{I}^{+}%
$\\\hline
\end{tabular}

\end{center}

\qquad\ \ \ \ \ \ \ \qquad\qquad\qquad\ \ \ \ \ \ \ \qquad\qquad\ \ \ \qquad\ \ 

At the round $n\in\omega$ player $\mathsf{I}$ plays $A_{n}\in\mathcal{I}$ and
$\mathsf{II}$ responds with $s_{n}\in\lbrack\omega\backslash$ $A_{n}%
]^{<\omega}.$ The player $\mathsf{II}$ wins in case $\bigcup s_{n}%
\in\mathcal{I}^{+}$. The following is a result of Laflamme.

\begin{proposition}
[Laflamme \cite{FilterGamesandCombinatorialPropertiesofStrategies}]Let
$\mathcal{I}$ be an ideal on $\omega.$

\begin{enumerate}
\item The following are equivalent:

\begin{enumerate}
\item $\mathsf{I}$ has a winning strategy in $\mathcal{L}\left(
\mathcal{I}\right)  .$

\item \textsf{FIN}$\times$\textsf{FIN }$\leq_{\mathsf{K}}\mathcal{I}.$
\end{enumerate}

\item The following are equivalent:

\begin{enumerate}
\item $\mathsf{II}$ has a winning strategy in $\mathcal{L}\left(
\mathcal{I}\right)  .$

\item There is $\left\{  X_{n}\mid n\in\omega\right\}  \subseteq\left(
\mathcal{I}^{<\omega}\right)  ^{+}$ such that for every $A\in\mathcal{I}$
there is $n\in\omega$ such that $A$ does not contain any element of $X_{n}.$
\end{enumerate}
\end{enumerate}
\end{proposition}

\qquad\ \ 

If $s_{0},...,s_{n}$ are finite non-empty sets of $\omega$, we say $a=\left\{
k_{0},...,k_{n}\right\}  \in\left[  \omega\right]  ^{<\omega}$ is a
\emph{selector }of $\left(  s_{0},...,s_{n}\right)  $ if $k_{i}\in s_{i}$ for
every $i\leq n.$

\ \qquad\ \ \qquad\qquad\qquad\ \ \ \ \ \ \ \ \ \ \ \ \ \ 

\begin{proposition}
If $\mathcal{I}$ is Shelah-Stepr\={a}ns then $\mathsf{II}$ does not have a
winning strategy in $\mathcal{L}\left(  \mathcal{I}\right)  .$
\end{proposition}

\begin{proof}
Let $\mathcal{I}$ be an ideal for which $\mathsf{II}$ has a winning strategy
in $\mathcal{L}\left(  \mathcal{I}\right)  ,$ we will prove that $\mathcal{I}$
is not Shelah-Stepr\={a}ns. Let $\left\{  X_{n}\mid n\in\omega\right\}
\subseteq\left(  \mathcal{I}^{<\omega}\right)  ^{+}$ such that for every
$A\in\mathcal{I}$ there is $n\in\omega$ such that $A$ does not contain any
element of $X_{n}.$ For every $n\in\omega$ enumerate\ $X_{n}=\left\{
t_{n}^{i}\mid i\in\omega\right\}  $ and $\prod\limits_{j<n}X_{j}=\{p_{n}%
^{i}\mid i<\omega\}.$

\qquad\ \ \ \ \ 

For every $n,m\in\omega$ and a selector $a\in\left[  \omega\right]  ^{<\omega
}$ of $\left(  t_{n}^{0},...,t_{n}^{m}\right)  $ we define $F_{\left(
n,m,a\right)  }=p_{n}^{m}\left(  0\right)  \cup...\cup p_{n}^{m}\left(
n-1\right)  \cup a$ (recall $p_{n}^{m}\in\prod\limits_{j<n}X_{j}$). Clearly
each $F_{\left(  n,m,a\right)  }$ is a non-empty finite set. Let $X$ be the
collection of all the $F_{\left(  n,m,a\right)  },$ we will prove that $X$
witnesses that $\mathcal{I}$ is not Shelah-Stepr\={a}ns.

\qquad\ \ 

We will first prove that $X\in\left(  \mathcal{I}^{<\omega}\right)  ^{+}.$
Letting $A\in\mathcal{I}$ we first find $n\in\omega$ such that $A$ does not
contain any element of $X_{n}.$ Since each $X_{j}\in\left(  \mathcal{I}%
^{<\omega}\right)  ^{+}$ for every $j<\omega$ there is $m\in\omega$ such that
$A$ is disjoint with $p_{n}^{m}\left(  0\right)  \cup...\cup p_{n}^{m}\left(
n-1\right)  .$ Finally, by the assumption of $X_{n}$ we can find a selector
$b$ of $\left(  t_{n}^{0},...,t_{n}^{m}\right)  $ such that $b\cap
A=\emptyset$ and therefore $A\cap F_{\left(  n,m,b\right)  }=\emptyset.$

\qquad\ \ 

Letting $Y\in\left[  X\right]  ^{\omega}$ we will show that $B=\bigcup
Y\in\mathcal{I}^{+}.$ There are two cases to consider: first assume there is
$n\in\omega$ for which there are infinitely many $\left(  m,a\right)  $ such
that $F_{\left(  n,m,a\right)  }\in Y.$ In this case, $B$ intersects every
element of $X_{n},$ hence $B\in\mathcal{I}^{+}.$ Now assume that for every
$n\in\omega$ there are only finitely many $\left(  m,a\right)  $ such that
$F_{\left(  n,m,a\right)  }\in Y.$ In this case, there must be infinitely many
$n\in\omega$ for which there is $\left(  m,a\right)  $ such that $F_{\left(
n,m,a\right)  }\in Y,$ hence $B$ must contain (at least) one element of every
$X_{k}.$ We can then conclude that $B\in\mathcal{I}^{+}.$
\end{proof}

\qquad\ \ \ \ \ \qquad\ \ \ \ \ \qquad\ \ 

We can now conclude the following (the equivalence of point 2 and 3 was proved
by Laczkovich and Rec\l aw in
\cite{IdeallimitsofSequencesofContinuousFunctions}, we include the proof for
the convenience of the reader).

\begin{corollary}
Let $\mathcal{I}$ be an ideal on $\omega.$ The following are equivalent:

\begin{enumerate}
\item $\mathcal{I}$ is not Shelah-Stepr\={a}ns.

\item The Player $\mathsf{II}$ has a winning strategy in $\mathcal{L}\left(
\mathcal{I}\right)  .$

\item There is an $F_{\sigma}$ set $F\subseteq\wp\left(  \omega\right)  $ such
that $\mathcal{I}$ $\subseteq F$ and $\mathcal{I}^{\ast}\cap F=\emptyset.$
\end{enumerate}
\end{corollary}

\begin{proof}
By the previous result, we know that 2 implies 1. We will now prove that 1
implies 3. Assume that $\mathcal{I}$ is not Shelah-Stepr\={a}ns, so there is
$X=\left\{  s_{n}\mid n\in\omega\right\}  \in\left(  \mathcal{I}^{<\omega
}\right)  ^{+}$ such that $%
{\textstyle\bigcup}
Y\in\mathcal{I}^{+}$ for every $Y\in\left[  X\right]  ^{\omega}.$ We know
define $F=\left\{  W\subseteq\omega\mid\forall^{\infty}n\left(  s_{n}%
\nsubseteq W\right)  \right\}  .$ It is easy to see that $F$ has the desired properties.

\qquad\ 

We will now prove that 3 implies 2. Assume there is an increaing sequence of
closed sets $\left\langle C_{n}\mid n\in\omega\right\rangle $ such that $F=%
{\textstyle\bigcup\limits_{n\in\omega}}
C_{n}$ contains $\mathcal{I}$ and is disjoint from $\mathcal{I}^{\ast}.$ We
will now describe a winning strategy for Player $\mathsf{II}$: In the first
round, if Player $\mathsf{I}$ plays $A_{0}\in\mathcal{I}$ then Player
$\mathsf{II}$ finds $s_{0}$ an initial segment of $\omega\setminus A_{0}$ such
that $\left\langle s_{0}\right\rangle =\left\{  Z\mid s_{0}\sqsubseteq
Z\right\}  $ is disjoint from $C_{0}$ (where $s_{0}\sqsubseteq Z$ means that
$s_{0}$ is an initial segment of $Z$). At round round $n+1,$ if Player
$\mathsf{I}$ plays $A_{n+1}\in\mathcal{I}$ then Player $\mathsf{II}$ finds
$s_{n+1}$ such that $t=%
{\textstyle\bigcup\limits_{i\leq n+1}}
s_{i}$ is an initial segment of $\left(  \omega\setminus A_{n}\right)  \cup%
{\textstyle\bigcup\limits_{j<n+1}}
s_{j}$ (we may assume $%
{\textstyle\bigcup\limits_{j<n+1}}
s_{j}\subseteq A_{n}$) and $\left\langle t\right\rangle $ is disjoint from
$C_{n+1}.$ It is easy to see that this is a winning strategy.
\end{proof}

Since every game with Borel payoff is determined, we can give a
characterization of the Borel ideals that are Shelah-Stepr\={a}ns.

\begin{corollary}
If $\mathcal{I}$ is a Borel ideal then $\mathcal{I}$ is Shelah-Stepr\={a}ns if
and only if \textsf{FIN}$\times$\textsf{FIN }$\leq_{\mathsf{K}}\mathcal{I}.$
\end{corollary}

\qquad\ \ 

We can extend the previous corollary under some large cardinal assumptions.
Fix a tree $T$ of height $\omega,$ $f:\left[  T\right]  \longrightarrow
\wp\left(  \omega\right)  $ a continuous function (where $\left[  T\right]  $
denotes the set of branches of $T$) and $\mathcal{W}\subseteq\wp\left(
\omega\right)  .$ We then define the game $\mathcal{G}\left(  T,f,\mathcal{W}%
\right)  $ as follows:\qquad\ \ \ \ \ \ \ \ \ \qquad\ \ \ \ \qquad\ \ \ \qquad\ \ 

\begin{center}%
\begin{tabular}
[c]{|l|l|l|l|l|}\hline
$\mathsf{I}$ & $...$ & $x_{n}$ &  & $...$\\\hline
$\mathsf{II}$ & $...$ &  & $y_{n}$ & $...$\\\hline
\end{tabular}

\qquad\ \ \ \ 
\end{center}

At the round $n\in\omega$ player $\mathsf{I}$ plays $x_{n}$ and $\mathsf{II}$
responds with $y_{n}$ with the requirement that $\left\langle x_{0}%
,y_{0},...,x_{n},y_{n}\right\rangle \in T.$ Then Player $\mathsf{I}$ wins if
$f\left(  b\right)  \in\mathcal{W}$ where $b$ is the branch constructed during
the game. The following is a well known extension of Martin's result (see
\cite{ForcingIdealized}):

\begin{proposition}
[\textsf{LC}]If $\mathcal{W}\in$ \textsf{L}$\left(  \mathbb{R}\right)  $ then
$\mathcal{G}\left(  T,f,\mathcal{W}\right)  $ is determined (\textsf{L}%
$\left(  \mathbb{R}\right)  $ denotes the smallest transitive model of
$\mathsf{ZFC}$ that contains all reals)
\end{proposition}

\qquad\ \ 

Where \textsf{LC }denotes a large cardinal assumption. We can then conclude
the following:

\begin{corollary}
[\textsf{LC}]\qquad\ \ \ \qquad\ \ \ \ \ \ \ \qquad
\ \ \ \ \ \ \ \ \ \ \ \ \ \ \ \ \ \ \qquad\qquad\ \ 

\begin{enumerate}
\item Let $\mathcal{I}\in$ \textsf{L}$\left(  \mathbb{R}\right)  $ be an ideal
on $\omega.$ Then $\mathcal{I}$ is Shelah-Stepr\={a}ns if and only if
\textsf{FIN}$\times$\textsf{FIN }$\leq_{\mathsf{K}}\mathcal{I}.$

\item Let $\mathcal{J}$ be a $\sigma$-ideal in $\omega^{\omega}$ such that
$\mathcal{J}\in$ \textsf{L}$\left(  \mathbb{R}\right)  $ and $X\in tr\left(
\mathcal{J}\right)  ^{+}.$ Then $tr\left(  \mathcal{J}\right)  \upharpoonright
X$ is Shelah-Stepr\={a}ns if and only if \textsf{FIN}$\times$\textsf{FIN
}$\leq_{\mathsf{K}}tr\left(  \mathcal{J}\right)  \upharpoonright X.$
\end{enumerate}
\end{corollary}

\begin{proof}
To prove the first item, let $Y$ be the set of all sequences $\left\langle
A_{0},s_{0},...,A_{n},s_{n}\right\rangle $ such that $A_{n}\in\mathcal{I}$ and
$s_{n}\in\left[  \omega\setminus A_{n}\right]  ^{<\omega}$ and $max\left(
s_{i}\right)  \subseteq A_{i+1}$ if $i<n.$ Let $T$ be the tree obtained by
closing $Y$ under restrictions. We know define $f:\left[  T\right]
\longrightarrow\wp\left(  \omega\right)  $ where $f\left(  b\right)  =%
{\textstyle\bigcup\limits_{n\in\omega}}
b\left(  2n+1\right)  $ where $b\in\left[  T\right]  .$ Clearly $\mathcal{L}%
\left(  \mathcal{I}\right)  $ is a game equivalent to $\mathcal{G}\left(
T,f,\mathcal{I}\right)  $ so the result follows from the previous results. The
second item is a consequence of the first.
\end{proof}

\qquad\ \ \ 

We say a \textsf{MAD} family $\mathcal{A}$ is Shelah-Stepr\={a}ns if
$\mathcal{I}\left(  \mathcal{A}\right)  $ is Shelah-Stepr\={a}ns.\ The
following is a very interesting result of Raghavan:

\begin{proposition}
[\cite{AModelwithnoStronglySeparableAlmostDisjointFamilies}]It is consistent
that there are no Shelah-Stepr\={a}ns \textsf{MAD} families.
\end{proposition}

The following result shows that Shelah-Stepr\={a}ns \textsf{MAD} families have
very strong properties:

\begin{corollary}
If $\mathcal{A}$ is Shelah-Stepr\={a}ns then:

\begin{enumerate}
\item $\mathcal{A}$ can not be extended to an $F_{\sigma\delta}$ ideal.

\item $\mathcal{A}$ is Cohen and Random indestructible.

\item \textsf{(LC)} If $\mathcal{J}$ is a $\sigma$-ideal in $\omega^{\omega}$
such that $\mathcal{J}\in$ \textsf{L}$\left(  \mathbb{R}\right)  $ for which
$\mathbb{P}_{\mathcal{J}}$ is proper, has the continuos reading of names and
does not add a dominating real (under any condition), then $\mathcal{A}$ is
$\mathbb{P}_{\mathcal{J}}$-indestructible.
\end{enumerate}
\end{corollary}

\begin{proof}
By results of Solecki, Laczkovich and Rec\l aw, no $F_{\sigma\delta}$ ideal is
Kat\u{e}tov above \textsf{FIN}$\times$\textsf{FIN} (see
\cite{FiltersandSequences} and
\cite{IdeallimitsofSequencesofContinuousFunctions}) this implies the first
item. We will now prove the third item. Let $\mathcal{J}\in$ \textsf{L}%
$\left(  \mathbb{R}\right)  $ be a $\sigma$-ideal in $\omega^{\omega}$ such
that such that $\mathbb{P}_{\mathcal{J}}$ is proper and has the continuos
reading of names. If there is $B\in\mathbb{P}_{\mathcal{J}}$ such that forcing
below $B$ destroys $\mathcal{A},$ then there is $X\in tr\left(  \mathcal{J}%
\right)  ^{+}$ such that $\mathcal{I}\left(  \mathcal{A}\right)  \leq
_{K}tr\left(  \mathcal{J}\right)  \upharpoonright X$. We can then conclude
that $tr\left(  \mathcal{J}\right)  \upharpoonright X$ is Shelah-Stepr\={a}ns
and by our definability hypothesis, we know that \textsf{FIN}$\times
$\textsf{FIN }$\leq_{\mathsf{K}}tr\left(  \mathcal{J}\right)  \upharpoonright
X$ so $\mathbb{P}_{\mathcal{J}}$ must add a dominating real below some
condition. Since the trace of the meager and null ideals is Borel, in this
case the large cardinals hypothesis is not needed.\ \ 
\end{proof}

\qquad\ \ \ \ \qquad\ \ \ \qquad\ \ \ \qquad\ \ \ \qquad\ \ \ 

We will now prove that such families exist under certain assumptions:

\begin{proposition}
If $\mathfrak{p=c}$ then Shelah-Stepr\={a}ns \textsf{MAD} families exist generically.
\end{proposition}

\begin{proof}
Let $\mathcal{A}$ be an AD family of size less than $\mathfrak{c}$ and
$X=\left\{  s_{n}\mid n\in\omega\right\}  \in\left(  \mathcal{I}\left(
\mathcal{A}\right)  ^{<\omega}\right)  ^{+}.$ We define the forcing
$\mathbb{P}$ as the set of all $p=\left(  t_{p},\mathcal{F}_{p}\right)  $
where $t_{p}\in2^{<\omega}$ and $\mathcal{F}_{p}\in\left[  \mathcal{A}\right]
^{<\omega}.$ If $p=\left(  t_{p},\mathcal{F}_{p}\right)  $ and $q=\left(
t_{q},\mathcal{F}_{q}\right)  $ then $p\leq q$ if the following holds:

\qquad\ \ 

\begin{enumerate}
\item $t_{q}\subseteq t_{p}$ and $\mathcal{F}_{q}\subseteq\mathcal{F}_{p}.$

\item In case $n\in dom\left(  t_{p}\right)  $ $\backslash dom\left(
t_{q}\right)  $ and $A\in\mathcal{F}_{q}$ if $t_{p}\left(  n\right)  =1$ then
$s_{n}\cap A=\emptyset.$
\end{enumerate}

\qquad\ \ 

For any $n\in\omega$ and $A\in\mathcal{A}$ let $D_{n,A}\subseteq\mathbb{P}$ be
the set of conditions $p=\left(  t_{p},\mathcal{F}_{p}\right)  $ such that
$t_{p}^{-1}\left(  1\right)  $ has size at least $n$ and $A\in\mathcal{F}%
_{p}.$ Since $X\in\left(  \mathcal{I}\left(  \mathcal{A}\right)  ^{<\omega
}\right)  ^{+}$, each $D_{n,A}$ is open dense. Clearly $\mathbb{P}$ is
$\sigma$-centered and since $\mathcal{A}$ has size less than $\mathfrak{p}$ we
can then force and find $Y\in\left[  X\right]  ^{\omega}$ such that $\bigcup
Y$ is almost disjoint with every element of $\mathcal{A}.$\qquad
\ \ \ \qquad\ \ \ \qquad\ \ \ 
\end{proof}

\qquad\ \ \ \qquad\ \ \ 

We already know that Shelah-Stepr\={a}ns \textsf{MAD} families are Cohen
indestructible, however, even more is true:

\begin{lemma}
If $\mathcal{A}$ is Shelah-Stepr\={a}ns then $\mathcal{A}$ is tight.
\end{lemma}

\begin{proof}
Assume $\mathcal{A}$ is a \textsf{MAD} family that is not tight as witnessed
by $\left\{  X_{n}\mid n\in\omega\right\}  \subseteq\mathcal{I}\left(
\mathcal{A}\right)  ^{+},$ then player $II$ can easily win in the game
$\mathcal{L}\left(  \mathcal{I}\left(  \mathcal{A}\right)  \right)  $ by
making the resulting set intersects every $X_{n}.$
\end{proof}

\qquad\ \ \ \ \ \ \ \ 

We will later prove that tightness does not imply being Shelah-Stepr\={a}ns.
We shall now introduce a stronger version of tightness:

\begin{definition}
$\mathcal{A}$ is \emph{strongly tight }whenever $\mathcal{W}=\left\{
X_{n}\mid n\in\omega\right\}  \subseteq\left[  \omega\right]  ^{\omega}$ is a
family such that

\begin{enumerate}
\item For every $n\in\omega$ there is $A_{n}\in\mathcal{A}$ such that
$X_{n}\subseteq A_{n}.$

\item For every $A\in\mathcal{A}$ the set $\left\{  n\mid A_{n}=A\right\}  $
is finite.
\end{enumerate}

\qquad\ \ 

There is $A\in\mathcal{I}\left(  \mathcal{A}\right)  $ such that $A\cap
X_{n}\neq\emptyset$ for every $n\in\omega.$
\end{definition}

\qquad\ \ 

Note that if $\mathcal{W}$ is as above, then for every $B\in\mathcal{I}\left(
\mathcal{A}\right)  $ the set $\left\{  X\in\mathcal{W}\mid B\cap X\in\left[
\omega\right]  ^{\omega}\right\}  $ is finite. We can prove the following lemma:

\begin{lemma}
Let $\mathcal{A}$ and $\mathcal{B}$ be two \textsf{MAD} families. If
$\mathcal{A}$ is strongly tight and $\mathcal{I}\left(  \mathcal{A}\right)
\leq_{K}\mathcal{I}\left(  \mathcal{B}\right)  $ then $\mathcal{B}$ is
strongly tight.
\end{lemma}

\begin{proof}
Fix $f$ a Katetov morphism from $\left(  \omega,\mathcal{I}\left(
\mathcal{B}\right)  \right)  $ to $\left(  \omega,\mathcal{I}\left(
\mathcal{A}\right)  \right)  $ and a family $\mathcal{W}=\left\{  X_{n}\mid
n\in\omega\right\}  $ such that for every $n\in\omega$ there is $B_{n}%
\in\mathcal{B}$ such that $X_{n}\subseteq B_{n}$ and for every $B\in
\mathcal{B}$ the set $\left\{  n\mid B_{n}=B\right\}  $ is finite. Let
$\mathcal{W}_{1}=\left\{  X\in\mathcal{W}\mid f\left[  X\right]  \in\left[
\omega\right]  ^{<\omega}\right\}  $ and for every $X\in\mathcal{W}$ we choose
$b_{X}\in f\left[  X\right]  $ such that $f^{-1}\left(  \left\{
b_{X}\right\}  \right)  $ is finite. We first claim that the set $Y=\left\{
b_{X}\mid X\in\mathcal{W}\right\}  $ is finite. If this was not the case, we
could find $A\in\mathcal{A}$ such that $A\cap Y$ is infinite. Since $f$ is a
Katetov morphism, we conclude that $f^{-1}\left(  A\right)  \in\mathcal{I}%
\left(  \mathcal{B}\right)  $ and $\left\{  X\in\mathcal{W}\mid f^{-1}\left(
A\right)  \cap X\in\left[  \omega\right]  ^{\omega}\right\}  $ is infinite,
but this is a contradiction. Using that $Y$ is finite, it is easy to see that
$\mathcal{W}_{1}$ must also be finite.

\qquad\ \ 

Letting $\mathcal{W}_{2}=\mathcal{W}\setminus\mathcal{W}_{1},$ for every
$X\in\mathcal{W}_{2}$ we choose $A_{X}\in\mathcal{I}\left(  \mathcal{A}%
\right)  $ such that $Y_{X}=A_{X}\cap f\left[  X\right]  $ is infinite. Note
that if $A\in\mathcal{A}$ then the set $\left\{  X\in\mathcal{W}_{2}\mid
A=A_{X}\right\}  $ must be finite. Since $\mathcal{A}$ is strongly tight we
can find $A\in\mathcal{I}\left(  A\right)  $ such that $A\cap Y_{X}%
\neq\emptyset$ for every $X\in\mathcal{W}_{2}.$ Since $f$ is a a Katetov
morphism, we may conclude that $B_{1}=f^{-1}\left(  A\right)  $ belongs to
$\mathcal{I}\left(  \mathcal{B}\right)  $ and $B_{1}\cap X\neq\emptyset$ for
every $X\in\mathcal{W}_{2}.$ Clearly $B_{1}\cup%
{\textstyle\bigcup}
\mathcal{W}_{1}$ has the desired properties.
\end{proof}

\qquad\ \ \ 

We now have the following:

\begin{proposition}
If $\mathcal{A}$ is strongly tight then $\mathfrak{d\leq}$ $\left\vert
\mathcal{A}\right\vert .$
\end{proposition}

\begin{proof}
Let $\left\{  A_{n}\mid n\in\omega\right\}  $ be a partition of $\omega$
contained in $\mathcal{A}$ and for each $n\in\omega$ let $P_{n}=\left\{
A_{n}\left(  i\right)  \mid i\in\omega\right\}  $ be a partition of $A_{n}$
into infinite pieces. Given $A\in\mathcal{I}\left(  \mathcal{A}\right)  $ we
define a function $f_{A}:\omega\longrightarrow\omega$ given by $f_{A}\left(
n\right)  =0$ if $A\cap A_{n}$ is infinite and in the other case $f_{A}\left(
n\right)  =max\left\{  i\mid A\cap A_{n}\left(  i\right)  \neq\emptyset
\right\}  +1.$ We claim that $\left\{  f_{A}\mid A\in\mathcal{I}\left(
\mathcal{A}\right)  \right\}  $ is a dominating family. Assume this is not the
case, so there is $g:\omega\longrightarrow\omega$ not dominated by any of the
$f_{A}.$ For each $n\in\omega$ define $X_{n}=A_{n}\left(  g\left(  n\right)
\right)  $ and $X=\left\{  X_{n}\mid n\in\omega\right\}  .$ Since
$\mathcal{A}$ is strongly tight, there must be $A\in\mathcal{I}\left(
\mathcal{A}\right)  $ such that $A\cap X_{n}\neq\emptyset$ for every
$n\in\omega.$ Pick any $m$ such that $f_{A}\left(  m\right)  <g\left(
m\right)  ;$ this implies that $A\cap A_{m}\left(  g\left(  m\right)  \right)
=\emptyset$ so $A\cap X_{m}=\emptyset$ which is a contradiction.
\end{proof}

\qquad\ \ \ \ \ \ \ \ \ \ \ \ \ \ \ \ \ \ \qquad\ \ \ 

Now we can conclude the following:

\begin{corollary}
There are no strongly tight \textsf{MAD} families in the Cohen model.
\end{corollary}

\begin{proof}
If there were then they must have size continuum, but since it is also tight,
it should have size $\omega_{1}.$\qquad\ \ \ \ 
\end{proof}

\qquad\ \ \ \ 

We will later prove that there are Shelah-Stepr\={a}ns \textsf{MAD} families
in the Cohen model, so Shelah-Stepr\={a}ns does not imply strong tightness. We
will now show that strongly tight \textsf{MAD} families may consistently exist:

\begin{lemma}
Let $\mathcal{A}$ be an \textsf{AD} family of size less than $\mathfrak{p}.$
Let $X=\left\{  X_{n}\mid n\in\omega\right\}  $ be a family of infinite
subsets of $\omega$ such that for every $A\in\mathcal{I}\left(  \mathcal{A}%
\right)  $ the set $\left\{  n\mid X_{n}\subseteq^{\ast}A\right\}  $ is
finite. Then there is $B\in\mathcal{A}^{\perp}$ such that $B\cap X_{n}%
\neq\emptyset$ for every $n\in\omega.$
\end{lemma}

\begin{proof}
We may assume that for every $n\in\omega$ there is $A_{n}\in\mathcal{A}$ such
that $X_{n}\subseteq A_{n}$ (note that if $A\in\mathcal{A}$ then the set
$\left\{  n\mid A_{n}=A\right\}  $ is finite). Let $\mathcal{B}=\left\{
A_{n}\mid n\in\omega\right\}  $ and $\mathcal{D}=\mathcal{A}\setminus
\mathcal{B}.$ We now define the forcing $\mathbb{P}$ whose elements are sets
of the form $p=\left(  s_{p},F_{p},G_{p}\right)  $ with the following properties:

\begin{enumerate}
\item $s_{p}\in\omega^{<\omega},$ $F_{p}\in\left[  \mathcal{D}\right]
^{<\omega}$ and $G_{p}\in\left[  \mathcal{B}\right]  ^{<\omega}.$

\item If $i\in dom\left(  s_{p}\right)  $ then $s_{p}\left(  i\right)  \in
X_{n}.$
\end{enumerate}

\qquad\ \ 

If $p,q\in\mathbb{P}$ then $p\leq q$ if the following conditions hold:

\begin{enumerate}
\item $s_{q}\subseteq s_{p},$ $F_{q}\subseteq F_{p}$ and $G_{q}\subseteq
G_{p}.$

\item For every $i\in dom\left(  s_{p}\right)  \setminus dom\left(
s_{q}\right)  $ the following holds:

\begin{enumerate}
\item $s_{p}\left(  i\right)  \notin%
{\textstyle\bigcup}
F_{q}.$

\item If $B\in G_{q}$ and $A_{i}\neq B$ then $s_{p}\left(  i\right)  \notin
B.$
\end{enumerate}
\end{enumerate}

\qquad\ \ 

It is easy to see that $\mathbb{P}$ is a $\sigma$-centered forcing and adds a
set almost disjoint with $\mathcal{A}$ that intersects every $X_{n}.$ Since
$\mathcal{A}$ has size less than $\mathfrak{p}$ the result follows.
\end{proof}

\qquad\ \ \ 

We can then conclude the following:

\begin{proposition}
If $\mathfrak{p=c}$ then strongly tight \textsf{MAD} families exist generically.
\end{proposition}

\qquad\ \ \ 

We know that Shelah-Stepr\={a}ns \textsf{MAD} families are indestructible by
many definable forcings that do not add dominating reals. Surprisingly, they
can be destroyed by forcings that do not add dominating or unsplitted reals,
as we will shortly see. We need the following definitions:

\begin{definition}
Let $\mathcal{I}$ be an ideal in $\omega.$

\begin{enumerate}
\item We say $\mathcal{I}$ is \emph{Canjar }if and only if for every $\left\{
X_{n}\mid n\in\omega\right\}  \subseteq\left(  \mathcal{I}^{<\omega}\right)
^{+}$ there are $Y_{n}\in\left[  X_{n}\right]  ^{<\omega}$ such that $%
{\textstyle\bigcup\limits_{n\in\omega}}
Y_{n}\in\left(  \mathcal{I}^{<\omega}\right)  ^{+}$ for every $A\in\left[
\omega\right]  ^{\omega}.$

\item We say $\mathcal{I}$ is \emph{Hurewicz }if and only if for every
$\left\{  X_{n}\mid n\in\omega\right\}  \subseteq\left(  \mathcal{I}^{<\omega
}\right)  ^{+}$ there are $Y_{n}\in\left[  X_{n}\right]  ^{<\omega}$ such that
$%
{\textstyle\bigcup\limits_{n\in A}}
Y_{n}\in\left(  \mathcal{I}^{<\omega}\right)  ^{+}$ for every $A\in\left[
\omega\right]  ^{\omega}.$
\end{enumerate}
\end{definition}

\qquad\ \ 

We will say that a \textsf{MAD }family $\mathcal{A}$ is Canjar (Hurewicz) if
$\mathcal{I}\left(  \mathcal{A}\right)  $ is Canjar (Hurewicz). Is
$\mathcal{I}$ is an ideal, we denote by $\mathbb{M}\left(  \mathcal{I}\right)
$ the \emph{Mathias forcing of }$\mathcal{I}$ as the set of all $\left(
s,A\right)  $ such that $s\in\left[  \omega\right]  ^{<\omega}$ and
$A\in\mathcal{I}.$ If $\left(  s,A\right)  ,\left(  t,B\right)  \in
\mathbb{M}\left(  \mathcal{I}\right)  $ then $\left(  s,A\right)  \leq\left(
t,B\right)  $ if $t\subseteq s,$ $B\subseteq A$ and $\left(  s\setminus
t\right)  \cap B=\emptyset.$ It is easy to see that $\mathbb{M}\left(
\mathcal{I}\right)  $ destroys $\mathcal{I}.$ We would like to mention the
following important results regarding Canjar and Hurewicz ideals:

\begin{proposition}
[\cite{MathiasPrikryandLaverPrikryTypeForcing}]$\mathcal{I}$ is Canjar if and
only if $\mathbb{M}\left(  \mathcal{I}\right)  $ does not add a dominating real.
\end{proposition}

\begin{proposition}
[\cite{MathiasForcingandCombinatorialCoveringPropertiesofFilters}]%
$\mathcal{I}$ is Canjar if and only if $\mathcal{I}$ is a Menger subspace of
$\wp\left(  \omega\right)  .$
\end{proposition}

\begin{proposition}
[\cite{MathiasForcingandCombinatorialCoveringPropertiesofFilters}]%
$\mathcal{I}$ is Hurewicz if and only if $\mathbb{M}\left(  \mathcal{I}%
\right)  $ preserves all unbounded families of the ground model.
\end{proposition}

\qquad\ \ 

We will need the following lemma:

\begin{lemma}
Let $\mathcal{I}$ be an ideal on $\omega.$ The following are equivalent:

\begin{enumerate}
\item $\mathcal{I}$ is Shelah-Stepr\={a}ns.

\item For every $\left\{  X_{n}\mid n\in\omega\right\}  \subseteq\left(
\mathcal{I}^{<\omega}\right)  ^{+}$ there is $B\in\mathcal{I}$ such that
$X_{n}\cap\left[  B\right]  ^{<\omega}$ is infinite for every $n\in\omega.$
\end{enumerate}
\end{lemma}

\begin{proof}
Clearly 2 implies 1 and if 2 fails then it is easy to see that Player
$\mathsf{II}$ has a winning strategy in $\mathcal{L}\left(  \mathcal{I}%
\right)  ,$ so 1 also fails.
\end{proof}

\qquad\ \ \ 

With the previous lemma we can then conclude the following:g:

\begin{proposition}
Every Shelah-Stepr\={a}ns \textsf{MAD} family is Hurewicz.
\end{proposition}

\begin{proof}
Let $\mathcal{A}$ be a Shelah-Stepr\={a}ns \textsf{MAD} family and $X=\left\{
X_{n}\mid n\in\omega\right\}  \subseteq\left(  \mathcal{I}\left(
\mathcal{A}\right)  ^{<\omega}\right)  ^{+}.$ Note that if $B\in
\mathcal{I}\left(  \mathcal{A}\right)  $ then $\left\{  X_{n}\setminus\left[
B\right]  ^{<\omega}\mid n\in\omega\right\}  \subseteq\left(  \mathcal{I}%
\left(  \mathcal{A}\right)  ^{<\omega}\right)  ^{+}.$ We can then recursively
find $\left\{  B_{n}\mid n\in\omega\right\}  \subseteq\mathcal{I}\left(
\mathcal{A}\right)  $ with the following properties:

\begin{enumerate}
\item If $n\neq m$ then there is no $A\in\mathcal{A}$ that has infinite
intersection with both $B_{n}$ and $B_{m}.$

\item If $n,m\in\omega$ then $B_{n}$ contains an element of $X_{m}.$
\end{enumerate}

\qquad\ \ \ 

For every $n\in\omega$ let $Y_{n}\in\left[  X_{n}\right]  ^{<\omega}$ such
that $Y_{n}\cap\left[  B_{i}\right]  ^{<\omega}\neq\emptyset$ for every $i\leq
n.$ It is then easy to see that if $D\in\left[  \omega\right]  ^{\omega}$ then
$%
{\textstyle\bigcup\limits_{n\in D}}
Y_{n}\in\left(  \mathcal{I}\left(  \mathcal{A}\right)  ^{<\omega}\right)
^{+}.$\qquad\ \ \qquad\ \ 
\end{proof}

\qquad\ \ \ 

We need the following definition:

\begin{definition}
\qquad\ \ \ \ \ \ \ \ \ \ \ \ \ \ \ \ \ \ \ \ \ \ 

\begin{enumerate}
\item We say that $\mathcal{S}=\left\{  S_{\alpha}\mid\alpha\in\omega
_{1}\right\}  \subseteq\left[  \omega\right]  ^{\omega}$ is a \emph{tail
block}-\emph{splitting family }if for every infinite set $P$ of finite
disjoint subsets of $\omega$ there is $\alpha<\omega_{1}$ such that
$S_{\gamma}$ block splits $P$ for every $\gamma>\alpha.$
\end{enumerate}
\end{definition}

\qquad\ \ \ 

It is easy to see that tail block splitting families exist if $\mathfrak{d}%
=\omega_{1}$ and tail block splitting families are splitting families. We say
that a forcing $\mathbb{P}$ preserves a tail block-splitting family if it
remains being tail block-splitting after forcing with $\mathbb{P}$. The
following result could be consider folklore:

\begin{proposition}
Let $\mathcal{I}$ be a Hurewicz ideal. If $\mathcal{S=}\left\{  S_{\alpha}%
\mid\alpha\in\omega_{1}\right\}  \subseteq\left[  \omega\right]  ^{\omega}$ is
a tail block-splitting family then $\mathbb{M}\left(  \mathcal{I}\right)  $
preserves $\mathcal{S}$ as a tail block-splitting family.
\end{proposition}

\begin{proof}
Let $\mathcal{I}$ be a Hurewicz ideal and $\mathcal{S}$ a tail block-splitting
family. Let $\dot{P}$ $\mathcal{=\{}\dot{p}_{n}\mid n\in\omega\}$ be a name
for an infinite set of pairwise disjoint finite subsets of $\omega$, we may
assume $\dot{p}_{n}$ is forced to be disjoint from $n$. For every $s\in\left[
\omega\right]  ^{<\omega}$ and $m\in\omega$ we define $X_{m}\left(  s\right)
$ as the set of all $t\in\left[  \omega\right]  ^{<\omega}$ such that
$max\left(  s\right)  <min\left(  t\right)  $ and there are $F_{\left(
t,m,s\right)  }\in\left[  \omega\right]  ^{<\omega}$ and $B\in\mathcal{I}$
such that $\left(  s\cup t,B\right)  \Vdash``\dot{p}_{m}=F_{\left(
t,m,s\right)  }\textquotedblright.$ It is easy to see that $\left\{
X_{m}\left(  s\right)  \mid m\in\omega\right\}  \subseteq\left(
\mathcal{I}^{<\omega}\right)  ^{+}$ and since $\mathcal{I}$ is Hurewicz, we
may find $Y_{m}\left(  s\right)  \in\left[  X_{m}\left(  s\right)  \right]
^{<\omega}$ such that if $W\in\left[  \omega\right]  ^{\omega}$ then $%
{\textstyle\bigcup\limits_{m\in W}}
Y_{m}\left(  s\right)  \in\left(  \mathcal{I}^{<\omega}\right)  ^{+}.$ Let
$Z_{m}\left(  s\right)  =%
{\textstyle\bigcup\limits_{t\in Y_{m}\left(  s\right)  }}
F_{\left(  t,m,s\right)  }.$ For every $s\in\left[  \omega\right]  ^{<\omega}$
we can then find $D\left(  s\right)  \in\left[  \omega\right]  ^{\omega}$ such
that $R\left(  s\right)  =\left\{  Z_{m}\left(  s\right)  \mid m\in D\left(
s\right)  \right\}  $ is pairwise disjoint.

\qquad\ \ 

Since $\mathcal{S}$ is tail block-splitting, we can find $\alpha$ such that if
$\gamma>\alpha$ then $S_{\gamma}$ block splits $R\left(  s\right)  $ for every
$s\in\left[  \omega\right]  ^{<\omega}.$ We claim that in this case,
$S_{\gamma}$ is forced to block split $\dot{P}.$ If this was not the case, we
could find $\left(  s,A\right)  \in\mathbb{M}\left(  \mathcal{I}\right)  $ and
$n\in\omega$ such that either $\left(  s,A\right)  \Vdash``%
{\textstyle\bigcup}
\left\{  \dot{p}_{m}\mid\dot{p}_{m}\subseteq S_{\gamma}\right\}  \subseteq
n\textquotedblright$ or $\left(  s,A\right)  \Vdash``%
{\textstyle\bigcup}
\left\{  \dot{p}_{m}\mid\dot{p}_{m}\cap S_{\gamma}=\emptyset\right\}
\subseteq n\textquotedblright.$ Assume the first case holds (the other one is
similar). Since $S_{\gamma}$ block splits $R\left(  s\right)  ,$ we know that
the set $W=\left\{  m>n\mid Z_{m}\left(  s\right)  \subseteq S_{\gamma
}\right\}  $ is infinite. Since $%
{\textstyle\bigcup\limits_{m\in W}}
Y_{m}\left(  s\right)  \in\left(  \mathcal{I}^{<\omega}\right)  ^{+}$ then
there is $m\in W$ and $t\in Y_{m}\left(  s\right)  $ such that $t\cap
A=\emptyset.$ We then know there is $B\in\mathcal{I}$ such that $\left(  s\cup
t,B\right)  \Vdash``\dot{p}_{m}=F_{\left(  t,m,s\right)  }\textquotedblright.$
Since $t\cap A=\emptyset$ then $\left(  s\cup t,A\cup B\right)  \leq\left(
s,A\right)  .$ But $\left(  s\cup t,A\cup B\right)  $ forces that $\dot{p}%
_{m}$ is a subset of $S_{\gamma},$ which is a contradiction. We then conclude
that $\mathcal{S}$ remains being a tail block-splitting family.

\qquad\ \ \ \ \qquad\ \ \ 

Clearly $\left\{  X_{m}\mid m\in\omega\right\}  \subseteq\left(
\mathcal{I}^{<\omega}\right)  ^{+}$ and since $\mathcal{I}$ is Hurewicz, we
may find $Y_{m}\in\left[  X_{m}\right]  ^{<\omega}$ such that if $A\in\left[
\omega\right]  ^{\omega}$ then $%
{\textstyle\bigcup\limits_{m\in A}}
Y_{m}\in\left(  \mathcal{I}^{<\omega}\right)  ^{+}.$ For every $t\in Y_{m}$
let $F_{t}$ such that there is $B$ for which $\left(  t,B\right)  \Vdash
``\dot{P}_{m}=F\textquotedblright,$ let $Z_{m}=%
{\textstyle\bigcup\limits_{t\in Y_{m}}}
F_{t}.$ Let $D\in\left[  \omega\right]  ^{\omega}$ such that if $i,j\in D$ and
$i<j$ then $max\left(  Z_{i}\right)  <min\left(  Z_{j}\right)  .$ Let
$R=\left\{  Z_{i}\mid i\in D\right\}  $ since $\mathcal{S}_{1}$ is
block-splitting, we may find $S\in\mathcal{S}_{1}$ that block splits $R.$
Since $S\in\mathcal{S}_{1}$ there is $C\in\mathcal{I}$ such that $\left(
s,C\right)  \Vdash``\left\{  k\mid\dot{P}_{k}\subseteq S\right\}  \subseteq
n\textquotedblright$ if $r=0$ and $\left(  s,C\right)  \Vdash``\left\{
k\mid\dot{P}_{k}\cap S=\emptyset\right\}  \subseteq n\textquotedblright$ if
$r=1.$ Let $D_{1}=\left\{  i\in D\mid Z_{i}\subseteq S\right\}  $ if $r=0$ and
$D_{1}=\left\{  i\in D\mid Z_{i}\cap S=\emptyset\right\}  $ if $r=1,$ in
either case, $D_{1}$ is infinite. Since $%
{\textstyle\bigcup\limits_{m\in D_{1}}}
Y_{m}\in\left(  \mathcal{I}^{<\omega}\right)  ^{+}$ there is $t\in\left[
C\right]  ^{<\omega}$ and $m\in D_{1}$ such that $t\in Y_{m}.$ In this way,
there is $E$ such that $\left(  t,E\right)  \Vdash``\dot{P}_{m}\subseteq
Z_{m}\textquotedblright$ and $\left(  t,E\cap C\right)  \leq\left(
t,E\right)  ,\left(  s,C\right)  .$ Therefore $\left(  t,E\cap C\right)
\Vdash``\dot{P}_{m}\subseteq S\textquotedblright$ if $r=0$ and $\left(
t,E\cap C\right)  \Vdash``\dot{P}_{m}\cap S=\emptyset\textquotedblright$ if
$r=1,$ in either case, we get a contradiction.
\end{proof}

\qquad\ 

Since Hurewicz ideals are Canjar ideals, we conclude the following:

\begin{corollary}
If $\mathcal{A}$ is Shelah-Stepr\={a}ns then $\mathcal{A}$ can be destroyed
with a ccc forcing that does not add dominating nor unsplit reals.
\end{corollary}

\qquad\ \ \ \ \qquad

We will now find a notion stronger than both strongly tight and
Shelah-Stepr\={a}ns:\qquad\qquad\ \ \ \ 

\begin{definition}
Let $\mathcal{I}$ be an ideal on $\omega.$

\qquad\ \ \ \qquad\ \ 

\begin{enumerate}
\item We say a family $X=\left\{  X_{n}\mid n\in\omega\right\}  $ such that
$X_{n}\subseteq\left[  \omega\right]  ^{<\omega}$ is \emph{locally finite
according to }$\mathcal{I}$ if for every $B\in\mathcal{I}$ for almost all
$n\in\omega$ there is $s\in X_{n}$ such that $s\cap B=\emptyset.$

\item We say $\mathcal{I}$ is \emph{raving }if for every family $X=\left\{
X_{n}\mid n\in\omega\right\}  $ that is locally finite according to\emph{
}$\mathcal{I}$ there is $B\in\mathcal{I}$ such that $B$ contains at least one
element of each $X_{n}.$ \qquad\ \ \ \ \qquad\ \ 
\end{enumerate}
\end{definition}

\qquad\ \ \ \ 

It is easy to see that every raving \textsf{MAD} family is both
Shelah-Stepr\={a}ns and strongly tight. The following lemma is easy and left
to the reader:\ 

\begin{lemma}
Let $\mathcal{A}$ and $\mathcal{B}$ be two \textsf{MAD} families. If
$\mathcal{A}$ is raving and $\mathcal{I}\left(  \mathcal{A}\right)  \leq
_{K}\mathcal{I}\left(  \mathcal{B}\right)  $ then $\mathcal{B}$ is raving.
\end{lemma}

\qquad\ \ \qquad\ \ 

We can construct such families with the parametrized diamond principles:

\begin{proposition}
$\ \ \qquad\ \ \ \qquad$

\begin{enumerate}
\item $\Diamond\left(  \mathfrak{b}\right)  $ implies there is a
Shelah-Stepr\={a}ns \textsf{MAD} family.

\item $\Diamond\left(  \mathfrak{d}\right)  $ implies there is a raving
\textsf{MAD} family.
\end{enumerate}
\end{proposition}

\begin{proof}
For every $\alpha<\omega_{1}$ fix an enumeration $\alpha=\left\{  \alpha
_{n}\mid n\in\omega\right\}  .$ We will first show that $\Diamond\left(
\mathfrak{b}\right)  $ implies there is a Shelah-Stepr\={a}ns \textsf{MAD}
family. With a suitable coding, the coloring $C$ will be defined for pairs
$t=\left(  \mathcal{A}_{t},X_{t}\right)  $ where $\mathcal{A}_{t}=\left\langle
A_{\xi}\mid\xi<\alpha\right\rangle $ and $X_{t}\subseteq\left[  \omega\right]
^{<\omega}$ (we identify $t$ with an element of $2^{\alpha}$). We define
$C\left(  t\right)  $ to be the constant $0$ function in case $\mathcal{A}%
_{t}$ is not an almost disjoint family or $X_{t}\notin\left(  \mathcal{I}%
\left(  \mathcal{A}_{t}\right)  ^{<\omega}\right)  ^{+}.$ In the other case,
define an increasing function $C\left(  t\right)  :\omega\longrightarrow
\omega$ such that if $n\in\omega$ then there is $s\in X_{t}$ such that
$s\subseteq C\left(  t\right)  \left(  n\right)  $ and $s\cap\left(
A_{\alpha_{0}}\cup...\cup A_{\alpha_{n}}\cup n\right)  =\emptyset.$

\qquad\ \qquad\ \ \ 

Using $\Diamond\left(  \mathfrak{b}\right)  $ let $G:\omega_{1}\longrightarrow
\omega^{\omega}$ be a guessing sequence for $C,$ by changing $G$ if necessary,
we may assume that all the $G\left(  \alpha\right)  $ are increasing and if
$\alpha<\beta$ then $G\left(  \alpha\right)  <^{\ast}G\left(  \beta\right)  .$
We will now define our \textsf{MAD} family: start by taking a partition
$\left\{  A_{n}\mid n\in\omega\right\}  $ of $\omega.$ Having defined $A_{\xi
}$ for all $\xi<\alpha,$ we proceed to define $A_{\alpha}=\bigcup
\limits_{n\in\omega}\left(  G\left(  \alpha\right)  \left(  n\right)
\backslash A_{\alpha_{0}}\cup...\cup A_{\alpha_{n}}\right)  $ in case this is
an infinite set, otherwise just take any $A_{\alpha}$ that is almost disjoint
with $\mathcal{A}_{\alpha}=\left\{  A_{\xi}\mid\xi<\alpha\right\}  .$ We will
see that $\mathcal{A}$ is a Shelah-Stepr\={a}ns \textsf{MAD} family. Let
$X\in\left(  \mathcal{I}\left(  \mathcal{A}\right)  ^{<\omega}\right)  ^{+}$.
Consider the branch $R=\left(  \left\langle A_{\xi}\mid\xi<\omega
_{1}\right\rangle ,X\right)  $ and pick $\beta>\omega$ such that $C\left(
R\upharpoonright\beta\right)  $ $^{\ast}\ngeq G\left(  \beta\right)  .$ It is
easy to see that $A_{\beta}$ contains infinitely many elements of $X.$

\qquad\ \ 

Now we will prove that $\Diamond\left(  \mathfrak{d}\right)  $ implies there
is a raving \textsf{MAD} family. With a suitable coding, the coloring $C$ will
be defined for pairs $t=\left(  \mathcal{A}_{t},X_{t}\right)  $ where
$\mathcal{A}_{t}=\left\langle A_{\xi}\mid\xi<\alpha\right\rangle $ and
$X_{t}=\left\{  X_{n}^{t}\mid n\in\omega\right\}  \subseteq\wp\left(  \left[
\omega\right]  ^{<\omega}\right)  $. We define $C\left(  t\right)  $ to be the
constant $0$ function in case $\mathcal{A}_{t}$ is not an almost disjoint
family or $X_{t}$ is not locally finite according to $\mathcal{I}\left(
\mathcal{A}_{t}\right)  .$ We will describe what to do in the other case. For
every $n\in\omega$ define $B_{n}=\bigcup\limits_{i<n}A_{\alpha_{i}}$ (hence
$B_{0}=\emptyset$) and let $d\left(  n\right)  $ be the smallest $k\geq n$
such that if $l\geq k$ then $B_{n}$ does not intersect every element of
$X_{l}^{t}.$ We define an increasing function $C\left(  t\right)
:\omega\longrightarrow\omega$ such that for every $n,i\in\omega,$
if\ $d\left(  n\right)  \leq i<d\left(  n+1\right)  $ then $C\left(  t\right)
\left(  n\right)  $ $\backslash$ $B_{n}$ contains an element of $X_{i}^{t}.$
The rest of the proof is similar to the case of $\Diamond\left(
\mathfrak{b}\right)  .$
\end{proof}

\qquad\ \ \ 

By the previous result, we conclude that there are Shelah-Stepr\={a}ns
\textsf{MAD} families in the Cohen model. Denote by $\mathbb{P}_{\mathsf{MAD}%
}$ the set of all countable AD families ordered by inclusion. This is a
$\sigma$-closed forcing and adds a \textsf{MAD} family $\mathcal{A}_{gen}.$ We
now have the following:

\begin{proposition}
$\mathbb{P}_{\mathsf{MAD}}$ forces that$\ \mathcal{A}_{gen}$ is raving.
\end{proposition}

\begin{proof}
Let $\mathcal{B\in}$ $\mathbb{P}_{\mathsf{MAD}}$ and $X=\left\{  X_{n}\mid
n\in\omega\right\}  $ such that $\mathcal{B}$ forces that $X$ is locally
finite according to $\mathcal{I}\left(  \mathcal{A}_{gen}\right)  .$ Let
$\mathcal{B}=\left\{  B_{n}\mid n\in\omega\right\}  $ and we define
$E_{n}=B_{0}\cup...\cup B_{n}$ for every $n\in\omega.$ We can then find an
interval partition $\mathcal{P}=\left\{  P_{n}\mid n\in\omega\right\}  $ such
that if $i\in P_{n+1}$ then $E_{n}$ does not intersect every element of
$X_{i}.$ For every $i\in\omega$ we choose $s_{i}\in X_{i}$ as follows: if
$i\in P_{0}$ let $s_{i}$ be any element of $X_{i}$ and if $i\in P_{n+1}$ we
choose $s_{i}\in X_{i}$ such that $s_{i}\cap E_{n}=\emptyset.$ Let $A=%
{\textstyle\bigcup\limits_{n\in\omega}}
s_{n}$ so $A\in\mathcal{B}^{\perp}$ and the condition $\mathcal{B\cup}\left\{
A\right\}  \in\mathbb{P}_{\mathsf{MAD}}$ is the extension of $\mathcal{B}$ we
were looking for.
\end{proof}

\qquad\ \ 

We will now comment on the reason for why we introduced the concept of raving
\textsf{MAD} families. First we take a look at the following theorems:

\begin{proposition}
[Todorcevic]An ultrafilter $\mathcal{U}$ is $\wp\left(  \omega\right)
\setminus$\emph{FIN} generic over $L\left(  \mathbb{R}\right)  $ if and only
if $\mathcal{U}$ is Ramsey.
\end{proposition}

\begin{proposition}
[Chodousky, Zapletal]Let $\mathcal{I}$ be an $F_{\sigma}$-ideal and
$\mathcal{U}$ an ultrafilter. $\mathcal{U}$ is $\wp\left(  \omega\right)
\setminus\mathcal{I}$ generic over $L\left(  \mathbb{R}\right)  $ if and only
if $\mathcal{I\cap U=\emptyset}$ and for every closed set $\mathcal{C}$ if
$\mathcal{C}\cap\mathcal{U=\emptyset}$ then there is $A\in\mathcal{U}$ such
that $A\cap Y\in\mathcal{I}$ for every $Y\in$ $\mathcal{C}.$
\end{proposition}

\qquad\ \ \ \ 

It would be interesting to find a similar characterization of the
$\mathbb{P}_{\mathsf{MAD}}$ generics over $L\left(  \mathbb{R}\right)  :$

\begin{problem}
Is there a combinatorial characterization of the ideal $\mathcal{I}\left(
\mathcal{A}\right)  \ $where $\mathcal{A}\ $is$\ \mathbb{P}_{\mathsf{MAD}}$
generic over $L\left(  \mathbb{R}\right)  $?
\end{problem}

\qquad\ \ \ \ 

The indestructibility if \textsf{MAD} families is a particular instance of the
following problem:

\begin{problem}
Let $\mathcal{X}$ be a set of ideals. Is there a \textsf{MAD} family that is
not Kat\v{e}tov below any element of $\mathcal{X}$?
\end{problem}

\qquad\ \ \ 

Of course, the answer depends on the nature of the set $\mathcal{X}.$ Recall
that every \textsf{MAD} family is Kat\v{e}tov below \textsf{FIN}$\times
$\textsf{FIN}$,$ so this imposes some condition on $\mathcal{X}.$ The
following is a particularly interesting instance of the problem:

\begin{definition}
\qquad\ \ \ \ \ \ \ \ \ \ \ \ \ \ \ \ \qquad\ \ \ \ \ \ \ \qquad
\ \ \ \ \ \qquad\ \ \ \ 

\begin{enumerate}
\item Let $\mathcal{I}$ be an ideal. $\mathcal{I}$ is \emph{Laflamme }if it is
not Kat\v{e}tov below any $F_{\sigma}$-ideal.

\item A \textsf{MAD} family $\mathcal{A}$ is \emph{Laflamme }if $\mathcal{I}%
\left(  \mathcal{A}\right)  $ is Laflamme.
\end{enumerate}
\end{definition}

\qquad\ \ \ \qquad\ \ 

David Meza and Michael Hru\v{s}\'{a}k proved that every ideal Kat\v{e}tov
above \textsf{conv} is Laflamme (see \cite{TesisDavid}). It is unknown if this
a is characterization for Borel ideals.\qquad\ \ \ 

\begin{lemma}
Let $\mathcal{K}$ be an ideal. The following are equivalent:

\begin{enumerate}
\item $\mathcal{K}$ is Laflamme.

\item $\mathcal{K}$ can not be extended to an $F_{\sigma}$-ideal.
\end{enumerate}
\end{lemma}

\begin{proof}
Let $\mathcal{K}$ be an ideal that can not be extended to an $F_{\sigma}%
$-ideal. Letting $\mathcal{I}$ be an $F_{\sigma}$ ideal and $f:\omega
\longrightarrow\omega.$ We will show that $f$ is not a Kat\v{e}tov morphism
from $\left(  \omega,\mathcal{I}\right)  $ to $\left(  \omega,\mathcal{K}%
\right)  .$ We now define $\mathcal{J=}\left\{  X\mid f^{-1}\left(  X\right)
\in\mathcal{I}\right\}  .$ Let $\varphi$ be a lower semicontinuous submeasure
such that $\mathcal{I}=$ \textsf{Fin}$\left(  \varphi\right)  .$ For every
$n\in\omega\,\ $we define $C_{n}=\left\{  X\mid\varphi\left(  f^{-1}\left(
X\right)  \right)  \leq n\right\}  .$ It is easy to see that each $C_{n}$ is a
closed set and $\mathcal{J}=%
{\textstyle\bigcup\limits_{n\in\omega}}
C_{n}.$ Since $\mathcal{K}$ is not contained in $\mathcal{J}$ the result follows.
\end{proof}

\qquad\ \ \ \ \ \ \ 

Clearly every Shelah-Stepr\={a}ns \textsf{MAD} family is Laflamme (it can not
even be extended to an $F_{\sigma\delta}$ ideal). In particular, we get the
following result of \cite{KatetovandKatetovBlassOrdersFsigmaIdeals}:

\begin{proposition}
[\cite{KatetovandKatetovBlassOrdersFsigmaIdeals}]If $\mathfrak{p=c}$ then
there is a Laflamme \textsf{MAD} family.
\end{proposition}

\chapter{There is a $+$-Ramsey \textsf{MAD} family}

\section{$+$-Ramsey \textsf{MAD} families}

In this chapter we introduce the concept of a $+$-Ramsey ideal, which is a
stronger notion than selectiveness and then we will prove that there is a
$+$-Ramsey \textsf{MAD} family, answering an old question of Hru\v{s}\'{a}k.
Letting $\mathcal{X}\subseteq\left[  \omega\right]  ^{\omega},$ we say a tree
$T\subseteq\omega^{<\omega}$ is a $\mathcal{X}$\emph{-branching tree} if
$suc_{T}\left(  s\right)  \in\mathcal{X}$ for every $s\in T$. If $\mathcal{B}$
$\subseteq\left[  \omega\right]  ^{\omega}$ is a centered family, we define
$\left\langle \mathcal{B}\right\rangle $ as the filter generated by
$\mathcal{B}.$ The following is an important theorem of Adrian Mathias:

\begin{proposition}
Let $\mathcal{I}$ be an ideal in $\omega.$ The following are equivalent:

\begin{enumerate}
\item $\mathcal{I}$ is selective.

\item For every $\mathcal{I}^{+}$-branching tree $T$ such that $\mathcal{B}%
=\left\{  suc_{T}\left(  s\right)  \mid s\in T\right\}  $ is centered and
$\left\langle \mathcal{B}\right\rangle \subseteq\mathcal{I}^{+}$ there is
$f\in\left[  T\right]  $ such that $im\left(  f\right)  \in\mathcal{I}^{+}.$
\end{enumerate}
\end{proposition}

\begin{proof}
Let $\mathcal{I}$ be a selective ideal and $T$ an $\mathcal{I}^{+}$-branching
tree such that $\mathcal{B}=\left\{  suc_{T}\left(  s\right)  \mid s\in
T\right\}  $ is centered and $\left\langle \mathcal{B}\right\rangle
\subseteq\mathcal{I}^{+}.$ We may assume that if $s\in T$ then $s$ is an
increasing sequence$.$ For every $n\in\omega,$ we define $Y_{n}=%
{\textstyle\bigcap}
\left\{  suc_{T}\left(  s\right)  \mid s\subseteq n\right\}  .$ Clearly
$\left\{  Y_{n}\mid n\in\omega\right\}  \subseteq\mathcal{I}^{+}$ and it forms
a decreasing sequence. Since $\mathcal{I}$ is selective, there is $X=\left\{
x_{n}\mid n\in\omega\right\}  \in\mathcal{I}^{+}$ such that $X\subseteq
Y_{0}=suc_{T}\left(  \emptyset\right)  $ and $X\setminus\left(  x_{n}%
+1\right)  \subseteq Y_{x_{n}}.$ It is easy to see that there is $f\in\left[
T\right]  $ such that $im\left(  f\right)  =X.$

\qquad\ \ 

Now assume $\mathcal{I}$ has the property stated in point 2. We will show that
$\mathcal{I}$ is selective. Let $\left\{  Y_{n}\mid n\in\omega\right\}
\subseteq\mathcal{I}^{+}$ be a decreasing sequence. We now recursively define
a tree $T\subseteq\omega^{<\omega}$ as follows:

\begin{enumerate}
\item $\emptyset\in T.$

\item If $s=\left\langle n_{0},...,n_{m}\right\rangle \in T$ then
$suc_{T}\left(  s\right)  =Y_{n_{m}}\setminus\left(  n_{m+1}\right)  .$
\end{enumerate}

\qquad\ \ 

Clearly $T$ is a tree such that $\mathcal{B}=\left\{  suc_{T}\left(  s\right)
\mid s\in T\right\}  $ is centered and $\left\langle \mathcal{B}\right\rangle
\subseteq\mathcal{I}^{+}.$ We can then find $f\in\left[  T\right]  $ such that
$im\left(  f\right)  \in\mathcal{I}^{+}.$ It is easy to see that $im\left(
f\right)  $ has the desired properties.
\end{proof}

\qquad\ \ \ 

We are now ready to introduce the notion of a $+$-Ramsey ideal:\qquad\ \ 

\begin{definition}
\qquad\ \ 

\begin{enumerate}
\item An ideal $\mathcal{I}$ is $+$\emph{-Ramsey }if for every $\mathcal{I}%
^{+}$-branching tree $T,$ there is $f\in\left[  T\right]  $ such that
$im\left(  f\right)  \in\mathcal{I}^{+}.$

\item An \textsf{AD} family $\mathcal{A}$ is $+$\emph{-Ramsey} if
$\mathcal{I}\left(  \mathcal{A}\right)  $ is $+$-Ramsey.\qquad\ \ \ \qquad
\ \ \qquad\ \ 
\end{enumerate}
\end{definition}

\qquad\ \ \ \qquad\ \ \ 

Obviously every $+$-Ramsey ideal is selective, but the converse is not true.
Recall that the ideals generated by \textsf{MAD} families are selective,
however we have the following:

\begin{lemma}
There is a \textsf{MAD} family that is not $+$-Ramsey.
\end{lemma}

\begin{proof}
Given $f\in\omega^{\omega}$ we let $\widehat{f}=\left\{  f\upharpoonright
n\mid n\in\omega\right\}  .$ Let $\mathcal{A}$ be a \textsf{MAD} family with
the following properties:

\begin{enumerate}
\item If $f\in\omega^{\omega}$ then $\widehat{f}\in\mathcal{A}.$

\item If $s\in\omega^{<\omega}$ then $\left\{  s^{\frown}n\mid n\in
\omega\right\}  \in\mathcal{I}\left(  \mathcal{A}\right)  ^{+}$
\end{enumerate}

\qquad\ \ \ 

It is easy to see that $\omega^{<\omega}$ is an $\mathcal{I}\left(
\mathcal{A}\right)  ^{+}$-branching tree without branches in $\mathcal{I}%
\left(  \mathcal{A}\right)  ^{+}.$
\end{proof}

\qquad\ \ \ 

Although we will not need the following interesting result of Michael
Hru\v{s}\'{a}k, we will include it.

\begin{proposition}
[\cite{SelectivityofAlmostDisjointFamilies}]\textsf{cov}$\left(
\mathcal{M}\right)  $ is the smallest cofinality of an ideal that is not $+$-Ramsey.
\end{proposition}

\begin{proof}
Letting $\mathcal{I}$ be an ideal with \textsf{cof}$\left(  \mathcal{I}%
\right)  <$ \textsf{cov}$\left(  \mathcal{M}\right)  $ we will show it is
$+$-Ramsey. Let $T\subseteq\omega^{<\omega}$ be an $\mathcal{I}^{+}$-branching
tree. We view $T$ is as a forcing notion, which clearly is equivalent to Cohen
forcing. We denote by $\dot{r}_{gen}$ the name of the generic branch. For
every $A\in\mathcal{I}$ and $n\in\omega$ we define $D_{n}\left(  A\right)
=\left\{  s\in T\mid im\left(  s\right)  \nsubseteq A\cup n\right\}  $. Since
$T$ is an $\mathcal{I}^{+}$-branching tree, each $D_{n}\left(  A\right)  $ is
an open dense set for every $A\in\mathcal{I}$ and $n\in\omega.$ Since
\textsf{cof}$\left(  \mathcal{I}\right)  <$ \textsf{cov}$\left(
\mathcal{M}\right)  $ the result follows.

\qquad\ \ 

We will now construct an ideal $\mathcal{I}$ that is not $+$-Ramsey such that
\textsf{cof}$\left(  \mathcal{I}\right)  $ is equal to \textsf{cov}$\left(
\mathcal{M}\right)  .$ Let $\left\{  T_{\alpha}\mid\alpha<\mathsf{cov}\left(
\mathcal{M}\right)  \right\}  $ be a family of well pruned trees of
$\omega^{<\omega}$ such that each $\left[  T_{\alpha}\right]  $ is nowhere
dense and $\omega^{\omega}=%
{\textstyle\bigcup}
\left[  T_{\alpha}\right]  .$ We define $\mathcal{I}$ as the ideal in
$\omega^{<\omega}$ generated by $\{T_{\alpha}\mid\alpha<$\textsf{ cov}$\left(
\mathcal{M}\right)  \}.$ We will show that $\mathcal{I}$ is not $+$-Ramsey. We
define a tree $T\subseteq\left(  \omega^{<\omega}\right)  ^{<\omega}$ as the
set of all sequences $\left(  s_{0},s_{1},...,s_{n}\right)  $ such that
$s_{0}\subsetneq s_{1}...\subsetneq s_{n}.$ It is easy to see that $T$ is the
tree we were looking for.
\end{proof}

\qquad\ \ \ \qquad\ \ \ \qquad\ \ 

We know that there are \textsf{MAD} families that are not $+$-Ramsey. On the
other hand, $+$-Ramsey \textsf{MAD} families can be constructed under
$\mathfrak{b=c},$ \textsf{cov}$\left(  \mathcal{M}\right)  =\mathfrak{c,}$
$\mathfrak{a}<$ \textsf{cov}$\left(  \mathcal{M}\right)  $ or $\Diamond\left(
\mathfrak{b}\right)  $ (see \cite{SelectivityofAlmostDisjointFamilies} and
\cite{OrderingMADFamiliesalaKatetov}). This led Michael Hru\v{s}\'{a}k to ask
the following,

\begin{problem}
[Hru\v{s}\'{a}k \cite{SelectivityofAlmostDisjointFamilies}]Is there a
$+$-Ramsey \textsf{MAD} family in \textsf{ZFC}?
\end{problem}

\qquad\ \ \ \qquad

We will provide a positive answer to this question. Our technique for
constructing a $+$-Ramsey \textsf{MAD} is based on the technique of Shelah for
constructing a completely separable \textsf{MAD} family (however, in this case
we are able to avoid the \textsf{PCF }hypothesis).

\qquad\ \ \ 

The first step to construct a $+$-Ramsey \textsf{MAD} family is to prove that
every Miller-indestructible \textsf{MAD} family has this property. In
\cite{OrderingMADFamiliesalaKatetov} it is proved that every tight
\textsf{MAD} family is $+\,$-Ramsey. We will prove that every
Miller-indestructible \textsf{MAD} family is $+$-Ramsey. This improves the
previous result since Miller-indestructibility follows from
Cohen-indestructibility. First we need the following
lemma:\ \ \ \ \ \ \ \ \ \ \ \ \ \ \ \ \ \ \ \ \ \ \ \ \ \ \ \ \ \qquad\ 

\begin{lemma}
Let $\mathcal{A}$ be a \textsf{MAD} family and $T$ an $\mathcal{I}\left(
\mathcal{A}\right)  ^{+}$-branching tree. Then there is a subtree $S\subseteq
T$ with the following properties:

\begin{enumerate}
\item $S$ is an\ $\mathcal{I}\left(  \mathcal{A}\right)  $-branching tree.

\item If $s\in S$ there is $A_{s}\in\mathcal{A}$ such that $suc_{S}\left(
s\right)  \subseteq A_{s}.$

\item If $s$ and $t$ are two different nodes of $S,$ then $A_{s}\neq A_{t}$
and $suc_{S}\left(  s\right)  \cap suc_{S}\left(  t\right)  =\emptyset.$
\end{enumerate}
\end{lemma}

\begin{proof}
Since $T$ is an $\mathcal{I}\left(  \mathcal{A}\right)  ^{+}$-branching tree
and $\mathcal{A}$ is \textsf{MAD}$\,$, $suc_{T}\left(  t\right)  $ intersects
infinitely many infinite elements of $\mathcal{A}$ for every $t\in T$.
Recursively, for every $t\in T$ we choose $A_{t}\in\mathcal{A}$ and $B_{t}%
\in\left[  A_{t}\cap suc_{T}\left(  t\right)  \right]  ^{\omega}$ such that
$B_{t}\cap B_{s}=\emptyset$ and $A_{s}\neq A_{t}$ whenever $t\neq s.$ We then
recursively construct $S\subseteq T$ such that if $s\in S$ then $suc_{S}%
\left(  s\right)  =B_{s}.$
\end{proof}

\qquad\ \ \ \qquad\qquad\ \ \ \qquad\ \ 

With the previous lemma we can now prove the following,

\begin{proposition}
If $\mathcal{A}$ is Miller-indestructible then $\mathcal{A}$ is $+$-Ramsey.
\end{proposition}

\begin{proof}
Let $\mathcal{A}$ be a Miller-indestructible \textsf{MAD} family and $T$ an
$\mathcal{I}\left(  \mathcal{A}\right)  ^{+}$-branching tree. Let $S$ be an
$\mathcal{I}\left(  \mathcal{A}\right)  $-branching subtree of $T$ as in the
previous lemma. We can then view $S$ as a Miller tree. Let $\dot{r}_{gen}$ be
the name of the generic real and $\dot{X}$ the name of the image of $\dot
{r}_{gen}.$

\qquad\ \ 

We will first argue that $S\Vdash``\dot{X}\notin\mathcal{I}\left(
\mathcal{A}\right)  \textquotedblright$. Assume this is not true, so there is
$S_{1}\leq S$ and $B\in\mathcal{I}\left(  \mathcal{A}\right)  $ ($B$ is an
element of $V$) such that $S_{1}\Vdash``\dot{X}\subseteq B\textquotedblright.$
In this way, if $t$ is a splitting node of $S_{1}$ then $suc_{S_{1}}\left(
t\right)  \subseteq B$ but note that if $t_{1}\neq t_{2}$ are two different
splitting nodes of $S_{2}$ then $suc_{S_{1}}\left(  t_{1}\right)  $ and
$suc_{S_{1}}\left(  t_{2}\right)  $ are two infinite sets contained in
different elements of $\mathcal{A}$, so then $B\in\mathcal{I}\left(
\mathcal{A}\right)  ^{+}$ which is a contradiction.

\qquad\ \ \ 

In this way, $\dot{X}$ is forced by $S$ to be an element of $\mathcal{I}%
\left(  \mathcal{A}\right)  ^{+}$ but since $\mathcal{A}$ is still
\textsf{MAD} after performing a forcing extension of Miller forcing, we then
conclude there are names $\{\dot{A}_{n}\mid n\in\omega\}$ for different
elements of $\mathcal{A}$ such that $S$ forces that $\dot{X}\cap\dot{A}_{n}$
is infinite. We then recursively build two sequences $\left\{  S_{n}\mid
n\in\omega\right\}  $ and $\left\{  B_{n}\mid n\in\omega\right\}  $ such that
for every $n\in\omega$ the following holds:

\begin{enumerate}
\item $S_{n}$ is a Miller tree and $B_{n}\in\mathcal{A}\mathbf{.}$

\item $S_{0}\leq S$ and if $n<m$ then $S_{m}\leq S_{n}.$

\item $S_{n}\Vdash``\dot{A}_{n}=B_{n}\textquotedblright$ (it then follows that
$B_{n}\neq B_{m}$ if $n\neq m$).

\item If $i\leq n$ then $stem\left(  S_{n}\right)  \cap B_{i}$ has size at
least $n.$
\end{enumerate}

\qquad\ \ \ \ \ \ \qquad\ \ 

We then define $r=%
{\textstyle\bigcup\limits_{n\in\omega}}
stem\left(  S_{n}\right)  $ then clearly $r\in\left[  S\right]  $ and
$im\left(  r\right)  \in\mathcal{I}\left(  \mathcal{A}\right)  ^{++}%
.$\ \ \qquad\ \ \qquad\ \ 
\end{proof}

\qquad\ \ \ \ \ 

The previous proposition has the following corollary, which is an unpublished
result of Michael Hru\v{s}\'{a}k, which he proved by completely different
means.\qquad\ \ \qquad\ \ 

\begin{corollary}
[Hru\v{s}\'{a}k]Every \textsf{MAD} family of size less than $\mathfrak{d}$ is
$+$-Ramsey. In particular, if $\mathfrak{a<d}$ then there is a $+$-Ramsey
\textsf{MAD} family.
\end{corollary}

\begin{proof}
Let $\mathcal{A}$ be a \textsf{MAD} family that is not $+$-Ramsey. Then
$\mathcal{A}$ is destructible by Miller forcing, so $\mathcal{I}\left(
\mathcal{A}\right)  \leq_{K}tr\left(  \mathcal{K}_{\sigma}\right)  $ and then
$\mathfrak{d=}$ \textsf{cov}$^{\ast}\left(  tr\left(  \mathcal{K}_{\sigma
}\right)  \right)  \leq$ \textsf{cov}$^{\ast}\left(  \mathcal{I}\left(
\mathcal{A}\right)  \right)  =$ $\left\vert \mathcal{A}\right\vert .$
\end{proof}

\qquad\ \ \qquad\ \qquad\ \ \ 

\section{The construction of a $+$-Ramsey \textsf{MAD} family}

In this chapter we will construct a $+$-Ramsey \textsf{MAD} family without any
extra hypothesis beyond \textsf{ZFC}$\mathsf{.}$ We will use the construction
of Shelah of a completely separable \textsf{MAD} family, however, the previous
result will help us avoid the need of a \textsf{PCF }hypothesis for our
construction. From now on, we will always assume that all Miller trees are
formed by increasing sequences.

\begin{definition}
Let $p$ be a Miller tree. Given $f\in\left[  p\right]  $ we define $Sp\left(
p,f\right)  =\left\{  f\left(  n\right)  \mid f\upharpoonright n\in
Split\left(  p\right)  \right\}  $ and $\left[  p\right]  _{split}=\left\{
Sp\left(  p,f\right)  \mid f\in\left[  p\right]  \right\}  .$
\end{definition}

\qquad\ \ \ \ 

We will need the following definitions,\qquad\ \ \ \ \qquad\ \ \ \qquad

\begin{definition}
Let $p$ be a Miller tree and $H:Split\left(  p\right)  \longrightarrow\omega.$
We then define:

\begin{enumerate}
\item $Catch_{\exists}\left(  H\right)  $ is the set \qquad

$\left\{  Sp\left(  f,p\right)  \mid f\in\left[  p\right]  \wedge
\exists^{\infty}n\left(  f\upharpoonright n\in Split\left(  p\right)  \wedge
f\left(  n\right)  <H\left(  f\upharpoonright n\right)  \right)  \right\}  .$

\item $Catch_{\forall}\left(  H\right)  $ is the set \qquad

$\left\{  Sp\left(  f,p\right)  \mid f\in\left[  p\right]  \wedge
\forall^{\infty}n\left(  f\upharpoonright n\in Split\left(  p\right)  \wedge
f\left(  n\right)  <H\left(  f\upharpoonright n\right)  \right)  \right\}  .$

\item Define $\mathcal{K}\left(  p\right)  $ as the collection of all
$A\subseteq\left[  p\right]  _{split}$ for which there is $G:Split\left(
p\right)  \longrightarrow\omega$ such that $A\subseteq Catch_{\exists}\left(
G\right)  .$
\end{enumerate}
\end{definition}

\qquad\ \ \ \qquad\ \ 

Note that if $\mathcal{B}=\left\{  f_{\alpha}\mid\alpha<\mathfrak{b}\right\}
\subseteq\omega^{\omega}$ is an unbounded family of increasing functions then
for every infinite partial function $g\subseteq\omega\times\omega$ there is
$\alpha<\mathfrak{b}$ such that the set $\left\{  n\in dom\left(  g\right)
\mid g\left(  n\right)  <f_{\alpha}\left(  n\right)  \right\}  $ is infinite.
With this simple observation we can prove the following lemma,

\begin{lemma}
$\mathcal{K}\left(  p\right)  $ is a $\sigma$-ideal in $\left[  p\right]
_{split}$ that contains all singletons and $\mathfrak{b=}$ \textsf{add}%
$\left(  \mathcal{K}\left(  p\right)  \right)  =$ \textsf{cov}$\left(
\mathcal{K}\left(  p\right)  \right)  .$
\end{lemma}

\begin{proof}
In order to prove that $\mathfrak{b}$ $\mathfrak{\leq}$ \textsf{add}$\left(
\mathcal{K}\left(  p\right)  \right)  $ it is enough to show that if
$\kappa<\mathfrak{b}$ and $\left\{  H_{\alpha}\mid\alpha<\kappa\right\}
\subseteq\omega^{Split\left(  p\right)  }$ then $%
{\textstyle\bigcup\limits_{\alpha<\kappa}}
Catch_{\exists}\left(  H_{\alpha}\right)  \in\mathcal{K}\left(  p\right)  .$
Since $\kappa$ is smaller than $\mathfrak{b}$, we can find $H:Split\left(
p\right)  \longrightarrow\omega$ such that if $\alpha<\kappa$ then $H_{\alpha
}\left(  s\right)  <H\left(  s\right)  $ for almost all $s\in Split\left(
p\right)  .$ Clearly $%
{\textstyle\bigcup\limits_{\alpha<\kappa}}
Catch_{\exists}\left(  H_{\alpha}\right)  \subseteq Catch_{\exists}\left(
H\right)  .$

\qquad\ \ \qquad\qquad\ \ \ \ \ 

Now we must prove that \textsf{cov}$\left(  \mathcal{K}\left(  p\right)
\right)  \leq\mathfrak{b}.$ Let $Split\left(  p\right)  =\left\{  s_{n}\mid
n<\omega\right\}  $ and $\mathcal{B}=\left\{  f_{\alpha}\mid\alpha
<\mathfrak{b}\right\}  \subseteq\omega^{\omega}$ be an unbounded family of
increasing functions. Given $\alpha<\mathfrak{b}$ define $H_{\alpha
}:Split\left(  p\right)  \longrightarrow\omega$ where $H_{\alpha}\left(
s_{n}\right)  =$ $f_{\alpha}\left(  n\right)  .$ We will show that $\left\{
Catch_{\exists}\left(  H_{\alpha}\right)  \mid\alpha<\mathfrak{b}\right\}  $
covers $\left[  p\right]  _{split}.$ Letting $f\in\left[  p\right]  $ define
$A=\left\{  n\mid s_{n}\sqsubseteq f\right\}  $ and construct the function
$g:A\longrightarrow\omega$ where $g\left(  n\right)  =f\left(  \left\vert
s_{n}\right\vert \right)  +1$ for every $n\in A.$ By the previous remark,
there is $\alpha<\mathfrak{b}$ such that $f_{\alpha}\upharpoonright A$ is not
dominated by $g\upharpoonright A.$ It is then clear that $S_{p}\left(
p,f\right)  \in Catch_{\exists}\left(  H_{\alpha}\right)  .$ \qquad
\ \ \qquad\ \ \ \qquad\ \ \ \ 
\end{proof}

\qquad\ \ \ \ \ \ \ \ \ \ \ \ \ \ \ \ \ \ \ \ \ \ \ \ \ \ \ \ \ \ \ \ \ \ \ \ \ \ \ \ \ \ \ \ \ \ \ \ \ \ \ 

Letting $p$ be a Miller tree and $S\in\left[  \omega\right]  ^{\omega},$ we
define the game $\mathcal{G}\left(  p,S\right)  $ as follows:

\qquad\ \ 

\begin{center}%
\begin{tabular}
[c]{|l|l|l|l|l|l|}\hline
$\mathsf{I}$ & $s_{0}$ &  & $s_{1}$ &  & $\cdots$\\\hline
$\mathsf{II}$ &  & $r_{0}$ &  & $r_{1}$ & \\\hline
\end{tabular}

\end{center}

\qquad\ \ \qquad\ 

\begin{enumerate}
\item Each $s_{i}$ is a splitting node of $p.$

\item $r_{i}\in\omega.$

\item $s_{i+1}$ extends $s_{i}.$

\item $s_{i+1}\left(  \left\vert s_{i}\right\vert \right)  \in S$ and is
bigger than $r_{i}.$
\end{enumerate}

\qquad\ \ \ \ \ 

Player $\mathsf{I}$ wins the game if she can continue playing for infinitely
many rounds. Given $S\in\left[  \omega\right]  ^{\omega}$ we denote by
$Hit\left(  S\right)  $ as the set of all subsets of $\omega$ that have
infinite intersection with $S.$

\begin{lemma}
Letting $p$ be a Miller tree and $S\in\left[  \omega\right]  ^{\omega}$, we
have the following:

\begin{enumerate}
\item Player $I$ has a winning strategy if and only if there is $q\leq p$ such
that $\left[  q\right]  _{split}\subseteq\left[  S\right]  ^{\omega}.$

\item Player $II$ has a winning strategy if and only if there is
$H:Split\left(  p\right)  \longrightarrow\omega$ such that if $f\in\left[
p\right]  $ then the set $\left\{  f\upharpoonright n\in Split\left(
p\right)  \mid f\left(  n\right)  \in S\right\}  $ is almost contained in
$\left\{  f\upharpoonright n\in Split\left(  p\right)  \mid f\left(  n\right)
<H\left(  f\upharpoonright n\right)  \right\}  $ (in particular $\left[
p\right]  _{split}\cap Hit\left(  S\right)  \in\mathcal{K}\left(  p\right)  $).
\end{enumerate}
\end{lemma}

\begin{proof}
The first equivalence is easy so we leave it for the reader. Now assume there
is a winning strategy $\pi$ for $II.$ We define $H:Split\left(  p\right)
\longrightarrow\omega$ such that if $s\in Split\left(  p\right)  $ then
$\pi\left(  \overline{x}\right)  <H\left(  s\right)  $ where $\overline{x}$ is
any partial play in which player $I$ has build $s$ and $II$ has played
according to $\pi$ (note there are only finitely many of those $\overline{x}$
so we can define $H\left(  s\right)  $). We want to prove that if $f\in\left[
p\right]  $ then $\left\{  f\upharpoonright n\in Split\left(  p\right)  \mid
f\left(  n\right)  \in S\right\}  $ is almost contained in the set $\left\{
f\upharpoonright n\in Split\left(  p\right)  \mid f\left(  n\right)  <H\left(
f\upharpoonright n\right)  \right\}  .$ Assume this is not the case. $\ $Let
$B$ be the set of all $n\in\omega$ such that $f\upharpoonright n\in
Split\left(  p\right)  $ with $f\left(  n\right)  \in S$ but $H\left(
f\upharpoonright n\right)  \leq f\left(  n\right)  .$ By our hypothesis $B$ is
infinite and then we enumerate it as $B=\left\{  b_{n}\mid n\in\omega\right\}
$ in increasing order. Consider the run of the game where $\mathsf{I}$ plays
$f\upharpoonright b_{n}$ at the $n$-th stage. This is possible since $f\left(
b_{n}\right)  \in S$ and $H\left(  f\upharpoonright b_{n}\right)  \leq
f\left(  b_{n}\right)  $ so $\mathsf{I}$ will win the game, which is a
contradiction. The other implication is easy.
\end{proof}

\qquad\ \ 

Since $\mathcal{G}\left(  p,S\right)  $ is an open fame for $\mathsf{II}$ by
the Gale-Stewart theorem (see \cite{Kechris}) it is determined, so we conclude
the following dichotomy:\ \ \ \ \ \ \ \ \ \ \ \ \ \ 

\begin{corollary}
If $p$ is a Miller tree and $S\in\left[  \omega\right]  ^{\omega}$ then one
and only one of the following holds:

\begin{enumerate}
\item There is $q\leq p$ such that $\left[  q\right]  _{split}\subseteq\left[
S\right]  ^{\omega}.$

\item There is $H:Split\left(  p\right)  \longrightarrow\omega$ such that if
$f\in\left[  p\right]  $ then the set defined as $\left\{  f\upharpoonright
n\in Split\left(  p\right)  \mid f\left(  n\right)  \in S\right\}  $ is almost
contained in the following set: $\left\{  f\upharpoonright n\in Split\left(
p\right)  \mid f\left(  n\right)  <H\left(  f\upharpoonright n\right)
\right\}  $ (and $\left[  p\right]  _{split}\cap Hit\left(  S\right)
\in\mathcal{K}\left(  p\right)  $).
\end{enumerate}
\end{corollary}

\qquad\ \ \ 

In particular, for every Miller tree $p$ and $S\in\left[  \omega\right]
^{\omega}$ there is $q\leq p$ such that either $\left[  q\right]
_{split}\subseteq\left[  S\right]  ^{\omega}$\ or\ $\left[  q\right]
_{split}\subseteq\left[  \omega\setminus S\right]  ^{\omega}$ \ (although this
fact can be proved easier without the game).\ \ \ \ \ \ \ \ 

\begin{definition}
Let $p$ be a Miller tree and $S\in\left[  \omega\right]  ^{\omega}.$ We say
$S$ \emph{tree-splits }$p$ if there are Miller trees $q_{0},q_{1}\leq p$ such
that $\left[  q_{0}\right]  _{split}\subseteq\left[  S\right]  ^{\omega}$ and
$\left[  q_{1}\right]  _{split}\subseteq\left[  \omega\setminus S\right]
^{\omega}.$ $\mathcal{S}$ is a \emph{Miller tree-splitting family }if every
Miller tree is tree-split by some element of $\mathcal{S}.$\ 
\end{definition}

\qquad\ \ \ 

It is easy to see that every Miller-tree splitting family is a splitting
family and it is also easy to see that every $\omega$-splitting family is a
Miller-tree splitting family. We will now prove there is a Miller-tree
splitting family of size $\mathfrak{s}.$\qquad\ \ \ \qquad\ \ \ \ \ \ \ \ 

\begin{proposition}
The smallest size of a Miller-tree splitting family is $\mathfrak{s}.$
\end{proposition}

\begin{proof}
We will construct a Miller-tree splitting family of size $\mathfrak{s.}$ In
case $\mathfrak{b\leq s}$ there is an $\omega$-splitting family of size
$\mathfrak{s}$ and this is a Miller tree-splitting family as remarked above.

\ \qquad\ \ \ \ 

Now assume $\mathfrak{s}<\mathfrak{b}.$ We will show that any $\left(
\omega,\omega\right)  $-splitting family is a Miller tree splitting family.
Let $\mathcal{S}=\left\{  S_{\alpha}\mid\alpha<\mathfrak{s}\right\}  $ be a
$\left(  \omega,\omega\right)  $-splitting family and $p$ a Miller tree. Let
$Split_{1}\left(  p\right)  =\left\{  s_{n}\mid n\in\omega\right\}  $ and for
every $n<\omega$ define $A_{n}$ as the set of all $\alpha<\mathfrak{s}$ such
that there is $i\left(  \alpha,n\right)  $ such that there is no $q\leq
p_{s_{n}}$ for which $\left[  q\right]  _{split}\subseteq\lbrack S_{\alpha
}^{i\left(  \alpha,n\right)  }]^{\omega}$ (hence $\left[  p_{s_{n}}\right]
_{split}\cap Hit(S_{\alpha}^{i\left(  \alpha,n\right)  })\in\mathcal{K}\left(
p_{s_{n}}\right)  $). Since $A_{n}$ has size less than $\mathfrak{b}=$
\textsf{cov}$\left(  \mathcal{K}\left(  p_{s_{n}}\right)  \right)  $ we can
find $f_{n}\in\left[  p_{s_{n}}\right]  $ such that $X_{n}=Sp\left(  p_{s_{n}%
},f_{n}\right)  \notin%
{\textstyle\bigcup\limits_{\alpha\in A_{n}}}
Hit(S_{\alpha}^{i\left(  \alpha,n\right)  }),$ which means $X_{n}$ is almost
disjoint with every $S_{\alpha}^{i\left(  \alpha,n\right)  }$ whenever
$\alpha\in A_{n}.$ Since $\mathcal{S}$ is an $\left(  \omega,\omega\right)
$-splitting family, there is $\alpha<\mathfrak{s}$ such that both $F=\left\{
n\mid\left\vert S_{\alpha}\cap X_{n}\right\vert =\omega\right\}  $ and
$G=\left\{  n\mid\left\vert \left(  \omega\setminus S_{\alpha}\right)  \cap
X_{n}\right\vert =\omega\right\}  $ are infinite (in particular, they are not
empty). Choose any $n\in F$ and $m\in G.$ We then know that $X_{n}$ is not
almost disjoint with $S_{\alpha}$ so then there must be $q_{0}\leq p_{s_{n}}$
for which $\left[  q_{0}\right]  _{split}\subseteq\lbrack S_{\alpha}]^{\omega
}.$ In the same way, since $m\in G$ \ there must be $q_{1}\leq p_{s_{n}}$ for
which $\left[  q_{1}\right]  _{split}\subseteq\lbrack\omega\setminus
S_{\alpha}]^{\omega}$ and then $S_{\alpha}$ tree-splits $p.$
\end{proof}

\qquad\ \ \ \qquad\ \ \ \qquad\ \ \ \ \qquad\ \ \ 

The following lemma is probably well known:

\begin{lemma}
Assume $\kappa<\mathfrak{d}$ and for every $\alpha<\kappa$ let $\mathcal{F}%
_{\alpha}\subseteq\left[  \omega\right]  ^{<\omega}$ be an infinite set of
disjoint finite subsets of $\omega$ and $g_{\alpha}:\bigcup\mathcal{F}%
_{\alpha}\longrightarrow\omega.$ Then there is $f:\omega\longrightarrow\omega$
such that for every $\alpha<\kappa$ there are infinitely many $X\in
\mathcal{F}_{\alpha}$ such that $g_{\alpha}\upharpoonright X<f\upharpoonright
X.$
\end{lemma}

\begin{proof}
Given $\alpha<\kappa$ find an interval partition $\mathcal{P}_{\alpha
}=\left\{  P_{\alpha}\left(  n\right)  \mid n\in\omega\right\}  $ such that
for every $n\in\omega$ there is $X\in\mathcal{F}_{\alpha}$ such that
$X\subseteq P_{\alpha}\left(  n\right)  $ (this is possible since
$\mathcal{F}_{\alpha}$ is infinite and its elements are pairwise disjoint).
Then define the function $\overline{g}_{\alpha}:\omega\longrightarrow\omega$
such that $\overline{g}_{\alpha}\upharpoonright P_{\alpha}\left(  n\right)  $
is the constant function $max\left\{  g_{\alpha}\left[  P_{\alpha}\left(
n+1\right)  \right]  \right\}  .$ Since $\kappa$ is smaller than
$\mathfrak{d},$ we can then find an increasing function $f:\omega
\longrightarrow\omega$ that is not dominated by any of the $\overline
{g}_{\alpha}.$ It is easy to prove that $f$ has the desired property.
\end{proof}

\qquad\ \ \ \ 

Now we can prove the following lemma that will be useful:

\begin{lemma}
Let $q$ be a Miller tree and $\kappa<\mathfrak{d}.$ If $\left\{  H_{\alpha
}\mid\alpha<\kappa\right\}  \subseteq\omega^{Split\left(  q\right)  }$ then
there is $r\leq q$ such that $Split\left(  r\right)  =Split\left(  q\right)
\cap r$ and $\left[  r\right]  _{split}\cap%
{\displaystyle\bigcup\limits_{\alpha<\kappa}}
Catch_{\forall}\left(  H_{\alpha}\right)  =\emptyset$.
\end{lemma}

\begin{proof}
We will first prove there is $G:Split\left(  q\right)  \longrightarrow\omega$
such that $%
{\displaystyle\bigcup\limits_{\alpha<\kappa}}
Catch_{\forall}\left(  H_{\alpha}\right)  $ is a subset of $Catch_{\exists
}\left(  G\right)  .$ Given $t\in Split\left(  q\right)  $ let $T\left(
t,\alpha\right)  $ the subtree of $q$ such that if $f\in\left[  T\left(
t,\alpha\right)  \right]  $ then $t\sqsubseteq f$ and if $t\sqsubseteq
f\upharpoonright n$ and $f\upharpoonright n\in Split\left(  q\right)  $ then
$f\left(  n\right)  \in H_{\alpha}\left(  f\upharpoonright n\right)  .$
Clearly $T\left(  t,\alpha\right)  $ is a finitely branching subtree of $q.$
Then define $\mathcal{F}\left(  t,\alpha\right)  =\left\{  Split_{n}\left(
q\right)  \cap T\left(  t,\alpha\right)  \mid n<\omega\right\}  $ which is an
infinite collection of pairwise disjoint finite sets and let $g_{\left(
t,\alpha\right)  }:\bigcup\mathcal{F}\left(  t,\alpha\right)  \longrightarrow
\omega$ given by $g_{\left(  t,\alpha\right)  }\left(  s\right)  =H_{\alpha
}\left(  s\right)  .$ Since $\kappa<\mathfrak{d}$ by the previous lemma, we
can find $G:Split\left(  q\right)  \longrightarrow\omega$ such that if
$\alpha<\kappa$ and $t\in Split\left(  q\right)  $ then there are infinitely
many $Y\in\mathcal{F}\left(  t,\alpha\right)  $ such that $g_{\left(
t,\alpha\right)  }\upharpoonright Y<G\upharpoonright Y.$ We will now prove
that $%
{\displaystyle\bigcup\limits_{\alpha<\kappa}}
Catch_{\forall}\left(  H_{\alpha}\right)  \subseteq Catch_{\exists}\left(
G\right)  .$ Let $\alpha<\kappa$ and $f\in Catch_{\forall}\left(  H_{\alpha
}\right)  .$ Find $t\in Split\left(  q\right)  $ such that $t\sqsubseteq f$
and if $t\sqsubseteq f\upharpoonright m$ and $f\upharpoonright m\in
Split\left(  q\right)  $ then $f\left(  m\right)  \in H_{\alpha}\left(
f\upharpoonright m\right)  .$ Note that $f$ is a branch through $T\left(
t,\alpha\right)  .$ Let $Y\in\mathcal{F}\left(  t,\alpha\right)  $ such that
$g_{\left(  t,\alpha\right)  }\upharpoonright Y<G\upharpoonright Y$ and since
$f\in\left[  T\left(  t,\alpha\right)  \right]  $, there is $n\in\omega$ such
that $f\upharpoonright n\in Y$ so $f\left(  n\right)  <H_{\alpha}\left(
f\upharpoonright n\right)  <G\left(  f\upharpoonright n\right)  .$

\qquad\ \ \ \ 

Define $r\leq q$ such that $Split\left(  r\right)  =Split\left(  q\right)
\cap r$ and $suc_{r}\left(  s\right)  =suc_{q}\left(  s\right)  \setminus
G\left(  s\right)  .$ Clearly $\left[  r\right]  _{split}$ is disjoint from
$Catch_{\exists}\left(  G\right)  .$
\end{proof}

\qquad\ \ \ 

We can then finally prove our main theorem.

\begin{theorem}
There is a $+$-Ramsey \textsf{MAD} family.
\end{theorem}

\begin{proof}
If $\mathfrak{a<s},$ then $\mathfrak{a}$ is smaller than $\mathfrak{d}$ so
then there is a $+$-Ramsey \textsf{MAD} family (in fact, there is a
Miller-indestructible \textsf{MAD} family). So we assume $\mathfrak{s}%
\leq\mathfrak{a}$ for the rest of the proof. Fix an $\left(  \omega
,\omega\right)  $-splitting family $\mathcal{S}=\left\{  S_{\alpha}\mid
\alpha<\mathfrak{s}\right\}  $ that is also a Miller-tree splitting family.
Let $\left\{  L,R\right\}  $ be a partition of the limit ordinals smaller than
$\mathfrak{c}$ such that both $L$ and $R$ have size continuum. Enumerate by
$\left\{  X_{\alpha}\mid\alpha\in L\right\}  $ all infinite subsets of
$\omega$ and by $\left\{  T_{\alpha}\mid\alpha\in R\right\}  $ all subtrees of
$\omega^{<\omega}.$ We will recursively construct $\mathcal{A}=\left\{
A_{\xi}\mid\xi<\mathfrak{c}\right\}  $ and $\left\{  \sigma_{\xi}\mid
\xi<\mathfrak{c}\right\}  $ as follows:

\begin{enumerate}
\item $\mathcal{A}$ is an \textsf{AD} family and $\sigma_{\alpha}%
\in2^{<\mathfrak{s}}$ for every $\alpha<\mathfrak{c}.$

\item If $\sigma_{\alpha}\in2^{\beta}$ and $\xi<\beta$ then $A_{\alpha
}\subseteq^{\ast}S_{\xi}^{\sigma_{\alpha}\left(  \xi\right)  }.$

\item If $\alpha\neq\beta$ then $\sigma_{\alpha}\neq\sigma_{\beta}.$

\item If $\delta\in L$ and $X_{\delta}\in\mathcal{I}\left(  \mathcal{A}%
_{\delta}\right)  ^{+}$ then $A_{\delta+n}\subseteq X_{\delta}$ for every
$n\in\omega$ (where $\mathcal{A}_{\delta}=\left\{  A_{\xi}\mid\xi
<\delta\right\}  $).

\item If $\delta\in R$ and $T_{\delta}$ is an $\mathcal{I}\left(
\mathcal{A}_{\delta}\right)  ^{+}$-branching tree then there is $f\in\left[
T_{\delta}\right]  $ such that $A_{\delta+n}\subseteq im\left(  f\right)  $
for every $n\in\omega.$
\end{enumerate}

\qquad\ \ 

It is clear that if we manage to perform the construction then $\mathcal{A}$
will be a $\mathfrak{+}$-Ramsey \textsf{MAD} family (and it will be completely
separable too). Let $\delta$ be a limit ordinal and assume we have constructed
$A_{\xi}$ for every $\xi<\delta.$ In case $\delta\in L$ we just proceed as in
the case of the completely separable \textsf{MAD} family, so assume $\delta\in
R.$ Since $\mathcal{A}_{\delta}=\left\{  A_{\xi}\mid\xi<\delta\right\}  $ is
nowhere-\textsf{MAD,} we can find $p$ an $\mathcal{A}_{\delta}^{\perp}%
$-branching subtree of $T_{\delta}.$

\qquad\ \ 

First note that since $\mathcal{S}$ is a Miller-tree splitting family, for
every Miller tree $q$ there is $\alpha<\mathfrak{s}$ and $\tau_{q}\in
2^{\alpha}$ such that:

\begin{enumerate}
\item $S_{\alpha}$ tree-splits $q.$

\item If $\xi<\alpha$ then there is no $q^{\prime}\leq q$ such that $\left[
q^{\prime}\right]  _{split}\subseteq\lbrack S_{\xi}^{1-\tau_{q}\left(
\xi\right)  }]^{\omega}.$
\end{enumerate}

\qquad\ \ \ 

Note that if $q^{\prime}\leq q$ then $\tau_{q^{\prime}}$ extends $\tau_{q}.$
If $q\leq p$ and $\tau_{q}\in2^{\alpha}$ we fix the following items:

\begin{enumerate}
\item $W_{0}\left(  q\right)  =\left\{  \xi<\alpha\mid\exists\beta
<\delta\left(  \sigma_{\beta}=\tau_{q}\upharpoonright\xi\right)  \right\}  $
and\qquad\ \ \ \ \ \ 

$W_{1}\left(  q\right)  =\left\{  \xi<\alpha\mid\exists\beta<\delta\left(
\Delta\left(  \sigma_{\beta},\tau_{q}\right)  =\xi\right)  \right\}  .$

\item Let $\xi\in W_{0}\left(  q\right)  $ we then find $\beta$ such that
$\sigma_{\beta}=\tau_{q}\upharpoonright\xi$ and define $G_{q,\xi}:Split\left(
q\right)  \longrightarrow\omega$ such that if $s\in Split\left(  q\right)  $
then $A_{\beta}\cap suc_{q}\left(  s\right)  \subseteq G_{q,\xi}\left(
s\right)  $ (this is possible since $q$ is $\mathcal{A}_{\delta}^{\perp}$-branching).

\item Given $\xi$ $\in W_{1}\left(  q\right)  $ we know there is no
$q^{\prime}\leq q$ such that $\left[  q^{\prime}\right]  _{split}%
\subseteq\lbrack S_{\xi}^{1-\tau_{q}\left(  \xi\right)  }]^{\omega}.$ We know
that there is $H_{q,\xi}:Split\left(  q\right)  \longrightarrow\omega$ such
that if $f\in\left[  q\right]  $, the set defined as $\{f\upharpoonright n\in
Split\left(  q\right)  \mid f\left(  n\right)  \in S_{\xi}^{1-\tau_{q}\left(
\xi\right)  }\}$ is almost contained in the set $\left\{  f\upharpoonright
n\in Split\left(  q\right)  \mid f\left(  n\right)  <H_{q,\xi}\left(
f\upharpoonright n\right)  \right\}  .$

\item If $U\in\left[  W_{0}\left(  q\right)  \right]  ^{<\omega}$ and
$V\in\left[  W_{1}\left(  q\right)  \right]  ^{<\omega}$ choose any \qquad\ \ \ \ 

$J_{q,U,V}:Split\left(  q\right)  \longrightarrow\omega$ such that if $s\in
Split\left(  q\right)  $ then $J_{q,U,V}\left(  s\right)  >max\left\{
G_{q,\xi}\left(  s\right)  \mid\xi\in U\right\}  ,$ $max\left\{  H_{q,\xi
}\left(  s\right)  \mid\xi\in V\right\}  .$

\item $\mathcal{A}\left(  q\right)  =\left\{  A_{\xi}\in\mathcal{A}_{\delta
}\mid\tau_{q}\nsubseteq\sigma_{\xi}\right\}  .$
\end{enumerate}

\qquad\ \ \ \qquad\ \ \ \ \ \ \ \ \ \qquad\ \ \ \ \ 

Note that if $\xi\in W_{0}\left(  q\right)  $ then there is a unique
$\beta<\delta$ such that $\sigma_{\beta}=\tau_{q}\upharpoonright\xi$ (although
the analogous remark is not true for the elements of $W_{1}\left(  q\right)
$). The following claim will play a fundamental role in the proof:

\begin{claim}
If $q\leq p$ then there is $r\leq q$ such that $\left[  r\right]
_{split}\subseteq\mathcal{I}\left(  \mathcal{A}\left(  q\right)  \right)
^{+}$.
\end{claim}

\qquad\ \ \ \ 

Let $\alpha<\mathfrak{s}$ such that $\tau_{q}\in2^{\alpha}.$ Since
$\mathfrak{s\leq d}$, we know there is $r\leq q$ such that $\left[  r\right]
_{split}$ is disjoint from $\bigcup\left\{  Catch_{\forall}\left(
J_{q,U,V}\right)  \mid U\in\left[  W_{0}\left(  q\right)  \right]  ^{<\omega
},V\in\left[  W_{1}\left(  q\right)  \right]  ^{<\omega}\right\}  $ and
$Split\left(  r\right)  =Split\left(  q\right)  \cap r.$ We will now prove
$\left[  r\right]  _{split}\subseteq\mathcal{I}\left(  \mathcal{A}\left(
q\right)  \right)  ^{+}$ but assume this is not the case. Therefore, there is
$f\in\left[  r\right]  $ and $F\in\left[  \mathcal{A}\left(  q\right)
\right]  ^{<\omega}$ such that $X=Sp\left(  r,f\right)  \subseteq^{\ast
}\bigcup F.$ Let $F=F_{1}\cup F_{2}$ and $U\in\left[  W_{0}\left(  q\right)
\right]  ^{<\omega}$, $V\in\left[  W_{1}\left(  q\right)  \right]  ^{<\omega}$
such that for every $A_{\beta}\in F_{1}$ there is $\xi_{\beta}\in U$ such that
$\sigma_{\beta}=\tau_{q}\upharpoonright\xi_{\beta}$ and for every $A_{\gamma
}\in F_{2}$ there is $\eta_{\gamma}\in V$ such that $\Delta\left(  \tau
_{q},\sigma_{\gamma}\right)  =\eta_{\gamma}.$ Let $D\subseteq\left\{  n\mid
f\upharpoonright n\in Split\left(  r\right)  \right\}  $ be the (infinite) set
of all $n<\omega$ such that the following holds:

\qquad\ \ 

\begin{enumerate}
\item $f\upharpoonright n\in Split\left(  r\right)  $ and $f\left(  n\right)
\in\bigcup F.$

\item If $\eta_{\gamma}\in V$ then $A_{\gamma}\setminus n\subseteq
S_{\eta_{\gamma}}^{1-\tau_{q}\left(  \eta_{\gamma}\right)  }.$

\item $f\left(  n\right)  >J_{q,U,V}\left(  f\upharpoonright n\right)  .$

\item If $\eta\in V$ and $f\left(  n\right)  \in S_{\eta}^{1-\tau_{q}\left(
\eta\right)  }$ then $f\left(  n\right)  <H_{q,\eta}\left(  f\upharpoonright
n\right)  <J_{q,U,V}\left(  f\upharpoonright n\right)  $ (recall that
$\{f\upharpoonright m\in Split\left(  q\right)  \mid f\left(  m\right)  \in
S_{\eta}^{1-\tau_{q}\left(  \eta\right)  }\}$ is almost contained in $\left\{
f\upharpoonright m\in Split\left(  q\right)  \mid f\left(  m\right)
<H_{q,\eta}\left(  f\upharpoonright m\right)  \right\}  $).
\end{enumerate}

\qquad\ \ 

We first claim that if $n\in D,\xi_{\beta}\in U$ and $\eta_{\gamma}\in$ $V$
then $f\left(  n\right)  \notin A_{\beta}\cup S_{\eta_{\gamma}}^{1-\tau
_{q}\left(  \eta_{\gamma}\right)  }.$ On one hand, since $A_{\beta}\cap
suc_{q}\left(  f\upharpoonright n\right)  \subseteq G_{q,\xi_{\beta}}\left(
f\upharpoonright n\right)  <J_{q,U,V}\left(  f\upharpoonright n\right)  $ and
$f\left(  n\right)  >J_{q,U,V}\left(  f\upharpoonright n\right)  $ then
$f\left(  n\right)  \notin A_{\beta}.$ On the other hand, if it was the case
that $f\left(  n\right)  \in S_{\eta_{\gamma}}^{1-\tau_{q}\left(  \eta
_{\gamma}\right)  }$ so $f\left(  n\right)  <H_{q,\eta}\left(
f\upharpoonright n\right)  <J_{q,U,V}\left(  f\upharpoonright n\right)  $ but
we already know that $f\left(  n\right)  >J_{q,U,V}\left(  f\upharpoonright
n\right)  .$ Since $n\leq f\left(  n\right)  $ (recall every branch through
$p$ is increasing) $f\left(  n\right)  \notin A_{\gamma}$ for every
$\eta_{\gamma}\in V$ because $A_{\gamma}\setminus n\subseteq S_{\eta_{\gamma}%
}^{1-\tau_{q}\left(  \eta_{\gamma}\right)  }.$ This implies $f\left(
n\right)  \notin\bigcup F$ which is a contradiction and finishes the proof of
the claim.

\qquad\ \ 

Back to the proof of the theorem, we recursively build a tree of Miller trees
$\left\{  p\left(  s\right)  \mid s\in2^{<\omega}\right\}  $ with the
following properties:

\qquad\ \ \ 

\begin{enumerate}
\item $p\left(  \emptyset\right)  =p.$

\item $p\left(  s^{\frown}i\right)  \leq p\left(  s\right)  $ and the stem of
$p\left(  s^{\frown}i\right)  $ has length at least $\left\vert s\right\vert
.$

\item $\tau_{p\left(  s^{\frown}0\right)  }$ and $\tau_{p\left(  s^{\frown
}1\right)  }$ are incompatible.

\item $\left[  p\left(  s^{\frown}i\right)  \right]  _{split}\subseteq
\mathcal{I}\left(  \mathcal{A}\left(  p\left(  s\right)  \right)  \right)
^{+}.$
\end{enumerate}

\qquad\ \ \ 

This is easy to do with the aid of the previous claim. For every
$g\in2^{\omega}$ let $\tau_{g}=\bigcup\tau_{p\left(  g\upharpoonright
m\right)  }.$ Note that if $g_{1}\neq g_{2}$ then $\tau_{g_{1}}$ and
$\tau_{g_{2}}$ are two incompatible nodes of $2^{<\mathfrak{s}}.$ Since
$\mathcal{A}_{\delta}$ has size less than the continuum, there is
$g\in2^{\omega}$ such that there is no $\beta<\delta$ such that $\sigma
_{\beta}$ extends $\tau_{g}$ and then $\mathcal{A}_{\delta}=\bigcup
\limits_{m\in\omega}\mathcal{A}\left(  p\left(  g\upharpoonright m\right)
\right)  .$ Let $f$ be the only element of $\bigcap\limits_{m\in\omega}\left[
p\left(  g\upharpoonright m\right)  \right]  .$ Obviously, $f$ is a branch
through $p$ and we claim that $im\left(  f\right)  \in\mathcal{I}\left(
\mathcal{A}_{\delta}\right)  ^{+}.$ This is easy since if $A_{\xi_{1}%
},...,A_{\xi_{n}}\in\mathcal{A}_{\delta}$ then we can find $m<\omega$ such
that $A_{\xi_{1}},...,A_{\xi_{n}}\in\mathcal{A}\left(  p\left(
g\upharpoonright m\right)  \right)  $ and then we know that $Sp\left(
p\left(  g\upharpoonright m+1\right)  ,f\right)  \nsubseteq^{\ast}A_{\xi_{1}%
}\cup...\cup A_{\xi_{n}}$ and since $Sp\left(  p\left(  g\upharpoonright
m+1\right)  ,f\right)  $ is contained in $im\left(  f\right)  $ we conclude
that $im\left(  f\right)  \in\mathcal{I}\left(  \mathcal{A}_{\delta}\right)
^{+}.$

\qquad\ \ 

Finally, find a partition $\left\{  Z_{n}\mid n\in\omega\right\}
\subseteq\mathcal{I}\left(  \mathcal{A}_{\delta}\right)  ^{+}$ of $im\left(
f\right)  $ and using the method of Shelah construct $A_{\delta+n}$ such that
$A_{\delta+n}\subseteq Z_{n}.$ This finishes the proof.
\end{proof}

\qquad\ \ 

\bigskip

\chapter{\textsf{MAD} examples}

We now repeat the main definitions of \textsf{MAD} families used in this thesis:

\begin{definition}
Let $\mathcal{A}$ be a \textsf{MAD} family.

\begin{enumerate}
\item $\mathcal{A}$ is $\mathbb{P}$\emph{-indestructible} if $\mathcal{A}$
remains \textsf{MAD} after forcing with $\mathbb{P}$ (we are mainly interested
where $\mathbb{P}$ is Cohen, random, Sacks or Miller forcing).

\item $\mathcal{A}$ is \emph{weakly tight }if for every $\left\{  X_{n}\mid
n\in\omega\right\}  \subseteq\mathcal{I}\left(  \mathcal{A}\right)  ^{+}$
there is $B\in\mathcal{I}\left(  \mathcal{A}\right)  $ such that $\left\vert
B\cap X_{n}\right\vert =\omega$ for infinitely many $n\in\omega.$\qquad

\item $\mathcal{A}$ is\ \emph{tight }if for every $\left\{  X_{n}\mid
n\in\omega\right\}  \subseteq\mathcal{I}\left(  \mathcal{A}\right)  ^{+}$
there is $B\in\mathcal{I}\left(  \mathcal{A}\right)  $ such that $B\cap X_{n}$
is infinite for every $n\in\omega.$

\item $\mathcal{A}$ is \emph{Laflamme }if $\mathcal{A}$ can not be extended to
an $F_{\sigma}$-ideal.

\item $\mathcal{A}$ is $+$\emph{-Ramsey} if for every $\mathcal{I}\left(
\mathcal{A}\right)  ^{+}$-branching tree $T,$ there is $f\in\left[  T\right]
$ such that $im\left(  f\right)  \in\mathcal{I}\left(  \mathcal{A}\right)
^{+}.$

\item $\mathcal{A}$ is \textsf{FIN}$\times$\textsf{FIN}-like if $\mathcal{I}%
\left(  \mathcal{A}\right)  \nleq_{\mathsf{K}}\mathcal{J}$ for every analytic
ideal $\mathcal{J}$ such that \textsf{FIN}$\times$\textsf{FIN }$\nleq
_{\mathsf{K}}\mathcal{J}.$

\item $\mathcal{A}$ is called \emph{Shelah-Stepr\={a}ns }if for every
$X\in\left(  \mathcal{I}\left(  \mathcal{A}\right)  ^{<\omega}\right)  ^{+}$
there is $Y\in\left[  X\right]  ^{\omega}$ such that $\bigcup Y\in
\mathcal{I}\left(  \mathcal{A}\right)  .$

\item $\mathcal{A}$ is \emph{strongly tight }if for every family $\left\{
X_{n}\mid n\in\omega\right\}  $ such that for every $B\in\mathcal{I}\left(
\mathcal{A}\right)  $ the set $\left\{  n\mid X\subseteq^{\ast}B\right\}  $ is
finite, there is $A\in\mathcal{A}$ such that $A\cap X_{n}\neq\emptyset$ for
every $n\in\omega.$

\item $\mathcal{A}$ is a \emph{raving \textsf{MAD} family if }for every family
$X=\left\{  X_{n}\mid n\in\omega\right\}  $ that is locally finite according
to\emph{ }$\mathcal{I}\left(  \mathcal{A}\right)  $ there is $B\in
\mathcal{I}\left(  \mathcal{A}\right)  $ such that $B$ contains at least one
element of each $X_{n}$ (a family $X=\left\{  X_{n}\mid n\in\omega\right\}  $
such that $X_{n}\subseteq\left[  \omega\right]  ^{<\omega}$ is \emph{locally
finite according to }$\mathcal{I}\left(  \mathcal{A}\right)  $ if for every
$B\in\mathcal{I}\left(  \mathcal{A}\right)  $ for almost all $n\in\omega$
there is $s\in X_{n}$ such that $s\cap B=\emptyset$).
\end{enumerate}
\end{definition}

\qquad\ \ \ \ \ \ 

In this chapter we will build (consistently) examples of the non implications
of the previous notions. Note that raving implies all the other properties.
The following definition will be useful in this chapter:

\begin{definition}
Let $\mathcal{I}$ be an ideal. We say that $\mathcal{I}$ is \emph{nowhere
Shelah-Stepr\={a}ns }if no restriction of$\mathcal{\ I}$ is Shelah-Stepr\={a}ns.
\end{definition}

\qquad\ \ \ \ 

It is easy to see that \textsf{nwd}$,$ $tr($\textsf{ctble}$),$ $tr\left(
\mathcal{N}\right)  ,$ $tr\left(  \mathcal{K}_{\sigma}\right)  $ and every
$F_{\sigma}$-ideal are nowhere Shelah-Stepr\={a}ns.

\begin{lemma}
Let $\mathcal{I},\mathcal{J}$ be two ideals such that $\mathcal{I}$ is nowhere
Shelah-Stepr\={a}ns and $\mathcal{J}\nleq_{K}\mathcal{I}.$ Let
$\mathcal{A\subseteq J}$ be a countable AD family and $f:\left(
\omega,\mathcal{I}\right)  \longrightarrow\left(  \omega,\mathcal{I}\left(
\mathcal{A}\right)  \right)  $ be a Kat\v{e}tov morphism. Then there is
$B\in\mathcal{J}\cap\mathcal{A}^{\perp}$ such that $f^{-1}\left(  B\right)
\in\mathcal{I}^{+}.$
\end{lemma}

\begin{proof}
Let $\mathcal{A}=\left\{  A_{n}\mid n\in\omega\right\}  .$ We know $f$ is a
Kat\v{e}tov morphism so the set $\left\{  f^{-1}\left(  A_{n}\right)  \mid
n\in\omega\right\}  $ is contained in $\mathcal{I}.$ Since $\mathcal{J}%
\nleq_{K}\mathcal{I}$ there is $D\in\mathcal{J}$ such that $C=f^{-1}\left(
D\right)  \in\mathcal{I}^{+}.$ Since $\mathcal{I}\upharpoonright C$ is not
Shelah-Stepr\={a}ns, there is $X\in\left(  \left(  \mathcal{I}\upharpoonright
C\right)  ^{<\omega}\right)  ^{+}$ such that no element of $\mathcal{I}$
contains infinitely many elements of $X.$ For each $n\in\omega$ we choose
$s_{n}\in X$ such that $s_{n}\cap\left(  f^{-1}\left(  A_{0}\cup
...A_{n}\right)  \right)  =\emptyset.$ We then know that $D=%
{\textstyle\bigcup}
s_{n}\in\mathcal{I}^{+}.$ It is easy to see that $B=f\left[  D\right]  $ has
the desired properties.
\end{proof}

\qquad\ \ 

By a simple bookkeeping argument we can then conclude the following:

\begin{proposition}
[$\mathsf{CH}$]Let $\mathcal{I},\mathcal{J}$ be ideals such that $\mathcal{I}$
is nowhere Shelah-Stepr\={a}ns and $\mathcal{J}\nleq_{K}\mathcal{I}.$ Then
there is a \textsf{MAD} family $\mathcal{A}\subseteq\mathcal{J}$ such that
$\mathcal{I}\left(  \mathcal{A}\right)  \nleq_{K}\mathcal{I}.$
\end{proposition}

\qquad\ \ 

The previous proposition shows that Miller indestructibility does not imply
Cohen indestructibility and that Random and Miller indestructibility are
incomparable notions (this is a result of \emph{Brendle and Yatabe}).\qquad\ \ \ 

\subsubsection{Cohen and Random indestructibility does not imply weak
tightness or Laflamme. FIN$\times$FIN-like implies weak tightness\qquad\ \ }

We will start with a useful characterization of weak tightness.

\begin{definition}
We define the ideal $\mathcal{WT}$ on $\omega\times\omega$ as follows:

\begin{enumerate}
\item $\mathcal{WT}\upharpoonright C_{n}$ is a copy of \textsf{FIN}$\times
$\textsf{FIN} (where $C_{n}=\left\{  \left(  n,m\right)  \mid m\in
\omega\right\}  $).

\item $\mathcal{WT}$ extends $\emptyset\times$\textsf{FIN}$.$
\end{enumerate}
\end{definition}

\qquad\ \ \ 

Note that if $B\subseteq\omega\times\omega$ has infinite intersection with
infinitely many columns then $B\in\mathcal{WT}^{+}.$\ \ \ 

\begin{proposition}
$\mathcal{WT}$ is strictly Kat\v{e}tov below \textsf{FIN}$\times$%
\textsf{FIN}$.$
\end{proposition}

\begin{proof}
Note that the identity mapping witnesses $\mathcal{WT}\leq_{K}$\textsf{FIN}%
$\times$\textsf{FIN}$.$ Now, we will show $\mathsf{II}\ $has a winning
strategy in $\mathcal{L}\left(  \mathcal{WT}\right)  .$ This is easy, since
every $C_{n}$ is not in $\mathcal{WT}$ and then player $\mathsf{II}$ can play
in such a way that the set she constructed at the end intersects infinitely
all the $C_{n},$ so it can not be an element of $\mathcal{WT}.$
\end{proof}

\qquad\ \ \ \ \ 

Then we have the following characterization:

\begin{proposition}
If $\mathcal{A}$ is a \textsf{MAD} family then $\mathcal{A}$ is weakly tight
if and only if $\mathcal{A\nleq}_{K}$ $\mathcal{WT}.$
\end{proposition}

\begin{proof}
We will prove that $\mathcal{A}$ is not weakly tight if and only if
$\mathcal{I}\left(  \mathcal{A}\right)  $ $\mathcal{\leq}_{K}$ $\mathcal{WT}.$
First assume $\mathcal{A}$ is not weakly tight, so we can find a partition
$X=\left\{  X_{n}\mid n\in\omega\right\}  \subseteq\mathcal{I}\left(
\mathcal{A}\right)  ^{+}$ such that if $A\in\mathcal{A}$ then $A\cap X_{n}$ is
finite for almost all $n\in\omega.$ Since $\mathcal{A}\upharpoonright X_{n}$
is an AD family, we know $\mathcal{A}\upharpoonright X_{n}$ $\mathcal{\leq
}_{K}$ \textsf{FIN}$\times$\textsf{FIN} so for each $n\in\omega$ fix
$h_{n}:C_{n}\longrightarrow X_{n}$ a Kat\v{e}tov morphism from $\left(
C_{n},\mathcal{WT\upharpoonright}C_{n}\right)  $ to $\left(  X_{n}%
,\mathcal{A}\upharpoonright X_{n}\right)  .$ Letting $h=%
{\textstyle\bigcup}
h_{n}$ we will show $h$ is a Kat\v{e}tov morphism from $\left(  \omega
\times\omega,\mathcal{WT}\right)  $ to $\left(  \omega,\mathcal{A}\right)  .$
If $A\in\mathcal{A}$ then we can find $B\in X^{\perp}$ and a finite set
$F\subseteq\omega$ such that $A=%
{\textstyle\bigcup\limits_{n\in F}}
\left(  A\cap X_{n}\right)  \cup B.$ Clearly $h^{-1}\left(  B\right)  \in$
$\mathcal{WT}$ since $h^{-1}\left(  B\right)  \in$ $\emptyset\times
$\textsf{FIN} and $h^{-1}\left(  A\cap X_{n}\right)  =h_{n}^{-1}\left(  A\cap
X_{n}\right)  $ which is an element of $\mathcal{WT}$ since $h_{n}$ is a
Kat\v{e}tov morphism. Therefore, $h^{-1}\left(  A\right)  \in\mathcal{WT}.$

\qquad\ \ \ 

For the second implication, assume $\mathcal{I}\left(  \mathcal{A}\right)  $
$\mathcal{\leq}_{K}$ $\mathcal{WT}$ so there is a Kat\v{e}tov morphism
$h:\omega\times\omega\longrightarrow\omega$ from $\left(  \omega\times
\omega,\mathcal{WT}\right)  $ to $\left(  \omega,\mathcal{A}\right)  .$ Let
$X_{n}=h\left[  C_{n}\right]  $ which is an element of $\mathcal{I}\left(
\mathcal{A}\right)  ^{+}$ since $h$ is a Kat\v{e}tov morphism. Note that if
$B\cap X_{n}$ is infinite for infinitely many $n\in\omega$ then $h^{-1}\left(
B\right)  \in\mathcal{WT}^{+},$ hence $\left\{  X_{n}\mid n\in\omega\right\}
$ witnesses that $\mathcal{A}$ is not weakly tight.
\end{proof}

\qquad\ \ \qquad\ \ 

We can then conclude the following:

\begin{corollary}
If $\mathcal{A}$ is \textsf{FIN}$\times$\textsf{FIN}-like then $\mathcal{A}$
is weakly tight.
\end{corollary}

\qquad\ \ 

Note that diagonalizing $\mathcal{WT}$ will result in adding a dominating
real, so $\mathcal{WT}$ is not Kat\v{e}tov below $tr\left(  \mathcal{N}%
\right)  $ or $tr\left(  \mathcal{M}\right)  .$ Furthermore, $tr\left(
\mathcal{M}\right)  $ is not Kat\v{e}tov above $\mathcal{ED}$ and $tr\left(
\mathcal{N}\right)  $ is not Kat\v{e}tov above the Solecki ideal (see
\cite{TesisDavid}). Therefore we conclude the following:

\begin{corollary}
[$\mathsf{CH}$]There is a Cohen and random indestructible \textsf{MAD} family
that is not weakly tight. Neither does Cohen or random indestructibility imply
being Laflamme.
\end{corollary}

\subsubsection{Laflamme does not imply Sacks indestructibility or weakly
tight.}

Since every $F_{\sigma}$-ideal is nowhere Shelah-Stepr\={a}ns we can conclude
the following:

\begin{lemma}
[$\mathsf{CH}$]Every Laflamme ideal contains a Laflamme \textsf{MAD} family.
\end{lemma}

\qquad\ \qquad\ \ 

Since both $tr($\textsf{ctble}$)$ and $\mathcal{WT}$ are Laflamme (they are
Kat\v{e}tov above \textsf{conv}) we conclude the following:

\begin{corollary}
[$\mathsf{CH}$]There is a Laflamme family that is destructible by Sacks
forcing and is not weakly tight.
\end{corollary}

\subsubsection{Laflamme does not imply $+$-Ramsey}

Let $\mathcal{J}$ be the ideal on $\omega^{<\omega}$ of all sets
$a\subseteq\omega^{<\omega}$ such that $\pi\left(  a\right)  $ is finite. It
is easy to see that $\mathcal{J}$ can not be extended to an $F_{\sigma}$-ideal
(this is because \textsf{conv }$\leq_{K}\mathcal{J}$).\qquad\ \ \qquad\ \ \ \ 

\begin{proposition}
[$\mathsf{CH}$]There is a Laflamme \textsf{MAD} family that is not $+$-Ramsey.
\end{proposition}

\begin{proof}
Let $\mathcal{BR}=\{\widehat{f}\mid f\in\omega^{\omega}\}$ (where $\widehat
{f}=\left\{  f\upharpoonright n\mid n\in\omega\right\}  $) and $\left\{
\mathcal{I}_{\alpha}\mid\alpha\in\omega_{1}\right\}  $ be the set of all
$F_{\sigma}\,$-ideals in $\omega^{<\omega}.$ We construct $\mathcal{A}%
=\left\{  A_{\alpha}\mid\alpha<\omega_{1}\right\}  $ such that the following holds:

\begin{enumerate}
\item $\mathcal{A\cup BR}$ is an \textsf{AD} family.

\item If $s\in\omega^{<\omega}$ then $\mathcal{A}$ contains a partition of
$suc\left(  s\right)  =\left\{  s^{\frown}n\mid n\in\omega\right\}  .$

\item If\emph{ }$\mathcal{A}_{\alpha}\mathcal{\cup BR}\subseteq\mathcal{I}%
_{\alpha}$ then $A_{\alpha}\notin\mathcal{I}_{\alpha}$ (where $\mathcal{A}%
_{\alpha}=\left\{  A_{\xi}\mid\xi<\alpha\right\}  $).$\qquad$

\item $\mathcal{A}_{\alpha}$ is countable.
\end{enumerate}

\qquad\ \ \ \qquad\ \ \ 

At step $\alpha$ assume that $\mathcal{BR}\cup\mathcal{A}_{\alpha}%
\subseteq\mathcal{I}_{\alpha}$. Since $\mathcal{I}_{\alpha}$ is an $F_{\sigma
}$-ideal and it contains all branches, there is $a\in\mathcal{I}_{\alpha}%
^{+}\cap\mathcal{J}.$ Let $\pi\left(  a\right)  =\left\{  f_{1},...,f_{n}%
\right\}  $ and we now define $b=a\setminus\widehat{f}_{1}\cup...\cup
\widehat{f}_{n}.$ Note that $\pi\left(  b\right)  =\emptyset$ and
$b\in\mathcal{I}_{\alpha}^{+}.$ Let $\varphi$ be a lower semicontinuous
submeasure such that $\mathcal{I}_{\alpha}=$ \textsf{Fin}$\left(
\varphi\right)  $ and let $\mathcal{A}_{\alpha}=\left\{  B_{n}\mid n\in
\omega\right\}  .$ We recursively find $s_{n}\subseteq b\setminus B_{0}%
\cup...\cup B_{n}$ such that $\varphi\left(  s_{n}\right)  \geq n$ (this is
possible since $b\in\mathcal{I}_{\alpha}^{+}$). Then $A_{\alpha}=%
{\textstyle\bigcup\limits_{n\in\omega}}
s_{n}$ is the set we were looking for. It is easy to see that $\mathcal{A\cup
BR}$ is a Laflamme \textsf{MAD} family that is not $+$-Ramsey.
\end{proof}

\subsubsection{$+$-Ramsey does not imply $\mathbb{S}$-indestructibility,
Laflamme or weakly tight}

We start with the following lemma:

\begin{lemma}
Let $\mathcal{A}$ be a countable \textsf{AD} family, $\mathcal{I}$ a tall
ideal such that $\mathcal{A}\subseteq\mathcal{I}$ and $T$ an $\mathcal{I}%
\left(  \mathcal{A}\right)  ^{+}$-tree. Then there is a countable AD family
$\mathcal{B}$ and $R\in\left[  T\right]  $ such that:

\begin{enumerate}
\item $\mathcal{A\subseteq B}.$

\item $\mathcal{B}\subseteq\mathcal{I}.$

\item $R\in\mathcal{I}\left(  \mathcal{B}\right)  ^{++}.$
\end{enumerate}
\end{lemma}

\begin{proof}
Since $\mathcal{A}$ is $+$-Ramsey we can find $R\in\left[  T\right]  $ such
$im\left(  R\right)  \in\mathcal{I}\left(  \mathcal{A}\right)  ^{+}.$ Since
$\mathcal{I}$ is a tall ideal and $\mathcal{A}$ is countable, we can find
$A\in\left[  im\left(  R\right)  \right]  ^{\omega}\cap\mathcal{I\cap
A}^{\perp}.$ We partition $A$ into countably many disjoint pieces and add them
to $\mathcal{A}.$
\end{proof}

\qquad\ \ \ 

With a simple bookkeeping argument we can then prove the following:

\begin{proposition}
[$\mathsf{CH}$]Every tall ideal contains a $+$-Ramsey \textsf{MAD} family.
\end{proposition}

\subsubsection{Weakly tight does not imply $+$-Ramsey}

Recall that if $f:\omega\longrightarrow\omega$ we define $\widehat{f}=\left\{
f\upharpoonright n\mid n\in\omega\right\}  .$ We now have the following:

\begin{lemma}
If $A\subseteq\omega^{<\omega}$ does not have infinite antichains then $A$ can
be covered with finitely many chains.
\end{lemma}

\begin{proof}
Define $S$ as the set of all unsplitting nodes of $A$ i.e. $s\in A$ if and
only if every two extensions of $s$ in $A$ are compatible. Note that
$S\subseteq A$ and every element of $A$ can be extended to an element of $S$
(otherwise $A$ would contain a Sacks tree and hence an infinite antichain).
Let $B\subseteq S$ be a maximal (finite) antichain. For every $s\in B$ let
$b_{s}\in\omega^{\omega}$ the unique branch such that $A\cap\left[  s\right]
\subseteq\widehat{b}_{s}.$ Then (by the maximality of $B$) we conclude
$A\subseteq%
{\textstyle\bigcup\limits_{s\in B}}
\widehat{b}_{s}.$
\end{proof}

\qquad\ \ 

We need the following lemma:\qquad\ \ \ \ \qquad\ \ \ 

\begin{lemma}
If $A=\left\{  A_{n}\mid n\in\omega\right\}  \subseteq\wp\left(
\omega^{<\omega}\right)  $ is a collection of infinite antichains, then there
is an antichain $B$ such that $B\cap A_{n}$ is infinite for infinitely many
$n\in\omega.$
\end{lemma}

\begin{proof}
We say $s\in\omega^{<\omega}$ \emph{watches }$A_{n}$ if $s$ has infinitely
many extensions in $A_{n}.$ Define $T\subseteq\omega^{<\omega}$ such that
$s\in T$ if and only if there are infinitely many $n\in\omega$ such that $s$
watches $A_{n}.$ Note that $T$ is a tree. First assume there is $s\in T$ that
is a maximal node. By shrinking $A$ if needed, we may assume $s$ watches every
element of $A.$ We now define the set $C=\left\{  A_{n}\mid\exists^{\infty
}m\left(  A_{n}\cap\left[  s^{\frown}m\right]  \neq\emptyset\right)  \right\}
.$ In case $C$ is infinite, we can find an antichain $B$ that has infinite
intersection with every element of $C.$ Now assume that $C$ is finite, by
shrinking $A$ we may assume $C$ is the empty set. In this way, for every
$A_{n}$ there is $m_{n}$ such that $s^{\frown}m_{n}$ watches $A_{n}.$ We can
then find an infinite set $X\in\left[  \omega\right]  ^{\omega}$ such that
$m_{n}\neq m_{r}$ whenever $n\neq r$ and $n,r\in X$ (recall that $s$ is
maximal). Then $B=%
{\textstyle\bigcup\limits_{n\in X}}
\left[  s^{\frown}m_{n}\right]  \cap A_{n}$ is the set we were looking for.

\qquad\ \ \ \qquad\ \ 

Now we may assume $T$ does not have maximal nodes. If $T$ contains a Sacks
tree then we can find an infinite antichain $Y\subseteq T$. For every $s\in Y$
we choose $n_{s}$ such that $s$ watches $A_{n_{s}}$ and if $s\neq t$ then
$A_{n_{s}}\neq A_{n_{t}}.$ Then $B=%
{\textstyle\bigcup\limits_{s\in Y}}
\left[  s\right]  \cap A_{n_{s}}$ is the set we were looking for.

\qquad\ \ 

The only case left is that there is $s\in T$ that does not split in $T$ nor is
maximal. Let $f\in\left[  T\right]  $ the only branch that extends $s.$ We may
assume $s$ watches every element of $A$ and every $A_{n}$ is disjoint from
$\widehat{f}$ (this is because $A_{n}$ is an antichain and $f$ is a branch).
We say $A_{n}$ \emph{is a comb with }$f$ if $\Delta(A_{n}\cap\left[  s\right]
,\widehat{f})$ is infinite. We may assume that either every element of $A$ is
a comb with $f$ or none is. In case all of them are combs we can easily find
the desired antichain. So assume none of them are combs. In this way, for
every $n\in\omega$ we can find $t_{n}$ extending $s$ but incompatible with $f$
of minimal length such that $t_{n}$ watches $A_{n}.$ Since $t_{n}\notin T$ we
can find $W\in\left[  \omega\right]  ^{\omega}$ such that $t_{n}\neq t_{m}$
for all $n,m\in W$ where $n\neq m.$ Then we recursively construct the desired antichain.
\end{proof}

\qquad\ \ \ \ \ \ \ \ \ \ \ \ \ \ \qquad\ \ \ \ \ \ \ \ \ \ \ \ \ \ \ \ \ \ \ \ \ 

We can then conclude the following:\ \ \ \ \ \ \ \ \ \ \ \ \ \ \ \qquad\ \ \ \ \ \ \ \ \ \ 

\begin{proposition}
[$\mathsf{CH}$]There is a weakly tight \textsf{MAD} family that is not $+$-Ramsey.
\end{proposition}

\begin{proof}
Let $\left\{  \overline{X}_{\alpha}\mid\omega\leq\alpha<\omega_{1}\right\}  $
enumerate all countable sequences of infinite pairwise disjoint subsets of
$\omega^{<\omega}.$ Let $\mathcal{BR}=\{\widehat{f}\mid f\in\omega^{\omega
}\},$ we construct $\mathcal{A}=\left\{  A_{\alpha}\mid\alpha<\omega
_{1}\right\}  $ such that the following holds:

\begin{enumerate}
\item $\mathcal{A\cup BR}$ is an \textsf{AD} family.

\item If $s\in\omega^{<\omega}$ then $\mathcal{A}$ contains a partition of
$suc\left(  s\right)  =\left\{  s^{\frown}n\mid n\in\omega\right\}  .$

\item For every $\omega\leq\alpha<\omega_{1}$ if $\overline{X}_{\alpha
}=\left\{  X_{n}\mid n\in\omega\right\}  \subseteq\left(  \mathcal{A}_{\alpha
}\mathcal{\cup BR}\right)  ^{+}$\emph{ }then $A_{\alpha}\cap X_{n}$ is
infinite for infinitely many $n\in\omega$ (where $\mathcal{A}_{\alpha
}=\left\{  A_{\xi}\mid\xi<\alpha\right\}  $).
\end{enumerate}

\qquad\ \ \ \ \ \ \ \ 

At step $\alpha=\left\{  \alpha_{n}\mid n\in\omega\right\}  $ assume
$\overline{X}_{\alpha}=\left\{  X_{n}\mid n\in\omega\right\}  \subseteq\left(
\mathcal{A}_{\alpha}\mathcal{\cup BR}\right)  ^{+}.$ Since each $X_{n}$ can
not be covered with a finite number of branches, we may assume every $X_{n}$
is an infinite antichain. Let $Y_{n}=X_{n}\setminus\left(  A_{\alpha_{0}}%
\cup...A_{\alpha_{n}}\right)  $ which is an infinite antichain. By the lemma
we can find an antichain $B$ and $W\in\left[  \omega\right]  ^{\omega}$ such
that if $n\in W$ then$\ \ B\cap Y_{n}$ is infinite. Let $A_{\alpha}=%
{\textstyle\bigcup\limits_{n\in W}}
\left(  B\cap Y_{n}\right)  $ then $A_{\alpha}$ is \textsf{AD} with
$\mathcal{A}_{\alpha}$ and since it is an antichain it is also disjoint from
$\mathcal{BR}.$ Clearly $\mathcal{A\cup BR}$ is not $+$-Ramsey (recall that
weakly tight families are maximal).
\end{proof}

\subsubsection{Weakly tight does not imply Sacks indestructibility}

First we need the following definition:

\begin{definition}
We say an ideal $\mathcal{I}$ is \emph{weakly }$\omega$\emph{-hitting }if for
every countable family $\left\{  X_{n}\mid n\in\omega\right\}  \subseteq
\left[  \omega\right]  ^{\omega}$ there is $A\in\mathcal{I}$ such that $A\cap
X_{n}$ is infinite for infinitely many $n\in\omega.$
\end{definition}

\qquad\ \ \ \qquad\ \ \ \qquad\ \ \qquad\ \ 

We have to prove the following result:

\begin{proposition}
[$\mathsf{CH}$]If $\mathcal{I}$ is weakly $\omega$-hitting then there is a
weakly tight \textsf{MAD} family contained in $\mathcal{A}.$
\end{proposition}

\begin{proof}
Let $\left\{  X_{\alpha}\mid\alpha<\omega_{1}\right\}  $ enumerate all
countable family of $\omega$ and $X_{\alpha}=\left\{  X_{\alpha}\left(
n\right)  \mid n\in\omega\right\}  .$ We will build $\mathcal{A}=\left\{
A_{\alpha}\mid\alpha<\omega_{1}\right\}  \subseteq\mathcal{I}$ such that for
every $\alpha<\omega_{1}$ if $X_{\alpha}\subseteq\mathcal{I}\left(
\mathcal{A}_{\alpha}\right)  ^{+}$ then $A_{\alpha}\cap X_{\alpha}\left(
n\right)  $ is infinite for infinitely many $n\in\omega.$

\qquad\ \ 

Assume we are at stage $\alpha$ and $X_{\alpha}\subseteq\mathcal{I}\left(
A_{\alpha}\right)  ^{+}.$ First for each $n\in\omega$ we find $Y_{n}\in\left[
X_{\alpha}\left(  n\right)  \right]  ^{\omega}$ such that $Y_{n}$ is
\textsf{AD} with $\mathcal{A}_{\alpha}$ and $Y_{n}\cap\left(  A_{\alpha_{0}%
}\cup...\cup A_{\alpha_{n}}\right)  =\emptyset.$ Since $\mathcal{I}$ is weakly
$\omega$-hitting we can find $B\in\mathcal{I}$ and $W\in\left[  \omega\right]
^{\omega}$ such that $B\cap Y_{n}$ is infinite for every $n\in W,$ we may even
assume $B=%
{\textstyle\bigcup\limits_{n\in\omega}}
\left(  B\cap Y_{n}\right)  .$ We then just define $A_{\alpha}=B.$
\end{proof}

\qquad\ \ \ \ \ \ 

We have the following:

\begin{lemma}
$tr($\textsf{ctble}$)$ is weakly $\omega$-hitting.
\end{lemma}

\begin{proof}
Let $\left\{  X_{n}\mid n\in\omega\right\}  \subseteq\left[  2^{<\omega
}\right]  ^{\omega}$ and we may assume each $X_{n}$ is either a chain or an
antichain. In case they are all antichains then the result follows by a
previous result (there is an infinite antichain having infinite intersection
with infinitely many of them). We only need to consider the case where all of
the $X_{n}$ are chains, so there are $r_{n}\in2^{\omega}$ such that
$X_{n}\subseteq\widehat{r}_{n}.$ In case there is $r\in2^{\omega}$ such that
$\left\{  n\mid r_{n}=r\right\}  $ is infinite then we just take $\widehat
{r}_{n}.$ So assume $r_{n}\neq r_{m}$ whenever $n\neq m.$ Since $2^{\omega}$
is compact, we can find $r\in2^{\omega}$ and $W\in\left[  \omega\right]
^{\omega}$ such that the sequence \thinspace$\left\langle r_{n}\right\rangle
_{n\in W}$ converges to $r.$ Then $B=%
{\textstyle\bigcup\limits_{n\in W}}
X_{n}$ is an element of $tr($\textsf{ctble}$).$
\end{proof}

\qquad\ \ \qquad

We can then conclude the following:

\begin{corollary}
[$\mathsf{CH}$]There is a weakly tight family that is destructible by Sacks forcing.
\end{corollary}

\subsubsection{If $\mathbb{P}$ is $\omega^{\omega}$-bounding then $\mathbb{P}%
$-indestructibility does not imply $+$-Ramsey}

We will now prove that (in particular) Sacks or random indestructibility are
not enough to get $+$-Ramseyness. We will say a family $\mathcal{A}$ on
$\omega^{<\omega}$ is a \emph{standard }$\mathcal{K}_{\sigma}$\emph{ family
}if the following holds:

\begin{enumerate}
\item $\mathcal{A}$ is an \textsf{AD} family.

\item If $A\in\mathcal{A}$ either $\pi\left(  A\right)  =\emptyset$ or $A$ is
a finitely branching tree on $\omega^{<\omega}.$

\item If $s\in\omega^{<\omega}$ then $\left\{  s^{\frown}n\mid n\in
\omega\right\}  \in\mathcal{I}\left(  \mathcal{A}\right)  ^{++}.$
\end{enumerate}

\qquad\ \ \ 

We now need the following lemma:

\begin{lemma}
Let $\mathbb{P}$ be an $\omega^{\omega}$-bounding forcing and $\mathcal{A}$ a
countable standard $\mathcal{K}_{\sigma}$ family. If $p\in\mathbb{P}$ and
$\dot{b}$ is a $\mathbb{P}$-name for an infinite subset of $\omega^{<\omega}$
such that $p\Vdash``\dot{b}\in\mathcal{A}^{\perp}\textquotedblright$ then
there are $q\leq p$ and $\mathcal{B}$ a countable standard $\mathcal{K}%
_{\sigma}$ family such that $\mathcal{A\subseteq B}$ and $q\Vdash``\dot
{b}\notin\mathcal{B}^{\perp}\textquotedblright.$
\end{lemma}

\begin{proof}
Let $\mathcal{A}=\left\{  T_{n}\mid n\in\omega\right\}  \cup\left\{  a_{n}\mid
n\in\omega\right\}  $ where $T_{n}$ is a finitely branching tree and
$\pi\left(  a_{n}\right)  =\emptyset$ for every $n\in\omega.$ We may assume
that $p$ forces that $\pi(\dot{b})$ is either empty or a singleton. We first
assume there is $\dot{r}$ such that $p\Vdash``\pi(\dot{b})=\left\{  \dot
{r}\right\}  \textquotedblright.$ Since $\mathbb{P}$ is $\omega^{\omega}%
$-bounding, we may find $p_{1}\leq p$ and $T\in V$ a finitely branching well
pruned tree such that $p_{1}\Vdash``\dot{r}\in\left[  T\right]
\textquotedblright.$ Once again, since $\mathbb{P}$ is $\omega^{\omega}%
$-bounding we may find $p_{2}\leq p_{1}$ and $f\in\omega^{\omega}$ such that
the following holds:

\begin{enumerate}
\item $f$ is an increasing function.

\item $p_{2}\Vdash``\left(  T_{n}\cup a_{n}\right)  \cap\widehat{r}%
\subseteq\omega^{<f\left(  n\right)  }\textquotedblright.$
\end{enumerate}

\qquad\ \ \ 

Define $K=\left(  T\cap\omega^{\leq f\left(  0\right)  }\right)  \cup%
{\textstyle\bigcup}
\left(  T\cap\omega^{\leq f\left(  n+1\right)  }\setminus\left(  T_{0}%
\cup...T_{n}\cup a_{0}\cup...a_{n}\right)  \right)  .$ It is easy to see that
$K$ is a finitely branching tree, $p_{2}\Vdash``\dot{r}\in\left[  K\right]
\textquotedblright$ and $K\in\mathcal{A}^{\perp}.$ We now simply define
$\mathcal{B}=\mathcal{A\cup}\left\{  K\right\}  .$

\qquad\ \ 

Now we consider the case where $\pi(\dot{b})$ is forced to be empty. Let
$\dot{S}$ be the tree of all $s\in\omega^{<\omega}$ such that $s$ has
infinitely many extensions in $\dot{b}.$ We will first assume there are
$p_{1}\leq p$ and $s$ such that $p_{1}$ forces that $s$ is a maximal node of
$\dot{S}.$ Since $\mathbb{P}$ is $\omega^{\omega}$-bounding, we can find a
ground model interval partition $\mathcal{P}=\left\{  P_{n}\mid n\in
\omega\right\}  $ and $p_{2}\leq p_{1}$ such that if $n\in\omega$ then $p_{2}$
forces that there is $\dot{m}_{n}\in P_{n}$ such that $(\left[  s^{\frown}%
\dot{m}_{n}\right]  \cap\dot{b})\setminus\left(  T_{0}\cup...T_{n}\cup
a_{0}\cup...\cup a_{n}\right)  \neq\emptyset.$ Given $n,m\in\omega$ we define
$K_{n,m}=\left\{  s^{\frown}i^{\frown}t\mid i\in P_{n}\wedge t\in
m^{m}\right\}  .$ Using once again that $\mathbb{P}$ is $\omega^{\omega}%
$-bounding, we may find $p_{3}\leq p_{2}$ and an increasing function
$f:\omega\longrightarrow\omega$ such that if $n\in\omega$ then $p_{3}$ forces
$(K_{n,f\left(  n\right)  }\cap\dot{b})\setminus\left(  T_{0}\cup...T_{n}\cup
a_{0}\cup...\cup a_{n}\right)  $ is non-empty for every $n\in\omega.$ We now
define $a=%
{\textstyle\bigcup\limits_{n\in\omega}}
K_{n,f\left(  n\right)  }\setminus\left(  T_{0}\cup...T_{n}\cup a_{0}%
\cup...\cup a_{n}\right)  .$ It is easy to see that $\pi\left(  a\right)
=\emptyset,$ $a\in\mathcal{A}^{\perp}$ and $p_{3}$ forces that $a$ and
$\dot{b}$ have infinite intersection.

\qquad\ \ 

Now we assume that $p$ forces that $\dot{S}$ does not have maximal nodes, let
$\dot{r}$ be a name for a branch of $\dot{S}.$ First assume that $\dot{r}$ is
forced to be a branch through some element of $\mathcal{A}.$ We may assume
that $p\Vdash``\dot{r}\in\left[  T_{0}\right]  \textquotedblright.$ Since
$\mathbb{P}$ is $\omega^{\omega}$-bounding, we may find $p_{1}\leq p$ and an
increasing ground model function $f:\omega\longrightarrow\omega$ such that if
$n\in\omega$ then $p_{1}$ forces that all extentions of $\dot{r}%
\upharpoonright f\left(  n\right)  $ to $\dot{b}$ are not in $T_{0}%
\cup...T_{n}\cup a_{0}\cup...\cup a_{n}.$ Once again, we may find $p_{2}\leq
p_{1}$ and $g:\omega\longrightarrow\omega$ such that if $n\in\omega$ then
$\dot{b}$ has non empty intersection with the set $\{\dot{r}\upharpoonright
f\left(  n\right)  ^{\frown}t\mid t\in g\left(  n\right)  ^{g\left(  n\right)
}\}\setminus\left(  T_{0}\cup...T_{n}\cup a_{0}\cup...\cup a_{n}\right)  .$ We
now define $a=%
{\textstyle\bigcup\limits_{s\in\left(  T_{0}\right)  _{f\left(  n\right)  }}}
(\{s^{\frown}t\mid t\in g\left(  n\right)  ^{g\left(  n\right)  }%
\}\setminus\left(  T_{0}\cup...T_{n}\cup a_{0}\cup...\cup a_{n}\right)  ).$ It
is easy to see that $a$ has the desired properties.

\qquad\ \ \ \ \ \ \ 

Finally, in case that $\dot{r}$ is not forced to be a branch through some
element of $\mathcal{A},$ we find a finitely branching tree $T\in
\mathcal{A}^{\perp}$ such that $p\Vdash``\dot{r}\in\left[  T\right]
\textquotedblright$ as we did at the beginning of the proof. If $T$ has
infinite intersection with $\dot{b}$ we are done and if not then we apply the
previous case.
\end{proof}

\qquad\ \ \ \ 

With a standard bookkeeping argument we can then conclude the following:

\begin{proposition}
[$\mathsf{CH}$]If $\mathbb{P}$ is a proper $\omega^{\omega}$-bounding forcing
of size $\omega_{1}$ then there is a \textsf{MAD} family $\mathcal{A}$ that is
$\mathbb{P}$ indestructible but is not $+$-Ramsey.
\end{proposition}

\subsubsection{\textsf{FIN}$\times$\textsf{FIN}-like does not imply tightness}

\qquad\ \ \ We will now prove te following:\qquad\ \ \ \ \ 

\begin{proposition}
[$\mathsf{CH}$]There is a \textsf{FIN}$\times$\textsf{FIN}-like \textsf{MAD}
family that is not tight.
\end{proposition}

\begin{proof}
Let $\left\{  \mathcal{I}_{\alpha}\mid\omega\leq\alpha<\omega_{1}\right\}  $
be an enumeration of all analytic ideals that are not Shelah-Stepr\={a}ns and
$X=\left\{  X_{n}\mid n\in\omega\right\}  $ be a partition of $\omega$ into
infinite sets. We will recursively construct an \textsf{AD} family
$\mathcal{A}=\left\{  A_{\alpha}\mid\alpha<\omega_{1}\right\}  $ such that for
every $\alpha$ the following conditions hold:\qquad\ \ \ 

\begin{enumerate}
\item $\left\{  A_{n}\mid n\in\omega\right\}  $ is a partition of $\omega$
refining $\left\{  X_{n}\mid n\in\omega\right\}  $ and every $X_{n}$ contains
infinitely many of the $A_{m}.$

\item There is $\xi\leq\alpha$ such that $A_{\xi}\notin\mathcal{I}_{\alpha}.$

\item If $B\in\mathcal{I}\left(  \mathcal{A}\right)  $ then there is
$n\in\omega$ such that $B\cap X_{n}$ is finite.
\end{enumerate}

\qquad\ \ \ 

Let $\mathcal{A}_{\alpha}=\left\{  A_{\xi}\mid\xi<\alpha\right\}  $ and assume
$\mathcal{A}_{\alpha}\subseteq\mathcal{I}_{\alpha}.$ Let $\alpha=\left\{
\alpha_{n}\mid n\in\omega\right\}  $ and define $L_{n}=A_{\alpha_{0}}%
\cup...\cup A_{\alpha_{n}}.$ Define $E_{n}=\left\{  m\mid\left\vert L_{n}\cap
X_{m}\right\vert <\omega\right\}  $ and note that $E=\left\langle
E_{n}\right\rangle _{n\in\omega}\,\ $is a decreasing sequence of infinite
sets. Find a pseudointersection $D$ of $E$ such that $\omega\setminus D$ also
contains a pseudointersection of $E.$ Define $T_{0}=%
{\textstyle\bigcup\limits_{n\in D}}
X_{n}$ and $T_{1}=%
{\textstyle\bigcup\limits_{n\notin D}}
X_{n}.$ Since \textsf{FIN}$\times$\textsf{FIN} $\mathcal{\nleq}_{K}$
$\mathcal{I}_{\alpha}$ we know that either \textsf{FIN}$\times$\textsf{FIN}
$\mathcal{\nleq}_{K}$ $\mathcal{I}_{\alpha}\upharpoonright T_{0}$ or
\textsf{FIN}$\times$\textsf{FIN} $\mathcal{\nleq}_{K}$ $\mathcal{I}_{\alpha
}\upharpoonright T_{1}.$ First assume \textsf{FIN}$\times$\textsf{FIN}
$\mathcal{\nleq}_{K}$ $\mathcal{I}_{\alpha}\upharpoonright T_{0}$ so then we
choose any $A_{\alpha}\in\left(  \mathcal{I}_{\alpha}\upharpoonright
T_{0}\right)  ^{+}$ that is almost disjoint with $\mathcal{A}_{\alpha
}\mathcal{\upharpoonright}T_{0}$ which implies it is \textsf{AD} with
$\mathcal{A}_{\alpha}.$ We now need to prove that for every $n<\omega$ there
is $X_{m}\ \ $such that$\ \left(  L_{n}\cup A_{\alpha}\right)  \cap X_{m}$ is
finite. Since $\omega\setminus D$ contains a pseudointersection of $E$, there
is $m\in E_{n}\setminus D$ and then both $L_{n}$ and $A_{\alpha}$ are almost
disjoint with $X_{m}.$ The other case is similar.
\end{proof}

\qquad\ \ \ \qquad\ \ 

\subsubsection{Strong tightness does not imply Laflamme or random
indestructible}

\qquad\ \ \ 

We start with the following proposition:

\begin{lemma}
Let $\mathcal{A}$ be a countable \textsf{AD} family contained in the summable
ideal. Let $X=\left\{  X_{n}\mid n\in\omega\right\}  \subseteq\left[
\omega\right]  ^{\omega}$ such that if $n\in\omega$ then $X_{n}$ is contained
in some $B_{n}\in\mathcal{A}$ and $B_{m}\neq B_{n}$ for almost all $m\in
\omega.$ Then there is $D\in\mathcal{A}^{\perp}\cap\mathcal{J}_{1/n}$ such
that $D\cap X_{n}\neq\emptyset$ for every $n\in\omega.$
\end{lemma}

\begin{proof}
Let $\mathcal{A}=\left\{  A_{n}\mid n\in\omega\right\}  $ and for each
$n\in\omega$ we define $F_{n}=\left\{  X_{i}\mid X_{i}\subseteq A_{n}\right\}
.$ We construct a sequence of finite sets $\left\{  s_{n}\mid n\in
\omega\right\}  \subseteq\left[  \omega\right]  ^{<\omega}$ such that:\ 

\begin{enumerate}
\item $max\left(  s_{n}\right)  <min\left(  s_{n+1}\right)  .$

\item $%
{\textstyle\sum\limits_{i\in s_{n}}}
\frac{1}{1+i}<\frac{1}{2^{n+1}}.$

\item $s_{n}$ has non empty intersection with every element of $F_{n}.$

\item If $m<n$ then $s_{n}$ is disjoint from $A_{m}.$
\end{enumerate}

\qquad\ \ 

Assuming we are at step $n,$ let $r$ such that $F_{n}=\left\{  X_{n_{1}%
},...,X_{n_{r}}\right\}  .$ Find $m$ such that $\frac{r}{1+m}<\frac{1}%
{2^{n+1}}$ and $s_{i}\subseteq m$ for every $i<n.$ For every $1\leq i\leq r$
we choose $k_{i}>m$ such that $k_{i}\in X_{n_{i}}\setminus%
{\textstyle\bigcup\limits_{j<n}}
A_{j}$ and let $s_{n}=\left\{  k_{i}\mid1\leq i\leq r\right\}  .$ It is easy
to see that $D=%
{\textstyle\bigcup\limits_{n\in\omega}}
s_{n}$ has the desired properties.
\end{proof}

\qquad\ \ \ 

In \cite{CardinalInvariantsofAnalyticPIdeals} it is proved that random forcing
destroys the summable ideal. We can then conclude the following:

\begin{proposition}
[$\mathsf{CH}$]There is a strongly tight family contained in the summable
ideal $\mathcal{J}_{1/n}$ (in particular, it is $\mathbb{B}$-destructible and
not Laflamme).
\end{proposition}

\chapter{Destroying $P$-points with Silver reals}

\qquad\ \ 

Recall that an ultrafilter $\mathcal{U}$ is called a \emph{P-point }if every
countable subfamily of $\mathcal{U}$ has a pseudointersection in
$\mathcal{U}.$ This special kind of ultrafilters has been extensively studied
by set theorists and topologists. It is possible to construct such
ultrafilters under $\mathfrak{d=c}$ (see \cite{HandbookBlass}) or if the
parametrized diamond $\Diamond\left(  \mathfrak{r}\right)  $ holds (see
\cite{ParametrizedDiamonds}). \footnote{In fact, $P$-points can also be
constructed assuming the existence of a \textquotedblleft$\mathfrak{d}%
$-pathway\textquotedblright\ (which generalizes the construction under
$\mathfrak{d=c}$). Pathways are interesting combinatorial structures, but
since we are not going to use them in this thesis, we will avoid defining
them.} On the other hand, it is a remarkable theorem of Shelah that the
existence of $P$-points can not be proved using only the axioms of
$\mathsf{ZFC}$ (see \cite{Barty}). The model of Shelah is obtained by
iterating the Grigorieff forcing of non-meager $P$-filters.

\qquad\ \ 

By a \textquotedblleft canonical model\textquotedblright\ we mean a model
obtained after performing a countable support iteration of Borel proper
partial orders of length $\omega_{2}$. In the \textquotedblleft Forcing and
its applications retrospective workshop\textquotedblright\ held at the Fields
institute, Michael Hru\v{s}\'{a}k asked the following:

\begin{problem}
Are there $P$-points in every canonical model?
\end{problem}

\qquad\ \ \ \ 

There will be a $P$-point in case the iteration adds unbounded reals or does
not add splitting reals (since at the end we will get a model of either
$\mathfrak{d=c}$ or $\Diamond\left(  \mathfrak{r}\right)  $). Therefore, we
only need to consider the Borel $\omega^{\omega}$-bounding forcings that add
splitting reals. The best known examples of this type of forcings are the
random and Silver forcings. We will answer the question of Michael
Hru\v{s}\'{a}k by proving that there are no $P$-points in the Silver model.
The existence of a model without $P$-points with the continuum larger than
$\omega_{2}$ was also an open question \cite{WohofskyThesis} and we will also
show that the side-by-side product of Silver forcing produces sucha model. The
results of this chapter were obtained in collaboration with David Chodounsk\'{y}.

\qquad\ \ \ 

Regarding the random model, in \cite{PpointsinRandomUniverses} it was claimed
that there is a $P$-point in this model; unfortunately, I discovered the proof
presented there is incorrect. It seems that the existence of $P$-points in
that model is an open question.

\qquad\ \ 

We start by fixing some notation and remarks that will be used in the proof.
By $-_{n}$ and $=_{n}$ we denote the substraction operation and congruence
relation modulo $n.$ The notation $j\in_{n}X$ is interpreted as
\textquotedblleft there is $x\in X$ such that $j=_{n}x$\textquotedblright. For
$X,Y\subseteq n$ we define $X-_{n}Y=\left\{  x-_{n}y\mid x\in X\wedge y\in
Y\right\}  .$\qquad\ \ 

We will need the following lemma:

\begin{lemma}
For every $n\in\omega$ there is $k\left(  n\right)  \in\omega$ such that for
each $C\in\left[  k\left(  n\right)  \right]  ^{n}$ there is $s\in k\left(
n\right)  $ such that $C\cap\left(  C-_{k\left(  n\right)  }s\right)
=\emptyset.$
\end{lemma}

\begin{proof}
It is easy to see that if $k\left(  n\right)  >n^{2}$ then the result holds.
\end{proof}

\qquad\ \ 

We now recursively define two functions $v:\omega\longrightarrow\omega$ and
$m:\omega\longrightarrow\omega$ as follows: We define $v\left(  0\right)  =0$
and $m\left(  0\right)  =k\left(  2\right)  .$ If $v\left(  n-1\right)  $ and
$m\left(  n-1\right)  $ are already defined then let $v\left(  n\right)  =%
{\textstyle\sum\limits_{i<n}}
m\left(  i\right)  $ and $m\left(  n\right)  =k\left(  \left(  n+1\right)
\left(  v\left(  n\right)  +2\right)  \right)  .$ We will need the following definitions:

\begin{definition}
Let $r\in\left[  \omega\right]  ^{\omega},$ $f:\omega\longrightarrow\omega$ be
an increasing function and $\overline{X}=\left\langle X_{n}\mid n\in
\omega\right\rangle $ where $X_{n}\in\left[  m\left(  n\right)  \right]
^{n+1}.$

\begin{enumerate}
\item Let $\mathcal{P}\left(  r\right)  =\left\{  P_{l}\left(  r\right)  \mid
l\in\omega\right\}  $ be the interval partition where $P_{0}\left(  r\right)
=min\left(  r\right)  $ and $P_{n+1}\left(  r\right)  =(max\left(
P_{n}\left(  r\right)  \right)  ,min\left(  r\setminus P_{n}\right)  ].$

\item For every $n\in\omega$ and every $i<m\left(  n\right)  $ we define the
set $D_{i}^{m\left(  n\right)  }\left(  r\right)  =%
{\textstyle\bigcup}
\left\{  P_{j}\left(  r\right)  \mid j=_{m\left(  n\right)  }i\right\}  .$

\item For every $n\in\omega$ let $E_{n}\left(  r,\overline{X}\right)  =%
{\textstyle\bigcup\limits_{i\in X_{n}}}
D_{i}^{m\left(  n\right)  }\left(  r\right)  .$

\item We define $E\left(  r,\overline{X},f\right)  =%
{\textstyle\bigcap\limits_{n\in\omega}}
\left(  E_{n}\left(  r,\overline{X}\right)  \cup f\left(  n\right)  \right)  $
\end{enumerate}
\end{definition}

\qquad\ \ 

Note that $E\left(  r,\overline{X},f\right)  $ is a pseudointersection of
$\left\{  E_{n}\left(  r,\overline{X}\right)  \mid n\in\omega\right\}  ,$
furthermore, $E\left(  r,\overline{X},f\right)  \setminus f\left(  n\right)
\subseteq$ $E_{n}\left(  r,\overline{X}\right)  $ for each $n\in\omega.$

\section{Doughnuts and $P$-points}

Given $x,y\in\left[  \omega\right]  ^{\omega}$ such that $x\subseteq y,$
$y\setminus x$ is infinite, we define $\left[  x,y\right]  $ as the set
$\left\{  z\subseteq\omega\mid x\subseteq z\subseteq y\right\}  .$ These sets
are often referred to \emph{doughnuts. }We say $N\subseteq\left[
\omega\right]  ^{\omega}$ is \emph{doughnut-null }if for every doughnut $p$
there is a doughnut $p^{\prime}\subseteq p$ such that $p^{\prime}\cap
N=\emptyset.$ If $p=\left[  x,y\right]  $ is a doughnut, we define $cod\left(
p\right)  =y\setminus x.$ The ideal $v_{0}$ is defined as the ideal generated
by all doughnut-null sets. Doughnuts were first defined by Di Prisco and Henle
(see \cite{DoughnutsFloatingOrdinalsSquareBracketsandUltraflitters}). The main
result is the following:

\begin{proposition}
Let $\mathcal{U}$ be a (non principal ultrafilter), $f:\omega\longrightarrow
\omega$ and $\overline{X}=\left\langle X_{n}\mid n\in\omega\right\rangle $
where $X_{n}\in\left[  m\left(  n\right)  \right]  ^{n+1}.$ Then the set
$N\left(  \mathcal{U},f,X\right)  =\{r\subseteq\omega\mid\forall
A\in\mathcal{U}\left(  A\cap E\left(  r,f,\overline{X}\right)  \neq
\emptyset\right)  \}$ is doughnut null.
\end{proposition}

\begin{proof}
Letting $p$ be a doughnut, we must show that $p$ can be shrunken in order to
avoid $N\left(  \mathcal{U},f,X\right)  .$ We now choose an interval partition
$\left\{  A_{n}\mid n\in\left\{  -1\right\}  \cup\omega\right\}  $ such that
for each $n\in\omega$ the following holds:

\begin{enumerate}
\item $f\left(  n\right)  <min\left(  A_{2n}\right)  $.

\item $m\left(  n\right)  <\left\vert A_{2n+j}\cap cod\left(  p\right)
\right\vert $ for each $j<2.$
\end{enumerate}

\qquad\ \ 

We can assume that $U_{0}=%
{\textstyle\bigcup}
\left\{  A_{2n+1}\mid n\in\omega\right\}  \in\mathcal{U},$ if this was not the
case, we take the interval partition $\left\langle A_{-1}\cup A_{0}%
,A_{1},A_{2},...\right\rangle $ instead (this is possible since $\mathcal{U}$
is not principal). Let $p_{1}=\left[  a_{1},b_{1}\right]  \subseteq p$ be a
doughnut such that $cod\left(  p_{1}\right)  \cap A_{2n-1}=\emptyset$ and
$\left\vert cod\left(  p_{1}\right)  \cap A_{2n}\right\vert =m\left(
n\right)  $ for each $n\in\omega.$ Note that $\left\vert cod\left(
p_{1}\right)  \cap min\left(  A_{2n}\right)  \right\vert =v\left(  n\right)
.$ We now define $C_{n}$ as the following set: $\left\{  j\in m\left(
n\right)  \mid j\in_{m\left(  n\right)  }\left(  X_{n}-_{m\left(  n\right)
}\left\{  i\mid i\in v\left(  n\right)  +2\right\}  \right)  \right\}  $ and
note that $\left\vert C_{n}\right\vert \,\ $is at most $\left(  n+1\right)
\left(  v\left(  n\right)  +2\right)  .$ Let $H_{n}=A_{2n+1}\cap%
{\textstyle\bigcup}
\left\{  P_{j}\left(  a_{1}\right)  \mid j\in\omega\wedge j\in_{m\left(
n\right)  }C_{n}\right\}  $ for every $n\in\omega.$

\qquad\ \ 

We will now distinguish two cases: first assume that
${\textstyle\bigcup}$
$\left\{  H_{n}\mid n\in\omega\right\}  \notin$ $\mathcal{U}.$ Therefore, $U=%
{\textstyle\bigcup}
\left\{  A_{2n+1}\setminus H_{n}\mid n\in\omega\right\}  \in\mathcal{U}.$ Let
$p_{2}=\left[  a_{2},b_{2}\right]  \subseteq p_{1}$ be a doughnut such that
$a_{2}=a_{1}$ and $\left\vert cod\left(  p_{2}\right)  \cap A_{2n}\right\vert
=1$ for each $n\in\omega.$ Note that if $r\in p_{2}$ and $i\in X_{n}$ then
$D_{i}^{m\left(  n\right)  }\left(  r\right)  \cap A_{2n+1}\subseteq H_{n}.$
Thus, $E_{n}\left(  r,\overline{X}\right)  \cap A_{2n+1}\subseteq H_{n}$ and
since $E\left(  r,\overline{X},f\right)  \setminus min\left(  A_{2n}\right)
\subseteq E_{n}\left(  r,\overline{X}\right)  $ we conclude that $E\left(
r,\overline{X},f\right)  \cap U=\emptyset,$ so $p_{2}$ is the doughnut we were
looking for.

\qquad\ \ 

We now consider the case where $U=$%
${\textstyle\bigcup}$
$\left\{  H_{n}\mid n\in\omega\right\}$ $\in$ $U$. By the previous
lemma, for each $n\in\omega$ we can find $s_{n}\in m\left(  n\right)  $ such
that $C_{n}\cap\left(  C_{n}-_{m\left(  n\right)  }\left\{  s_{n}\right\}
\right)  =\emptyset.$ We then choose a doughnut $p_{2}=\left[  a_{2}%
,b_{2}\right]  \subseteq p_{1}$ such that $\left\vert cod\left(  p_{2}\right)
\cap A_{2n}\right\vert =1$ and $\left\vert a_{2}\cap A_{2n}\right\vert
=\left\vert a_{1}\cap A_{2n}\right\vert +s_{n}$ for each $n\in\omega$ (such
$a_{2}$ exists since $cod\left(  p_{1}\right)  \cap A_{2n}$ has size $m\left(
n\right)  $). For every $n\in\omega$ we define $\overline{H}_{n}$ as the set
$A_{2n+1}\cap%
{\textstyle\bigcup}
\left\{  P_{j}\left(  a_{1}\right)  \mid j\in\omega\wedge j\in_{m\left(
n\right)  }\left(  C_{n}-_{m\left(  n\right)  }\left\{  s_{n}\right\}
\right)  \right\}  .$ It is not hard to see that $\overline{H}_{n}%
=A_{2n+1}\cap%
{\textstyle\bigcup}
\left\{  P_{j}\left(  a_{2}\right)  \mid j\in\omega\wedge j\in_{m\left(
n\right)  }C_{n}\right\}  .$ Notice that $H_{n}\cap\overline{H}_{n}%
=\emptyset.$ Now, if $r\in p_{2}$ and $i\in X_{n}$ then $D_{i}^{m\left(
n\right)  }\left(  r\right)  \cap A_{2n+1}\subseteq\overline{H}_{n}$ which as
before implies that $E\left(  r,\overline{X},f\right)  \cap U=\emptyset.$
\end{proof}

\qquad\ \ 

We can then prove the following result:

\begin{proposition}
The inequality \textsf{cof}$\left(  \mathcal{N}\right)  <$ \textsf{cov}%
$\left(  v_{0}\right)  $ implies that there are no $P$-points.
\end{proposition}

\begin{proof}
Let $\mathcal{U}$ be a non principal ultrafilter, we will show that
$\mathcal{U}$ is not a $P$-point. Let $\mathcal{S}=\{\overline{X}_{\alpha}%
\mid\alpha\in$ \textsf{cof}$\left(  \mathcal{N}\right)  \}$ be a family with
the following properties:

\begin{enumerate}
\item $\overline{X}_{\alpha}=\left\langle X_{n}^{\alpha}\mid n\in
\omega\right\rangle $ where $X_{n}^{\alpha}\in\left[  m\left(  n\right)
\right]  ^{n+1}$ for $n\in\omega$ and $\alpha<$ \textsf{cof}$\left(
\mathcal{N}\right)  .$

\item For every $h:\omega\longrightarrow\omega$ such that $h\left(  n\right)
$ $<m\left(  n\right)  $ for every $n\in\omega,$ there is $\alpha<$
\textsf{cof}$\left(  \mathcal{N}\right)  $ such that $h\left(  n\right)  \in
X_{n}^{\alpha}$ for every $n\in\omega.$
\end{enumerate}

\qquad\ 

Let $\{f_{\beta}\mid\beta\in$\textsf{ cof}$\left(  \mathcal{N}\right)
\}\subseteq\omega^{\omega}$ be a $\leq$-dominating family of functions. We
know that $\mathcal{B}=\{N\left(  \mathcal{U},f_{\beta},\overline{X}_{\alpha
}\right)  \mid\alpha,\beta<$\textsf{ cof}$\left(  \mathcal{N}\right)  \}$ is a
family of doughnut null sets. Since \textsf{cof}$\left(  \mathcal{N}\right)
<$ \textsf{cov}$\left(  v_{0}\right)  $ we can find $r\notin N\left(
\mathcal{U},f_{\beta},\overline{X}_{\alpha}\right)  $ for every $\alpha$ and
$\beta$ smaller than \textsf{cof}$\left(  \mathcal{N}\right)  .$ We can then
find $h:\omega\longrightarrow\omega$ such that $D_{h\left(  n\right)
}^{m\left(  n\right)  }\left(  r\right)  \in\mathcal{U}$ for every $n\in
\omega$ (note that $h\left(  n\right)  $ $<m\left(  n\right)  $). We will now
prove that $\{D_{h\left(  n\right)  }^{m\left(  n\right)  }\left(  r\right)
\mid n\in\omega\}$ has no pseudointersection in $\mathcal{U}.$ Let $Y$ be a
pseudointersection. We first find $\alpha<$ \textsf{cof}$\left(
\mathcal{N}\right)  $ such that $h\left(  n\right)  \in X_{n}^{\alpha}$ for
every $n\in\omega$ and then we find $\beta<$ \textsf{cof}$\left(
\mathcal{N}\right)  $ such that $Y\setminus f_{\beta}\left(  n\right)
\subseteq D_{h\left(  n\right)  }^{m\left(  n\right)  }\left(  r\right)  $ for
every $n\in\omega.$ Since $r\notin N\left(  \mathcal{U},f_{\beta},\overline
{X}_{\alpha}\right)  $ there is $A\in\mathcal{U}$ such that $A\cap E\left(
r,\overline{X}_{\alpha},f_{\beta}\right)  =\emptyset.$ Since $Y\subseteq
^{\ast}E\left(  r,\overline{X}_{\alpha},f_{\beta}\right)  $ the result follows.
\end{proof}

\section{There are no $P$-points in the Silver model\qquad\qquad\ \ }

The \emph{Silver forcing }(also known as \emph{Silver-Prikry forcing})\emph{
}consists of all partial functions $p\subseteq\omega\times2$ such that
$\omega\setminus dom\left(  p\right)  $ is infinite. We say that\emph{ }$p\leq
q$ in case $q\subseteq p.$ We will denote Silver forcing by $\mathbb{PS}.$
Note that the set of all conditions $p\in\mathbb{PS}$ such that $p^{-1}\left(
1\right)  $ is infinite forms an open dense set, so we will assume all
conditions have this property. It is well known that Silver forcing is proper.
By the \emph{Silver model }we refer to the model obtained by iterating Silver
forcing $\omega_{2}$ times over a model of the Continuum Hypothesis. The
following results is well known:

\begin{proposition}
The equality \textsf{cov}$\left(  v_{0}\right)  =\mathfrak{c}$ holds in the
Silver model.
\end{proposition}

\begin{proof}
Let $G\subseteq\mathbb{PS}_{\omega_{2}}$ be a generic filter, we will prove
that $V\left[  G\right]  \models$ \textsf{cov}$\left(  v_{0}\right)
=\omega_{2}.$ Let $\mathcal{B}=\left\{  N_{\beta}\mid\beta\in\omega
_{1}\right\}  $ be a family of doughnut null sets. By a reflection argument,
we can find $\alpha<\omega_{2}$ such that for every doughnut $d\in V\left[
G_{\alpha}\right]  $ and for every $\beta<\omega_{1}$ there is a subdoughnut
$d_{1}\in V\left[  G_{\alpha}\right]  $ such that $d_{1}\subseteq d$ and
$d_{1}\cap N_{\beta}=\emptyset.$ It is easy to see that the next generic real
will avoid all of the $N_{\beta}.$
\end{proof}

\qquad\ \ 

It is well known that \textsf{cof}$\left(  \mathcal{N}\right)  =\omega_{1}$
holds in the Silver model (see \cite{Combinatorial}).

\begin{corollary}
There are no $P$-points in the Silver model.
\end{corollary}

\qquad\ \ \ 

It is possible to modify the previous argumentin order to construct models
with no $P$-points where the continuum is arbitrarily large:

\begin{proposition}
Assume $V$ is a model of $\mathsf{GCH}$ and $\kappa>\omega_{1}$ is a regular
cardinal. If $\otimes_{\kappa}\mathbb{PS}$ is the countable support product of
$\kappa$ many Silver forcings and $G\subseteq\otimes_{\kappa}\mathbb{PS}$ is a
generic filter, then $V\left[  G\right]  \models``$There are no $P$-points and
$\mathfrak{c}=\kappa\textquotedblright.$\qquad\ \ \ 
\end{proposition}

\qquad\ \ \ \ 

The previous proposition does not seem to follow formally from the previous
results (it is not clear that \textsf{cov}$\left(  v_{0}\right)  $ is bigger
than $\omega_{1}$ after adding many Silver reals by the countable support
product). Nevertheless, the proof is almost the same as before. The reader may
consult \cite{NoPpointsintheSilverModel} for more details.

\chapter{Open problems on \textsf{MAD} families}

In this small chapter, we gather some important open problems regarding
\textsf{MAD} families. This is not supposed to be an exhaustive list, it just
reflects the personal interests of the author. Perhaps the most famous problem
is the following:

\begin{problem}
[Roitman]Does $\mathfrak{d}=\omega_{1}$ imply $\mathfrak{a}=\omega_{1}$?
\end{problem}

\qquad\ \ \ 

It is a result of Shelah that the inequality $\mathfrak{d<a}$ is consistent
(see \cite{ShelahTemplates} and \cite{BrendleTemplates}). Furthermore, it is
known that $\mathfrak{a}=\omega_{1}$ follows from the diamond principle
$\Diamond_{\mathfrak{d}}$ (see \cite{Diamantesubd}) which is slightly stronger
than $\mathfrak{d}=\omega_{1}.$ The previous question is essentially
equivalent to the following:

\begin{problem}
Can every \textsf{MAD} family be destroyed with a proper $\omega^{\omega}%
$-bounding forcing?
\end{problem}

\qquad\ \ \ 

As mentioned on the chapter of indestructibility, it is known that every
\textsf{MAD} family can be destroyed with a proper forcing that does not add
dominating reals. The following stronger version of Roitman's question is also open:

\begin{problem}
[Brendle \cite{MobandMad}]Does $\mathfrak{b}=\mathfrak{s=}$ $\omega_{1}$ imply
$\mathfrak{a}=\omega_{1}$?
\end{problem}

\qquad\ \ \ 

It is a well known result of Shelah that $\omega_{1}=\mathfrak{b<a}$ is
consistent (see \cite{ProperandImproper}). A natural attempt to try to solve
the problem of Brendle would be to show under $\mathsf{CH}$ that every
\textsf{MAD} family can be extended to a Hurewicz ideal. Unfortunately, in an
unpublish work, Raghavan showed that this might not be the case:

\begin{proposition}
[Raghavan]If $\Diamond\left(  S\right)  $ holds for every stationary
$S\subseteq\omega_{1}$ then there is a \textsf{MAD} family that can not be
extended to a Hurewicz ideal.
\end{proposition}

\qquad\ \ \qquad\ \ 

The question of Brendle is essentially equivalent to the following question:
\textit{Assuming }$\mathsf{CH},$ \textit{can every }\textsf{MAD }%
\textit{family be destroyed by a proper forcing that does not add dominating
or unsplit reals? }Since every Shelah-Stepr\={a}ns \textsf{MAD} family has
this property, we conjecture that this must be the case for every \textsf{MAD}
family, however as the result of Raghavan shows, this forcing can not be the
Mathias forcing of an ideal extending the respective \textsf{MAD} family.

\begin{problem}
[Laflamme \cite{ZappingSmallFilters}]Is there a Laflamme \textsf{MAD} family?
\end{problem}

\qquad\ \ \ 

We already saw that the answer is positive under $\mathfrak{d=c}$ (see also
\cite{KatetovandKatetovBlassOrdersFsigmaIdeals}). In an unpublished work,
Raghavan has found more conditions that imply the existence of such families.
Regarding the indestructibility, we have the following:

\begin{problem}
[Hru\v{s}\'{a}k \cite{MADFamiliesandtheRationals}]Is there a Sacks
indestructible \textsf{MAD} family?
\end{problem}

\qquad\ \ \ 

The answer is positive under $\mathfrak{b=c},$ \textsf{cov}$\left(
\mathcal{M}\right)  $ $\mathfrak{=c}$ or $\Diamond\left(  \mathfrak{b}\right)
$ (see \cite{ForcingIndestructibilityofMADFamilies}). The following variant of
the previous question is also open:

\begin{problem}
[Hru\v{s}\'{a}k \cite{MADFamiliesandtheRationals}]Is there a Sacks
indestructible \textsf{MAD} family of size $\mathfrak{c}$ in the Sacks model?
\end{problem}

\qquad\ \ 

The similar problem for Cohen forcing is also open:

\begin{problem}
[Stepr\={a}ns \cite{CombinatorialConsequencesofAddingCohenReals}]Is there a
Cohen indestructible \textsf{MAD} family?
\end{problem}

\qquad\ \ 

The answer is positive under $\mathfrak{b=c}$ or $\Diamond\left(
\mathfrak{b}\right)  $. Since Cohen indestructibility implies Sacks
indestructibility, a positive answer to the previous question will give a
positive answer to the problem of Hru\v{s}\'{a}k. Recall that the existence of
a Cohen indestructible \textsf{MAD} family is equivalent to the existence of a
tight \textsf{MAD} family. The existence of weakly tight \textsf{MAD} families
is also open:

\begin{problem}
[Garcia Ferreira, Hru\v{s}\'{a}k \cite{OrderingMADFamiliesalaKatetov}]Is there
a weakly tight \textsf{MAD} family?
\end{problem}

\qquad\ \ \ 

The answer is positive under $\mathfrak{s\leq b}$ (see
\cite{OnWeaklytightFamilies}). In an unpublished work, Raghavan has found more
conditions that imply the existence of such families. The notion of raving
\textsf{MAD} families is one of the strongest notions considered so far in the
literature (and in this case, we know that they consistently do not exist).
However, we do not know the answer to the following problem:

\begin{problem}
Is it consistent to have a raving \textsf{MAD} family of size bigger than
$\omega_{1}$?
\end{problem}

\qquad\ \ \ 

A \textsf{MAD} family $\mathcal{A}$ is called \emph{Canjar }if $\mathbb{M}%
\left(  \mathcal{I}\left(  \mathcal{A}\right)  \right)  $ does not add a
dominating real (note that any Hurewicz \textsf{MAD} family is Canjar). It is
easy to construct a non Canjar \textsf{MAD} family, however, the following is
still open:

\begin{problem}
Is there a Canjar \textsf{MAD} family?
\end{problem}

\qquad\ \ \ 

It is known that the answer is positive in case $\mathfrak{d=c}$ (see
\cite{MathiasForcingandCombinatorialCoveringPropertiesofFilters}). Regarding
the Kat\v{e}tov order on \textsf{MAD} families, we have the following:

\begin{problem}
[Garcia Ferreira, Hru\v{s}\'{a}k \cite{OrderingMADFamiliesalaKatetov}]Is there
a Kat\v{e}tov-top \textsf{MAD} family? (i.e. a \textsf{MAD} family Kat\v{e}tov
above any other \textsf{MAD} family)
\end{problem}

\qquad\ \ \ \ 

It was shown by Garc\'{\i}a Ferreira and Hru\v{s}\'{a}k that the answer is
negative under $\mathfrak{b=c}$ (see \cite{OrderingMADFamiliesalaKatetov}). In
an unpublished work, the author proved that the answer is also negative under
$\mathfrak{s\leq b}$ and under some strenghtening of the principle
$\Diamond\left(  \mathfrak{b}\right)  .$ Most likely, the previous question
has a negative answer in $\mathsf{ZFC.}$ We can then wonder about
Kat\v{e}tov-maximality instead of Kat\v{e}tov-top:

\begin{problem}
[Garcia Ferreira, Hru\v{s}\'{a}k \cite{OrderingMADFamiliesalaKatetov}]Is there
a Kat\v{e}tov maximal \textsf{MAD} family?
\end{problem}

\qquad\qquad\ \ 

In \cite{InvariancePropertiesofAlmostDisjointFamilies} the authors proved that
the answer is positive under $\mathfrak{p=c}$ (in contrast with the previous
question). In an unpublished work, the author proved the following results:

\begin{proposition}
\qquad\ \ 

\begin{enumerate}
\item There is a Kat\v{e}tov maximal \textsf{MAD} family under
$\mathfrak{b=c.}$

\item $\Diamond\left(  \mathfrak{d}\right)  $ implies that there is a
Kat\v{e}tov maximal \textsf{MAD} family of size $\omega_{1}.$

\item There is no Kat\v{e}tov maximal \textsf{MAD} family of size $\omega_{1}$
in the Cohen model.
\end{enumerate}
\end{proposition}

\qquad\ \ 

The relationship between $\mathfrak{a}$ and some other cardinal invariants is
still unclear, we will mention more examples. The cardinal invariant
$\mathfrak{a}_{T}$ is defined as the least size of a maximal AD family
consisting of finitely branching subtrees of of $\omega^{\omega}$. It is known
that $\mathfrak{d\leq a}_{T}$ and that $\mathfrak{d<a}_{T}$ is consistent (see
\cite{SelectivityofAlmostDisjointFamilies} and \cite{PartitionNumbers}). The
following question is still unknown:

\begin{problem}
[Hru\v{s}\'{a}k]Is the inequality $\mathfrak{a}_{T}<\mathfrak{a}$ consistent?
\end{problem}

\qquad\ \ \ 

Unfortunately, the template framework does not seem to help in this situation.
This is also the case for the following question of Jerry Vaughan:

\begin{problem}
[Vaughan \cite{SmallUncountableCardinalsandTopology}]Is the inequality
$\mathfrak{i}<\mathfrak{a}$ consistent?
\end{problem}

\qquad\ \ \ 

By forcing along a template (with the aid of a measurable cardinal) it can be
shown that $\mathfrak{u<a}$ is consistent (see
\cite{MADFamiliesandUltrafilters}). However, the following question are still open:

\begin{problem}
[Shelah \cite{understand}]\qquad\ \ 

\begin{enumerate}
\item Does $\mathfrak{u=}$ $\omega_{1}$ imply $\mathfrak{a=}$ $\omega_{1}$?

\item Is the measurable cardinal necessary for the consistency of
$\mathfrak{u<a}$?
\end{enumerate}
\end{problem}

\qquad\ \ \ 

The cardinal invariant $\mathfrak{hm}$ is one of the largest Borel invariants
considered so far. However, the following is still unknown:

\begin{problem}
[Weinert]Is the inequality $\mathfrak{hm}<\mathfrak{a}$ consistent?
\end{problem}

\qquad\ \ \ 

In \cite{ShelahPartyForcing} Shelah constructed a model of $\mathfrak{i<u}.$
In an unpublished work, the author showed that $\mathfrak{a=hm<u}$ holds in
that model.

\qquad\ \ \qquad\ \ \ 

The cardinal invariant $\mathfrak{a}_{closed}$ is defined as the smallest
number of closed sets such that its union is a \textsf{MAD} family. Clearly it
is below $\mathfrak{a}$ and it is uncountable by a result of Mathias. Unlike
the almost disjointness number, $\mathfrak{a}_{closed}$ is known to be
incomparable to $\mathfrak{b}$ (see \cite{BrendleKhomskii} and
\cite{BrendleRaghavan}). It is known that $\mathfrak{p\leq a}_{closed}$, but
the following is still unknown:

\begin{problem}
[Raghavan]Is the inequality $\mathfrak{a}_{closed}<\mathfrak{h}$ consistent?
\end{problem}

\qquad\ \ \ \ 

This is really a question about computing $\mathfrak{a}_{closed}$ in the
Mathias model.

\qquad\ \ \ \ 

As mentioned before, by a result of Shelah, it is known that the inequality
$\mathfrak{b<a}$ is consistent. However, the following is still open:

\begin{problem}
[Brendle
\cite{CardinalInvariantsoftheContinuumandCombinatoricsonUncountableCardinals}%
]Does $\mathfrak{b=}$ $\omega_{1}$ and $\clubsuit$ imply $\mathfrak{a}%
=\omega_{1}$?
\end{problem}

\qquad\ \ \ 

At last but not least, we would like to mention the problem of Erd\"{o}s and Shelah:

\begin{problem}
[Erd\"{o}s, Shelah \cite{ErdosShelah}]Is there a completely separable
\textsf{MAD} family?
\end{problem}

\qquad\ \ \ 

The answer is positive (see \cite{SANEplayer}) under $\mathfrak{s\leq a}$ and
under $\mathfrak{a<s}$ plus some $\mathsf{PCF}$ hypothesis (it is unknown if
this hypothesis can even fail).

\bibliographystyle{plain}
\bibliography{Tesis}
\qquad\ \ \qquad\ \ 
\end{document}